\documentclass[a4paper,10pt,reqno]{amsart}

\usepackage{amssymb}
\usepackage{graphicx}
\usepackage[margin=1in]{geometry}  

\usepackage{color}
\definecolor{Chocolat}{rgb}{0, 0.15, 0.05}
\definecolor{BleuTresFonce}{rgb}{0, 0.15, 0.05}
\usepackage[utf8]{inputenc}
\usepackage[colorlinks,hyperindex]{hyperref}
\hypersetup{linkcolor=Chocolat,urlcolor=Chocolat,citecolor=Chocolat}

 \newcommand{\ac}{{\scriptstyle \text{\rm !`}}}
 
\theoremstyle{plain}

\newtheorem{thm}{Theorem}
\newtheorem{lemma}{Lemma}

\newtheorem*{claim*}{Claim}
\newtheorem*{thm*}{Theorem}

\theoremstyle{definition}
\newtheorem{definition}{Definition}
\newtheorem{convention}{Convention}

\newtheorem{remark}{Remark}

\newtheorem{example}{\sc Example}

\newcommand{\cal}[1]{\mathcal{#1}}              

 \usepackage{multicol}

\usepackage{float}
\usepackage{amssymb}
\usepackage{stmaryrd}
\usepackage{mathalfa}
\usepackage{euscript}
\usepackage{eufrak}
\usepackage{extarrows}
\usepackage[normalem]{ulem}
\usepackage{amsmath}
\usepackage{amsthm}
 
\numberwithin{equation}{section}

\allowdisplaybreaks
\usepackage{lscape}
\usepackage{tabularx}
\usepackage[dvipsnames]{xcolor}
\usepackage{framed}
\usepackage[title]{appendix}
\usepackage{tikz}
\usetikzlibrary{decorations.markings}
\usetikzlibrary{shapes.geometric,positioning}
\usepackage{booktabs}
\usepackage{euscript}
\usepackage[normalem]{ulem}
\usepackage{array}
\usepackage{collectbox}
\usetikzlibrary{arrows,matrix}

\usepackage{footnote}
\usepackage{makecell}
\pgfdeclarelayer{edgelayer}
\pgfdeclarelayer{nodelayer}
\pgfsetlayers{edgelayer,nodelayer,main}
\usepackage{array}
\usepackage{amssymb}
\usetikzlibrary{decorations.pathmorphing, patterns,shapes}

\usepackage{multirow}

 \usetikzlibrary{calc}
\usepackage{accents}
\definecolor{bluen}{RGB}{0,100,200}

\makeatletter
\newcommand\xqed[1]{%
  \leavevmode\unskip\penalty9999 \hbox{}\nobreak\hfill
  \quad\hbox{#1}}
\newcommand\demo{\xqed{$\triangle$}}
\pgfarrowsdeclare{open cap}{open cap}
{\pgfarrowsleftextend{+0pt}\pgfarrowsrightextend{+0.5\pgflinewidth}}
{
  \pgfmathsetlength{\pgfutil@tempdimb}{.5*\pgflinewidth-.5*\pgfinnerlinewidth}%
  \pgfsetlinewidth{\pgfutil@tempdimb}
  \pgfsetbuttcap
  \pgfsetdash{}{0pt}
  \pgfmathsetlength{\pgfutil@tempdima}{.5*\pgfutil@tempdimb+.5*\pgfinnerlinewidth}%
  \pgfpathmoveto{\pgfqpoint{0pt}{\pgfutil@tempdima}}
  \pgfpatharc{270}{90}{-\pgfutil@tempdima}
  \pgfusepathqstroke
}
\makeatother

\tikzset{snake it/.style={decorate, decoration=snake}}
\usetikzlibrary{decorations.pathreplacing}
\newcommand\mydots{\hbox to 1em{.\hss.\hss.}}
\subjclass[2010]{Primary 18D50; Secondary 52B20, 18G55.}


\thanks{This work was supported  by the Praemium Academiae of M. Markl and RVO:67985840.}

\title[Combinatorial homotopy theory for   operads]{Combinatorial homotopy theory for   operads}

\author{Jovana Obradovi\' c}
\address{Institute of Mathematics CAS, 
\v Zitn\' a 25,
115 67   Prague,
Czech Republic}
\email{obradovic@math.cas.cz}

\date{\today}
 
\begin{document}

\begin{abstract}
We introduce an explicit combinatorial characterization of the minimal model ${\EuScript O}_{\infty}$ of the coloured  operad   ${\EuScript O}$ encoding non-symmetric  operads. In our description of ${\EuScript O}_{\infty}$, the spaces of operations are defined in terms of hypergraph polytopes and the composition structure  generalizes the one of  the $A_{\infty}$-operad.  As further generalizations of this construction,  we present a combinatorial description of the $W$-construction applied on ${\EuScript O}$, as well as of the minimal model of  the coloured  operad ${\EuScript C}$ encoding non-symmetric cyclic operads.  
\end{abstract}

\maketitle

\thispagestyle{empty}

\tableofcontents

\section*{Introduction}

\noindent Sullivan's classical construction of minimal models of rational homotopy theory
  has been made available to operad theory by Markl, in his  paper \cite{m2}, together with the subsequent papers of Hinich \cite{Hinich}, Spitzweck \cite{Spitzweck}, Vogt \cite{Vogt}, Berger-Moerdijk \cite{BM0}, Cisinski-Moerdijk \cite{CM} and Robertson \cite{Marcy},  in which  model structures of various categories of  operads have been investigated. In  \cite[Theorem 3.1]{m2}, Markl  introduced the notion of a minimal model  of a  monochrome dg operad and he proved that any such operad ${\EuScript P}$, with ${\EuScript P}(0)=0$ and ${\EuScript P}(1)={\Bbbk}$, with ${\Bbbk}$ being a field of characteristic zero, admits a
minimal model, which is unique up to isomorphism. In \cite[Definition 2]{m1},  Markl  generalized the notion of a minimal model to  coloured dg operads. We recall his definition below.

\smallskip

\begin{definition}\label{minimalmodel}
Let ${\EuScript P}=({\EuScript P},d_{\EuScript P})$ be a $C$-coloured dg operad. A minimal model of ${\EuScript P}$ is a $C$-coloured dg operad $\mathfrak{M}_{\EuScript P}=({\cal T}_{C}(E),d_{\mathfrak M})$, where ${\cal T}_C(E)$ is the free $C$-coloured operad on a $C$-coloured collection $E$, together with a map $\alpha_{\EuScript P}:{\mathfrak M}_{\EuScript P}\rightarrow {\EuScript P}$ of dg coloured operads, such that  $\alpha_{\EuScript P}:{\mathfrak M}_{\EuScript P}\rightarrow {\EuScript P}$  is a quasi-isomorphism,
 and
 $d_{\mathfrak M}(E)$ consists of decomposable elements of   ${\cal T}_C(E)$, i.e.  $d_{\mathfrak M}(E)\subseteq {\cal T}_C(E)^{(\geq 2)}$, where ${\cal T}_C(E)^{(\geq 2)}\subseteq {\cal T}_C(E)$ is determined by  trees with at least two vertices. 
\end{definition}

\smallskip

For  Koszul operads, Markl's notion of minimal model coincides with the cobar
construction on the Koszul  dual of an operad,   given by  Ginzburg-Kapranov  \cite{GK94} and  Getzler-Jones \cite{GJ94}, and, in particular, provides us with the structure  encoding higher operations of most classical strongly homotopy algebras, such as $A_{\infty}$-, $L_{\infty}$- and $C_{\infty}$-algebras. A detailed description of these algebras can be found in \cite[Section 3.10]{mss}. 
  In recent applications of   homotopy theory of algebras over operads, especially in theoretical physics, an explicit description of the structure maps of
minimal models remains essential; see \cite{jurco} for an  up to date review on how higher homotopy structures naturally govern field theories.  Such a description is often obtained by a direct calculation of a particular model, which tends to be a rather involved task and calls for new methods and  conceptual approaches for understanding   the homotopy properties of algebraic structures.

\medskip 

In this paper, we introduce an explicit combinatorial characterization of the minimal model of the coloured  operad ${\EuScript O}$ encoding non-symmetric  operads, introduced by Van der Laan in his work \cite{VdL} on extending Koszul duality theory of Ginzburg-Kapranov  and  Getzler–Jones  to coloured operads. The  novelty of our characterization is its   interpretation in terms of  hypergraph polytopes,  introduced  by Do\v sen and Petri\' c in \cite{DP-HP}  and further developed by Curien, Ivanovi\' c and the author in \cite{CIO},   whose hypergraphs arise in a certain way from rooted trees -- we refer to them as operadic polytopes. In particular, each operadic polytope is a truncated simplex displaying the homotopy replacing   the ``pre-Lie'' relations for the  partial composition operations pertinent to the corresponding rooted tree. In this way, our operad  structure generalizes the structure of  Stasheff's topological $A_{\infty}$-operad \cite{Sta1}: the family of (combinatorial) associahedra corresponds to the suboperad determined by linear rooted trees.  We then introduce a combinatorial description of the cubical subdivision of  operadic polytopes, obtaining in this way   the   Boardman-Vogt-Berger-Moerdijk resolution of  ${\EuScript O}$, i.e. its $W$-construction, introduced   in \cite{BM2}. Finally, by modifying the underlying formalism of trees, we obtain   the minimal model of  the coloured  operad ${\EuScript C}$ encoding non-symmetric cyclic operads, whose algebras yield a notion of strongly homotopy cyclic operads for which  the   relations for the partial composition operations are  coherently relaxed up to homotopy, while   the relations involving the action of cyclic permutations are kept strict.

\medskip

We hope that our explicit construction of operadic polytopes, together with the fact that they admit the  structure  of a strict infinity operad, will be of interest in the context of recent developments around Koszulity in operadic categories of Batanin and Markl \cite{BatMar}. From a different, but closely related point of view,    we believe that it provides a valuable addition to Ward's recent work \cite{Ward}, proving that the operad encoding modular operads is Koszul and indicating that such a proof can be given in terms of cellular chains on a family of polytopes that generalizes graph associahedra.

\medskip\smallskip

\noindent{\em Acknowledgements.}  I wish to express my
gratitude to  M. Livernet, M. Markl, F. Wierstra,  and R. Kaufmann for many useful discussions. I am especially indebted to  P.-L. Curien and B. Vallette for detailed comments that greatly improved the final version of this paper. I gratefully acknowledge the financial support of the Praemium Academiae of M. Markl and RVO:67985840.
\subsection*{Notation and conventions} 

\subsubsection*{Operads} We work with ${\mathbb N}$-coloured  reduced operads in the symmetric monoidal category ${\mathsf{dgVect}}$ of dg vector
spaces over a field ${\Bbbk}$ of characteristic 0.    In ${\mathsf{dgVect}}$, the monoidal structure is given by the classical tensor product $\otimes$, and the  switching map $\tau: V\otimes W\rightarrow W\otimes V$ is defined by $\tau(v\otimes w):=(-1)^{|v||w|} w\otimes v$, where $v$ and $w$ are homogeneous elements of degrees $|v|$ and $|w|$, respectively.  We   use the classical Koszul sign convention. We work with homological grading; the differential is a map of degree $-1$. We denote with ${\cal T}_{C}(K)$ the free   $C$-coloured (symmetric) operad on a   $C$-coloured (symmetric) collection $K$. A detailed construction of ${\cal T}_{C}(K)$ is given in \cite[Section 3]{BM2}. Our main references for the general theory of operads and related notions are \cite{mss} and \cite{LV}.

\subsubsection*{Ordinals} We denote by $[n]$ the set $\{1,\dots,n\}$, and by $\Sigma_n$ the symmetric group on $[n]$.

\subsubsection*{Trees} A rooted  tree ${\cal T}$ is a finite connected contractible graph on a non-empty vertex set, together with a distinguished external edge $r({\cal T})$, called the root  of ${\cal T}$.  We shall denote by ${\it v}({\cal T})$,  ${\it e}({\cal T})$ and $l({\cal T})$ the sets of vertices, (internal) edges and external edges (or leaves) of ${\cal T}$, respectively.  We shall write $E({\cal T})$ for the union $e({\cal T})\cup l({\cal T})$ of all the  edges of ${\cal T}$. We shall write $i({\cal T})$ for the set ${\it l}({\cal T})\backslash \{{r}({\cal T})\}$, and we shall refer to the elements of $i({\cal T})$ as the inputs (or the input leaves) of ${\cal T}$. The set of inputs $i(v)$ and the root $r(v)$ of a vertex $v\in v({\cal T})$ are defined in the standard way through the source and  target maps obtained by reading ${\cal T}$ from the input leaves to the root. The notation  for all these various sets defining a rooted tree will often also be used for their respective cardinalities. We shall denote by $\rho({\cal T})$  the unique vertex of ${\cal T}$ whose root is $r({\cal T})$, and we shall refer to it as the root vertex of ${\cal T}$. Throughout   the paper, {\em edge} will always mean an {\em internal edge}.

\medskip

\indent A rooted tree  ${\cal T}$ is called   planar if each vertex of ${\cal T}$ comes equipped with an ordering of its inputs. In this case, the inputs of   ${\cal T}$ admit a canonical labeling  from left to right, relative to the planar embedding of ${\cal T}$, by $1$ through $n$, for $n=|{\it i}({\cal T})|$.  Planar  rooted  trees   are  isomorphic if there exists an isomorphism of the correspondig graphs that  preserves the root and the planar structure. We   denote by ${\tt Tree}(n)$
the set of  planar rooted trees with $n$ inputs.

\medskip

There are two principal constructions on planar rooted trees: grafting and substitution. For  trees ${\cal T}_1\in {\tt Tree}(n)$ and ${\cal T}_2\in {\tt Tree}(m)$ and an index $1\leq i\leq n$,  the   grafting of ${\cal T}_2$  to ${\cal T}_1$   along the  input $i$ is the tree  ${\cal T}_1\circ_i {\cal T}_2$, obtained by identifying the root of ${\cal T}_2$ with the $i$-th input of ${\cal T}_1$. If $v\in V({\cal T}_1)$ is such that $|{\it i}(v)|=m$, the substitution of     the vertex $v$ of ${\cal T}_1$ by ${\cal T}_2$ is the tree  ${\cal T}_1\bullet_v {\cal T}_2$, obtained by replacing the vertex $v$ by the tree ${\cal T}_2$, identifying the $m$ inputs of $v$ with the $m$ inputs of ${\cal T}_2$, using the respective planar structures.  The trees ${\cal T}_1\circ_i {\cal T}_2$ and ${\cal T}_1\bullet_v {\cal T}_2$ can be rigorously defined either in terms of disjoint unions of sets
of vertices, edges and leaves of ${\cal T}_1$ and ${\cal T}_2$, or by preassuming the appropriate disjointness of sets and
taking the ordinary union instead; we   take the latter convention. Moreover, we shall assume
that all the edges and leaves that need to be identified in these two constructions are a priori the same.

\medskip
  
 A  corolla is a rooted tree  with only one vertex.  Each planar rooted tree ${\cal T}$ is either a planar corolla, or there exist planar rooted trees ${\cal T}_1,\dots,{\cal T}_p$, a corolla  $t_n$ with $n$ inputs, for $n\geq p$,  and a monotone injection $(p,n):[p]\rightarrow [n]$,  such that ${\cal T}$ is obtained by identifying the roots of ${\cal T}_i$'s with the inputs $(p,n)(i)$ of     $t_n$. In the latter case, we write ${\cal T}=t_n({\cal T}_1,\dots,{\cal T}_p)$, implicitly bookkeeping the data of the correspondence $(p,n)$. Note that this recursive definition allows one for an inductive reasoning.

\medskip

A   subtree of a planar rooted tree ${\cal T}$ is a connected subgraph ${\cal S}$ of ${\cal T}$ which is itself a planar rooted tree, such that ${r}({\cal S})={r}(v)$, for some $v\in v({\cal T})$, and such that, if a vertex $v$ of ${\cal T}$ is present in ${\cal S}$, then all  the inputs and the root of $v$ in ${\cal T}$  must also be present in ${\cal S}$; it is assumed that the source and the target maps of ${\cal S}$  are the appropriate restrictions of the ones of ${\cal T}$, and that the planar structure of ${\cal S}$ is inherited from ${\cal T}$. In this way, each  subtree of  ${\cal T}$  is completely determined by a subset of vertices of ${\cal T}$, and  therefore  also by a subset of internal edges of ${\cal T}$ (by taking all the  vertices adjacent with those edges). We can, therefore, speak about the   subtree ${\cal T}(X)$    of ${\cal T}$  determined by a subset  $X$ of vertices  (resp. of edges) of ${\cal T}$.    For an edge $e\in e({\cal T})\cup \{r({\cal T})\}$, the subtree of ${\cal T}$   rooted at $e$  is the subtree of ${\cal T}$     determined by all the vertices of ${\cal T}$ that are descendants of the vertex $v$ whose root is $e$, including $v$ itself. 

\medskip

 In this paper, we shall work with three different kinds of rooted trees. In order to help the reader navigate between them, in the following table we briefly summarize their characterizations and the corresponding notational conventions. 
 
 \medskip
 
\begin{center}
  \begin{tabular}{ccccc}  
    \toprule
   {\small\textsf{Composite trees}} &\enspace&  {\small\textsf{Operadic trees}}  &\enspace&   {\small\textsf{Constructs}} \\
    \midrule 
  \multicolumn{1}{m{4.5cm}}{\footnotesize{ planar rooted trees with ${\mathbb N}$-coloured edges, whose vertices encode the operadic $\circ_i$ operations}}    &&    \multicolumn{1}{m{4.5cm}}{\footnotesize{planar rooted trees with monochrome edges and totally ordered vertex sets}}   &&  \multicolumn{1}{m{4.5cm}}{\footnotesize{non-planar trees labeling the faces of hypergraph polytopes}}\\ [0.5cm]
 \raisebox{3em}{\small $T=$} {\resizebox{!}{2.4cm}{\begin{tikzpicture}  
\node(1) [rectangle,draw=none,minimum size=0mm,inner sep=0cm] at (-0.3,0.6) {\tiny $3$};
\node(2) [rectangle,draw=none,minimum size=0mm,inner sep=0cm] at (0.3,0.6) {\tiny $1$};
\node(3) [rectangle,draw=none,minimum size=0mm,inner sep=0cm] at (0.9,0) {\tiny $2$};
\node(n) [rectangle,draw=none,minimum size=0mm,inner sep=0cm] at (-0.39,0.33) {\tiny $5$};
\node(k) [rectangle,draw=none,minimum size=0mm,inner sep=0cm] at (0.39,0.33) {\tiny $2$};
\node(k) [rectangle,draw=none,minimum size=0mm,inner sep=0cm] at (0.975,-0.3) {\tiny $2$};
\node(k1) [rectangle,draw=none,minimum size=0mm,inner sep=0cm] at (0.1,-0.55) {\tiny $6$};
\node(k1) [rectangle,draw=none,minimum size=0mm,inner sep=0cm] at (0.55,-1.4) {\tiny $7$};
\node(j) [circle,draw=black,thick,minimum size=0mm,inner sep=0.05cm] at (0.45,-1) {\small $6$};
 \node(i) [circle,draw=black,thick,minimum size=0.2mm,inner sep=0.05cm] at (0,0) {\small $3$}; 
\draw[-,thick] (i)--(-0.3,0.45);\draw[-,thick] (i)--(0.3,0.45);
\draw[-,thick] (j)--(0.9,-0.15);
\draw[-,thick] (i)--(j)--(0.45,-1.55);\end{tikzpicture}}} &&  \raisebox{3em}{\small ${\cal T}=$} \begin{tikzpicture}
    \node (E)[circle,draw=black,minimum size=4mm,inner sep=0.1mm] at (0,0) {\small $3$};
    \node (F) [circle,draw=black,minimum size=4mm,inner sep=0.1mm] at (0,1) {\small $1$};
    \node (A) [circle,draw=black,minimum size=4mm,inner sep=0.1mm] at (1,1) {\small $2$};
 \draw[-] (-0.2,1.8) -- (F)--(0.2,1.8);   
 \draw[-] (0.8,1.8) -- (A)--(1.2,1.8);   
  \draw[-] (-0.4,0.8)--(E)--(0,-0.55); 
 \draw[-] (0.4,0.8)--(E)--(-1,0.8); 
    \draw[-] (E)--(F) node {};
    \draw[-] (E)--(A) node  {};
   \end{tikzpicture} &&  \raisebox{3em}{\small $C=$} {\resizebox{2cm}{!}{\begin{tikzpicture}
\node (b) [circle,fill=cyan,draw=black,minimum size=0.1cm,inner sep=0.2mm,label={[yshift=-0.45cm]{\footnotesize $x$}}] at (4,-0.2) {};
\node (c) [circle,fill=cyan,draw=black,minimum size=0.1cm,inner sep=0.2mm,label={[xshift=-0.2cm,yshift=-0.3cm]{\footnotesize $y$}}] at (3.8,0.15) {};
 \node (d) [circle,fill=cyan,draw=black,minimum size=0.1cm,inner sep=0.2mm,label={[xshift=0.45cm,yshift=-0.31cm]{\footnotesize $\{u,v\}$}}] at (4.2,0.15) {};
  \node (2d) [circle,fill=none,draw=none,minimum size=0.1cm,inner sep=0.2mm] at (4.2,-1.1) {};
 \draw[thick]  (c)--(b)--(d);\end{tikzpicture}}}     \\ 

\midrule
  \end{tabular}
\end{center}
 
 \medskip
 A planar unrooted tree is  a finite connected contractible graph on a non-empty vertex set, each of whose vertices comes equipped with  a cyclic ordering of all the adjacent edges. Planar unrooted trees are isomorphic if there exists an isomorphism of the corresponding graphs that   preserves the  
cyclic orderings of the  sets of edges adjacent to vertices.  By forgetting the data of the root of a planar rooted tree ${\cal T}$, one canonically obtains a planar unrooted tree that we shall denote by ${\cal U}({\cal T})$.

\section{Hypergraph polytopes} 
\noindent A hypergraph polytope is a  polytope that may be characterized as a truncated simplex, whereby the  truncations are only performed on the faces of the original simplex and not on the faces already  obtained as a result of a truncation. In particular, in each dimension, the family of hypergraph polytopes consists of an interval of simple polytopes starting with a simplex and ending with a permutohedron. As an illustration, here is a sequence of truncations of the 3-dimensional simplex that leads to a polytope called {\em hemiassociahedron}:

\medskip

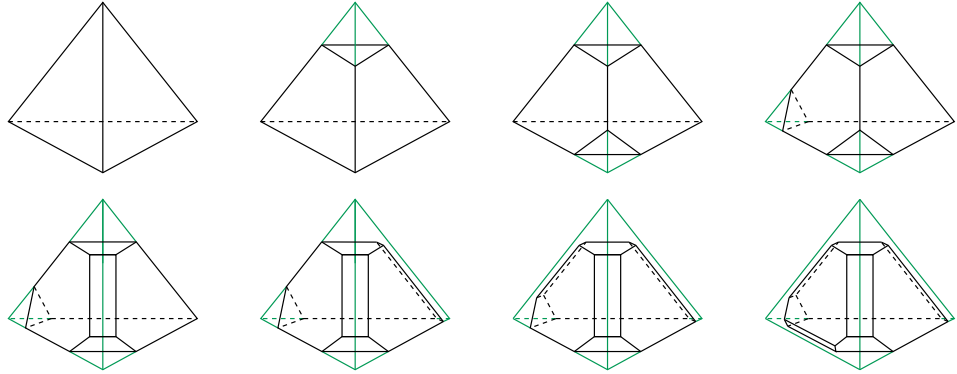
\begin{center}
\resizebox{2.5cm}{!}{\begin{tikzpicture}[thick,scale=2]
\coordinate (A1) at (0,2);
\coordinate (A2) at (0,0); 
\coordinate (A3) at (-1.1,0.6);
\coordinate (A4) at (1.1,0.6);
 \draw[draw=black,fill=none] (A1) -- (A3) -- (A2) -- (A4) -- cycle;
\draw[dashed]  (A3) -- (A4);
\draw (A1) -- (A2);
\end{tikzpicture}} \quad\quad \resizebox{2.5cm}{!}{\begin{tikzpicture}[thick,scale=2]
\coordinate (A1) at (0,2);
\coordinate (A11) at (-0.39,1.5);
\coordinate (A12) at (0.39,1.5);
\coordinate (A13) at (0,1.25);
\coordinate (A2) at (0,0); 
\coordinate (A3) at (-1.1,0.6);
\coordinate (A4) at (1.1,0.6);
 \draw[draw=black,fill=none] (A11) -- (A3) -- (A2) -- (A4) -- (A12) --cycle;
\draw[draw=ForestGreen] (A11)--(A1)--(A12);
\draw[draw=ForestGreen] (A13)--(A1);
\draw[dashed]  (A3) -- (A4);
\draw (A13) -- (A2);
\draw (A11) -- (A13) -- (A12);
\end{tikzpicture}} \quad\quad \resizebox{2.5cm}{!}{\begin{tikzpicture}[thick,scale=2]
\coordinate (A1) at (0,2);
\coordinate (A11) at (-0.39,1.5);
\coordinate (A12) at (0.39,1.5);
\coordinate (A13) at (0,1.25);
\coordinate (A2) at (0,0); 
\coordinate (A21) at (-0.387,0.213); 
\coordinate (A22) at (0.387,0.213); 
\coordinate (A23) at (0,0.5); 
\coordinate (A3) at (-1.1,0.6);
\coordinate (A4) at (1.1,0.6);
\draw[draw=ForestGreen] (A11)--(A1)--(A12);
\draw[draw=ForestGreen] (A21)--(A2)--(A22);
\draw[draw=ForestGreen] (A2) -- (A23);
\draw[draw=ForestGreen] (A13)--(A1);
 \draw[draw=black,fill=none] (A11) -- (A3) --(A21)-- (A22) -- (A4) -- (A12) --cycle;
\draw  (A21)--(A23)--(A22);
\draw[dashed]  (A3) -- (A4);
\draw (A13) -- (A23);
\draw (A11) -- (A13) -- (A12);
\end{tikzpicture}} \quad\quad \resizebox{2.5cm}{!}{\begin{tikzpicture}[thick,scale=2]
\coordinate (A1) at (0,2);
\coordinate (A11) at (-0.39,1.5);
\coordinate (A12) at (0.39,1.5);
\coordinate (A13) at (0,1.25);
\coordinate (A2) at (0,0); 
\coordinate (A21) at (-0.387,0.213); 
\coordinate (A22) at (0.387,0.213); 
\coordinate (A23) at (0,0.5); 
\coordinate (A3) at (-1.1,0.6);
\coordinate (A31) at (-0.9,0.49);
\coordinate (A32) at (-0.8,0.976);
\coordinate (A33) at (-0.6,0.6);
\coordinate (A4) at (1.1,0.6);
\draw[draw=ForestGreen] (A32)--(A3)--(A31);
\draw[draw=ForestGreen] (A11)--(A1)--(A12);
\draw[draw=ForestGreen] (A21)--(A2)--(A22);
\draw[draw=ForestGreen] (A2) -- (A23);
\draw[draw=ForestGreen] (A13)--(A1);
\draw[draw=ForestGreen,dashed]  (A33) -- (A3);
 \draw[draw=black,fill=none] (A11) -- (A32)--(A31) --(A21)-- (A22) -- (A4) -- (A12) --cycle;
\draw  (A21)--(A23)--(A22);
\draw[dashed]  (A33) -- (A4);
\draw (A13) -- (A23);
\draw[dashed] (A32) -- (A33) -- (A31);
\draw (A11) -- (A13) -- (A12);  
\end{tikzpicture}}\\[0.3cm]  \resizebox{2.5cm}{!}{\begin{tikzpicture}[thick,scale=2]
\coordinate (A1) at (0,2);
\coordinate (A11) at (-0.39,1.5);
\coordinate (A12) at (0.39,1.5);
\coordinate (A13) at (0,1.25);
\coordinate (A131) at (-0.153,1.35); 
\coordinate (A132) at (0.153,1.35); 
\coordinate (A2) at (0,0); 
\coordinate (A21) at (-0.387,0.213); 
\coordinate (A22) at (0.387,0.213); 
\coordinate (A23) at (0,0.5); 
\coordinate (A231) at (-0.153,0.388); 
\coordinate (A232) at (0.153,0.388); 
\coordinate (A3) at (-1.1,0.6);
\coordinate (A31) at (-0.9,0.49);
\coordinate (A32) at (-0.8,0.976);
\coordinate (A33) at (-0.6,0.6);
\coordinate (A4) at (1.1,0.6);
\draw[draw=ForestGreen] (A32)--(A3)--(A31);
\draw[draw=ForestGreen] (A11)--(A1)--(A12);
\draw[draw=ForestGreen] (A21)--(A2)--(A22);
\draw[draw=ForestGreen] (A2) -- (A1);
\draw[draw=ForestGreen] (A13)--(A1);
\draw[draw=ForestGreen,dashed]  (A33) -- (A3);
 \draw[draw=black,fill=none] (A11) -- (A32)--(A31) --(A21)-- (A22) -- (A4) -- (A12) --cycle;
\draw  (A21)--(A231)--(A232)--(A22);
\draw[dashed]  (A33) -- (A4);
\draw (A231) -- (A131) -- (A132) -- (A232) -- cycle;
\draw[dashed] (A32) -- (A33) -- (A31);
\draw (A11) -- (A131) --(A132) -- (A12);
\end{tikzpicture}}  \quad\quad \resizebox{2.5cm}{!}{\begin{tikzpicture}[thick,scale=2]
\coordinate (A1) at (0,2);
\coordinate (A11) at (-0.39,1.5);
\coordinate (A12) at (0.39,1.5);
\coordinate (A121) at (0.25,1.498);
\coordinate (A122) at (0.33,1.46);
\coordinate (A13) at (0,1.25);
\coordinate (A131) at (-0.153,1.35); 
\coordinate (A132) at (0.153,1.35); 
\coordinate (A2) at (0,0); 
\coordinate (A21) at (-0.387,0.213); 
\coordinate (A22) at (0.387,0.213); 
\coordinate (A23) at (0,0.5); 
\coordinate (A231) at (-0.153,0.388); 
\coordinate (A232) at (0.153,0.388); 
\coordinate (A3) at (-1.1,0.6);
\coordinate (A31) at (-0.9,0.49);
\coordinate (A32) at (-0.8,0.976);
\coordinate (A33) at (-0.6,0.6);
\coordinate (A4) at (1.1,0.6);
\coordinate (A41) at (1.027,0.565);
\coordinate (A42) at (0.95,0.6);
\draw[draw=ForestGreen] (A32)--(A3)--(A31);
\draw[draw=ForestGreen] (A11)--(A1)--(A12);
\draw[draw=ForestGreen] (A21)--(A2)--(A22);
\draw[draw=ForestGreen] (A2) -- (A1);
\draw[draw=ForestGreen] (A13)--(A1);
\draw[draw=ForestGreen,dashed]  (A33) -- (A3);
\draw[draw=ForestGreen,dashed]  (A42) -- (A4);
\draw[draw=ForestGreen] (A41)--(A4)--(A12);
 \draw[draw=black,fill=none] (A11) -- (A32)--(A31) --(A21)-- (A22) -- (A41);
\draw (A11)--(A121);
\draw  (A21)--(A231)--(A232)--(A22);
\draw[dashed]  (A33) -- (A42);
\draw (A231) -- (A131) -- (A132) -- (A232) -- cycle;
\draw[dashed] (A32) -- (A33) -- (A31);
\draw (A11) -- (A131) --(A132)--(A122);
\draw[dashed]  (A41) -- (A42) -- (A121); 
\draw (A41)--(A122) --(A121);
\end{tikzpicture}}  \quad\quad \resizebox{2.5cm}{!}{\begin{tikzpicture}[thick,scale=2]
\coordinate (A1) at (0,2);
\coordinate (A11) at (-0.39,1.5);
\coordinate (A111) at (-0.25,1.498);
\coordinate (A112) at (-0.33,1.46);
\coordinate (A12) at (0.39,1.5);
\coordinate (A121) at (0.25,1.498);
\coordinate (A122) at (0.33,1.46);
\coordinate (A13) at (0,1.25);
\coordinate (A131) at (-0.153,1.35); 
\coordinate (A132) at (0.153,1.35); 
\coordinate (A2) at (0,0); 
\coordinate (A21) at (-0.387,0.213); 
\coordinate (A22) at (0.387,0.213); 
\coordinate (A23) at (0,0.5); 
\coordinate (A231) at (-0.153,0.388); 
\coordinate (A232) at (0.153,0.388); 
\coordinate (A3) at (-1.1,0.6);
\coordinate (A31) at (-0.9,0.49);
\coordinate (A32) at (-0.8,0.976);
\coordinate (A321) at (-0.825,0.845);
\coordinate (A322) at (-0.745,0.875);
\coordinate (A33) at (-0.6,0.6);
\coordinate (A4) at (1.1,0.6);
\coordinate (A41) at (1.027,0.565);
\coordinate (A42) at (0.95,0.6);
\draw[draw=ForestGreen]  (A3)--(A31);
\draw[draw=ForestGreen]  (A1)--(A12);
\draw[draw=ForestGreen] (A21)--(A2)--(A22);
\draw[draw=ForestGreen] (A2) -- (A1);
\draw[draw=ForestGreen] (A3)--(A1);
\draw[draw=ForestGreen,dashed]  (A33) -- (A3);
\draw[draw=ForestGreen,dashed]  (A42) -- (A4);
\draw[draw=ForestGreen] (A41)--(A4)--(A12);
 \draw[draw=black,fill=none]   (A321)--(A31) --(A21)-- (A22) -- (A41);
\draw (A111)--(A121);
\draw (A111)--(A112);
\draw (A112)--(A321);
\draw[dashed] (A321)--(A322)--(A111);
\draw  (A21)--(A231)--(A232)--(A22);
\draw[dashed]  (A33) -- (A42);
\draw (A231) -- (A131) -- (A132) -- (A232) -- cycle;
\draw[dashed] (A322) -- (A33) -- (A31);
\draw (A112) -- (A131) --(A132)--(A122);
\draw[dashed]  (A41) -- (A42) -- (A121); 
\draw (A41)--(A122) --(A121);
\end{tikzpicture}}  \quad\quad \resizebox{2.5cm}{!}{\begin{tikzpicture}[thick,scale=2]
\coordinate (A1) at (0,2);
\coordinate (A11) at (-0.39,1.5);
\coordinate (A111) at (-0.25,1.498);
\coordinate (A112) at (-0.33,1.46);
\coordinate (A12) at (0.39,1.5);
\coordinate (A121) at (0.25,1.498);
\coordinate (A122) at (0.33,1.46);
\coordinate (A13) at (0,1.25);
\coordinate (A131) at (-0.153,1.35); 
\coordinate (A132) at (0.153,1.35); 
\coordinate (A2) at (0,0); 
\coordinate (A21) at (-0.387,0.213); 
\coordinate (A211) at (-0.29,0.2795); 
\coordinate (A212) at (-0.28,0.213); 
\coordinate (A22) at (0.387,0.213); 
\coordinate (A23) at (0,0.5); 
\coordinate (A231) at (-0.153,0.388); 
\coordinate (A232) at (0.153,0.388); 
\coordinate (A3) at (-1.1,0.6);
\coordinate (A31) at (-0.9,0.49);
\coordinate (A311) at (-0.88,0.59);
\coordinate (A312) at (-0.84,0.51);
\coordinate (A32) at (-0.8,0.976);
\coordinate (A321) at (-0.825,0.845);
\coordinate (A322) at (-0.745,0.875);
\coordinate (A33) at (-0.6,0.6);
\coordinate (A4) at (1.1,0.6);
\coordinate (A41) at (1.027,0.565);
\coordinate (A42) at (0.95,0.6);
\draw[draw=ForestGreen]  (A3)--(A2);
\draw[draw=ForestGreen]  (A1)--(A12);
\draw[draw=ForestGreen] (A2)--(A22);
\draw[draw=ForestGreen] (A2) -- (A1);
\draw[draw=ForestGreen] (A3)--(A1);
\draw[draw=ForestGreen,dashed]  (A33) -- (A3);
\draw[draw=ForestGreen,dashed]  (A42) -- (A4);
\draw[draw=ForestGreen] (A41)--(A4)--(A12);
 \draw[draw=black,fill=none]   (A321)--(A311);
 \draw[draw=black,fill=none] (A212)-- (A22) -- (A41);
\draw (A311)--(A312);
\draw (A111)--(A121);
\draw (A111)--(A112);
\draw (A311)--(A211)--(A212) --(A312);
\draw (A112)--(A321);
\draw[dashed] (A321)--(A322)--(A111);
\draw  (A211)--(A231)--(A232)--(A22);
\draw[dashed]  (A33) -- (A42);
\draw (A231) -- (A131) -- (A132) -- (A232) -- cycle;
\draw[dashed] (A322) -- (A33) -- (A312);
\draw (A112) -- (A131) --(A132)--(A122);
\draw[dashed]  (A41) -- (A42) -- (A121); 
\draw (A41)--(A122) --(A121);
\end{tikzpicture}} 
\end{center} 

\medskip

\noindent The hemiassociahedron is not as well-known as certain other  notable members of the family of hypergraph polytopes, like simplices, hypercubes, associahedra, cyclohedra and permutohedra,  
but, like all those polytopes, it also has a role in characterizing infinity structures: it displays a particular homotopy of strongly homotopy operads. The hemiassociahedron will be {\em our favourite polytope} in this article. 

\medskip

  The attribute {\em hypergraph} in the designation {\em hypergraph polytopes} is meant to indicate the particular style of  combinatorial description of the polytopes from this family: the face lattice of each hypergraph polytope can be derived from the data of a hypergraph whose hyperedges   encode the truncations of the simplex that lead to the polytope in question. This particular  characterization of truncated simplices has been introduced by Do\v sen and Petri\' c in \cite{DP-HP} and further developed by Curien, Ivanovi\' c and the author in \cite{CIO}. Truncated simplices  were originally investigated by   Feichtner and Sturmfels in \cite{FS05} and by 
Postnikov in \cite{Pos}, by means of different -- and predating --  combinatorial tools: {\em nested sets} and {\em tubings}, while,  in \cite{Pos2}, they first appeared under the name of {\em nestohedra}. The familly of {\em graph-associahedra}, introduced by Carr and Devados in \cite{CD}, is  the subfamily of the family of   hypergraph polytopes determined by polytopes whose face lattices can be encoded by the data of a genuine graph.
 \medskip

This section is a recollection on the combinatorial description  of the familly of hypergraph polytopes and is entiriely based on \cite{DP-HP} and \cite{CIO}. In particular, we shall consider hypergraph polytopes as abstract polytopes only, disregarding their geometric characterization  as a bounded intersection of a finite set of half-spaces, which is also given in the two references. We refer to \cite{abs-pol} for  the definition of an
abstract polytope and related notions.

\smallskip

\subsection{Hypergraph terminology}
A hypergraph is a generalization of a graph for which an edge can relate an arbitrary number of vertices. Formally, a hypergraph ${\bf H}$ is given by a  set $H$ of {\em vertices} and a subset ${\bf H}\subseteq {\cal P}(H)\backslash\emptyset$ of {\em hyperedges}, such that $\bigcup {\bf H}=H$.  Note the abuse of notation here: we used the  bold letter ${\bf H}$ to denote both  the hypergraph itself and its set of hyperedges. We justify this identification by requiring all our hypergraphs to be {\em atomic}, meaning that $\{x\}\in {\bf H}$ for all $x\in H$. Additionally,  we shall  assume that all our hypergraphs are {\em non-empty}, meaning that $H\neq \emptyset$, {\em finite}, meaning that $H$ is finite and {\em connected}, meaning that there are  no non-trivial partitions $H=H_1\cup H_2$, such that ${\bf H}=\{X\in {\bf H}\,|\, X\subseteq H_1\}\cup\{Y\in {\bf H}\,|\, Y\subseteq H_2\}$. There is one more property of hypergraphs that we shall encounter (but not a priori ask for) in the construction of hypergraph polytopes: the property of being {\em saturated}. We say that a hypergraph ${\bf H}$ is     saturated when, for every $X,Y\in{\bf H}$ such that $X\cap Y\neq\emptyset$, we have that $X\cup Y\in {\bf H}$. Every hypergraph can be saturated by adding the missing (unions of) hyperedges. Let us introduce the notation $${\bf{H}}_X:=\{Z\in {\bf{H}}\,|\, Z\subseteq X\},$$ for a hypergraph ${\bf H}$ and $X\subseteq H$.
 The {\em saturation} of    ${\bf{H}}$ is then formally defined  as the hypergraph
$${\it Sat}({\bf{H}}):=\{ X\,|\, {\emptyset\subsetneq X\subseteq H\;\mbox{and}\;{\bf{H}}_X\;\mbox{is connected}}\}.$$ 

\medskip

\begin{example}\label{hypergraphex1} The hypergraph 
 $${\bf H}=\{\{x\},\{y\},\{u\},\{v\},\{x,y\},\{x,u\},\{x,v\},\{u,v\},\{x,u,v\}\}$$    can be represented pictorially as follows:
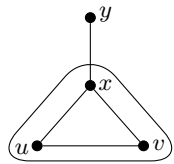
\begin{center}
\begin{tikzpicture}
\node (x) [circle,draw=black,fill=black,inner sep=0mm,minimum size=1.3mm,label={[xshift=0.2cm,yshift=-0.25cm]{\small $x$}}] at (1,1) {};
\node (u) [circle,draw=black,fill=black,inner sep=0mm,minimum size=1.3mm,label={[xshift=0.2cm,yshift=-0.25cm]{\small $y$}}] at (1,1.9) {};
\node (y) [circle,draw=black,fill=black,inner sep=0mm,minimum size=1.3mm,label={[xshift=-0.2cm,yshift=-0.3cm]{\small $u$}}] at (0.3,0.2) {};
\node (z) [circle,draw=black,fill=black,inner sep=0mm,minimum size=1.3mm,,label={[xshift=0.2cm,yshift=-0.25cm]{\small $v$}}] at (1.7,0.2) {};
\draw (u)--(x)--(y)--(z);
\draw (x)--(z);
\draw [rounded corners=5mm,fill=none] (1,1.5)--(-0.3,0)--(2.3,0)--cycle;
\end{tikzpicture}
\end{center}
Here, the hyperedge $\{x,u,v\}$ is represented by the circled-out  area aroud the vertices $x$, $u$ and $v$.
The hypergraph ${\bf H}$  is not saturated. The saturation of ${\bf H}$ is the hypergraph $${\it Sat}({\bf H})=\{\{x\},\{y\},\{u\},\{v\},\{x,y\},\{x,u\},\{x,v\},\{u,v\},\{x,y,u\},\{x,y,v\},\{x,u,v\},\{x,y,u,v\}\}.$$ \demo
 \end{example}
 
\medskip

We additionally import the following notational conventions and terminology from \cite{CIO}. For a hypergraph ${\bf H}$ and $X\subseteq H$, we set $${\bf H}\backslash X:={\bf{H}}_{H\backslash X}.$$ Observe that for each (not necessarily connected) finite hypergraph there exists a partition
$H=H_1\cup\ldots\cup H_m$, such that each hypergraph ${\bf{H}}_{H_i}$ is connected and ${\bf{H}}=\bigcup({\bf{H}}_{H_i})$.  The ${\bf{H}}_{H_i}$'s are called the {\em connected components} of ${\bf{H}}$. We shall  write
${\bf{H}}_i$  for
${\bf{H}}_{H_i}$.     We shall use the notation 
\begin{center}
${\bf{H}}\backslash X  \leadsto {\bf H}_1,\ldots,{\bf H}_n$ \quad\quad (resp.  ${\bf{H}}\backslash X  \leadsto \{{\bf H}_i \,|\,1\leq i\leq n\}$)
\end{center}
to indicate that  ${\bf H}_1,\ldots,{\bf H}_n$ are  the (resp. $ \{{\bf H}_i \,|\, 1\leq i\leq n\}$ is the set of)  connected components of ${\bf H}\backslash X$.  

\smallskip

 \subsection{The abstract polytope of a hypergraph}\label{abspol} We next define the abstract polytope $${\cal A}({\bf H})=(A({\bf H})\cup\{\emptyset\},\leq_{\bf H})$$ associated to a hypergraph ${\bf H}$.  We shall recall the representation of ${\cal A}({\bf H})$ given in \cite{CIO}, which coincides, up to isomorphism, with the one of \cite{DP-HP}. The advantage of the representation of  ${\cal A}({\bf H})$ given in \cite{CIO} lies in the  {\em tree notation} for all the faces
of hypergraph polytopes that encodes   face
inclusion as  edge contraction --  this combinatorial decription reveals  the operad structure on a particular   subfamily of the family of hypergraph polytopes and was essential for the main purpose of this paper.
 
 \medskip
 
The elements of the set $A({\bf H})$, to which we refer as the {\em  constructs} of   ${\bf H}$, are  the non-planar, vertex-decorated rooted trees defined recursively as follows.  Let $\emptyset\neq Y\subseteq H$ be a non-empty subset of vertices of ${\bf H}$. \\[-0.3cm]
\begin{itemize}
\item  If $Y = H$, then  the  tree with a single vertex decorated with $H$ and without any inputs, is a construct of ${\bf{H}}$; we denote it by $H$.\\[-0.3cm]
\item  If $Y\subsetneq H$, if   ${\bf H}\backslash Y  \leadsto {\bf H}_1,\ldots,{\bf H}_n$, and if  $C_1,\ldots,C_n$ are constructs of ${\bf H}_1,\ldots,{\bf H}_n$, respectively, then the
tree whose root vertex is decorated by $Y$ and that has $n$ inputs, on which the respective $C_i\,$'s are grafted,  
 is a construct of ${\bf H}$; we denote it by $Y\{C_1,\ldots,C_n\}$.\\[-0.3cm]
\end{itemize}
 We   write $C:{\bf H}$ to indicate that $C$ is a  construct of ${\bf H}$.  A {\em construction} is a construct  whose vertices are all decorated with singletons. 

\smallskip
 
\begin{example} Let us go through the   recursive definition of constructs by unwinding the construction of the following three constructs of the hypergraph ${\bf H}$ from Example \ref{hypergraphex1}:  $$C_1=\{x,y,u,v\},\quad C_2=\{x\}\{\{u,v\},\{y\}\}, \quad \mbox{ and }\quad C_3=\{x\}\{\{u\}\{\{v\}\},\{y\}\}.$$
The construct $C_1$ is obtained by the first rule in the above definition. The constructs $C_2$ and $C_3$ are both obtained by choosing the set $\{x\}$ to be the decoration of the root vertex. The connected components of ${\bf H}\backslash\{x\}$ are   
${\bf H}_1=\{\{u\},\{v\},\{u,v\}\}$ and  ${\bf H}_2=\{\{y\}\}$. The construct $C_2$ is then obtained by grafting to the root vertex $\{x\}$ the constructs $\{u,v\}:{\bf H}_1$ and $\{y\}:{\bf H}_2$ (both obtained by the first rule in the above definition), and $C_3$ is obtained by choosing $\{u\}\{\{v\}\}$   instead of $\{u,v\}$ as a construct of ${\bf H}_1$.  The construct $C_3$ is a construction. \demo
\end{example}
 \smallskip
 
 \begin{convention}\label{condr} In order to facilitate the notation for constructs, we shall represent their singleton vertices without the braces. For example, instead of $\{x\}\{\{u,v\},\{y\}\}$ and $\{x\}\{\{u\}\{\{v\}\},\{y\}\}$, we shall write $x\{\{u,v\},y\}$ and $x\{u\{v\},y\}$. Also, we shall
freely confuse the vertices of constructs with the sets decorating them, since they are a fortiori all distinct. In particular, we shall denote the vertices of constructs with capital letters specifying those sets. Finally,  in order to provide more intuition for the partial order on constructs that we are about to define in terms of edge contraction, and later also for the composition of constructs   underlying our infinity operad structure,  we shall use the  graphical representation for constructs indicated in the Introduction. For example, 
\begin{center} the constructs $x\{y,\{u,v\}\}$ and $x\{u\{v\},y\}$ will be drawn as
\raisebox{-1em}{\resizebox{2cm}{!}{\begin{tikzpicture}
\node (b) [circle,fill=cyan,draw=black,minimum size=0.1cm,inner sep=0.2mm,label={[yshift=-0.45cm]{\footnotesize $x$}}] at (4,-0.2) {};
\node (c) [circle,fill=cyan,draw=black,minimum size=0.1cm,inner sep=0.2mm,label={[xshift=-0.45cm,yshift=-0.31cm]{\footnotesize $\{u,v\}$}}] at (3.8,0.15) {};
 \node (d) [circle,fill=cyan,draw=black,minimum size=0.1cm,inner sep=0.2mm,label={[xshift=0.2cm,yshift=-0.3cm]{\footnotesize $y$}}] at (4.2,0.15) {};
 \draw[thick]  (c)--(b)--(d);\end{tikzpicture}}}   and \enspace\raisebox{-1.2em}{\resizebox{1.4cm}{!}{\begin{tikzpicture}
\node (b) [circle,fill=ForestGreen,draw=black,minimum size=0.1cm,inner sep=0.2mm,label={[yshift=-0.45cm]{\footnotesize $x$}}] at (4,-0.2) {};
\node (c) [circle,fill=ForestGreen,draw=black,minimum size=0.1cm,inner sep=0.2mm,label={[xshift=0.2cm,yshift=-0.3cm]{\footnotesize $y$}}] at (4.2,0.15) {};
 \node (d) [circle,fill=ForestGreen,draw=black,minimum size=0.1cm,inner sep=0.2mm,label={[xshift=-0.2cm,yshift=-0.25cm]{\footnotesize $u$}}] at (3.8,0.15) {};
  \node (a) [circle,fill=ForestGreen,draw=black,draw=black,minimum size=0.1cm,inner sep=0.2mm,label={[xshift=-0.2cm,yshift=-0.2cm]{\footnotesize $v$}}] at (3.8,0.5) {};
 \draw[thick]  (c)--(b)--(d)--(a);\end{tikzpicture}}},
\end{center}
respectively. Notice that for constructs we  do not draw the root. 
  \end{convention}

\smallskip

The partial order $\leq_{\bf H}$ on $A({\bf H})\cup\{\emptyset\}$ is defined by the following three rules.\\[-0.3cm]
\begin{itemize}
\item For all $C:{\bf H}$,  $\emptyset \leq_{\bf{H}} C$.\\[-0.3cm]
\item If ${\bf H}\backslash Y\leadsto {\bf H}_1,\dots,{\bf H}_n$, ${\bf H}_1\backslash X\leadsto {\bf H}_{11},\dots,{\bf H}_{1m}$, $C_{1j}:{\bf H}_{1j}$ for $1\leq j\leq m$, and ${C}_i:{\bf H}_i$ for $2\leq i\leq n$, then $$Y\{X\{C_{11},\ldots,C_{1m}\},C_2,\ldots,C_n\}\leq_{\bf H}(Y\cup X)\{C_{11},\ldots,C_{1m},C_2,\ldots,C_n\}.$$
\item If ${\bf H}\backslash Y\leadsto {\bf H}_1,\ldots,{\bf H}_n$, $C_i:{\bf H}_i$ for $2\leq i\leq n$, and $C_1\leq_{{\bf H}_1}C'_1$, then $$Y\{C_1,C_2,\ldots,C_n\}  \leq_{\bf{H}} Y\{C'_1,C_2,\ldots,C_n\}.$$ 
\end{itemize}
  
  \smallskip
  
\noindent Therefore, given a  construct $C:{\bf H}$,  one can obtain a larger  construct  by  contracting an edge of $C$ and merging the decorations of the vertices adjacent with that edge. 
Note that the partial order $\leq_{\bf H}$  is well-defined, in the sense that, if $C_1:{\bf H}$ and if $C_1\leq_{\bf H} C_2$ is inferred, then $C_2:{\bf H}$ can be inferred.  

\medskip

 The faces of ${\cal A}({\bf H})$ are ranked by integers  ranging from $-1$ to $|H|-1$. The face $\emptyset$ is the unique face of rank $-1$, whereas the rank of a construct $C:{\bf H}$ is  $|H|-|v(C)|$. In particular, constructions are faces of rank $0$, whereas  the  construct $H:{\bf H}$   is the unique face of rank $|H|-1$. We take the usual convention to name the faces of rank $0$   {\em vertices}, the faces of rank $1$  {\em edges} and the faces of rank $|H|-2$ {\em facets}. The ranks of the faces of ${\cal A}({\bf H})$  correspond to their actual dimension when realized in Euclidean space. We refer to \cite[Section 9]{DP-HP} and \cite[Section 3]{CIO} for a geometric realization of ${\cal A}({\bf H})$. In the next section, in conformity with this geometric realization, we shall provide examples of hypergraph polytopes.
 
 \medskip
 
  The fact that the poset ${\cal A}({\bf H})$ is indeed an abstract polytope of rank $|H|-1$  follows by translating the definition of  ${\cal A}({\bf H})$, using  the order isomorphism \cite[Proposition 2]{CIO}, to the formalism of hypergraph polytopes presented in \cite{DP-HP}, as a
consequence of \cite[Section 8]{DP-HP}, where the axioms of abstract polytopes  are verified for the latter presentation of ${\cal A}({\bf H})$.

\smallskip

\subsection{Examples} This section contains examples of various hypergraph polytopes; in \cite[Appendix B]{DP-HP} and \cite[Section 2.4, Section 2.6]{CIO},   the reader can find more of them. Given that our hypergraph vocabulary is now settled, before we give the individual examples,  let us first provide the intuition on the very first characterization of hypergraph polytopes that we have mentioned: the geometric description  in terms of  truncated simplices.

\medskip

If ${\bf H}$ is a hypergraph with the vertex set $H=\{x_1,\dots,x_{n+1}\}$, then ${\bf H}$ encodes the truncation instructions to be applied to the $(|H|-1)$-dimensional simplex, as follows. Start by labeling the facets of the $(|H|-1)$-dimensional  simplex by the vertices of ${\bf H}$. Then, for each hyperedge $X\in{\it Sat}({\bf H})\backslash (\{H\}\cup \{\{x_i\}\,|\, x_i\in H\})$, truncate the face of the simplex defined as the intersection of the facets contained in $X$. This intuition is formalized as the geometric realization  of hypergraph polytopes in \cite[Section 8]{DP-HP} and \cite[Section 3]{CIO}.

\smallskip

\subsubsection{Simplex} The hypergraph encoding the $n$-dimensional simplex is the hypergraph with $n+1$ vertices and no non-trivial hyperedges: $${\bf S}_{n+1}=\{\{x_1\},\dots,\{x_{n+1}\},\{x_1,\dots,x_{n+1}\}\}.$$
 In dimension $2$, the poset of constructs of the hypergraph ${\bf S}_3=\{\{x\},\{y\},\{z\},\{x,y,z\}\}$ can be realized as a triangle:
 
 \smallskip
 
 \begin{center}
\resizebox{5.5cm}{!}
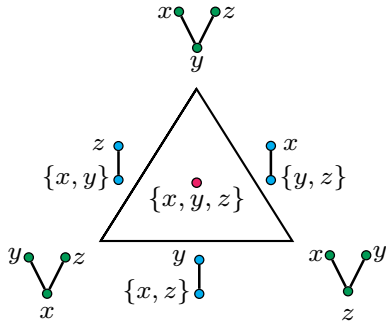
{\begin{tikzpicture}
\coordinate (A1) at (0,0);
\coordinate (A4) at (2.5,0);
\coordinate (A2) at (1.25,2);
\draw[thick,draw=black]  (A1) --(A2)--(A4)--(A1)--(A2);
\node at (A1) [xshift=-0.7cm,yshift=-0.55cm]  { {\resizebox{1.3cm}{!}{\begin{tikzpicture}
\node (L) [circle,draw=black,fill=ForestGreen,minimum size=0.1cm,inner sep=0.2mm,label={[yshift=-0.45cm]{\footnotesize $x$}}] at (4,-0.25) {};
\node (B1) [circle,draw=black,fill=ForestGreen,minimum size=0.1cm,inner sep=0.2mm,label={[xshift=-0.15cm,yshift=-0.25cm]{\footnotesize $y$}}] at (3.8,0.15) {};
 \node (B4) [circle,draw=black,fill=ForestGreen,minimum size=0.1cm,inner sep=0.2mm,label={[xshift=0.15cm,yshift=-0.25cm]{\footnotesize $z$}}] at (4.2,0.15) {};
 \draw[thick]  (L)--(B1);
\draw [thick] (L)--(B4);\end{tikzpicture}}}};
\node at (A4) [xshift=0.72cm,yshift=-0.55cm] { {\resizebox{1.3cm}{!}{\begin{tikzpicture}
\node (L) [circle,draw=black,fill=ForestGreen,minimum size=0.1cm,inner sep=0.2mm,label={[yshift=-0.5cm]{\footnotesize $z$}}] at (4,-0.25) {};
\node (B1) [circle,draw=black,fill=ForestGreen,minimum size=0.1cm,inner sep=0.2mm,label={[xshift=-0.15cm,yshift=-0.25cm]{\footnotesize $x$}}] at (3.8,0.15) {};
 \node (B4) [circle,draw=black,fill=ForestGreen,minimum size=0.1cm,inner sep=0.2mm,label={[xshift=0.15cm,yshift=-0.25cm]{\footnotesize $y$}}] at (4.2,0.15) {};
 \draw[thick]  (L)--(B1);
\draw [thick] (L)--(B4);\end{tikzpicture}}}};
\node at (A2) [xshift=0cm,yshift=0.65cm]  { {\resizebox{1.3cm}{!}{\begin{tikzpicture}
\node (L) [circle,draw=black,fill=ForestGreen,minimum size=0.1cm,inner sep=0.2mm,label={[yshift=-0.45cm]{\footnotesize $y$}}] at (4,-0.25) {};
\node (B1) [circle,draw=black,fill=ForestGreen,minimum size=0.1cm,inner sep=0.2mm,label={[xshift=-0.15cm,yshift=-0.25cm]{\footnotesize $x$}}] at (3.8,0.15) {};
 \node (B4) [circle,draw=black,fill=ForestGreen,minimum size=0.1cm,inner sep=0.2mm,label={[xshift=0.15cm,yshift=-0.25cm]{\footnotesize $z$}}] at (4.2,0.15) {};
 \draw[thick]  (L)--(B1);
\draw [thick] (L)--(B4);\end{tikzpicture}}}};
\node   at (0.75,-0.48) {{\resizebox{1.2cm}{!}{\begin{tikzpicture}
\node (L) [circle,draw=black,fill=cyan,draw=black,minimum size=0.1cm,inner sep=0.2mm,label={[xshift=-0.47cm,yshift=-0.3cm]{\footnotesize $\{x,z\}$}}] at (0,0) {};
\node (B1) [circle,draw=black,fill=cyan,draw=black,minimum size=0.1cm,inner sep=0.2mm,label={[xshift=-0.22cm,yshift=-0.25cm]{\footnotesize $y$}}] at (0,0.38) {};
 \draw[thick]  (L)--(B1);\end{tikzpicture}}}};
\node   at (2.75,1) {{\resizebox{1.2cm}{!}{\begin{tikzpicture}
\node (L) [circle,draw=black,fill=cyan,draw=black,minimum size=0.1cm,inner sep=0.2mm,label={[xshift=0.47cm,yshift=-0.3cm]{\footnotesize $\{y,z\}$}}] at (0,0) {};
\node (B1) [circle,draw=black,fill=cyan,draw=black,minimum size=0.1cm,inner sep=0.2mm,label={[xshift=0.22cm,yshift=-0.225cm]{\footnotesize $x$}}] at (0,0.38) {};
 \draw[thick]  (L)--(B1);\end{tikzpicture}}}};
\node   at (-0.3,1) {{\resizebox{1.2cm}{!}{\begin{tikzpicture}
\node (L) [circle,draw=black,fill=cyan,minimum size=0.1cm,inner sep=0.2mm,label={[xshift=-0.47cm,yshift=-0.3cm]{\footnotesize $\{x,y\}$}}] at (0,0) {};
\node (B1) [circle,draw=black,fill=cyan,minimum size=0.1cm,inner sep=0.2mm,label={[xshift=-0.22cm,yshift=-0.225cm]{\footnotesize $z$}}] at (0,0.38) {};
 \draw[thick]  (L)--(B1);\end{tikzpicture}}}};
\node   at (1.25,0.55)  {{\resizebox{1.55cm}{!}{\begin{tikzpicture}
\node (L) [circle,fill=WildStrawberry,draw=black,minimum size=0.1cm,inner sep=0.2mm,label={[xshift=0cm,yshift=-0.5cm]{\footnotesize $\{x,y,z\}$}}] at (0,0) {};
\end{tikzpicture}}}};
\end{tikzpicture}}  
 \end{center}
 
  \smallskip
 
 Let us now illustrate  how  truncations arise  by adding non-trivial hyperedges to the ``bare'' simplex hypergraph ${\bf S}_3$. Consider the hypergraph ${\bf S}_3\cup \{\{y,z\}\}$. For this hypergraph, the vertex $x\{y,z\}$ is no longer well-defined, since  $({\bf S}_3\cup \{\{y,z\}\})\backslash\{x\}$ no longer contains two connected components, but only one. In the polytope associated to ${\bf S}_3\cup \{\{y,z\}\}$, the vertex $x\{y,z\}$ gets replaced by two new vertices: $x\{y\{z\}\}$ and $x\{z\{y\}\}$, and the edge $x\{\{y,z\}\}$ between them,
 which can be realized by truncating $x\{y,z\}$ in the above realization of the triangle: 
 
   \smallskip
 
 \begin{center}
 \resizebox{3.6cm}{!}
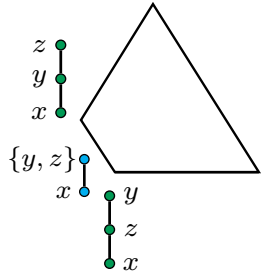
{\begin{tikzpicture}
\coordinate (A11) at (0.8,0);
\coordinate (A12) at (0.4,0.62);
\coordinate (A4) at (2.5,0);
\coordinate (A2) at (1.25,2);
\draw[thick,draw=black]  (A4) --(A11)--(A12) --(A2)--(A4)--(A11);
\node at (A1) [xshift=0.95cm,yshift=-0.675cm]  {\resizebox{0.55cm}{!}{\begin{tikzpicture}
\node (b) [circle,draw=black,fill=ForestGreen,minimum size=0.1cm,inner sep=0.2mm,label={[xshift=0.22cm,yshift=-0.25cm]{\footnotesize $x$}}] at (0,0) {};
\node (c) [circle,draw=black,fill=ForestGreen,minimum size=0.1cm,inner sep=0.2mm,label={[xshift=0.22cm,yshift=-0.25cm]{\footnotesize $z$}}] at (0,0.35) {};
 \node (d) [circle,draw=black,fill=ForestGreen,minimum size=0.1cm,inner sep=0.2mm,label={[xshift=0.22cm,yshift=-0.27cm]{\footnotesize $y$}}] at (0,0.7) {};
 \draw[thick]  (b)--(c)--(d);\end{tikzpicture}}};
\node   at (-0.05,1.1)  {\resizebox{0.55cm}{!}{\begin{tikzpicture}
\node (b) [circle,draw=black,fill=ForestGreen,minimum size=0.1cm,inner sep=0.2mm,label={[xshift=-0.22cm,yshift=-0.25cm]{\footnotesize $x$}}] at (0,0) {};
\node (c) [circle,draw=black,fill=ForestGreen,minimum size=0.1cm,inner sep=0.2mm,label={[xshift=-0.22cm,yshift=-0.27cm]{\footnotesize $y$}}] at (0,0.35) {};
 \node (d) [circle,draw=black,fill=ForestGreen,minimum size=0.1cm,inner sep=0.2mm,label={[xshift=-0.22cm,yshift=-0.25cm]{\footnotesize $z$}}] at (0,0.7) {};
 \draw[thick]  (b)--(c)--(d);\end{tikzpicture}}};
\node   at (-0.05,0) {\resizebox{1.1cm}{!}{\begin{tikzpicture}
\node (c) [circle,fill=cyan,draw=black,minimum size=0.1cm,inner sep=0.2mm,label={[xshift=-0.45cm,yshift=-0.32cm]{\footnotesize $\{y,z\}$}}] at (0,0.35) {};
 \node (d) [circle,fill=cyan,draw=black,minimum size=0.1cm,inner sep=0.2mm,label={[xshift=-0.22cm,yshift=-0.25cm]{\footnotesize $x$}}] at (0,0) {};
 \draw[thick]  (c)--(d);\end{tikzpicture}}};
\end{tikzpicture}}  
 \end{center}
 
   \smallskip
   
\noindent  By additionally adding the hyperedge $\{x,y\}$ to ${\bf S}_3\cup \{\{y,z\}\}$, the vertex $z\{x,y\}$ will also be truncated. This leads us to our second example. 
 
 \smallskip
\subsubsection{Associahedron}\label{assoc} The hypergraph encoding the $n$-dimensional associahedron is the linear graph with $n+1$ vertices: $${\bf A}_{n+1}=\{\{x_1\},\dots,\{x_{n+1}\},\{x_1,x_2\},\dots,\{x_n,x_{n+1}\}\}.$$
  In dimension $2$, the poset of constructs of the hypergraph  ${\bf A}_3=\{\{x\},\{y\},\{z\},\{x,y\},\{y,z\}\}$ (i.e. of the hypergraph ${\bf S}_3\cup \{\{x,y\},\{y,z\}\}$) encodes the face lattice of   a pentagon as follows: 
  
    \smallskip
  
   \begin{center}
 \resizebox{5.5cm}{!}
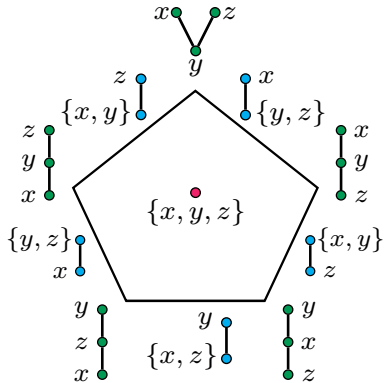
{\begin{tikzpicture}
\coordinate (A1) at (0.15,0);
\coordinate (A4) at (1.85,0);
\coordinate (A2) at (-0.5,1.4);
\coordinate (A3) at (2.5,1.4);
\coordinate (A5) at (1,2.6);
\draw[thick,draw=black]  (A4) --(A1) --(A2)--(A5) --(A3)--(A4)--(A1);
\node at (A1) [xshift=-0.5cm,yshift=-0.5cm] {\resizebox{0.55cm}{!}{\begin{tikzpicture}
\node (b) [circle,fill=ForestGreen,draw=black,minimum size=0.1cm,inner sep=0.2mm,label={[xshift=-0.22cm,yshift=-0.25cm]{\footnotesize $x$}}] at (0,0) {};
\node (c) [circle,fill=ForestGreen,draw=black,minimum size=0.1cm,inner sep=0.2mm,label={[xshift=-0.22cm,yshift=-0.25cm]{\footnotesize $z$}}] at (0,0.35) {};
 \node (d) [circle,fill=ForestGreen,draw=black,minimum size=0.1cm,inner sep=0.2mm,label={[xshift=-0.22cm,yshift=-0.27cm]{\footnotesize $y$}}] at (0,0.7) {};
 \draw[thick]  (b)--(c)--(d);\end{tikzpicture}}};
\node at (A4) [xshift=0.5cm,yshift=-0.5cm] {\resizebox{0.55cm}{!}{\begin{tikzpicture}
\node (b) [circle,fill=ForestGreen,draw=black,minimum size=0.1cm,inner sep=0.2mm,label={[xshift=0.22cm,yshift=-0.25cm]{\footnotesize $z$}}] at (0,0) {};
\node (c) [circle,fill=ForestGreen,draw=black,minimum size=0.1cm,inner sep=0.2mm,label={[xshift=0.22cm,yshift=-0.25cm]{\footnotesize $x$}}] at (0,0.35) {};
 \node (d) [circle,fill=ForestGreen,draw=black,minimum size=0.1cm,inner sep=0.2mm,label={[xshift=0.22cm,yshift=-0.27cm]{\footnotesize $y$}}] at (0,0.7) {};
 \draw[thick]  (b)--(c)--(d);\end{tikzpicture}}};
\node at (A2) [xshift=-0.5cm,yshift=0.3cm]   {\resizebox{0.55cm}{!}{\begin{tikzpicture}
\node (b) [circle,fill=ForestGreen,draw=black,minimum size=0.1cm,inner sep=0.2mm,label={[xshift=-0.22cm,yshift=-0.25cm]{\footnotesize $x$}}] at (0,0) {};
\node (c) [circle,fill=ForestGreen,draw=black,minimum size=0.1cm,inner sep=0.2mm,label={[xshift=-0.22cm,yshift=-0.27cm]{\footnotesize $y$}}] at (0,0.35) {};
 \node (d) [circle,fill=ForestGreen,draw=black,minimum size=0.1cm,inner sep=0.2mm,label={[xshift=-0.22cm,yshift=-0.25cm]{\footnotesize $z$}}] at (0,0.7) {};
 \draw[thick]  (b)--(c)--(d);\end{tikzpicture}}};
\node at (A5) [yshift=0.6cm]    { {\resizebox{1.3cm}{!}{\begin{tikzpicture}
\node (L) [circle,draw=black,fill=ForestGreen,minimum size=0.1cm,inner sep=0.2mm,label={[yshift=-0.45cm]{\footnotesize $y$}}] at (4,-0.25) {};
\node (B1) [circle,draw=black,fill=ForestGreen,minimum size=0.1cm,inner sep=0.2mm,label={[xshift=-0.15cm,yshift=-0.25cm]{\footnotesize $x$}}] at (3.8,0.15) {};
 \node (B4) [circle,draw=black,fill=ForestGreen,minimum size=0.1cm,inner sep=0.2mm,label={[xshift=0.15cm,yshift=-0.25cm]{\footnotesize $z$}}] at (4.2,0.15) {};
 \draw[thick]  (L)--(B1);
\draw [thick] (L)--(B4);\end{tikzpicture}}}};
\node at (A3) [xshift=0.5cm,yshift=0.3cm]   {\resizebox{0.55cm}{!}{\begin{tikzpicture}
\node (b) [circle,fill=ForestGreen,draw=black,minimum size=0.1cm,inner sep=0.2mm,label={[xshift=0.22cm,yshift=-0.25cm]{\footnotesize $z$}}] at (0,0) {};
\node (c) [circle,fill=ForestGreen,draw=black,minimum size=0.1cm,inner sep=0.2mm,label={[xshift=0.22cm,yshift=-0.27cm]{\footnotesize $y$}}] at (0,0.35) {};
 \node (d) [circle,fill=ForestGreen,draw=black,minimum size=0.1cm,inner sep=0.2mm,label={[xshift=0.22cm,yshift=-0.25cm]{\footnotesize $x$}}] at (0,0.7) {};
 \draw[thick]  (b)--(c)--(d);\end{tikzpicture}}};
\node   at (-0.2,2.5) {{\resizebox{1.2cm}{!}{\begin{tikzpicture}
\node (L) [circle,fill=cyan,draw=black,minimum size=0.1cm,inner sep=0.2mm,label={[xshift=-0.47cm,yshift=-0.3cm]{\footnotesize $\{x,y\}$}}] at (0,0) {};
\node (B1) [circle,fill=cyan,draw=black,minimum size=0.1cm,inner sep=0.2mm,label={[xshift=-0.22cm,yshift=-0.225cm]{\footnotesize $z$}}] at (0,0.38) {};
 \draw[thick]  (L)--(B1);\end{tikzpicture}}}};
\node   at (-0.9,0.6) {\resizebox{1.1cm}{!}{\begin{tikzpicture}
\node (c) [circle,fill=cyan,draw=black,minimum size=0.1cm,inner sep=0.2mm,label={[xshift=-0.45cm,yshift=-0.32cm]{\footnotesize $\{y,z\}$}}] at (0,0.35) {};
 \node (d) [circle,fill=cyan,draw=black,minimum size=0.1cm,inner sep=0.2mm,label={[xshift=-0.22cm,yshift=-0.25cm]{\footnotesize $x$}}] at (0,0) {};
 \draw[thick]  (c)--(d);\end{tikzpicture}}};
\node   at (2.9,0.6) {\resizebox{1.1cm}{!}{\begin{tikzpicture}
\node (c) [circle,fill=cyan,draw=black,minimum size=0.1cm,inner sep=0.2mm,label={[xshift=0.45cm,yshift=-0.32cm]{\footnotesize $\{x,y\}$}}] at (0,0.35) {};
 \node (d) [circle,fill=cyan,draw=black,minimum size=0.1cm,inner sep=0.2mm,label={[xshift=0.22cm,yshift=-0.25cm]{\footnotesize $z$}}] at (0,0) {};
 \draw[thick]  (c)--(d);\end{tikzpicture}}};
\node   at (2.15,2.5) {{\resizebox{1.2cm}{!}{\begin{tikzpicture}
\node (L) [circle,fill=cyan,draw=black,minimum size=0.1cm,inner sep=0.2mm,label={[xshift=0.47cm,yshift=-0.3cm]{\footnotesize $\{y,z\}$}}] at (0,0) {};
\node (B1) [circle,fill=cyan,draw=black,minimum size=0.1cm,inner sep=0.2mm,label={[xshift=0.22cm,yshift=-0.225cm]{\footnotesize $x$}}] at (0,0.38) {};
 \draw[thick]  (L)--(B1);\end{tikzpicture}}}};
\node   at (0.85,-0.5) {{\resizebox{1.2cm}{!}{\begin{tikzpicture}
\node (L) [circle,fill=cyan,draw=black,minimum size=0.1cm,inner sep=0.2mm,label={[xshift=-0.47cm,yshift=-0.3cm]{\footnotesize $\{x,z\}$}}] at (0,0) {};
\node (B1) [circle,fill=cyan,draw=black,minimum size=0.1cm,inner sep=0.2mm,label={[xshift=-0.22cm,yshift=-0.25cm]{\footnotesize $y$}}] at (0,0.38) {};
 \draw[thick]  (L)--(B1);\end{tikzpicture}}}};
\node   at (1,1.1) {{\resizebox{1.5cm}{!}{\begin{tikzpicture}
\node (L) [circle,fill=WildStrawberry,draw=black,minimum size=0.1cm,inner sep=0.2mm,label={[xshift=0cm,yshift=-0.53cm]{\footnotesize $\{x,y,z\}$}}] at (0,0) {};
\end{tikzpicture}}}};
\end{tikzpicture}}
 \end{center}
 
   \smallskip
 
 Starting from the construct representation of the $n$-dimensional associahedron, one can retrieve Stasheff's original representation  in terms of (partial) parenthesisations of a word $a_1\cdots a_{n+2}$ on $n+2$ letters, or, equivalently, of planar rooted trees with $n+2$ leaves, as follows. The idea is to consider each  vertex $x_i$ of ${\bf A}_{n+1}$ as the mutiplication of letters $a_{i}$ and $a_{i+1}$, as suggested in the following expression: $$a_1 \cdot_{x_1} a_2 \cdot_{x_2} a_3 \cdot_{x_3} \,\,\cdots\,\, \cdot_{x_{n+1}} a_{n+2}.$$ A given construct should then be read from the leaves to the root, interpreting each vertex as an instruction for inserting a pair of parentheses around the   group of (possibly already partially parenthesised) letters  spanned by all the multiplications  determined by the vertex. For example, in dimension $2$, and taking $$a \cdot_{x} b\cdot_{y} c \cdot_z d$$ as the layout for building the parentheses, the constructs $$x\{y\{z\}\},\enspace \{x,y\}\{z\},\enspace\mbox{ and }\enspace y\{x,z\}$$ correspond to parenthesised words $$(a(b(cd))),\enspace (ab(cd)),\enspace\mbox{ and }\enspace ((ab)(cd)),$$ respectively. In the other direction, the construct corresponding to a planar rooted tree  with $n+2$ leaves is recovered as follows. First, label  the $n+1$ intervals between the  leaves of a given tree
  by $x_1,\dots,x_{n+1}$. Then, considering $x_i$'s as balls, let them fall, and decorate  each vertex of the tree by
the set of balls which end up falling  to that vertex. Finally, remove the input leaves of the starting tree. For example,  in dimension 2 again, and writing $x,y,z$ for $x_1,x_2,x_3$, respectively, the planar rooted trees 
\smallskip
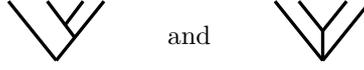
\begin{center}
{\resizebox{1.3cm}{!}{\begin{tikzpicture}
\coordinate (L) at (4,-0.25);
\coordinate (M2) at (4.075,0.05);
\coordinate (M1) at (4.15 ,-0.05);
\coordinate (B1) at (3.6,0.25);
 \coordinate (B2) at (3.925 ,0.25);
 \coordinate (B3) at (4.225,0.25);
 \coordinate (B4) at (4.4,0.25);
 \draw[thick] (M1)--(B2);
\draw[thick] (L)--(B1);
\draw[thick] (M2)--(B3);
\draw[thick] (L)--(B4);\end{tikzpicture}}} \quad\quad \raisebox{0.7em}{and} \quad\quad \resizebox{1.3cm}{!}{\begin{tikzpicture}
\coordinate (Aw) at (4,-0.25);
\coordinate (Cw) at (4,-0);
\coordinate (B1w) at (3.6,0.25);
 \coordinate (B2w) at (3.8,0.25);
 \coordinate (B3w) at (4.2,0.25);
 \coordinate (B4w) at (4.4,0.25);
 \draw[thick] (Aw)--(B1w);
\draw[thick] (Aw)--(Cw);
\draw[thick] (Cw)--(B3w);
\draw[thick] (Cw)--(B2w);
\draw[thick] (Aw)--(B4w);\end{tikzpicture}}
\end{center}
\smallskip
\noindent correspond to constructs 
 
\begin{center}
$x\{z\{y\}\}$ \quad\quad and \quad\quad $\{x,z\}\{y\}$,
\end{center}
 respectively.
 
 \smallskip
\subsubsection{Permutohedron} The hypergraph encoding the $n$-dimensional permutoheron is the complete graph with $n+1$ vertices: $${\bf P}_{n+1}=\{\{x_1\},\dots,\{x_{n+1}\}\}\cup\{\{x_i,x_j\}\,|\, i,j\in\{1,\dots,n+1\} \mbox{ and } i\neq j\}.$$ 
  In dimension $2$, the hypergraph ${\bf P}_3=\{\{x\},\{y\},\{z\},\{x,y\},\{y,z\},\{z,x\}\}$ (i.e. the hypergraph ${\bf A}_3\cup \{\{z,x\}\}$) encodes   the following  set of constructs:
  
    \smallskip
  
    \begin{center}
\raisebox{4.5em}{ \begin{tabular}{rcl}
{\textsf{Rank}} && {\textsf{Faces}}\\
\toprule\\[-0.3cm]
 $2$  &&  $\{x,y,z\}$   \\[0.1cm] 
$1$ && $\{x,y\}\{z\}$, $\{y,z\}\{x\}$, $\{x,z\}\{y\}$, $x\{\{y,z\}\}$, $z\{\{x,y\}\}$, $y\{\{x,z\}\}$\\[0.1cm]
$0$ && $x\{y\{z\}\}$, $x\{z\{y\}\}$, $y\{x\{z\}\}$, $y\{z\{x\}\}$, $z\{x\{y\}\}$, $z\{y\{x\}\}$ \\[0.1cm]
$-1$ && $\emptyset$ 
\end{tabular}} 
\end{center}
The corresponding realization is obtained by truncating the top vertex of the 2-dimensional associahedron from \S\ref{assoc}.
  \smallskip

\subsubsection{Hemiassociahedron}\label{sec.hemiassociahedron} We finish this section with the   description of the  $3$-dimensional hemiassociahedron, whose construction in terms of simplex truncation we illustrated in the introduction to this section. The hypergraph encoding the  $3$-dimensional hemiassociahedron is the hypergraph ${\bf H}$ from Example \ref{hypergraphex1}:
 $${\bf H}=\{\{x\},\{y\},\{u\},\{v\},\{x,y\},\{x,u\},\{x,v\},\{u,v\},\{x,u,v\}\}.$$ 
 As an exercise, the reader may now label the facets of the $3$-dimensional simplex in such a way that the sequence of truncations from page 3 can be read in terms of  non-trivial hyperedges  of ${\it Sat}({\bf H})$. The following table, listing the  constructs of  rank 2 of ${\bf H}$, might come in handy:
 
   \smallskip
 
     \begin{center}
\raisebox{4.5em}{ \begin{tabular}{rcl}
{\textsf{Rank}} && {\textsf{Faces}}\\
\toprule\\[-0.3cm]
 $2$  &&  {\textsf{Hexagons}} \enspace\, $y\{\{x,u,v\}\}$, $\{y,u,v\}\{x\}$,   $\{x,u,v\}\{y\}$\\
 &&  {\textsf{Pentagons}}\enspace\,\,$\{x,y,v\}\{u\}$, $v\{\{x,y,u\}\}$,  $\{x,y,u\}\{v\}$, $u\{\{x,y,v\}\}$\\
 &&  {\textsf{Squares}}\enspace \enspace\enspace \,    $\{x,y\}\{\{u,v\}\}$,  $\{u,v\}\{\{x,y\}\}$, $\{y,v\}\{\{x,u\}\}$, $\{y,u\}\{\{x,v\}\}$ \\[-0.75cm]
\end{tabular}}
\end{center}

  \smallskip

\noindent together with the following realization, in which we   labeled the vertices of the hemiassociahedron:

  \smallskip

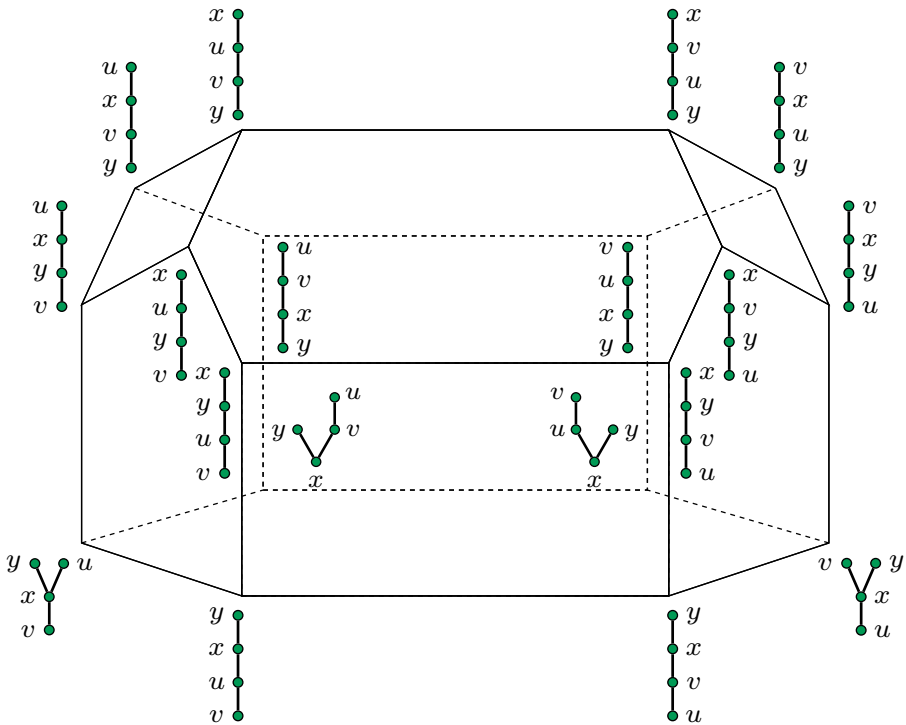
\begin{center}
 \resizebox{12.25cm}{!}{\begin{tikzpicture}[thick,scale=23]
\coordinate (A1) at (-0.2,0);
\coordinate (A2) at (0.2,0); 
\coordinate (D1) at (-0.35,0.05);
\coordinate (D2) at (0.35,0.05);
\coordinate (D3) at (-0.18,0.1);
\coordinate (D4)  at (0.18,0.1);
\coordinate (E3) at (-0.18,0.34);
\coordinate (E4) at (0.18,0.34);
\coordinate (A3) at (-0.2,0.22);
\coordinate (A4) at (0.2,0.22);
 \coordinate (C1) at (-0.25,0.33);
\coordinate (C2) at (0.25,0.33);
 \coordinate (F1) at (-0.35,0.275);
\coordinate (F2) at (0.35,0.275);
 \coordinate (G1) at (-0.3,0.385);
\coordinate (G2) at (0.3,0.385);
 \coordinate (B3) at (-0.2,0.44);
\coordinate (B4) at (0.2,0.44);
\draw (A1) -- (A2) -- (A4) -- (A3) -- cycle;
\draw (A3) -- (A4) -- (C2) -- (B4) -- (B3) -- (C1)-- cycle;
\draw[dashed] (D3) -- (D4) -- (E4) -- (E3) -- cycle;
\draw[dashed] (A1) -- (A2) -- (A4) -- (A3) -- (A1);
\draw (A3) -- (A4) -- (C2) -- (B4) -- (B3) -- (C1)-- (A3);
\draw (G1) -- (B3) -- (B4) -- (G2);
\draw (G2) -- (F2);
\draw (G1) -- (F1);
\draw (A3)--(A1) -- (A2)--(A4);
\draw[dashed]   (G1) -- (E3); 
\draw[dashed] (E4) -- (G2);
 \draw  (A1) -- (D1) -- (F1) -- (C1) -- (A3) --  cycle;
 \draw  (A2) -- (D2) -- (F2) -- (C2) -- (A4) --  cycle;
 \draw  (F1) -- (G1) -- (B3) -- (C1)  --  cycle;
\draw (F2) -- (G2) -- (B4) -- (C2)  --  cycle;
\draw[dashed] (D3) -- (D4) -- (E4) -- (E3) -- cycle;
\draw[dashed] (D1)   -- (D3);
\draw[dashed] (D2)   -- (D4);
\draw (F1)--(D1)--(A1)--(A2)--(D2)--(F2);
\node at   (0.145,0.28) {\resizebox{1cm}{!}{\begin{tikzpicture}
\node (b) [circle,fill=ForestGreen,draw=black,minimum size=0.1cm,inner sep=0.2mm,label={[xshift=-0.22cm,yshift=-0.27cm]{\footnotesize $y$}}] at (0,0) {};
\node (c) [circle,fill=ForestGreen,draw=black,minimum size=0.1cm,inner sep=0.2mm,label={[xshift=-0.22cm,yshift=-0.25cm]{\footnotesize $x$}}] at (0,0.35) {};
 \node (d) [circle,fill=ForestGreen,draw=black,minimum size=0.1cm,inner sep=0.2mm,label={[xshift=-0.22cm,yshift=-0.25cm]{\footnotesize $u$}}] at (0,0.7) {};
  \node (a) [circle,fill=ForestGreen,draw=black,minimum size=0.1cm,inner sep=0.2mm,label={[xshift=-0.22cm,yshift=-0.25cm]{\footnotesize $v$}}] at (0,1.05) {};
 \draw[thick]  (b)--(c)--(d)--(a);\end{tikzpicture}}};
\node at   (-0.13,0.15) { {\resizebox{2.4cm}{!}{\begin{tikzpicture}
\node (b) [circle,fill=ForestGreen,draw=black,minimum size=0.1cm,inner sep=0.2mm,label={[yshift=-0.45cm]{\footnotesize $x$}}] at (4,-0.2) {};
\node (c) [circle,fill=ForestGreen,draw=black,minimum size=0.1cm,inner sep=0.2mm,label={[xshift=-0.2cm,yshift=-0.3cm]{\footnotesize $y$}}] at (3.8,0.15) {};
 \node (d) [circle,fill=ForestGreen,draw=black,minimum size=0.1cm,inner sep=0.2mm,label={[xshift=0.2cm,yshift=-0.25cm]{\footnotesize $v$}}] at (4.2,0.15) {};
  \node (a) [circle,fill=ForestGreen,draw=black,draw=black,minimum size=0.1cm,inner sep=0.2mm,label={[xshift=0.2cm,yshift=-0.2cm]{\footnotesize $u$}}] at (4.2,0.5) {};
 \draw[thick]  (c)--(b)--(d)--(a);\end{tikzpicture}}}};
\node at  (-0.145,0.28) {\resizebox{1cm}{!}{\begin{tikzpicture}
\node (b) [circle,fill=ForestGreen,draw=black,minimum size=0.1cm,inner sep=0.2mm,label={[xshift=0.22cm,yshift=-0.27cm]{\footnotesize $y$}}] at (0,0) {};
\node (c) [circle,fill=ForestGreen,draw=black,minimum size=0.1cm,inner sep=0.2mm,label={[xshift=0.22cm,yshift=-0.25cm]{\footnotesize $x$}}] at (0,0.35) {};
 \node (d) [circle,fill=ForestGreen,draw=black,minimum size=0.1cm,inner sep=0.2mm,label={[xshift=0.22cm,yshift=-0.25cm]{\footnotesize $v$}}] at (0,0.7) {};
  \node (a) [circle,fill=ForestGreen,draw=black,minimum size=0.1cm,inner sep=0.2mm,label={[xshift=0.22cm,yshift=-0.25cm]{\footnotesize $u$}}] at (0,1.05) {};
 \draw[thick]  (b)--(c)--(d)--(a);\end{tikzpicture}}};
\node at  (0.13,0.15) { {\resizebox{2.4cm}{!}{\begin{tikzpicture}
\node (b) [circle,fill=ForestGreen,draw=black,minimum size=0.1cm,inner sep=0.2mm,label={[yshift=-0.45cm]{\footnotesize $x$}}] at (4,-0.2) {};
\node (c) [circle,fill=ForestGreen,draw=black,minimum size=0.1cm,inner sep=0.2mm,label={[xshift=0.2cm,yshift=-0.3cm]{\footnotesize $y$}}] at (4.2,0.15) {};
 \node (d) [circle,fill=ForestGreen,draw=black,minimum size=0.1cm,inner sep=0.2mm,label={[xshift=-0.2cm,yshift=-0.25cm]{\footnotesize $u$}}] at (3.8,0.15) {};
  \node (a) [circle,fill=ForestGreen,draw=black,draw=black,minimum size=0.1cm,inner sep=0.2mm,label={[xshift=-0.2cm,yshift=-0.2cm]{\footnotesize $v$}}] at (3.8,0.5) {};
 \draw[thick]  (c)--(b)--(d)--(a);\end{tikzpicture}}}};
\node at  (0.32,0.45) {\resizebox{1cm}{!}{\begin{tikzpicture}
\node (b) [circle,fill=ForestGreen,draw=black,minimum size=0.1cm,inner sep=0.2mm,label={[xshift=0.22cm,yshift=-0.27cm]{\footnotesize $y$}}] at (0,0) {};
\node (c) [circle,fill=ForestGreen,draw=black,minimum size=0.1cm,inner sep=0.2mm,label={[xshift=0.22cm,yshift=-0.25cm]{\footnotesize $u$}}] at (0,0.35) {};
 \node (d) [circle,fill=ForestGreen,draw=black,minimum size=0.1cm,inner sep=0.2mm,label={[xshift=0.22cm,yshift=-0.25cm]{\footnotesize $x$}}] at (0,0.7) {};
  \node (a) [circle,fill=ForestGreen,draw=black,minimum size=0.1cm,inner sep=0.2mm,label={[xshift=0.22cm,yshift=-0.25cm]{\footnotesize $v$}}] at (0,1.05) {};
 \draw[thick]  (b)--(c)--(d)--(a);\end{tikzpicture}}};
\node at    (0.22,0.5)  {\resizebox{1cm}{!}{\begin{tikzpicture}
\node (b) [circle,fill=ForestGreen,draw=black,minimum size=0.1cm,inner sep=0.2mm,label={[xshift=0.22cm,yshift=-0.27cm]{\footnotesize $y$}}] at (0,0) {};
\node (c) [circle,fill=ForestGreen,draw=black,minimum size=0.1cm,inner sep=0.2mm,label={[xshift=0.22cm,yshift=-0.25cm]{\footnotesize $u$}}] at (0,0.35) {};
 \node (d) [circle,fill=ForestGreen,draw=black,minimum size=0.1cm,inner sep=0.2mm,label={[xshift=0.22cm,yshift=-0.25cm]{\footnotesize $v$}}] at (0,0.7) {};
  \node (a) [circle,fill=ForestGreen,draw=black,minimum size=0.1cm,inner sep=0.2mm,label={[xshift=0.22cm,yshift=-0.25cm]{\footnotesize $x$}}] at (0,1.05) {};
 \draw[thick]  (b)--(c)--(d)--(a);\end{tikzpicture}}};
\node at   (-0.22,0.5)  {\resizebox{1cm}{!}{\begin{tikzpicture}
\node (b) [circle,fill=ForestGreen,draw=black,minimum size=0.1cm,inner sep=0.2mm,label={[xshift=-0.22cm,yshift=-0.27cm]{\footnotesize $y$}}] at (0,0) {};
\node (c) [circle,fill=ForestGreen,draw=black,minimum size=0.1cm,inner sep=0.2mm,label={[xshift=-0.22cm,yshift=-0.25cm]{\footnotesize $v$}}] at (0,0.35) {};
 \node (d) [circle,fill=ForestGreen,draw=black,minimum size=0.1cm,inner sep=0.2mm,label={[xshift=-0.22cm,yshift=-0.25cm]{\footnotesize $u$}}] at (0,0.7) {};
  \node (a) [circle,fill=ForestGreen,draw=black,minimum size=0.1cm,inner sep=0.2mm,label={[xshift=-0.22cm,yshift=-0.25cm]{\footnotesize $x$}}] at (0,1.05) {};
 \draw[thick]  (b)--(c)--(d)--(a);\end{tikzpicture}}};
\node at   (-0.32,0.45) {\resizebox{1cm}{!}{\begin{tikzpicture}
\node (b) [circle,fill=ForestGreen,draw=black,minimum size=0.1cm,inner sep=0.2mm,label={[xshift=-0.22cm,yshift=-0.27cm]{\footnotesize $y$}}] at (0,0) {};
\node (c) [circle,fill=ForestGreen,draw=black,minimum size=0.1cm,inner sep=0.2mm,label={[xshift=-0.22cm,yshift=-0.25cm]{\footnotesize $v$}}] at (0,0.35) {};
 \node (d) [circle,fill=ForestGreen,draw=black,minimum size=0.1cm,inner sep=0.2mm,label={[xshift=-0.22cm,yshift=-0.25cm]{\footnotesize $x$}}] at (0,0.7) {};
  \node (a) [circle,fill=ForestGreen,draw=black,minimum size=0.1cm,inner sep=0.2mm,label={[xshift=-0.22cm,yshift=-0.25cm]{\footnotesize $u$}}] at (0,1.05) {};
 \draw[thick]  (b)--(c)--(d)--(a);\end{tikzpicture}}};
\node at  (0.273,0.255) {\resizebox{1cm}{!}{\begin{tikzpicture}
\node (b) [circle,fill=ForestGreen,draw=black,minimum size=0.1cm,inner sep=0.2mm,label={[xshift=0.22cm,yshift=-0.25cm]{\footnotesize $u$}}] at (0,0) {};
\node (c) [circle,fill=ForestGreen,draw=black,minimum size=0.1cm,inner sep=0.2mm,label={[xshift=0.22cm,yshift=-0.27cm]{\footnotesize $y$}}] at (0,0.35) {};
 \node (d) [circle,fill=ForestGreen,draw=black,minimum size=0.1cm,inner sep=0.2mm,label={[xshift=0.22cm,yshift=-0.25cm]{\footnotesize $v$}}] at (0,0.7) {};
  \node (a) [circle,fill=ForestGreen,draw=black,minimum size=0.1cm,inner sep=0.2mm,label={[xshift=0.22cm,yshift=-0.25cm]{\footnotesize $x$}}] at (0,1.05) {};
 \draw[thick]  (b)--(c)--(d)--(a);\end{tikzpicture}}};
\node at  (0.385,0.32)  {\resizebox{1cm}{!}{\begin{tikzpicture}
\node (b) [circle,fill=ForestGreen,draw=black,minimum size=0.1cm,inner sep=0.2mm,label={[xshift=0.22cm,yshift=-0.25cm]{\footnotesize $u$}}] at (0,0) {};
\node (c) [circle,fill=ForestGreen,draw=black,minimum size=0.1cm,inner sep=0.2mm,label={[xshift=0.22cm,yshift=-0.27cm]{\footnotesize $y$}}] at (0,0.35) {};
 \node (d) [circle,fill=ForestGreen,draw=black,minimum size=0.1cm,inner sep=0.2mm,label={[xshift=0.22cm,yshift=-0.25cm]{\footnotesize $x$}}] at (0,0.7) {};
  \node (a) [circle,fill=ForestGreen,draw=black,minimum size=0.1cm,inner sep=0.2mm,label={[xshift=0.22cm,yshift=-0.25cm]{\footnotesize $v$}}] at (0,1.05) {};
 \draw[thick]  (b)--(c)--(d)--(a);\end{tikzpicture}}};
\node at   (-0.273,0.255) {\resizebox{1cm}{!}{\begin{tikzpicture}
\node (b) [circle,fill=ForestGreen,draw=black,minimum size=0.1cm,inner sep=0.2mm,label={[xshift=-0.22cm,yshift=-0.25cm]{\footnotesize $v$}}] at (0,0) {};
\node (c) [circle,fill=ForestGreen,draw=black,minimum size=0.1cm,inner sep=0.2mm,label={[xshift=-0.22cm,yshift=-0.27cm]{\footnotesize $y$}}] at (0,0.35) {};
 \node (d) [circle,fill=ForestGreen,draw=black,minimum size=0.1cm,inner sep=0.2mm,label={[xshift=-0.22cm,yshift=-0.25cm]{\footnotesize $u$}}] at (0,0.7) {};
  \node (a) [circle,fill=ForestGreen,draw=black,minimum size=0.1cm,inner sep=0.2mm,label={[xshift=-0.22cm,yshift=-0.25cm]{\footnotesize $x$}}] at (0,1.05) {};
 \draw[thick]  (b)--(c)--(d)--(a);\end{tikzpicture}}};
\node at   (-0.385,0.32) {\resizebox{1cm}{!}{\begin{tikzpicture}
\node (b) [circle,fill=ForestGreen,draw=black,minimum size=0.1cm,inner sep=0.2mm,label={[xshift=-0.22cm,yshift=-0.25cm]{\footnotesize $v$}}] at (0,0) {};
\node (c) [circle,fill=ForestGreen,draw=black,minimum size=0.1cm,inner sep=0.2mm,label={[xshift=-0.22cm,yshift=-0.27cm]{\footnotesize $y$}}] at (0,0.35) {};
 \node (d) [circle,fill=ForestGreen,draw=black,minimum size=0.1cm,inner sep=0.2mm,label={[xshift=-0.22cm,yshift=-0.25cm]{\footnotesize $x$}}] at (0,0.7) {};
  \node (a) [circle,fill=ForestGreen,draw=black,minimum size=0.1cm,inner sep=0.2mm,label={[xshift=-0.22cm,yshift=-0.25cm]{\footnotesize $u$}}] at (0,1.05) {};
 \draw[thick]  (b)--(c)--(d)--(a);\end{tikzpicture}}};
\node at  (0.2325,0.1625)  {\resizebox{1cm}{!}{\begin{tikzpicture}
\node (b) [circle,fill=ForestGreen,draw=black,minimum size=0.1cm,inner sep=0.2mm,label={[xshift=0.22cm,yshift=-0.25cm]{\footnotesize $u$}}] at (0,0) {};
\node (c) [circle,fill=ForestGreen,draw=black,minimum size=0.1cm,inner sep=0.2mm,label={[xshift=0.22cm,yshift=-0.25cm]{\footnotesize $v$}}] at (0,0.35) {};
 \node (d) [circle,fill=ForestGreen,draw=black,minimum size=0.1cm,inner sep=0.2mm,label={[xshift=0.22cm,yshift=-0.27cm]{\footnotesize $y$}}] at (0,0.7) {};
  \node (a) [circle,fill=ForestGreen,draw=black,minimum size=0.1cm,inner sep=0.2mm,label={[xshift=0.22cm,yshift=-0.25cm]{\footnotesize $x$}}] at (0,1.05) {};
 \draw[thick]  (b)--(c)--(d)--(a);\end{tikzpicture}}};
\node at  (-0.2325,0.1625)  {\resizebox{1cm}{!}{\begin{tikzpicture}
\node (b) [circle,fill=ForestGreen,draw=black,minimum size=0.1cm,inner sep=0.2mm,label={[xshift=-0.22cm,yshift=-0.25cm]{\footnotesize $v$}}] at (0,0) {};
\node (c) [circle,fill=ForestGreen,draw=black,minimum size=0.1cm,inner sep=0.2mm,label={[xshift=-0.22cm,yshift=-0.25cm]{\footnotesize $u$}}] at (0,0.35) {};
 \node (d) [circle,fill=ForestGreen,draw=black,minimum size=0.1cm,inner sep=0.2mm,label={[xshift=-0.22cm,yshift=-0.27cm]{\footnotesize $y$}}] at (0,0.7) {};
  \node (a) [circle,fill=ForestGreen,draw=black,minimum size=0.1cm,inner sep=0.2mm,label={[xshift=-0.22cm,yshift=-0.25cm]{\footnotesize $x$}}] at (0,1.05) {};
 \draw[thick]  (b)--(c)--(d)--(a);\end{tikzpicture}}};
\node at  (-0.22,-0.065)   {\resizebox{1cm}{!}{\begin{tikzpicture}
\node (b) [circle,fill=ForestGreen,draw=black,minimum size=0.1cm,inner sep=0.2mm,label={[xshift=-0.22cm,yshift=-0.25cm]{\footnotesize $v$}}] at (0,0) {};
\node (c) [circle,fill=ForestGreen,draw=black,minimum size=0.1cm,inner sep=0.2mm,label={[xshift=-0.22cm,yshift=-0.25cm]{\footnotesize $u$}}] at (0,0.35) {};
 \node (d) [circle,fill=ForestGreen,draw=black,minimum size=0.1cm,inner sep=0.2mm,label={[xshift=-0.22cm,yshift=-0.25cm]{\footnotesize $x$}}] at (0,0.7) {};
  \node (a) [circle,fill=ForestGreen,draw=black,minimum size=0.1cm,inner sep=0.2mm,label={[xshift=-0.22cm,yshift=-0.27cm]{\footnotesize $y$}}] at (0,1.05) {};
 \draw[thick]  (b)--(c)--(d)--(a);\end{tikzpicture}}};
\node at  (0.22,-0.065)  {\resizebox{1cm}{!}{\begin{tikzpicture}
\node (b) [circle,fill=ForestGreen,draw=black,minimum size=0.1cm,inner sep=0.2mm,label={[xshift=0.22cm,yshift=-0.25cm]{\footnotesize $u$}}] at (0,0) {};
\node (c) [circle,fill=ForestGreen,draw=black,minimum size=0.1cm,inner sep=0.2mm,label={[xshift=0.22cm,yshift=-0.25cm]{\footnotesize $v$}}] at (0,0.35) {};
 \node (d) [circle,fill=ForestGreen,draw=black,minimum size=0.1cm,inner sep=0.2mm,label={[xshift=0.22cm,yshift=-0.25cm]{\footnotesize $x$}}] at (0,0.7) {};
  \node (a) [circle,fill=ForestGreen,draw=black,minimum size=0.1cm,inner sep=0.2mm,label={[xshift=0.22cm,yshift=-0.27cm]{\footnotesize $y$}}] at (0,1.05) {};
 \draw[thick]  (b)--(c)--(d)--(a);\end{tikzpicture}}};
\node at  (-0.38,0)  {\resizebox{2.35cm}{!}{\begin{tikzpicture}
\node (b) [circle,fill=ForestGreen,draw=black,minimum size=0.1cm,inner sep=0.2mm,label={[xshift=-0.22cm,yshift=-0.25cm]{\footnotesize $v$}}] at (0,0) {};
\node (c) [circle,fill=ForestGreen,draw=black,minimum size=0.1cm,inner sep=0.2mm,label={[xshift=-0.22cm,yshift=-0.25cm]{\footnotesize $x$}}] at (0,0.35) {};
 \node (d) [circle,fill=ForestGreen,draw=black,minimum size=0.1cm,inner sep=0.2mm,label={[xshift=-0.22cm,yshift=-0.27cm]{\footnotesize $y$}}] at (-0.15,0.7) {};
  \node (a) [circle,fill=ForestGreen,draw=black,minimum size=0.1cm,inner sep=0.2mm,label={[xshift=0.22cm,yshift=-0.25cm]{\footnotesize $u$}}] at (0.15,0.7) {};
 \draw[thick]  (b)--(c)--(d);
 \draw[thick](c)--(a);\end{tikzpicture}}};
\node at  (0.38,0)  {\resizebox{2.35cm}{!}{\begin{tikzpicture}
\node (b) [circle,fill=ForestGreen,draw=black,minimum size=0.1cm,inner sep=0.2mm,label={[xshift=0.22cm,yshift=-0.25cm]{\footnotesize $u$}}] at (0,0) {};
\node (c) [circle,fill=ForestGreen,draw=black,minimum size=0.1cm,inner sep=0.2mm,label={[xshift=0.22cm,yshift=-0.25cm]{\footnotesize $x$}}] at (0,0.35) {};
 \node (d) [circle,fill=ForestGreen,draw=black,minimum size=0.1cm,inner sep=0.2mm,label={[xshift=-0.22cm,yshift=-0.25cm]{\footnotesize $v$}}] at (-0.15,0.7) {};
  \node (a) [circle,fill=ForestGreen,draw=black,minimum size=0.1cm,inner sep=0.2mm,label={[xshift=0.22cm,yshift=-0.27cm]{\footnotesize $y$}}] at (0.15,0.7) {};
 \draw[thick]  (b)--(c)--(d);
 \draw[thick](c)--(a);\end{tikzpicture}}};
\end{tikzpicture}}
\end{center}

  \smallskip

\indent Notice that we did not specify  the hypergraph for the general case of an $n$-dimensional hemiassociahedron.  Indeed, the question of the generalization of the $3$-dimensional hemiassociahedron to an arbitrary finite dimension has more than one possible answer. For example, we might define the hypergraph encoding the $n$-dimensional hemiassociahedron by ${\bf H}_{n+1}={\bf P}_{n}\cup \{\{y\},\{x_i,y\}\}$, where $y$ is different from all the vertices of the permutohedron hypergraph ${\bf P}_{n}$ and $x_i$ is one of the vertices of ${\bf P}_{n}$, but other possibilities exist as well. In the framework of strongly homotopy structures, an appropriate generalization should be such that the resulting family of hemiassociahedra is closed under the composition product of the structure in question. Finding such a  generalization  seems like an interesting task.

\smallskip

\section{The combinatorial homotopy theory for operads}\label{operads}
\noindent This section contains the combinatorial description of the minimal model ${\EuScript O}_{\infty}$ of the coloured operad ${\EuScript O}$ encoding non-symmetric non-unital reduced operads. The operad ${\EuScript O}$ is the quadratic coloured operad whose generators and relations presentation is given by the ``non-symmetric portion" of   \cite[Definition 5]{DV}, where the coloured operad encoding  symmetric operads is defined. This definition describes ${\EuScript O}$ in terms of {\em composite trees} (i.e. binary  trees whose vertices encode the $\circ_i$ operations) and grafting. Under the name $\mbox{PsOpd}$, and  by specifying its spaces of operations, the operad ${\EuScript O}$  is defined earlier in \cite[Definition 4.1]{VdL},  where it is proven to be self-dual Koszul, and where ${\EuScript O}_{\infty}$ is subsequently defined as  the cobar construction $\Omega\,\mbox{PsOpd}^{\ac}$ of the cooperad $\mbox{PsOpd}^{\ac}$. This alternative characterizatiom describes ${\EuScript O}$ in terms of {\em operadic trees} (i.e., trees with vertices indexed bijectively by $[k]$)  and substitution. The same style of the definition of ${\EuScript O}$ can also be found in \cite[Example 1.5.6]{BM2}, where arity $0$ and units are additionally allowed.

 \medskip

We start this section by recalling and relating in \S\ref{theoperad} the two equivalent definitions of the operad ${\EuScript O}$. The combinatorial description of the minimal  model ${\EuScript O}_{\infty}$ that we  construct in \S \ref{oinf} will be directly tied to  the   characterization of ${\EuScript O}$ given by \cite[Definition 4.1]{VdL}. To each operadic tree ${\cal T}$ with $k$ vertices, we shall associate in a particular way a hypergraph ${\bf H}_{\cal T}$ with $k-1$ vertices, in such a way that  the faces of the abstract  polytope ${\cal A}({\bf H}_{\cal T})$ become the operations of ${\EuScript O}_{\infty}$ that replace (or split) ${\cal T}$, and that the order relation of ${\cal A}({\bf H}_{\cal T})$ determines the differential of ${\EuScript O}_{\infty}$. As further generalizations of our construction, in \S\ref{Wconstruction}, we  introduce a cubical subdivision of  operadic polytopes, obtaining in this way  precisely the  combinatorial Boardman-Vogt-Berger-Moerdijk resolution of ${\EuScript O}$, i.e. the $W$-construction for coloured operads, introduced   in \cite{BM2}, applied on ${\EuScript O}$. Finally, in \S\ref{cyc}, by switching from operadic trees to {\em cyclic operadic trees}, we obtain the combinatorial description of the minimal model of  the coloured  operad ${\EuScript C}$ encoding non-symmetric non-unital reduced cyclic operads.
 \smallskip
 
\subsection{Operads  as algebras over the colored operad ${\EuScript O}$}\label{theoperad}

\noindent In this section, we recall from \cite[Definition 5]{DV} and  \cite[Definition 4.1]{VdL} the two definitions of the coloured operad ${\EuScript O}$ encoding non-symmetric operads. We relate these two characterizations through a correspondence between the underlying formalisms of composite and operadic trees.

\subsubsection{The operad ${\EuScript O}$ in terms of composite trees} Below is the  ``non-symmetric portion" of   \cite[Definition 5]{DV}, obtained from \cite[Definition 5]{DV} by leaving out the generators encoding the action of the symmetric groups. Note that the resulting operad ${\EuScript O}$ (called ${\EuScript O}_{\it ns}$ in \cite{DV}) itself remains a symmetric coloured operad. As a final remark before the definition, we note that ${\EuScript O}$ will incipiently be defined as  a coloured operad in the category of sets, and that it is turned into a dg operad by the   strong symmetric monoidal functor sending a set $X$ to the direct sum $\bigoplus_{x\in X} {\Bbbk}$.

\smallskip

\begin{definition}\label{1} The coloured operad ${\EuScript O}$ is the   ${\mathbb N}$-coloured operad defined by ${\EuScript O}={\cal T}_{\mathbb N}(E)/(R)$, where the set of generators $E$ is given by   binary operations 
\smallskip
   \begin{center} \raisebox{1.8em}{${E}(n,k;n+k-1)=\Bigg\{$}
\raisebox{0.5em}{\resizebox{!}{1.3cm}{\begin{tikzpicture}  
\node(n) [rectangle,draw=none,minimum size=0mm,inner sep=0cm] at (-0.3,0.6) {\tiny $1$};
\node(k) [rectangle,draw=none,minimum size=0mm,inner sep=0cm] at (0.3,0.6) {\tiny $2$};
\node(n) [rectangle,draw=none,minimum size=0mm,inner sep=0cm] at (-0.4,0.4) {\tiny $n$};
\node(k) [rectangle,draw=none,minimum size=0mm,inner sep=0cm] at (0.4,0.4) {\tiny $k$};
\node(k) [rectangle,draw=none,minimum size=0mm,inner sep=0cm] at (0.5,-0.45) {\tiny $n\!+\!k\!-\!1$};
\node(j) [rectangle,draw=none,minimum size=0mm,inner sep=0cm] at (0,-0.5) {};
 \node(i) [circle,draw=black,thick,minimum size=0.2mm,inner sep=0.05cm] at (0,0) {$i$}; 
\draw[-,thick] (i)--(-0.3,0.45);\draw[-,thick] (i)--(0.3,0.45);
\draw[-,thick] (i)--(j);\end{tikzpicture}}} \raisebox{1.8em}{, \enspace $1\leq i\leq n$ \Bigg\} $\cup$\, $\Bigg\{$}
\raisebox{0.5em}{\resizebox{!}{1.3cm}{\begin{tikzpicture}  
\node(n) [rectangle,draw=none,minimum size=0mm,inner sep=0cm] at (-0.3,0.6) {\tiny $2$};
\node(k) [rectangle,draw=none,minimum size=0mm,inner sep=0cm] at (0.3,0.6) {\tiny $1$};
\node(n) [rectangle,draw=none,minimum size=0mm,inner sep=0cm] at (-0.4,0.4) {\tiny $k$};
\node(k) [rectangle,draw=none,minimum size=0mm,inner sep=0cm] at (0.4,0.4) {\tiny $n$};
\node(k) [rectangle,draw=none,minimum size=0mm,inner sep=0cm] at (0.5,-0.45) {\tiny $n\!+\!k\!-\!1$};
\node(j) [rectangle,draw=none,minimum size=0mm,inner sep=0cm] at (0,-0.5) {};
 \node(i) [circle,draw=black,thick,minimum size=0.2mm,inner sep=0.0325cm] at (0,0) {\small $j$}; 
\draw[-,thick] (i)--(-0.3,0.45);\draw[-,thick] (i)--(0.3,0.45);
\draw[-,thick] (i)--(j);\end{tikzpicture}}} \raisebox{1.8em}{, \enspace $1\leq j\leq k$ \Bigg\},}  
\end{center}
\smallskip
for $n,k\geq 1$, equipped with the action of the transposition $(21)$ that sends \raisebox{-1em}{\resizebox{!}{1.3cm}{\begin{tikzpicture}  
\node(n) [rectangle,draw=none,minimum size=0mm,inner sep=0cm] at (-0.3,0.6) {\tiny $1$};
\node(k) [rectangle,draw=none,minimum size=0mm,inner sep=0cm] at (0.3,0.6) {\tiny $2$};
\node(n) [rectangle,draw=none,minimum size=0mm,inner sep=0cm] at (-0.4,0.4) {\tiny $n$};
\node(k) [rectangle,draw=none,minimum size=0mm,inner sep=0cm] at (0.4,0.4) {\tiny $k$};
\node(k) [rectangle,draw=none,minimum size=0mm,inner sep=0cm] at (0.5,-0.45) {\tiny $n\!+\!k\!-\!1$};
\node(j) [rectangle,draw=none,minimum size=0mm,inner sep=0cm] at (0,-0.5) {};
 \node(i) [circle,draw=black,thick,minimum size=0.2mm,inner sep=0.05cm] at (0,0) {$i$}; 
\draw[-,thick] (i)--(-0.3,0.45);\draw[-,thick] (i)--(0.3,0.45);
\draw[-,thick] (i)--(j);\end{tikzpicture}}}\, to  \raisebox{-1em}{\resizebox{!}{1.3cm}{\begin{tikzpicture}  
\node(n) [rectangle,draw=none,minimum size=0mm,inner sep=0cm] at (-0.3,0.6) {\tiny $2$};
\node(k) [rectangle,draw=none,minimum size=0mm,inner sep=0cm] at (0.3,0.6) {\tiny $1$};
\node(n) [rectangle,draw=none,minimum size=0mm,inner sep=0cm] at (-0.4,0.4) {\tiny $n$};
\node(k) [rectangle,draw=none,minimum size=0mm,inner sep=0cm] at (0.41,0.4) {\tiny $k$};
\node(k) [rectangle,draw=none,minimum size=0mm,inner sep=0cm] at (0.5,-0.45) {\tiny $n\!+\!k\!-\!1$};
\node(j) [rectangle,draw=none,minimum size=0mm,inner sep=0cm] at (0,-0.5) {};
 \node(i) [circle,draw=black,thick,minimum size=0.2mm,inner sep=0.05cm] at (0,0) {$i$}; 
\draw[-,thick] (i)--(-0.3,0.45);\draw[-,thick] (i)--(0.3,0.45);
\draw[-,thick] (i)--(j);\end{tikzpicture}}}, and the set $R$ of relations    by relations 
\smallskip
\begin{center}
\begin{tabular}{clccl}
\hypertarget{A1}{\texttt{(A1)}}  & $$\raisebox{-8 ex}{\resizebox{!}{2.5cm}{\begin{tikzpicture}  
\node(1) [rectangle,draw=none,minimum size=0mm,inner sep=0cm] at (-0.85,-0.0) {\tiny $1$};
\node(2) [rectangle,draw=none,minimum size=0mm,inner sep=0cm] at (-0.32,0.6) {\tiny $2$};
\node(3) [rectangle,draw=none,minimum size=0mm,inner sep=0cm] at (0.32,0.6) {\tiny $3$};
\node(n) [rectangle,draw=none,minimum size=0mm,inner sep=0cm] at (-0.42,0.4) {\tiny $n$};
\node(k) [rectangle,draw=none,minimum size=0mm,inner sep=0cm] at (0.42,0.4) {\tiny $k$};
\node(k) [rectangle,draw=none,minimum size=0mm,inner sep=0cm] at (-0.95,-0.25) {\tiny $m$};
\node(j) [circle,draw=black,thick,minimum size=0.2mm,inner sep=0.07cm]  at (-0.45,-1) {\small $i$};
 \node(i) [circle,draw=black,thick,minimum size=0.2mm,inner sep=0.035cm] at (0,0) {\small $j$}; 
\draw[-,thick] (i)--(-0.3,0.45);\draw[-,thick] (i)--(0.3,0.45);
\draw[-,thick] (j)--(-0.85,-0.15);
\draw[-,thick] (i)--(j)--(-0.45,-1.6);\end{tikzpicture}}} \enspace =  \enspace  \raisebox{-8 ex}{\resizebox{!}{2.5cm}{\begin{tikzpicture}  
\node(1) [rectangle,draw=none,minimum size=0mm,inner sep=0cm] at (-0.3,0.6) {\tiny $1$};
\node(2) [rectangle,draw=none,minimum size=0mm,inner sep=0cm] at (0.3,0.6) {\tiny $2$};
\node(3) [rectangle,draw=none,minimum size=0mm,inner sep=0cm] at (0.87,0) {\tiny $3$};
\node(n) [rectangle,draw=none,minimum size=0mm,inner sep=0cm] at (-0.42,0.4) {\tiny $m$};
\node(k) [rectangle,draw=none,minimum size=0mm,inner sep=0cm] at (0.42,0.4) {\tiny $n$};
\node(k) [rectangle,draw=none,minimum size=0mm,inner sep=0cm] at (0.925,-0.25) {\tiny $k$};
\node(j) [circle,draw=black,thick,minimum size=0mm,inner sep=0.02cm] at (0.45,-1) {\small $j'$};
 \node(i) [circle,draw=black,thick,minimum size=0.2mm,inner sep=0.07cm] at (0,0) {\small $i$}; 
\draw[-,thick] (i)--(-0.3,0.45);\draw[-,thick] (i)--(0.3,0.45);
\draw[-,thick] (j)--(0.85,-0.15);
\draw[-,thick] (i)--(j)--(0.45,-1.6);\end{tikzpicture}}} &\quad\quad\quad\quad\quad& \hypertarget{A2}{\texttt{(A2)}} & \raisebox{-8 ex}{\resizebox{!}{2.5cm}{\begin{tikzpicture}  
\node(1) [rectangle,draw=none,minimum size=0mm,inner sep=0cm] at (-0.3,0.6) {\tiny $1$};
\node(2) [rectangle,draw=none,minimum size=0mm,inner sep=0cm] at (0.3,0.6) {\tiny $2$};
\node(3) [rectangle,draw=none,minimum size=0mm,inner sep=0cm] at (0.87,0) {\tiny $3$};
\node(n) [rectangle,draw=none,minimum size=0mm,inner sep=0cm] at (-0.42,0.4) {\tiny $m$};
\node(k) [rectangle,draw=none,minimum size=0mm,inner sep=0cm] at (0.42,0.4) {\tiny $n$};
\node(k) [rectangle,draw=none,minimum size=0mm,inner sep=0cm] at (0.925,-0.25) {\tiny $k$};
\node(j) [circle,draw=black,thick,minimum size=0mm,inner sep=0.07cm] at (0.45,-1) {\small $i$};
 \node(i) [circle,draw=black,thick,minimum size=0.2mm,inner sep=0.035cm] at (0,0) {$\small j$}; 
\draw[-,thick] (i)--(-0.3,0.45);\draw[-,thick] (i)--(0.3,0.45);
\draw[-,thick] (j)--(0.85,-0.15);
\draw[-,thick] (i)--(j)--(0.45,-1.6);\end{tikzpicture}}}  \enspace = \enspace  \raisebox{-8 ex}{\resizebox{!}{2.5cm}{\begin{tikzpicture}  
\node(1) [rectangle,draw=none,minimum size=0mm,inner sep=0cm] at (-0.3,0.6) {\tiny $1$};
\node(2) [rectangle,draw=none,minimum size=0mm,inner sep=0cm] at (0.3,0.6) {\tiny $3$};
\node(3) [rectangle,draw=none,minimum size=0mm,inner sep=0cm] at (0.87,0) {\tiny $2$};
\node(n) [rectangle,draw=none,minimum size=0mm,inner sep=0cm] at (-0.42,0.4) {\tiny $m$};
\node(k) [rectangle,draw=none,minimum size=0mm,inner sep=0cm] at (0.42,0.4) {\tiny $k$};
\node(k) [rectangle,draw=none,minimum size=0mm,inner sep=0cm] at (0.925,-0.25) {\tiny $n$};
\node(j) [circle,draw=black,thick,minimum size=0mm,inner sep=0.02cm] at (0.45,-1) {\small $j'$};
 \node(i) [circle,draw=black,thick,minimum size=0.2mm,inner sep=0.07cm] at (0,0) {\small $i$}; 
\draw[-,thick] (i)--(-0.3,0.45);\draw[-,thick] (i)--(0.3,0.45);
\draw[-,thick] (j)--(0.85,-0.15);
\draw[-,thick] (i)--(j)--(0.45,-1.6);\end{tikzpicture}}} 
\end{tabular}
\end{center}
where, in \hyperlink{A1}{\texttt{(A1)}}, $j'=j+i-1$, and, in \hyperlink{A2}{\texttt{(A2)}},  it is assumed that $i<j$ and $j'=j+k-1$.
\end{definition}
\smallskip
\begin{remark}
The operad ${\EuScript O}$ is a dg coloured operad: for each $k\geq 2$, the vector  space\linebreak ${\cal T}_{\mathbb N}(E)(n_1,\dots,n_k;n)/R(n_1,\dots,n_k;n)$ is concentrated in degree zero, and the differential   is trivial. Therefore, the homology of $${\EuScript O}(n_1,\dots,n_k;n)=\bigoplus_{m\geq 0} ({\cal T}_{\mathbb N}(E)(n_1,\dots,n_k;n)/R(n_1,\dots,n_k;n))_m$$ is trivial for $m\neq 0$, while, for $m=0$, we have $H_0({\EuScript O}(n_1,\dots,n_k;n),0)\cong {\EuScript O}(n_1,\dots,n_k;n)$.
\end{remark}
\smallskip

Let $({\it End}(A),d_{\it End})$ be the  dg coloured endomorphism operad on a dg ${\mathbb N}$-module\linebreak  $A=\{(A(n),d_{A(n)})\}_{n\geq 1}$, i.e. the ${\mathbb N}$-coloured dg operad defined by
$${\it End}(A)(n_1,\dots,n_k;n):= {\mbox{Hom}}(A(n_1)\otimes\cdots\otimes A(n_k);A(n)),\quad\quad k\geq 1,$$ 
where ${\mbox{Hom}}_p(A(n_1)\otimes\cdots\otimes A(n_k);A(n))$ is the vector space of homogeneous degree $p$ linear maps $f:A(n_1)\otimes\cdots\otimes A(n_k)\rightarrow A(n)$, with the partial composition operations (resp. the action of the symmetric groups)  induced by substitution
(resp. permutation respecting the Koszul sign rule) of the tensor factors, and with the differential $d_{\it End}$ defined on $f\in {\mbox{Hom}}_p(A(n_1)\otimes\cdots\otimes A(n_k);A(n))$ by $$d_{\it End}(f):=d_{A(n)}\circ f-(-1)^p \sum_{i=1}^k f\circ_i d_{A(n_i)}.$$
 (Note that when the formula defining $d_{\it End}$ is applied to elements, additional signs appear due to the Koszul sign rule.)
 \smallskip

\begin{lemma}\label{algoperad}
Algebras over the coloured operad ${\EuScript O}$ are  non-unital, non-symmetric, reduced dg operads.
\end{lemma}
\begin{proof}

By definition, an ${{\EuScript O}}$-algebra is a degree 0 homomorphism of ${\mathbb N}$-coloured dg operads $\chi:({{\EuScript O}},0)\rightarrow ({\it End}(A),d_{\it End})$. Therefore, 
an ${{\EuScript O}}$-algebra  is a  dg ${\mathbb N}$-module  $(A,d)=\{(A(n),d_{A(n)})\}_{n\geq 1}$  endowed with operations 
 $$\raisebox{-2.75ex}{\resizebox{!}{1.3cm}{\begin{tikzpicture}  
\node(n) [rectangle,draw=none,minimum size=0mm,inner sep=0cm] at (-0.3,0.6) {\tiny $1$};
\node(k) [rectangle,draw=none,minimum size=0mm,inner sep=0cm] at (0.3,0.6) {\tiny $2$};
\node(n) [rectangle,draw=none,minimum size=0mm,inner sep=0cm] at (-0.4,0.4) {\tiny $n$};
\node(k) [rectangle,draw=none,minimum size=0mm,inner sep=0cm] at (0.4,0.4) {\tiny $k$};
\node(k) [rectangle,draw=none,minimum size=0mm,inner sep=0cm] at (0.5,-0.45) {\tiny $n\!+\!k\!-\!1$};
\node(j) [rectangle,draw=none,minimum size=0mm,inner sep=0cm] at (0,-0.5) {};
 \node(i) [circle,draw=black,thick,minimum size=0.2mm,inner sep=0.05cm] at (0,0) {$i$}; 
\draw[-,thick] (i)--(-0.3,0.45);\draw[-,thick] (i)--(0.3,0.45);
\draw[-,thick] (i)--(j);\end{tikzpicture}}} :\, A(n)\otimes A(k)\rightarrow A(n+k-1)$$ 
satisfying the obvious associativity axioms, whereby  the equality $\chi\circ 0=d_{{\it End}}\circ \chi$  satisfied by $\chi$ guarantees that those operations are compatible with $d$.
\end{proof}
 \smallskip
\subsubsection{The operad ${\EuScript O}$ in terms of operadic trees}\label{operadictrees} We next recall from \cite[Definition 4.1]{VdL} and \cite[Example 1.5.6]{BM2}  the characterization of ${\EuScript O}$ in terms of   operadic trees.

\smallskip

Denote, for $k\geq 2$, $n_1,\dots,n_k\geq 1$, and $n=\big(\sum^k_{i=1} n_i\big)-k+1$, with ${\tt Tree}(n_1,\dots,n_k;n)$ the set of equivalence classes of pairs $(\cal T,\sigma)$, where ${\cal T}\in{\tt Tree}(n)$ has $k$ vertices and $\sigma :[k]\rightarrow v({\cal T})$ is a bijection such that the
vertex $\sigma(i)$ has $n_i$ inputs, under the equivalence relation defined by:
\begin{center} $({\cal T}_1,\sigma_{1})\sim ({\cal T}_2,\sigma_{2})$  if  there exists an isomorphim $\varphi: {\cal T}_1\rightarrow {\cal T}_2$, such that $\varphi\circ\sigma_{1}=\sigma_{2}$.
\end{center}
(In the equality $\varphi\circ\sigma_{1}=\sigma_{2}$ above, we abuse the notation by writing $\varphi$ for what is actually the vertex component of $\varphi$. We shall continue with this practice whenever specifying  compatibilities involving tree isomorphisms.) We refer to pairs $({\cal T},\sigma)$ as {\em operadic trees}.

\smallskip

\begin{lemma}\label{representation}
A ${\Bbbk}$-linear basis of the vector space ${\EuScript O}(n_1,\dots,n_k;n)$ is given by the equivalence classes of operadic trees from ${\tt Tree}(n_1,\dots,n_k;n)$.
\end{lemma}
\begin{proof}
By \cite[Proposition 3]{DV},  the coloured operad ${\EuScript O}$ is spanned by   binary planar ${\mathbb N}$-coloured left combs  whose vertex decorations,  read from top to bottom, are nondecreasing, together with a labeling of the leaves with a permutation on the number of them.
 Formally, this basis is the set of normal forms of the confluent and terminating rewriting system obtained by orienting the relations \hyperlink{A1}{\texttt{(A1)}} and \hyperlink{A2}{\texttt{(A2)}} from left to right.
To each left comb 
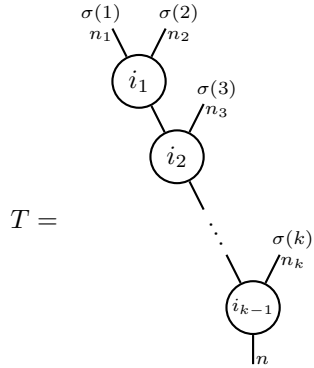
\begin{center} ${T}=$ 
\raisebox{-12 ex}{\begin{tikzpicture}
 \node(i_1) [circle,draw=black,thick,minimum size=7mm,inner sep=0.01cm] at (0,0) {\scriptsize $i_{k-1}$}; 
 \node(i_2) [circle,draw=black,thick,minimum size=7mm,inner sep=0.05cm] at (-1,2) {$i_2$}; 
 \node(i_3) [circle,draw=black,thick,minimum size=7mm,inner sep=0.05cm] at (-1.5,3) { $i_1$}; 
 \node (n_1)[circle,draw=none] at (-2,3.6) {\scriptsize $n_1$};
 \node (n_2)[circle,draw=none] at (-0.975,3.6) {\scriptsize $n_2$};
 \node (n_3)[circle,draw=none] at (-0.475,2.6) {\scriptsize $n_3$};
 \node (n_4)[circle,draw=none] at (0.525,0.6) {\scriptsize $n_k$};
  \node (n_1)[circle,draw=none] at (-2,3.9) {\scriptsize $\sigma(1)$};
 \node (n_2)[circle,draw=none] at (-0.975,3.9) {\scriptsize $\sigma(2)$};
 \node (n_3)[circle,draw=none] at (-0.475,2.9) {\scriptsize $\sigma(3)$};
 \node (n_4)[circle,draw=none] at (0.525,0.9) {\scriptsize $\sigma(k)$};
 \node (n)[rectangle,draw=none,inner sep=0.0cm] at (0.125,-0.7) {\scriptsize $n$};
\node(s)[rotate=117] at (-0.5,1) {\small $\dots$};
\draw[thick] (i_3)--(i_2);
\draw[thick] (-1.85,3.7)--(i_3)--(-1.15,3.7);
\draw[thick] (-0.65,1.3)--(i_2)--(-0.65,2.7);
\draw[thick] (-0.35,0.7)--(i_1)--(0.35,0.7);
\draw[thick] (0,-0.75)--(i_1);
\end{tikzpicture}  } 
\end{center}
such that $i_1\leq \cdots\leq i_{k-1}$, we associate an operadic
tree $\omega(T)$, as follows:  denote with $t_i$ the planar corolla with $n_i$ inputs, decorated with $\sigma(i)$, and define $$\omega(T)=(\cdots((t_1\circ_{i_1} t_2)\circ_{i_2} t_3)\cdots)\circ_{i_{k-1}} t_k ,$$ where $\circ_{i_j}$, $1\leq j\leq k-1$, denotes the grafting operation on rooted trees (that preserves vertex decorations).  In particular, the correspondence between the generators of ${\EuScript O}$ and operadic trees with two vertices is given by 
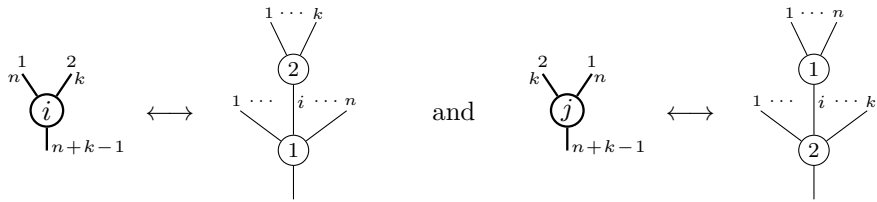
\begin{center}
 $\raisebox{-3ex}{\resizebox{!}{1.3cm}{\begin{tikzpicture}  
 \node(n) [rectangle,draw=none,minimum size=0mm,inner sep=0cm] at (-0.3,0.6) {\tiny $1$};
\node(k) [rectangle,draw=none,minimum size=0mm,inner sep=0cm] at (0.3,0.6) {\tiny $2$};
\node(n) [rectangle,draw=none,minimum size=0mm,inner sep=0cm] at (-0.4,0.4) {\tiny $n$};
\node(k) [rectangle,draw=none,minimum size=0mm,inner sep=0cm] at (0.4,0.4) {\tiny $k$};
\node(k) [rectangle,draw=none,minimum size=0mm,inner sep=0cm] at (0.5,-0.45) {\tiny $n\!+\!k\!-\!1$};
\node(j) [rectangle,draw=none,minimum size=0mm,inner sep=0cm] at (0,-0.5) {};
 \node(i) [circle,draw=black,thick,minimum size=0.2mm,inner sep=0.05cm] at (0,0) {$i$}; 
\draw[-,thick] (i)--(-0.3,0.45);\draw[-,thick] (i)--(0.3,0.45);
\draw[-,thick] (i)--(j);\end{tikzpicture}}}\enspace  \longleftrightarrow$  $\enspace\raisebox{-7ex}{\begin{tikzpicture}
 \node (s1)[circle,draw=none] at (-0.4,0.65) {\tiny $\dots$};
 \node (s2)[circle,draw=none] at (0.45,0.65) {\tiny $\dots$};
 \node (s5)[circle,draw=none] at (0,1.82) {\tiny $\dots$};
 \node (1)[circle,draw=none] at (-0.75,0.65) {\tiny $1$};
 \node (m)[circle,draw=none] at (0.75,0.65) {\tiny $n$};
 \node (2)[circle,draw=none] at (-0.32,1.82) {\tiny $1$};
 \node (n)[circle,draw=none] at (0.32,1.82) {\tiny $k$};
    \node (E)[circle,draw=black,minimum size=4mm,inner sep=0.1mm] at (0,0) {\footnotesize $1$};
    \node (F) [circle,draw=black,minimum size=4mm,inner sep=0.1mm] at (0,1.075) {\footnotesize $2$};
 \draw[-] (0.7,0.525) -- (E)--(-0.7,0.525); 
 \draw[-] (E)--(0,-0.65); 
 \draw[-] (0.3,1.7) -- (F)--(-0.3,1.7); 
    \draw[-] (E)--(F) node [midway,right,xshift=-0.075cm,yshift=0.1cm] {\scriptsize $i$};
\end{tikzpicture}}$ \quad\quad and\quad\quad  $\raisebox{-3ex}{\resizebox{!}{1.3cm}{\begin{tikzpicture}  
 \node(n) [rectangle,draw=none,minimum size=0mm,inner sep=0cm] at (-0.3,0.6) {\tiny $2$};
\node(k) [rectangle,draw=none,minimum size=0mm,inner sep=0cm] at (0.3,0.6) {\tiny $1$};
\node(n) [rectangle,draw=none,minimum size=0mm,inner sep=0cm] at (-0.4,0.4) {\tiny $k$};
\node(k) [rectangle,draw=none,minimum size=0mm,inner sep=0cm] at (0.4,0.4) {\tiny $n$};
\node(k) [rectangle,draw=none,minimum size=0mm,inner sep=0cm] at (0.5,-0.45) {\tiny $n\!+\!k\!-\!1$};
\node(j) [rectangle,draw=none,minimum size=0mm,inner sep=0cm] at (0,-0.5) {};
 \node(i) [circle,draw=black,thick,minimum size=0.2mm,inner sep=0.0325cm] at (0,0) {\small $j$}; 
\draw[-,thick] (i)--(-0.3,0.45);\draw[-,thick] (i)--(0.3,0.45);
\draw[-,thick] (i)--(j);\end{tikzpicture}}}\enspace  \longleftrightarrow$  $\enspace\raisebox{-7ex}{\begin{tikzpicture}
 \node (s1)[circle,draw=none] at (-0.4,0.65) {\tiny $\dots$};
 \node (s2)[circle,draw=none] at (0.45,0.65) {\tiny $\dots$};
 \node (s5)[circle,draw=none] at (0,1.82) {\tiny $\dots$};
 \node (1)[circle,draw=none] at (-0.75,0.65) {\tiny $1$};
 \node (m)[circle,draw=none] at (0.75,0.65) {\tiny $k$};
 \node (2)[circle,draw=none] at (-0.32,1.82) {\tiny $1$};
 \node (n)[circle,draw=none] at (0.32,1.82) {\tiny $n$};
    \node (E)[circle,draw=black,minimum size=4mm,inner sep=0.1mm] at (0,0) {\footnotesize $2$};
    \node (F) [circle,draw=black,minimum size=4mm,inner sep=0.1mm] at (0,1.075) {\footnotesize $1$};
 \draw[-] (0.7,0.525) -- (E)--(-0.7,0.525); 
 \draw[-] (E)--(0,-0.65); 
 \draw[-] (0.3,1.7) -- (F)--(-0.3,1.7); 
    \draw[-] (E)--(F) node [midway,right,xshift=-0.075cm,yshift=0.1cm] {\scriptsize $i$};
\end{tikzpicture}}$
\end{center}
Notice that   the fact that the vertex decorations of ${T}$ are nondecreasing means that $\omega(T)$ is defined by grafting the corollas in the {\em left-recursive way}, i.e.  from bottom to top and from left to right.   This property is used for the definition of the inverse of $\omega$: a composite tree is recovered by traversing an operadic tree in the left-recursive manner,  as we illustrate in Example \ref{compositionexample} that follows.
 \end{proof}

\smallskip
The partial composition operations $\circ_i$  of the coloured operad ${\EuScript O}$ translate to the basis given by Lemma \ref{representation} as follows: for $({\cal T}_1,\sigma_{1})\in {\EuScript O}(n_1,\dots,n_k;n)$ and $({\cal T}_2,\sigma_{2})\in {\EuScript O}(m_1,\dots,m_l;n_i)$, we have
 $$({\cal T}_1,\sigma_{1}) \circ_i ({\cal T}_2,\sigma_{2}) = ({\cal T}_1\bullet_i {\cal T}_2, \sigma_1\bullet_i\sigma_2),$$ 
where ${\cal T}_1\bullet_i {\cal T}_2$ is the planar rooted tree obtained by replacing the vertex $\sigma_{1}(i)$ (i.e. the vertex indexed by $i$) of ${\cal T}_1$ by the tree ${\cal T}_2$, identifying the $n_i$
inputs   of $\sigma_{1}(i)$ in ${\cal T}_1$ with the $n_i$
input
edges of ${\cal T}_2$ using the respective planar structures, and  $\sigma_1\bullet_i\sigma_2$ is   defined by $$(\sigma_1\bullet_i\sigma_2)(j)=\begin{cases} \sigma_{1}(j), \enspace  j<i \\ \sigma_{2}(j-i+1),\enspace   i\leq j\leq i+l-1 \\  \sigma_{1}(j-l+1), \enspace i+l\leq j\leq k+l-1.\end{cases}$$  
Indeed,  it can be shown that 
$$ ({\cal T}_1\bullet_i {\cal T}_2, \sigma_1\bullet_i\sigma_2)=\omega({\it nf}(\omega^{-1}({\cal T}_1,\sigma_1)\circ_i \omega^{-1}({\cal T}_2,\sigma_2))),$$ where  $\omega$ is the bijection from the proof of Lemma \ref{representation}, and  ${\it nf}$ is the normal form function of the rewriting system generated by orienting the relations \hyperlink{A1}{\texttt{(A1)}} and \hyperlink{A2}{\texttt{(A2)}} from left to right.   The action of the symmetric group is defined by $({\cal T},\sigma)^{\kappa}=({\cal T}, \sigma\circ\kappa)$.

\smallskip

The following example illustrates the correspondence between the partial composition operation of ${\EuScript O}$ in terms of composite trees and grafting, and operadic trees and substitution.
\smallskip
\begin{example}\label{compositionexample}
For operadic trees 
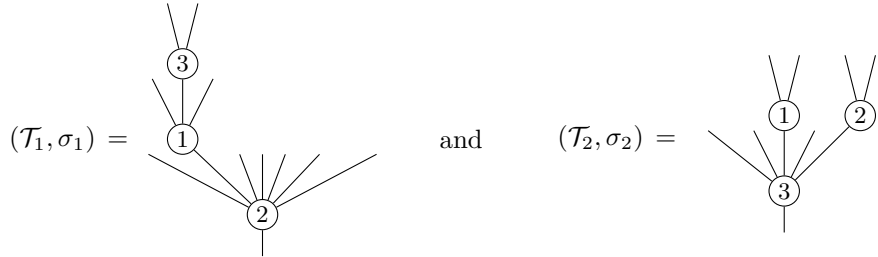
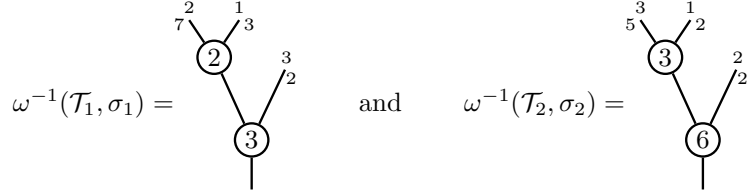
\begin{center}
\begin{tikzpicture}
  \node (t1)[circle,draw=none] at (-1.8,1) { $({\cal T}_1,\sigma_1)\,=$};
    \node (E)[circle,draw=black,minimum size=4mm,inner sep=0.1mm] at (0.75,0) {\small $2$};
    \node (F) [circle,draw=black,minimum size=4mm,inner sep=0.1mm] at (-0.3,1) {\small $1$};
    \node (A) [circle,draw=black,minimum size=4mm,inner sep=0.1mm] at (-0.3,2) {\small $3$};
 \draw[-] (-0.1,2.8) -- (A)--(-0.5,2.8); 
 \draw[-] (0.75,0.8) -- (E)--(-0.75,0.8); 
  \draw[-] (1.5,0.8) -- (E)--(2.25,0.8); 
    \draw[-] (0.45,0.8) -- (E)--(1.05,0.8); 
 \draw[-] (E)--(0.75,-0.55); 
 \draw[-] (0.1,1.8) -- (F)--(-0.7,1.8); 
    \draw[-] (E)--(F) node  {};
    \draw[-] (F)--(A) node {};
\end{tikzpicture}   \quad \quad \raisebox{4.1em}{and} \quad \quad  \raisebox{0.9em}{\begin{tikzpicture}
  \node (t1)[circle,draw=none] at (-2.2,0.7) { $({\cal T}_2,\sigma_2)\,=$};
    \node (E)[circle,draw=black,minimum size=4mm,inner sep=0.1mm] at (0,0) {\small $3$};
    \node (F) [circle,draw=black,minimum size=4mm,inner sep=0.1mm] at (0,1) {\small $1$};
    \node (A) [circle,draw=black,minimum size=4mm,inner sep=0.1mm] at (1,1) {\small $2$};
 \draw[-] (-0.2,1.8) -- (F)--(0.2,1.8);   
 \draw[-] (0.8,1.8) -- (A)--(1.2,1.8);   
  \draw[-] (-0.4,0.8)--(E)--(0,-0.55); 
 \draw[-] (0.4,0.8)--(E)--(-1,0.8); 
    \draw[-] (E)--(F) node {};
    \draw[-] (E)--(A) node  {};
   \end{tikzpicture}}
\end{center}  we  have

\begin{center}
\raisebox{3em}{$\omega^{-1}({\cal T}_1,\sigma_1)=$}\resizebox{!}{2.5cm}{\begin{tikzpicture}  
\node(1) [rectangle,draw=none,minimum size=0mm,inner sep=0cm] at (-0.3,0.6) {\tiny $2$};
\node(2) [rectangle,draw=none,minimum size=0mm,inner sep=0cm] at (0.3,0.6) {\tiny $1$};
\node(3) [rectangle,draw=none,minimum size=0mm,inner sep=0cm] at (0.87,0) {\tiny $3$};
\node(n) [rectangle,draw=none,minimum size=0mm,inner sep=0cm] at (-0.42,0.4) {\tiny $7$};
\node(k) [rectangle,draw=none,minimum size=0mm,inner sep=0.0cm] at (0.42,0.4) {\tiny $3$};
\node(k) [rectangle,draw=none,minimum size=0mm,inner sep=0cm] at (0.925,-0.25) {\tiny $2$};
\node(j) [circle,draw=black,thick,minimum size=0mm,inner sep=0.05cm] at (0.45,-1) {\small $3$};
 \node(i) [circle,draw=black,thick,minimum size=0.2mm,inner sep=0.05cm] at (0,0) {\small $2$}; 
\draw[-,thick] (i)--(-0.3,0.45);\draw[-,thick] (i)--(0.3,0.45);
\draw[-,thick] (j)--(0.85,-0.15);
\draw[-,thick] (i)--(j)--(0.45,-1.6);\end{tikzpicture}} \quad\quad \raisebox{3em}{and} \quad\quad \raisebox{3em}{$\omega^{-1}({\cal T}_2,\sigma_2)=$}\resizebox{!}{2.5cm}{\begin{tikzpicture}  
\node(1) [rectangle,draw=none,minimum size=0mm,inner sep=0cm] at (-0.3,0.6) {\tiny $3$};
\node(2) [rectangle,draw=none,minimum size=0mm,inner sep=0cm] at (0.3,0.6) {\tiny $1$};
\node(3) [rectangle,draw=none,minimum size=0mm,inner sep=0cm] at (0.87,0) {\tiny $2$};
\node(n) [rectangle,draw=none,minimum size=0mm,inner sep=0cm] at (-0.42,0.4) {\tiny $5$};
\node(k) [rectangle,draw=none,minimum size=0mm,inner sep=0.0cm] at (0.42,0.4) {\tiny $2$};
\node(k) [rectangle,draw=none,minimum size=0mm,inner sep=0cm] at (0.925,-0.25) {\tiny $2$};
\node(j) [circle,draw=black,thick,minimum size=0mm,inner sep=0.05cm] at (0.45,-1) {\small $6$};
 \node(i) [circle,draw=black,thick,minimum size=0.2mm,inner sep=0.05cm] at (0,0) {\small $3$}; 
\draw[-,thick] (i)--(-0.3,0.45);\draw[-,thick] (i)--(0.3,0.45);
\draw[-,thick] (j)--(0.85,-0.15);
\draw[-,thick] (i)--(j)--(0.45,-1.6);\end{tikzpicture}}
\end{center}
The normalizing sequence for  $\omega^{-1}({\cal T}_1,\sigma_1)\circ_2 \omega^{-1}({\cal T}_2,\sigma_2)$ is given by
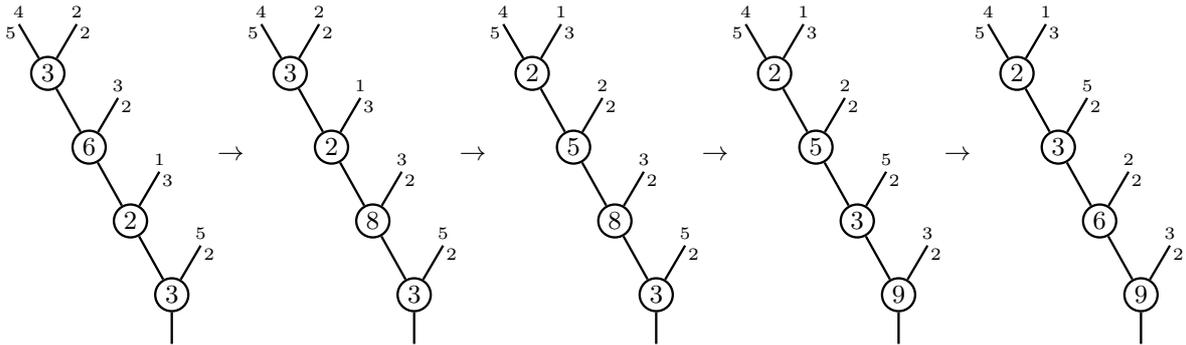
\begin{center}
\resizebox{!}{4.7cm}{\begin{tikzpicture}
 \node(i_0) [circle,draw=black,thick,minimum size=4mm,inner sep=0.05cm] at (0,0) {\small $3$}; 
 \node(i_1) [circle,draw=black,thick,minimum size=4mm,inner sep=0.05cm] at (-0.5,0.9) {\small $2$}; 
 \node(i_2) [circle,draw=black,thick,minimum size=4mm,inner sep=0.05cm] at (-1,1.8) {\small $6$}; 
 \node(i_3) [circle,draw=black,thick,minimum size=4mm,inner sep=0.05cm] at (-1.5,2.7) {\small  $3$}; 
 \node (n_1)[circle,draw=none] at (-1.95,3.2) {\tiny $5$};
 \node (n_2)[circle,draw=none] at (-1.05,3.2) {\tiny $2$};
 \node (n_3)[circle,draw=none] at (-0.55,2.3) {\tiny $2$};
 \node (n_4)[circle,draw=none] at (-0.05,1.4) {\tiny $3$};
 \node (n_5)[circle,draw=none] at (0.45,0.5) {\tiny $2$};
  \node (s_1)[circle,draw=none] at (-1.85,3.45) {\tiny $4$};
 \node (s_2)[circle,draw=none] at (-1.15,3.45) {\tiny $2$};
 \node (s_3)[circle,draw=none] at (-0.65,2.55) {\tiny $3$};
 \node (s_4)[circle,draw=none] at (-0.15,1.65) {\tiny $1$};
 \node (s_5)[circle,draw=none] at (0.35,0.75) {\tiny $5$};
\draw[thick] (i_3)--(i_2)--(i_1)--(i_0);
\draw[thick] (i_0)--(0,-0.6);
\draw[thick] (-1.85,3.3)--(i_3)--(-1.15,3.3);
\draw[thick] (i_2)--(-0.65,2.4);
\draw[thick] (i_1)--(-0.15,1.5);
\draw[thick] (i_0)--(0.35,0.6);
\end{tikzpicture}}  \hspace{-0.4cm} \raisebox{7em}{$\rightarrow$} \hspace{-0.4cm}  \resizebox{!}{4.7cm}{\begin{tikzpicture}
 \node(i_0) [circle,draw=black,thick,minimum size=4mm,inner sep=0.05cm] at (0,0) {\small $3$}; 
 \node(i_1) [circle,draw=black,thick,minimum size=4mm,inner sep=0.05cm] at (-0.5,0.9) {\small $8$}; 
 \node(i_2) [circle,draw=black,thick,minimum size=4mm,inner sep=0.05cm] at (-1,1.8) {\small $2$}; 
 \node(i_3) [circle,draw=black,thick,minimum size=4mm,inner sep=0.05cm] at (-1.5,2.7) {\small  $3$}; 
\node (n_1)[circle,draw=none] at (-1.95,3.2) {\tiny $5$};
 \node (n_2)[circle,draw=none] at (-1.05,3.2) {\tiny $2$};
 \node (n_3)[circle,draw=none] at (-0.55,2.3) {\tiny $3$};
 \node (n_4)[circle,draw=none] at (-0.05,1.4)  {\tiny $2$};
\node (n_5)[circle,draw=none] at (0.45,0.5) {\tiny $2$};
  \node (s_1)[circle,draw=none] at (-1.85,3.45) {\tiny $4$};
 \node (s_2)[circle,draw=none] at (-1.15,3.45) {\tiny $2$};
 \node (s_3)[circle,draw=none] at (-0.65,2.55) {\tiny $1$};
 \node (s_4)[circle,draw=none] at (-0.15,1.65) {\tiny $3$};
 \node (s_5)[circle,draw=none] at (0.35,0.75) {\tiny $5$};
\draw[thick] (i_3)--(i_2)--(i_1)--(i_0);
\draw[thick] (i_0)--(0,-0.6);
\draw[thick] (-1.85,3.3)--(i_3)--(-1.15,3.3);
\draw[thick] (i_2)--(-0.65,2.4);
\draw[thick] (i_1)--(-0.15,1.5);
\draw[thick] (i_0)--(0.35,0.6);
\end{tikzpicture}} \hspace{-0.4cm} \raisebox{7em}{$\rightarrow$} \hspace{-0.4cm}  \resizebox{!}{4.7cm}{\begin{tikzpicture}
 \node(i_0) [circle,draw=black,thick,minimum size=4mm,inner sep=0.05cm] at (0,0) {\small $3$}; 
 \node(i_1) [circle,draw=black,thick,minimum size=4mm,inner sep=0.05cm] at (-0.5,0.9) {\small $8$}; 
 \node(i_2) [circle,draw=black,thick,minimum size=4mm,inner sep=0.05cm] at (-1,1.8) {\small $5$}; 
 \node(i_3) [circle,draw=black,thick,minimum size=4mm,inner sep=0.05cm] at (-1.5,2.7) {\small  $2$}; 
 \node (n_1)[circle,draw=none] at (-1.95,3.2) {\tiny $5$};
 \node (n_2)[circle,draw=none] at (-1.05,3.2) {\tiny $3$};
 \node (n_3)[circle,draw=none] at (-0.55,2.3) {\tiny $2$};
 \node (n_4)[circle,draw=none] at (-0.05,1.4) {\tiny $2$};
 \node (n_5)[circle,draw=none] at (0.45,0.5) {\tiny $2$};
  \node (s_1)[circle,draw=none] at (-1.85,3.45) {\tiny $4$};
 \node (s_2)[circle,draw=none] at (-1.15,3.45) {\tiny $1$};
 \node (s_3)[circle,draw=none] at (-0.65,2.55) {\tiny $2$};
 \node (s_4)[circle,draw=none] at (-0.15,1.65) {\tiny $3$};
 \node (s_5)[circle,draw=none] at (0.35,0.75) {\tiny $5$};
\draw[thick] (i_3)--(i_2)--(i_1)--(i_0);
\draw[thick] (i_0)--(0,-0.6);
\draw[thick] (-1.85,3.3)--(i_3)--(-1.15,3.3);
\draw[thick] (i_2)--(-0.65,2.4);
\draw[thick] (i_1)--(-0.15,1.5);
\draw[thick] (i_0)--(0.35,0.6);
\end{tikzpicture}}   \hspace{-0.4cm} \raisebox{7em}{$\rightarrow$} \hspace{-0.4cm}  \resizebox{!}{4.7cm}{\begin{tikzpicture}
 \node(i_0) [circle,draw=black,thick,minimum size=4mm,inner sep=0.05cm] at (0,0) {\small $9$}; 
 \node(i_1) [circle,draw=black,thick,minimum size=4mm,inner sep=0.05cm] at (-0.5,0.9) {\small $3$}; 
 \node(i_2) [circle,draw=black,thick,minimum size=4mm,inner sep=0.05cm] at (-1,1.8) {\small $5$}; 
 \node(i_3) [circle,draw=black,thick,minimum size=4mm,inner sep=0.05cm] at (-1.5,2.7) {\small  $2$}; 
 \node (n_1)[circle,draw=none] at (-1.95,3.2) {\tiny $5$};
 \node (n_2)[circle,draw=none] at (-1.05,3.2) {\tiny $3$};
 \node (n_3)[circle,draw=none] at (-0.55,2.3) {\tiny $2$};
 \node (n_4)[circle,draw=none] at (-0.05,1.4) {\tiny $2$};
 \node (n_5)[circle,draw=none] at (0.45,0.5) {\tiny $2$};
  \node (s_1)[circle,draw=none] at (-1.85,3.45) {\tiny $4$};
 \node (s_2)[circle,draw=none] at (-1.15,3.45) {\tiny $1$};
 \node (s_3)[circle,draw=none] at (-0.65,2.55) {\tiny $2$};
 \node (s_4)[circle,draw=none] at (-0.15,1.65) {\tiny $5$};
 \node (s_5)[circle,draw=none] at (0.35,0.75) {\tiny $3$};
\draw[thick] (i_3)--(i_2)--(i_1)--(i_0);
\draw[thick] (i_0)--(0,-0.6);
\draw[thick] (-1.85,3.3)--(i_3)--(-1.15,3.3);
\draw[thick] (i_2)--(-0.65,2.4);
\draw[thick] (i_1)--(-0.15,1.5);
\draw[thick] (i_0)--(0.35,0.6);
\end{tikzpicture}} \hspace{-0.4cm} \raisebox{7em}{$\rightarrow$} \hspace{-0.4cm}  \resizebox{!}{4.7cm}{\begin{tikzpicture}
 \node(i_0) [circle,draw=black,thick,minimum size=4mm,inner sep=0.05cm] at (0,0) {\small $9$}; 
 \node(i_1) [circle,draw=black,thick,minimum size=4mm,inner sep=0.05cm] at (-0.5,0.9) {\small $6$}; 
 \node(i_2) [circle,draw=black,thick,minimum size=4mm,inner sep=0.05cm] at (-1,1.8) {\small $3$}; 
 \node(i_3) [circle,draw=black,thick,minimum size=4mm,inner sep=0.05cm] at (-1.5,2.7) {\small  $2$}; 
 \node (n_1)[circle,draw=none] at (-1.95,3.2) {\tiny $5$};
 \node (n_2)[circle,draw=none] at (-1.05,3.2) {\tiny $3$};
 \node (n_3)[circle,draw=none] at (-0.55,2.3) {\tiny $2$};
 \node (n_4)[circle,draw=none] at (-0.05,1.4) {\tiny $2$};
 \node (n_5)[circle,draw=none] at (0.45,0.5) {\tiny $2$};
  \node (s_1)[circle,draw=none] at (-1.85,3.45) {\tiny $4$};
 \node (s_2)[circle,draw=none] at (-1.15,3.45) {\tiny $1$};
 \node (s_3)[circle,draw=none] at (-0.65,2.55) {\tiny $5$};
 \node (s_4)[circle,draw=none] at (-0.15,1.65) {\tiny $2$};
 \node (s_5)[circle,draw=none] at (0.35,0.75) {\tiny $3$};
\draw[thick] (i_3)--(i_2)--(i_1)--(i_0);
\draw[thick] (i_0)--(0,-0.6);
\draw[thick] (-1.85,3.3)--(i_3)--(-1.15,3.3);
\draw[thick] (i_2)--(-0.65,2.4);
\draw[thick] (i_1)--(-0.15,1.5);
\draw[thick] (i_0)--(0.35,0.6);
\end{tikzpicture}} 
\end{center}
The operadic tree corresponding to the last composite tree in the sequence is 
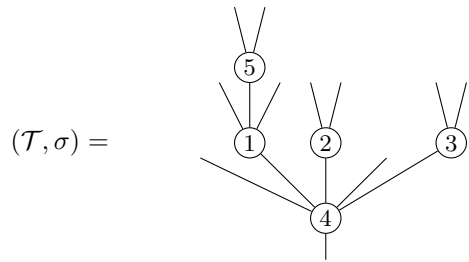
\begin{center}
    \begin{tikzpicture}
 \node (E)[circle,draw=none,minimum size=4mm,inner sep=0.1mm] at (-3.5,1) {  $({\cal T},\sigma)=$};
    \node (E)[circle,draw=black,minimum size=4mm,inner sep=0.1mm] at (-0,0) {\small $4$};
    \node (F) [circle,draw=black,minimum size=4mm,inner sep=0.1mm] at (-1,1) {\small $1$};
    \node (A) [circle,draw=black,minimum size=4mm,inner sep=0.1mm] at (-1,2) {\small $5$};
 \node (q) [circle,draw=black,minimum size=4mm,inner sep=0.1mm] at (0,1) {\small $2$};
    \node (r) [circle,draw=black,minimum size=4mm,inner sep=0.1mm] at (1.65,1) {\small $3$};
 \draw[-] (0.8,0.8) -- (E)--(-1.65,0.8); 
\draw[-] (-1.2,2.8) -- (A)--(-0.8,2.8); 
 \draw[-] (-0.2,1.8) -- (q)--(0.2,1.8);   
 \draw[-] (1.85,1.8) -- (r)--(1.45,1.8); 
 \draw[-] (E)--(0,-0.55); 
 \draw[-] (-1.4,1.8) -- (F)--(-0.6,1.8);   
    \draw[-] (E)--(F) node {};
 \draw[-] (E)--(q) node  {};
 \draw[-] (E)--(r) node {};
    \draw[-] (F)--(A) node {};
   \end{tikzpicture} 
\end{center} and we indeed have that $({\cal T}_1,\sigma_1)\circ_2 ({\cal T}_2,\sigma_2)=({\cal T},\sigma)$.\demo
\end{example}
\smallskip
\begin{remark}  Observe that, although  ${\EuScript O}$ is an  operad with the free action of the symmetric group, the relation \hyperlink{A2}{\texttt{(A2)}} contains a non-trivial  permutation of the inputs, making it a  {\em non-regular} operad. This means that  ${\EuScript O}$ cannot be characterized starting from a non-symmetric operad, by tensoring the space of operations with the regular representation of ${\Bbbk}[\Sigma_n]$, and by tensoring the partial composition operation with the composition map of the symmetric operad ${\it Ass}$. 

\smallskip

Indeed, such a  characterization would require that the restriction of  the structure of  ${\EuScript O}$ to {\em left-recursive} operadic trees,  i.e. operadic trees with a canonical order of vertices that we define below, is closed under the operadic composition of ${\EuScript O}$, which fails to be true.

\end{remark}
  
\smallskip

 A left-recursive operadic tree is an operadic tree $({\cal T},\sigma)$, for which   $\sigma:[k]\rightarrow v({\cal T})$ is the following canonical indexing of the vertices of ${\cal T}$:\\[-0.3cm]

\begin{itemize}
\item if ${\cal T}$ is a corolla $t_n$, then $\sigma:\{1\}\rightarrow v(t_n)$ is trivially defined by $\sigma(1)=\rho(t_n)$;\\[-0.3cm]
\item if ${\cal T}=t_m({\cal T}_1,\dots,{\cal T}_p)$ and if   $\leq_i$ is the linear order on $v({\cal T}_i)$ determined by the left-recursive structure of ${{\cal T}_i}$, then $\sigma$ is derived from the following linear order on $v({\cal T})$:
$$u\leq v \quad \Longleftrightarrow \quad \begin{cases} u=\rho(t_m) \\ u,v\in v({\cal T}_i) \enspace\mbox{ and }\enspace u\leq_i v \\ u\in v({\cal T}_i),\, v\in v({\cal T}_j)  \enspace\mbox{ and }\enspace {\cal T}_i < {\cal T}_j, \end{cases}$$
where ${\cal T}_i<{\cal T}_j$ means that  ${\cal T}_i$ comes before   ${\cal T}_j$ with respect to the order of inputs of $\rho(t_m)$.\\[-0.3cm]
\end{itemize}
Hence, in a left-recursive operadic tree, the vertices are indexed  from bottom to top and from left to right by $1$ through $k$. Observe that this indexing is invariant under planar isomorphisms.  In what follows, when refering to a left-recursive operadic tree $({\cal T},\sigma)$, given that $\sigma$ is canonically determined, we shall write simply ${\cal T}$.

\smallskip

The reader may now want to compose the operadic trees ${\cal T}_1$ and ${\cal T}_2$ from Example \ref{compositionexample}, considered as left-recursive trees, to see that the result will not be a left-recursive operadic tree. Nevertheless, note that the composition $({\cal T}_1,\sigma_1)\circ_2 ({\cal T}_2,\sigma_2)$ from that example can be calculated by  the substitution operation  on ${\cal T}_1$ and ${\cal T}_2$ considered as left-recursive  operadic trees, followed by   the reindexing of the vertices of the resulting (non-left-recursive) tree in a uniquely  determined way.

\medskip

\begin{convention}\label{namesnumbers} The data of an operadic tree ${\cal T}$ involves {\em non-skeletal} and {\em skeletal}  identifications of its the edges and vertices: the non-skeletal data is given by the {\em names} of edges and vertices as elements of  $e({\cal T})\cup i({\cal T})$  and $v({\cal T})$, respectively, and the skeletal data is the index of an edge (resp. vertex) given by the planar structure (resp. by the left-recursive indexing). We shall freely mix these two ways of specifying edges and vertices and use whatever is more suitable for the purpose at hand. In particular, note that the non-skeletal description of edges  eases the portrayal of operadic composition operation, as it bypasses the reindexing involved in the skeletal setting. 
\end{convention} 
 
\medskip 

\subsection{The combinatorial ${\EuScript O}_{\infty}$ operad}\label{oinf} In this section, we define the combinatorial ${\EuScript O}_{\infty}$ operad as the dg operad defined on the faces of {\em operadic polytopes}, i.e. hypergraph polytopes whose hypergraphs are the edge-graphs of operadic  trees, with the differential determined by the partial order on those faces. We start by formalizing the latter type of hypergraphs.
\subsubsection{The edge-graph of a planar rooted tree}\label{edgegraph1}
 
The {\em edge-graph} of a planar rooted tree ${\cal T}$ is the hypergraph ${\bf H}_{\cal T}$ defined as follows: the vertices of ${\bf H}_{\cal T}$ are the (internal) edges of ${\cal T}$ (identified in the non-skeletal manner) and two vertices are connected by an edge in ${\bf H}_{\cal T}$ whenever, as edges of ${\cal T}$, they share a common vertex.    Notice that the names (and possible indexing) of the vertices of ${\cal T}$, as well as the leaves of ${\cal T}$, play no role in the definition of the edge-graph  of ${\cal T}$.  

\smallskip

\begin{example}\label{edgegraph} With the non-skeletal identification of the edges of operadic trees $({\cal T}_1,\sigma_1)$, $({\cal T}_2,\sigma_2)$ and $({\cal T},\sigma)$ from Example \ref{compositionexample} given by
 
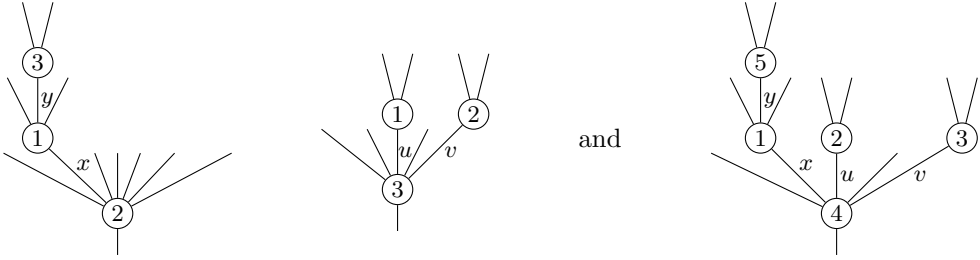
\begin{center}
\begin{tikzpicture}
    \node (E)[circle,draw=black,minimum size=4mm,inner sep=0.1mm] at (0.75,0) {\small $2$};
    \node (F) [circle,draw=black,minimum size=4mm,inner sep=0.1mm] at (-0.3,1) {\small $1$};
    \node (A) [circle,draw=black,minimum size=4mm,inner sep=0.1mm] at (-0.3,2) {\small $3$};
\node (x) [circle,draw=none,minimum size=4mm,inner sep=0.1mm] at (0.3,0.64) {\small $x$};
    \node (y) [circle,draw=none,minimum size=4mm,inner sep=0.1mm] at (-0.175,1.5) {\small $y$};
 \draw[-] (-0.1,2.8) -- (A)--(-0.5,2.8); 
 \draw[-] (0.75,0.8) -- (E)--(-0.75,0.8); 
  \draw[-] (1.5,0.8) -- (E)--(2.25,0.8); 
    \draw[-] (0.45,0.8) -- (E)--(1.05,0.8); 
 \draw[-] (E)--(0.75,-0.55); 
 \draw[-] (0.1,1.8) -- (F)--(-0.7,1.8); 
    \draw[-] (E)--(F) node  {};
    \draw[-] (F)--(A) node {};
\end{tikzpicture}   \quad \quad\quad  \raisebox{0.9em}{\begin{tikzpicture}
    \node (E)[circle,draw=black,minimum size=4mm,inner sep=0.1mm] at (0,0) {\small $3$};
    \node (F) [circle,draw=black,minimum size=4mm,inner sep=0.1mm] at (0,1) {\small $1$};
    \node (A) [circle,draw=black,minimum size=4mm,inner sep=0.1mm] at (1,1) {\small $2$};
\node (u) [circle,draw=none,minimum size=4mm,inner sep=0.1mm] at (0.1,0.5) {\small $u$};
    \node (v) [circle,draw=none,minimum size=4mm,inner sep=0.1mm] at (0.7,0.5) {\small $v$};
 \draw[-] (-0.2,1.8) -- (F)--(0.2,1.8);   
 \draw[-] (0.8,1.8) -- (A)--(1.2,1.8);   
  \draw[-] (-0.4,0.8)--(E)--(0,-0.55); 
 \draw[-] (0.4,0.8)--(E)--(-1,0.8); 
    \draw[-] (E)--(F) node {};
    \draw[-] (E)--(A) node  {};
   \end{tikzpicture}} \quad \quad\quad \raisebox{4.1em}{and} \quad \quad\quad   \begin{tikzpicture}
    \node (E)[circle,draw=black,minimum size=4mm,inner sep=0.1mm] at (-0,0) {\small $4$};
    \node (F) [circle,draw=black,minimum size=4mm,inner sep=0.1mm] at (-1,1) {\small $1$};
    \node (A) [circle,draw=black,minimum size=4mm,inner sep=0.1mm] at (-1,2) {\small $5$};
 \node (q) [circle,draw=black,minimum size=4mm,inner sep=0.1mm] at (0,1) {\small $2$};
    \node (r) [circle,draw=black,minimum size=4mm,inner sep=0.1mm] at (1.65,1) {\small $3$};
\node (x) [circle,draw=none,minimum size=4mm,inner sep=0.1mm] at (-0.4,0.64) {\small $x$};
    \node (y) [circle,draw=none,minimum size=4mm,inner sep=0.1mm] at (-0.875,1.5) {\small $y$};
\node (u) [circle,draw=none,minimum size=4mm,inner sep=0.1mm] at (0.14,0.5) {\small $u$};
    \node (v) [circle,draw=none,minimum size=4mm,inner sep=0.1mm] at (1.1,0.5) {\small $v$};
 \draw[-] (0.8,0.8) -- (E)--(-1.65,0.8); 
\draw[-] (-1.2,2.8) -- (A)--(-0.8,2.8); 
 \draw[-] (-0.2,1.8) -- (q)--(0.2,1.8);   
 \draw[-] (1.85,1.8) -- (r)--(1.45,1.8); 
 \draw[-] (E)--(0,-0.55); 
 \draw[-] (-1.4,1.8) -- (F)--(-0.6,1.8);   
    \draw[-] (E)--(F) node {};
 \draw[-] (E)--(q) node  {};
 \draw[-] (E)--(r) node {};
    \draw[-] (F)--(A) node {};
   \end{tikzpicture} 
\end{center}
respectively, the corresponding edge-graphs are 
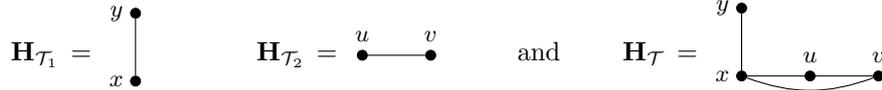
\begin{center}
\raisebox{1.3em}{${\bf H}_{{\cal T}_1}\,=$} {\begin{tikzpicture} \node(m) at (-0.25,0) {\small $x$}; \node(n) at (-0.25,0.9) {\small $y$};\node(a) [circle,draw=black,fill=black,minimum size=1.3mm,inner sep=0mm] at (0,0) {}; \node(b) [circle,draw=black,fill=black,minimum size=1.3mm,inner sep=0mm] at (0,0.9) {};    \draw  (a)--(b); \end{tikzpicture}} \quad\quad\quad \enspace\enspace
\raisebox{1.3em}{${\bf H}_{{\cal T}_2}\,=$} \raisebox{1.3em}{{\begin{tikzpicture}  \node(c) [circle,draw=black,fill=black,minimum size=1.3mm,inner sep=0mm] at (2.5,0) {}; \node(d) [circle,draw=black,fill=black,minimum size=1.3mm,inner sep=0mm] at (3.4,0) {}; \draw (d)--(c); \node(u) at (2.5,0.25) {\small $u$}; \node(v) at (3.4,0.25) {\small $v$};    \end{tikzpicture}}} \quad\quad \raisebox{1.3em}{\mbox{ and }}\quad\quad
\raisebox{1.3em}{${\bf H}_{{\cal T}}\,=$} {\begin{tikzpicture} \node(m) at (-0.25,0) {\small $x$}; \node(n) at (-0.25,0.9) {\small $y$};\node(a) [circle,draw=black,fill=black,minimum size=1.3mm,inner sep=0mm] at (0,0) {}; \node(b) [circle,draw=black,fill=black,minimum size=1.3mm,inner sep=0mm] at (0,0.9) {}; \node(c) [circle,draw=black,fill=black,minimum size=1.3mm,inner sep=0mm] at (0.9,0) {}; \node(d) [circle,draw=black,fill=black,minimum size=1.3mm,inner sep=0mm] at (1.8,0) {}; \draw (d)--(c)--(a)--(b); \node(u) at (0.9,0.25) {\small $u$}; \node(v) at (1.8,0.25) {\small $v$}; \draw (a) to [out=-20,in=200] (d); \end{tikzpicture}}
\end{center}
respectively. Observe that the edge-graph ${\bf H}_{\cal T}$ of the tree ${\cal T}$ is precisely the hypergraph of the hemiassociahedron (cf. \S\ref{sec.hemiassociahedron}), making ${\cal T}$  {\em our favourite operadic tree}.\demo
\end{example}

\smallskip

Observe, in Example \ref{edgegraph}, the additional data given by the relative position of vertices of ${\bf H}_{{\cal T}_1}$ (one above the other) and ${\bf H}_{{\cal T}_2}$ (one next to the other). This data is  implicitly present  in the edge-graphs of planar rooted trees: since edge-graphs inherit their structure from planar rooted trees,  their vertices can naturally be arranged in levels, both vertically (from bottom to top) and horizontally (from left to right). This observation is essential for the interpretation of the edges of operadic polytopes in terms of  homotopies replacing the relations \hyperlink{A1}{\texttt{(A1)}} and \hyperlink{A2}{\texttt{(A2)}} defining the operad ${\EuScript O}$. The latter interpretation has been defined in detail in \cite[Section 4]{CIO}. Let us recall here the idea.

\medskip

Recall from \S\ref{abspol} that the  vertices  of a hypergraph polytope  are encoded by the constructions of the corresponding hypergraph, i.e. by the constructs whose vertices are decorated by singletons only. In addition, the edges of a hypergraph polytope  are encoded by the constructs whose vertices are all singletons,
except one, which is a two-element set. Let ${\cal T}$ be an operadic tree and let $C$ be a construct encoding an edge of the operadic polytope ${\bf H}_{\cal T}$; suppose that $\{x,y\}$ is the unique two-element set vertex of $C$. We  show how ${\bf H}_{\cal T}$,
together with its bipartition of vertical and horizontal edges, determines the type 
of $C$ in terms of   homotopies for the relations \hyperlink{A1}{\texttt{(A1)}} and \hyperlink{A2}{\texttt{(A2)}}, as well as the direction of the corresponding edge corresponding to the orientation of \hyperlink{A1}{\texttt{(A1)}} and \hyperlink{A2}{\texttt{(A2)}} from left to right. (Strictly speaking, in \cite{CIO}, the authors worked in the non-skeletal operadic setting and with the opposite orientation of \hyperlink{A1}{\texttt{(A1)}}. In the non-skeletal environment,  the colours of ${\EuScript O}$ are arbitrary finite sets and  the vertices of composite trees are decorated by the elements of those sets. This in particular means that  the   non-skeletal variant of the relation \hyperlink{A2}{\texttt{(A2)}} does not admit a natural orientation, as opposed to the skeletal one.)
In order to state the criterion, we shall use the fact that, among all the paths between two vertices of ${\bf H}_{\cal T}$, there exists a unique one of minimal length; this fact is proven in \cite[Lemma 11]{CIO}.  The criterion is the following:

\smallskip

\begin{quote}
 If the shortest path between $x$ and $y$ in ${\bf H}_{\cal T}$ is made up of vertical edges only, then the edge encoded by $C$  corresponds to the homotopy for the relation \hyperlink{A1}{\texttt{(A1)}}, and is oriented towards the vertex encoded by the construction in which the vertex $x$ appears above the vertex $y$ if and only if the vertical level of $x$ is inferior
to the vertical level of $y$ in ${\bf H}_{\cal T}$. Otherwise,  the edge encoded by $C$  corresponds to the homotopy for the relation \hyperlink{A2}{\texttt{(A2)}}, and is oriented  towards the vertex encoded by the construction in which the vertex $x$ appears above the vertex $y$ if and only if  the horizontal level of $x$ is inferior to the horizontal level of $y$ in ${\bf H}_{\cal T}$.
\end{quote}
 
 \smallskip
 
 \begin{example}
 Let us derive the edge information for the facet  of the hemiassociahedron given by the marked  square in the realization below:
 \begin{center}
 \resizebox{8cm}{!}
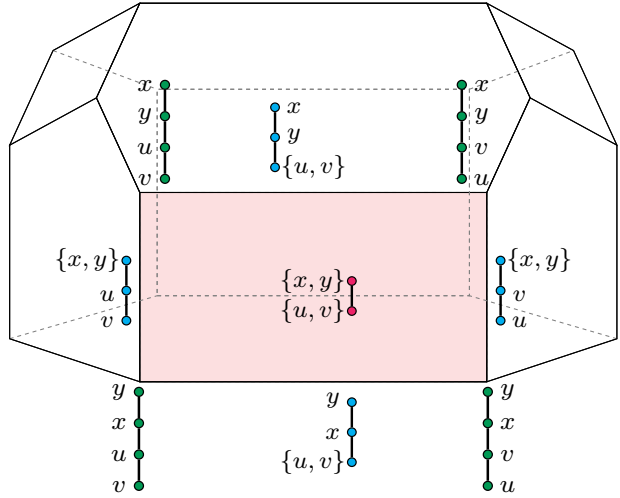
{\begin{tikzpicture}[thick,scale=23]
\coordinate (A1) at (-0.2,0);
\coordinate (A2) at (0.2,0); 
\coordinate (D1) at (-0.35,0.05);
\coordinate (D2) at (0.35,0.05);
\coordinate (D3) at (-0.18,0.1);
\coordinate (D4)  at (0.18,0.1);
\coordinate (E3) at (-0.18,0.34);
\coordinate (E4) at (0.18,0.34);
\coordinate (A3) at (-0.2,0.22);
\coordinate (A4) at (0.2,0.22);
 \coordinate (C1) at (-0.25,0.33);
\coordinate (C2) at (0.25,0.33);
 \coordinate (F1) at (-0.35,0.275);
\coordinate (F2) at (0.35,0.275);
 \coordinate (G1) at (-0.3,0.385);
\coordinate (G2) at (0.3,0.385);
 \coordinate (B3) at (-0.2,0.44);
\coordinate (B4) at (0.2,0.44);
\fill[fill=WildStrawberry!20,opacity=0.75] (A1)--(A2)--(A4)--(A3)--(A1);
\draw (A1) -- (A2) -- (A4) -- (A3) -- cycle;
\draw (A3) -- (A4) -- (C2) -- (B4) -- (B3) -- (C1)-- cycle;
\draw[dashed,draw=gray] (D3) -- (D4) -- (E4) -- (E3) -- cycle;
\draw[dashed,draw=gray] (A1) -- (A2) -- (A4) -- (A3) -- (A1);
\draw (A3) -- (A4) -- (C2) -- (B4) -- (B3) -- (C1)-- (A3);
\draw (G1) -- (B3) -- (B4) -- (G2);
\draw (G2) -- (F2);
\draw (G1) -- (F1);
\draw (A3)--(A1) -- (A2)--(A4);
\draw[dashed,draw=gray]   (G1) -- (E3); 
\draw[dashed,draw=gray] (E4) -- (G2);
 \draw  (A1) -- (D1) -- (F1) -- (C1) -- (A3) --  cycle;
 \draw  (A2) -- (D2) -- (F2) -- (C2) -- (A4) --  cycle;
 \draw  (F1) -- (G1) -- (B3) -- (C1)  --  cycle;
\draw (F2) -- (G2) -- (B4) -- (C2)  --  cycle;
\draw[dashed,draw=gray] (D3) -- (D4) -- (E4) -- (E3) -- cycle;
\draw[dashed,draw=gray] (D1)   -- (D3);
\draw[dashed,draw=gray] (D2)   -- (D4);
\draw (F1)--(D1)--(A1)--(A2)--(D2)--(F2);
\node at  (0.19,0.29)  {\resizebox{1.15cm}{!}{\begin{tikzpicture}
\node (b) [circle,fill=ForestGreen,draw=black,minimum size=0.1cm,inner sep=0.2mm,label={[xshift=0.22cm,yshift=-0.25cm]{\footnotesize $u$}}] at (0,0) {};
\node (c) [circle,fill=ForestGreen,draw=black,minimum size=0.1cm,inner sep=0.2mm,label={[xshift=0.22cm,yshift=-0.25cm]{\footnotesize $v$}}] at (0,0.35) {};
 \node (d) [circle,fill=ForestGreen,draw=black,minimum size=0.1cm,inner sep=0.2mm,label={[xshift=0.22cm,yshift=-0.27cm]{\footnotesize $y$}}] at (0,0.7) {};
  \node (a) [circle,fill=ForestGreen,draw=black,minimum size=0.1cm,inner sep=0.2mm,label={[xshift=0.22cm,yshift=-0.25cm]{\footnotesize $x$}}] at (0,1.05) {};
 \draw[thick]  (b)--(c)--(d)--(a);\end{tikzpicture}}};
\node at  (-0.19,0.29)  {\resizebox{1.15cm}{!}{\begin{tikzpicture}
\node (b) [circle,fill=ForestGreen,draw=black,minimum size=0.1cm,inner sep=0.2mm,label={[xshift=-0.22cm,yshift=-0.25cm]{\footnotesize $v$}}] at (0,0) {};
\node (c) [circle,fill=ForestGreen,draw=black,minimum size=0.1cm,inner sep=0.2mm,label={[xshift=-0.22cm,yshift=-0.25cm]{\footnotesize $u$}}] at (0,0.35) {};
 \node (d) [circle,fill=ForestGreen,draw=black,minimum size=0.1cm,inner sep=0.2mm,label={[xshift=-0.22cm,yshift=-0.27cm]{\footnotesize $y$}}] at (0,0.7) {};
  \node (a) [circle,fill=ForestGreen,draw=black,minimum size=0.1cm,inner sep=0.2mm,label={[xshift=-0.22cm,yshift=-0.25cm]{\footnotesize $x$}}] at (0,1.05) {};
 \draw[thick]  (b)--(c)--(d)--(a);\end{tikzpicture}}};
\node at  (-0.22,-0.065)   {\resizebox{1.15cm}{!}{\begin{tikzpicture}
\node (b) [circle,fill=ForestGreen,draw=black,minimum size=0.1cm,inner sep=0.2mm,label={[xshift=-0.22cm,yshift=-0.25cm]{\footnotesize $v$}}] at (0,0) {};
\node (c) [circle,fill=ForestGreen,draw=black,minimum size=0.1cm,inner sep=0.2mm,label={[xshift=-0.22cm,yshift=-0.25cm]{\footnotesize $u$}}] at (0,0.35) {};
 \node (d) [circle,fill=ForestGreen,draw=black,minimum size=0.1cm,inner sep=0.2mm,label={[xshift=-0.22cm,yshift=-0.25cm]{\footnotesize $x$}}] at (0,0.7) {};
  \node (a) [circle,fill=ForestGreen,draw=black,minimum size=0.1cm,inner sep=0.2mm,label={[xshift=-0.22cm,yshift=-0.27cm]{\footnotesize $y$}}] at (0,1.05) {};
 \draw[thick]  (b)--(c)--(d)--(a);\end{tikzpicture}}};
\node at  (0.22,-0.065)  {\resizebox{1.15cm}{!}{\begin{tikzpicture}
\node (b) [circle,fill=ForestGreen,draw=black,minimum size=0.1cm,inner sep=0.2mm,label={[xshift=0.22cm,yshift=-0.25cm]{\footnotesize $u$}}] at (0,0) {};
\node (c) [circle,fill=ForestGreen,draw=black,minimum size=0.1cm,inner sep=0.2mm,label={[xshift=0.22cm,yshift=-0.25cm]{\footnotesize $v$}}] at (0,0.35) {};
 \node (d) [circle,fill=ForestGreen,draw=black,minimum size=0.1cm,inner sep=0.2mm,label={[xshift=0.22cm,yshift=-0.25cm]{\footnotesize $x$}}] at (0,0.7) {};
  \node (a) [circle,fill=ForestGreen,draw=black,minimum size=0.1cm,inner sep=0.2mm,label={[xshift=0.22cm,yshift=-0.27cm]{\footnotesize $y$}}] at (0,1.05) {};
 \draw[thick]  (b)--(c)--(d)--(a);\end{tikzpicture}}};
  \node at  (0.26,0.11)  {\resizebox{2.3cm}{!}{\begin{tikzpicture}
\node (b) [circle,fill=cyan,draw=black,minimum size=0.1cm,inner sep=0.2mm,label={[xshift=0.22cm,yshift=-0.25cm]{\footnotesize $u$}}] at (0,0) {};
\node (c) [circle,fill=cyan,draw=black,minimum size=0.1cm,inner sep=0.2mm,label={[xshift=0.22cm,yshift=-0.3cm]{\footnotesize $v$}}] at (0,0.35) {};
 \node (d) [circle,fill=cyan,draw=black,minimum size=0.1cm,inner sep=0.2mm,label={[xshift=0.45cm,yshift=-0.32cm]{\footnotesize $\{x,y\}$}}] at (0,0.7) {};
 \draw[thick]  (b)--(c)--(d);\end{tikzpicture}}};
  \node at  (-0.26,0.11)  {\resizebox{2.3cm}{!}{\begin{tikzpicture}
\node (b) [circle,fill=cyan,draw=black,minimum size=0.1cm,inner sep=0.2mm,label={[xshift=-0.22cm,yshift=-0.25cm]{\footnotesize $v$}}] at (0,0) {};
\node (c) [circle,fill=cyan,draw=black,minimum size=0.1cm,inner sep=0.2mm,label={[xshift=-0.22cm,yshift=-0.3cm]{\footnotesize $u$}}] at (0,0.35) {};
 \node (d) [circle,fill=cyan,draw=black,minimum size=0.1cm,inner sep=0.2mm,label={[xshift=-0.45cm,yshift=-0.32cm]{\footnotesize $\{x,y\}$}}] at (0,0.7) {};
 \draw[thick]  (b)--(c)--(d);\end{tikzpicture}}};
  \node at  (0,0.28)  {\resizebox{2.3cm}{!}{\begin{tikzpicture}
\node (b) [circle,fill=cyan,draw=black,minimum size=0.1cm,inner sep=0.2mm,label={[xshift=0.22cm,yshift=-0.25cm]{\footnotesize $x$}}] at (0,0.7) {};
\node (c) [circle,fill=cyan,draw=black,minimum size=0.1cm,inner sep=0.2mm,label={[xshift=0.22cm,yshift=-0.26cm]{\footnotesize $y$}}] at (0,0.35) {};
 \node (d) [circle,fill=cyan,draw=black,minimum size=0.1cm,inner sep=0.2mm,label={[xshift=0.45cm,yshift=-0.32cm]{\footnotesize $\{u,v\}$}}] at (0,0) {};
 \draw[thick]  (b)--(c)--(d);\end{tikzpicture}}};
   \node at  (0,-0.06)  {\resizebox{2.3cm}{!}{\begin{tikzpicture}
\node (b) [circle,fill=cyan,draw=black,minimum size=0.1cm,inner sep=0.2mm,label={[xshift=-0.22cm,yshift=-0.25cm]{\footnotesize $y$}}] at (0,0.7) {};
\node (c) [circle,fill=cyan,draw=black,minimum size=0.1cm,inner sep=0.2mm,label={[xshift=-0.22cm,yshift=-0.26cm]{\footnotesize $x$}}] at (0,0.35) {};
 \node (d) [circle,fill=cyan,draw=black,minimum size=0.1cm,inner sep=0.2mm,label={[xshift=-0.45cm,yshift=-0.32cm]{\footnotesize $\{u,v\}$}}] at (0,0) {};
 \draw[thick]  (b)--(c)--(d);\end{tikzpicture}}};
    \node at  (0,0.1)  {\resizebox{2.3cm}{!}{\begin{tikzpicture}
\node (c) [circle,fill=WildStrawberry,draw=black,minimum size=0.1cm,inner sep=0.2mm,label={[xshift=-0.45cm,yshift=-0.32cm]{\footnotesize $\{x,y\}$}}] at (0,0.35) {};
 \node (d) [circle,fill=WildStrawberry,draw=black,minimum size=0.1cm,inner sep=0.2mm,label={[xshift=-0.45cm,yshift=-0.32cm]{\footnotesize $\{u,v\}$}}] at (0,0) {};
 \draw[thick]  (c)--(d);\end{tikzpicture}}};
\end{tikzpicture}}
\end{center}
According to the criterion, the edges $\{u,v\}\{y\{x\}\}$ and $\{u,v\}\{x\{y\}\}$ encode  the homotopies for  \hyperlink{A2}{\texttt{(A2)}}, whereas the edges $v\{u\{\{x,y\}\}\}$ and $u\{v\{\{x,y\}\}\}$ encode  the homotopies for  \hyperlink{A1}{\texttt{(A1)}}. From the point of view of {\em categorified operads} \cite{DP-HP-o}, corresponding to strongly homotopy operads for which the operations given by operadic trees with more than three vertices   vanish, the construct $\{u,v\}\{\{x,y\}\}$, encoding the entire square, is the homotopy identity for the naturality relation
\smallskip
\begin{center}
\begin{tikzpicture}
\node(a) at (0,0) {\small $(((f {\circ_{x}} g)\circ_{y}h)\circ_u k)\circ_v l$};
\node(b) at (5.5,0) {\small  $(((f {\circ_{x}} g)\circ_{y}h)\circ_v l)\circ_u k$};
\node(1) at (0,-1.75) {\small $((f {\circ_{x}} (g\circ_{y}h))\circ_u k)\circ_v l$};
\node(2) at (5.5,-1.75) {\small  $((f {\circ_{x}} (g\circ_{y}h))\circ_v l)\circ_u k$};
\draw[->] (b)--(a) node [midway,above] {\small \hyperlink{A2}{\texttt{(A2)}}};
\draw[->] (1)--(a) node [midway,left] {\small \hyperlink{A1}{\texttt{(A1)}}};
\draw[->] (2)--(1) node [midway,above] {\small \hyperlink{A2}{\texttt{(A2)}}};
\draw[->] (2)--(b) node [midway,right] {\small \hyperlink{A1}{\texttt{(A1)}}};
\end{tikzpicture}
\end{center}
pertaining to the operation
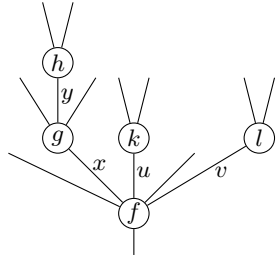
\begin{center}
\begin{tikzpicture}
    \node (E)[circle,draw=black,minimum size=4mm,inner sep=0.1mm] at (-0,0) {\small $f$};
    \node (F) [circle,draw=black,minimum size=4mm,inner sep=0.1mm] at (-1,1) {\small $g$};
    \node (A) [circle,draw=black,minimum size=4mm,inner sep=0.1mm] at (-1,2) {\small $h$};
 \node (q) [circle,draw=black,minimum size=4mm,inner sep=0.1mm] at (0,1) {\small $k$};
    \node (r) [circle,draw=black,minimum size=4mm,inner sep=0.1mm] at (1.65,1) {\small $l$};
            \node (x)[circle,draw=none,minimum size=4mm,inner sep=0.1mm] at (-0.45,0.65) {\small $x$};
    \node (y) [circle,draw=none,minimum size=4mm,inner sep=0.1mm] at (-0.875,1.55) {\small $y$};
                \node (u)[circle,draw=none,minimum size=4mm,inner sep=0.1mm] at (0.13,0.55) {\small $u$};
    \node (v) [circle,draw=none,minimum size=4mm,inner sep=0.1mm] at (1.15,0.55) {\small $v$};
 \draw[-] (0.8,0.8) -- (E)--(-1.65,0.8); 
\draw[-] (-1.2,2.8) -- (A)--(-0.8,2.8); 
 \draw[-] (-0.2,1.8) -- (q)--(0.2,1.8);   
 \draw[-] (1.85,1.8) -- (r)--(1.45,1.8); 
 \draw[-] (E)--(0,-0.55); 
 \draw[-] (-1.5,1.8) -- (F)--(-0.5,1.8);   
    \draw[-] (E)--(F) node {};
 \draw[-] (E)--(q) node  {};
 \draw[-] (E)--(r) node {};
    \draw[-] (F)--(A) node {};
   \end{tikzpicture}\vspace{-0.4cm}
\end{center}
 \demo
 \end{example}
 \medskip
 
The following two lemmas  are straightforward consequences of the definition of the edge-graph of an operadic tree.
\smallskip
\begin{lemma}\label{connectededges}  
The subtrees  of an operadic tree $({\cal T},\sigma)$ that have at least two vertices, considered as left-recursive operadic trees, are in a one-to-one correspondence with the connected subsets of ${H}_{\cal T}$, i.e. non-empty  subsets $X$ of vertices of ${\bf H}_{\cal T}$
such that the hypergraph $({\bf H}_{\cal T})_{X}$ is connected.
\end{lemma}

 \smallskip
\begin{remark}
Thanks to Lemma \ref{connectededges}, for an operadic tree $({\cal T},\sigma)$ and $\emptyset\neq X\subseteq e({\cal T})$, we can   index  the connected components of ${\bf H}_{\cal T}\backslash X$ by the corresponding left-recursive subtrees of ${\cal T}$, by  writing  $${\bf H}_{\cal T}\backslash X\leadsto {\bf H}_{{\cal T}_1},\dots,{\bf H}_{{\cal T}_n}.$$  
However, one must be careful with the induced decomposition on the level of trees! Observe that the subtrees ${\cal T}_1,\dots,{\cal T}_n$ of ${\cal T}$ do not in general make a decomposition of ${\cal T}$, in the sense that the removal  of the edges from the set $X$ may result in a number of subtrees of ${\cal T}$ reduced to a corolla.
\end{remark}
 \smallskip
 
 \begin{lemma}\label{wd0}
Suppose that $({\cal T},\sigma)=({\cal T}_1,\sigma_1)\circ_i ({\cal T}_2,\sigma_2)$, and that, for a subset  $\emptyset\neq X\subseteq e({\cal T}_1)$ of edges of ${\cal T}_1$, we have ${\bf H}_{{\cal T}_1}\backslash X\leadsto {\bf H}_{({\cal T}_1)_1},\dots,{\bf H}_{({\cal T}_1)_p}$. If there exists an index $1\leq j\leq p$,   such that the subtree $({\cal T}_{1})_j$ of ${\cal T}_1$ contains the vertex $v$ indexed by $i$ in ${\cal T}_1$, and if $l$ is the index that the vertex $v$ gets in the left-recursive ordering of the vertices of $({\cal T}_{1})_j$, then
 $${\bf H}_{{\cal T}_1\bullet_i {\cal T}_2}\backslash X\leadsto\{({\bf H}_{1})_k\,|\, 1\leq k\leq p , k\neq j\}\cup\{{\bf H}_{({\cal T}_{1})_j\,\bullet_l\, {\cal T}_2}\}.$$
 Otherwise, we have that $${\bf H}_{{\cal T}_1\bullet_i {\cal T}_2}\backslash X\leadsto\{({\bf H}_{1})_k\,|\, 1\leq k\leq p \}\cup\{{\bf H}_{{\cal T}_2}\}.$$ 
  \end{lemma}

 \smallskip

An isomorphism of  planar rooted trees induces an isomorphism on the corresponding hypergraphs and their constructs  in the natural way: the form of the hypergraph  matters, not the names of the hyperedges. For an  isomorphism $\varphi:{\cal T}_1\rightarrow {\cal T}_2$ of  planar rooted trees and constructs $C_1:{\bf H}_{{\cal T}_1}$ and $C_2:{\bf H}_{{\cal T}_2}$, we shall write $C_1\sim_{\varphi} C_2$ to denote that  $C_1$ and $C_2$ are isomorphic via $\varphi$. In addition, for a hypergraph ${\bf H}$,  the polytope ${\cal A}({\bf H})$ will be considered  modulo  renaming of the vertices of ${\bf H}$.
\smallskip
\subsubsection{The  operad ${\EuScript O}_{\infty}$ as an operad of vector spaces}\label{Oinf}
Define, for $k\geq 2$, $n_1,\dots,n_k\geq 1$, and $n=\big(\sum^k_{i=1} n_i\big)-k+1$, the vector space ${\EuScript O}_{\infty}\displaystyle(n_1,n_2,\dots,n_k;n)$  to be the ${\Bbbk}$-linear span of the set of triples $({\cal T},\sigma,C)$, such that  $({\cal T},\sigma)\in {\tt Tree}(n_1,\dots,n_k;n)$  and $C:{\bf H}_{\cal T}$, subject to the equivalence relation generated by: 
\medskip
\begin{quote} 
$({\cal T}_1,\sigma_1,C_1)\sim ({\cal T}_2,\sigma_2,C_2)$   if  there exists an isomorphism $\varphi:{\cal T}_1\rightarrow {\cal T}_2$, such that  $\varphi\circ\sigma_1=\sigma_2$ and  $ C_1 \sim_{\varphi}C_2$.  \end{quote}
\medskip
Hence, for a fixed operadic tree $({\cal T},\sigma)\in {\tt Tree}(n_1,\dots,n_k;n)$,   the subspace of ${\EuScript O}_{\infty}\displaystyle(n_1,n_2,\dots,n_k;n)$ determined by $({\cal T},\sigma)$ is spanned by all the (isomorphism classes of) constructs of the hypergraph ${\bf H}_{\cal T}$:

 \medskip
$${\EuScript O}_{\infty}(n_1,n_2,\dots,n_k;n):=\displaystyle {\textsf{Span}}_{{\Bbbk}}\bigg(\bigoplus_{({\cal T},\sigma)\in \,{\tt Tree}(n_1,\dots,n_k;n)} {A}({\bf H}_{\cal T})\bigg).$$
 \medskip

\noindent Note that for $n\neq \big(\sum^k_{i=1} n_i\big)-k+1$, we set ${\EuScript O}_{\infty}(n_1,n_2,\dots,n_k;n)$ to be the zero vector space.
 The ${\mathbb N}$-coloured collection
 \smallskip
  $$\{{\EuScript O}_{\infty}(\displaystyle n_1,n_2,\dots,n_k;n) \,\,|\,\, n_1,\dots,n_k\geq 1\}$$ 
  \smallskip
   admits the following  operad structure. The composition operation 
  \smallskip
$$\circ_i : {\EuScript O}_{\infty}\displaystyle(n_1,\dots,n_k;n) \otimes {\EuScript O}_{\infty}\displaystyle(m_1,\dots,m_l;n_i)\rightarrow {\EuScript O}_{\infty}\displaystyle(n_1,\dots,n_{i-1},m_1,\dots,m_l,n_{i+1},\dots n_k;n)$$
\smallskip
 is defined by
 \smallskip
$$({\cal T}_1,\sigma_1,C_1)\circ_i ({\cal T}_2,\sigma_2,C_2)=({\cal T}_1\bullet_i {\cal T}_2,\sigma_1\bullet_i\sigma_2,C_1\bullet_i C_2),$$ 
\smallskip
where the composition on the level of operadic trees  is determined by the composition product of the operad ${\EuScript O}$, and the construct $C_1\bullet_i C_2:{\bf H}_{{\cal T}_1\bullet_i {\cal T}_2}$ is defined as follows.\\[-0.3cm]
\begin{itemize}
\item If $C_1=H_{{\cal T}_1}$, then $$C_1\bullet_i C_2:=H_{{\cal T}_1}\{C_2\}.$$
\item Suppose that $C_1=X\{C_{11},\dots,C_{1p}\}$, where ${\bf H}_{{\cal T}_1}\backslash X\leadsto ({\bf H}_{{\cal T}_1})_1,\dots, ({\bf H}_{{\cal T}_1})_p$ and  $C_{1q}:({\bf H}_{{\cal T}_1})_q$.  If there exists an index $1\leq j\leq p$, such that the subtree $({\cal T}_1)_j$ of ${\cal T}_1$ contains the vertex $v$ indexed by $i$ in ${\cal T}_1$, we define $$C_1\bullet_i C_2:=X\{C_{11},\dots,C_{1j}\bullet_l C_2,\dots,C_{1p}\},$$ where $l$ is the left-recursive index of the vertex $v$ in ${\cal T}_1$. Otherwise, we define $$C_1\bullet_i C_2:=X\{C_{11},\dots,\dots,C_{1p},C_2\}.$$ 
\end{itemize}
\smallskip
The action of the symmetric group is defined by $({\cal T},\sigma,C)^{\kappa}=({\cal T}, \sigma\circ\kappa,C)$.
\smallskip
 \begin{lemma}\label{wd}
The composition operation of ${\EuScript O}_{\infty}$ is well-defined.
\end{lemma}
 \begin{proof} We prove that  $C_1\bullet_i C_2$ is indeed a construct of ${\bf H}_{{\cal T}_1\bullet_i {\cal T}_2}$. As for the first case defining $C_1\bullet_i C_2$,   by Lemma \ref{connectededges}, since ${\cal T}_2$ is a subtree of ${\cal T}_1\bullet_i {\cal T}_2$, we have that  ${\bf H}_{{\cal T}_1\bullet_i {\cal T}_2}\backslash H_{{\cal T}_1}\leadsto {\bf H}_{{\cal T}_2}$. Therefore, since $C_2:{\bf H}_{{\cal T}_2}$, we indeed have that $H_{{\cal T}_1}\{C_2\}:{\bf H}_{{\cal T}_1\bullet_i {\cal T}_2}$. The legitimacy of the second case defining $C_1\bullet_i C_2$ is a direct consequence of Lemma \ref{wd0}.
\end{proof}
\smallskip

The following lemma provides a non-inductive characterization of $C_1\bullet_i C_2$.
\smallskip
\begin{lemma}
The construct $C_1\bullet_i C_2$ is the unique construct of the hypergraph ${\bf H}_{{\cal T}_1\bullet_i {\cal T}_2}$, such that $v(C_1\bullet_i C_2)=v(C_1)\cup v(C_2)$, $\rho(C_1\bullet_i C_2)=\rho(C_1)$, and such that there exists an edge of $C_1\bullet_i C_2$ whose removal results precisely in $C_1$ and $C_2$.
\end{lemma}

\begin{remark}
Note that, if $({\cal T}_1,\sigma_1,C_1)\circ_i ({\cal T}_2,\sigma_2,C_2)=({\cal T}_1\bullet_i {\cal T}_2,\sigma_1\bullet_i\sigma_2,C_1\bullet_i C_2)$, then $C_1\bullet_i C_2\leq H_{{\cal T}_1}\{H_{{\cal T}_2}\}$ in ${\cal A}({\bf H}_{{\cal T}_1\bullet_i {\cal T}_2})$.
\end{remark}
 
\smallskip

\begin{example}\label{compositionconstructs}
The picture below displays  all the 9 instances of the partial composition 
\smallskip
$$\circ_2:{\EuScript O}_{\infty}(3,7,2;10)\otimes {\EuScript O}_{\infty}(2,2,5;7)\rightarrow {\EuScript O}_{\infty}(3,2,2,5,2;10)$$
\smallskip
 determined by operadic trees $({\cal T}_1,\sigma_1)$ and $({\cal T}_2,\sigma_2)$ from Example \ref{compositionexample}.  The resulting 9 constructs  are the faces of    the square $\{x,y\}\{\{u,v\}\}$ of
 the 3-dimensional hemiassociahedron.    
 \smallskip
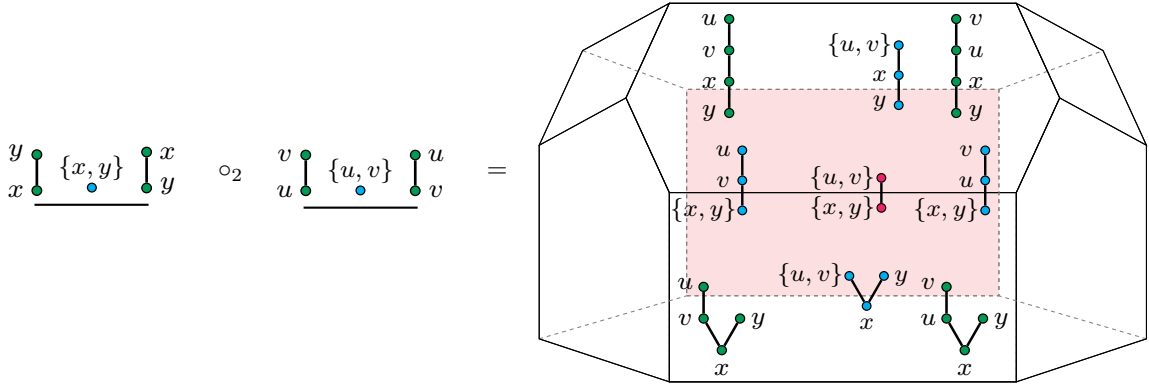
\begin{center}
\raisebox{-1.5em}{\begin{tikzpicture}
\coordinate (A1) at (-0.75,0);
\coordinate (A2) at (0.75,0);
\draw[thick] (A1)--(A2);
\node at   (-0.95,0.45) {{\resizebox{0.6cm}{!}{\begin{tikzpicture}
\node (L) [circle,fill=ForestGreen,draw=black,minimum size=0.1cm,inner sep=0.2mm,label={[xshift=-0.22cm,yshift=-0.25cm]{\footnotesize $x$}}] at (0,0) {};
\node (B1) [circle,fill=ForestGreen,draw=black,minimum size=0.1cm,inner sep=0.2mm,label={[xshift=-0.22cm,yshift=-0.25cm]{\footnotesize $y$}}] at (0,0.38) {};
 \draw[thick]  (L)--(B1);\end{tikzpicture}}}};
\node at   (0.95,0.45) {{\resizebox{0.6cm}{!}{\begin{tikzpicture}
\node (L) [circle,fill=ForestGreen,draw=black,minimum size=0.1cm,inner sep=0.2mm,label={[xshift=0.22cm,yshift=-0.25cm]{\footnotesize $y$}}] at (0,0) {};
\node (B1) [circle,fill=ForestGreen,draw=black,minimum size=0.1cm,inner sep=0.2mm,label={[xshift=0.22cm,yshift=-0.25cm]{\footnotesize $x$}}] at (0,0.38) {};
 \draw[thick]  (L)--(B1);\end{tikzpicture}}}};
\node at   (0,0.475) {{\resizebox{1.15cm}{!}{\begin{tikzpicture}
\node (L) [circle,fill=cyan,draw=black,minimum size=0.1cm,inner sep=0.2mm,label={[xshift=0cm,yshift=-0.1cm]{\footnotesize $\{x,y\}$}}] at (0,0) {};
\end{tikzpicture}}}};
\end{tikzpicture}} \enspace $\circ_2$\enspace 
\raisebox{-1.5em}{\begin{tikzpicture}
\coordinate (A1) at (-0.75,0);
\coordinate (A2) at (0.75,0);
\draw[thick] (A1)--(A2);
\node at   (-0.95,0.45) {{\resizebox{0.6cm}{!}{\begin{tikzpicture}
\node (L) [circle,fill=ForestGreen,draw=black,minimum size=0.1cm,inner sep=0.2mm,label={[xshift=-0.22cm,yshift=-0.25cm]{\footnotesize $u$}}] at (0,0) {};
\node (B1) [circle,fill=ForestGreen,draw=black,minimum size=0.1cm,inner sep=0.2mm,label={[xshift=-0.22cm,yshift=-0.25cm]{\footnotesize $v$}}] at (0,0.38) {};
 \draw[thick]  (L)--(B1);\end{tikzpicture}}}};
\node at   (0.95,0.45) {{\resizebox{0.6cm}{!}{\begin{tikzpicture}
\node (L) [circle,fill=ForestGreen,draw=black,minimum size=0.1cm,inner sep=0.2mm,label={[xshift=0.22cm,yshift=-0.25cm]{\footnotesize $v$}}] at (0,0) {};
\node (B1) [circle,fill=ForestGreen,draw=black,minimum size=0.1cm,inner sep=0.2mm,label={[xshift=0.22cm,yshift=-0.25cm]{\footnotesize $u$}}] at (0,0.38) {};
 \draw[thick]  (L)--(B1);\end{tikzpicture}}}};
\node at   (0,0.475) {{\resizebox{1.15cm}{!}{\begin{tikzpicture}
\node (L) [circle,fill=cyan,draw=black,minimum size=0.1cm,inner sep=0.2mm,label={[xshift=0cm,yshift=-0.1cm]{\footnotesize $\{u,v\}$}}] at (0,0) {};
\end{tikzpicture}}}};
\end{tikzpicture}} \enspace $=$ \enspace
\raisebox{-8em}{ \resizebox{8cm}{!}{\begin{tikzpicture}[thick,scale=23]
\coordinate (A1) at (-0.2,0);
\coordinate (A2) at (0.2,0); 
\coordinate (D1) at (-0.35,0.05);
\coordinate (D2) at (0.35,0.05);
\coordinate (D3) at (-0.18,0.1);
\coordinate (D4)  at (0.18,0.1);
\coordinate (E3) at (-0.18,0.34);
\coordinate (E4) at (0.18,0.34);
\coordinate (A3) at (-0.2,0.22);
\coordinate (A4) at (0.2,0.22);
 \coordinate (C1) at (-0.25,0.33);
\coordinate (C2) at (0.25,0.33);
 \coordinate (F1) at (-0.35,0.275);
\coordinate (F2) at (0.35,0.275);
 \coordinate (G1) at (-0.3,0.385);
\coordinate (G2) at (0.3,0.385);
 \coordinate (B3) at (-0.2,0.44);
\coordinate (B4) at (0.2,0.44);
\fill[fill=WildStrawberry!20,opacity=0.75] (D3)--(D4)--(E4)--(E3)--(D3);
\draw (A1) -- (A2) -- (A4) -- (A3) -- cycle;
\draw (A3) -- (A4) -- (C2) -- (B4) -- (B3) -- (C1)-- cycle;
\draw[dashed,draw=gray] (D3) -- (D4) -- (E4) -- (E3) -- cycle;
\draw[dashed,draw=gray] (A1) -- (A2) -- (A4) -- (A3) -- (A1);
\draw (A3) -- (A4) -- (C2) -- (B4) -- (B3) -- (C1)-- (A3);
\draw (G1) -- (B3) -- (B4) -- (G2);
\draw (G2) -- (F2);
\draw (G1) -- (F1);
\draw (A3)--(A1) -- (A2)--(A4);
\draw[dashed,draw=gray]   (G1) -- (E3); 
\draw[dashed,draw=gray] (E4) -- (G2);
 \draw  (A1) -- (D1) -- (F1) -- (C1) -- (A3) --  cycle;
 \draw  (A2) -- (D2) -- (F2) -- (C2) -- (A4) --  cycle;
 \draw  (F1) -- (G1) -- (B3) -- (C1)  --  cycle;
\draw (F2) -- (G2) -- (B4) -- (C2)  --  cycle;
\draw[dashed,draw=gray] (D3) -- (D4) -- (E4) -- (E3) -- cycle;
\draw[dashed,draw=gray] (D1)   -- (D3);
\draw[dashed,draw=gray] (D2)   -- (D4);
\draw (F1)--(D1)--(A1)--(A2)--(D2)--(F2);
\node at  (0.15,0.365)  {\resizebox{1.15cm}{!}{\begin{tikzpicture}
\node (b) [circle,fill=ForestGreen,draw=black,minimum size=0.1cm,inner sep=0.2mm,label={[xshift=0.22cm,yshift=-0.27cm]{\footnotesize $y$}}] at (0,0) {};
\node (c) [circle,fill=ForestGreen,draw=black,minimum size=0.1cm,inner sep=0.2mm,label={[xshift=0.22cm,yshift=-0.25cm]{\footnotesize $x$}}] at (0,0.35) {};
 \node (d) [circle,fill=ForestGreen,draw=black,minimum size=0.1cm,inner sep=0.2mm,label={[xshift=0.22cm,yshift=-0.25cm]{\footnotesize $u$}}] at (0,0.7) {};
  \node (a) [circle,fill=ForestGreen,draw=black,minimum size=0.1cm,inner sep=0.2mm,label={[xshift=0.22cm,yshift=-0.25cm]{\footnotesize $v$}}] at (0,1.05) {};
 \draw[thick]  (b)--(c)--(d)--(a);\end{tikzpicture}}};
\node at  (-0.15,0.365)  {\resizebox{1.15cm}{!}{\begin{tikzpicture}
\node (b) [circle,fill=ForestGreen,draw=black,minimum size=0.1cm,inner sep=0.2mm,label={[xshift=-0.22cm,yshift=-0.27cm]{\footnotesize $y$}}] at (0,0) {};
\node (c) [circle,fill=ForestGreen,draw=black,minimum size=0.1cm,inner sep=0.2mm,label={[xshift=-0.22cm,yshift=-0.25cm]{\footnotesize $x$}}] at (0,0.35) {};
 \node (d) [circle,fill=ForestGreen,draw=black,minimum size=0.1cm,inner sep=0.2mm,label={[xshift=-0.22cm,yshift=-0.25cm]{\footnotesize $v$}}] at (0,0.7) {};
  \node (a) [circle,fill=ForestGreen,draw=black,minimum size=0.1cm,inner sep=0.2mm,label={[xshift=-0.22cm,yshift=-0.25cm]{\footnotesize $u$}}] at (0,1.05) {};
 \draw[thick]  (b)--(c)--(d)--(a);\end{tikzpicture}}};
\node at (-0.16,0.23)  {\resizebox{2.3cm}{!}{\begin{tikzpicture}
\node (b) [circle,fill=cyan,draw=black,minimum size=0.1cm,inner sep=0.2mm,label={[xshift=-0.22cm,yshift=-0.25cm]{\footnotesize $u$}}] at (0,0.7) {};
\node (c) [circle,fill=cyan,draw=black,minimum size=0.1cm,inner sep=0.2mm,label={[xshift=-0.22cm,yshift=-0.26cm]{\footnotesize $v$}}] at (0,0.35) {};
 \node (d) [circle,fill=cyan,draw=black,minimum size=0.1cm,inner sep=0.2mm,label={[xshift=-0.45cm,yshift=-0.32cm]{\footnotesize $\{x,y\}$}}] at (0,0) {};
 \draw[thick]  (b)--(c)--(d);\end{tikzpicture}}};
\node at   (-0.14,0.065)  {{\resizebox{2.9cm}{!}{\begin{tikzpicture}
\node (b) [circle,fill=ForestGreen,draw=black,minimum size=0.1cm,inner sep=0.2mm,label={[yshift=-0.45cm]{\footnotesize $x$}}] at (4,-0.2) {};
\node (c) [circle,fill=ForestGreen,draw=black,minimum size=0.1cm,inner sep=0.2mm,label={[xshift=0.2cm,yshift=-0.3cm]{\footnotesize $y$}}] at (4.2,0.15) {};
 \node (d) [circle,fill=ForestGreen,draw=black,minimum size=0.1cm,inner sep=0.2mm,label={[xshift=-0.2cm,yshift=-0.25cm]{\footnotesize $v$}}] at (3.8,0.15) {};
  \node (a) [circle,fill=ForestGreen,draw=black,draw=black,minimum size=0.1cm,inner sep=0.2mm,label={[xshift=-0.2cm,yshift=-0.2cm]{\footnotesize $u$}}] at (3.8,0.5) {};
 \draw[thick]  (c)--(b)--(d)--(a);\end{tikzpicture}}}};
 \node at (0.02,0.36)  {\resizebox{2.3cm}{!}{\begin{tikzpicture}
\node (b) [circle,fill=cyan,draw=black,minimum size=0.1cm,inner sep=0.2mm,label={[xshift=-0.22cm,yshift=-0.25cm]{\footnotesize $y$}}] at (0,0) {};
\node (c) [circle,fill=cyan,draw=black,minimum size=0.1cm,inner sep=0.2mm,label={[xshift=-0.22cm,yshift=-0.25cm]{\footnotesize $x$}}] at (0,0.35) {};
 \node (d) [circle,fill=cyan,draw=black,minimum size=0.1cm,inner sep=0.2mm,label={[xshift=-0.45cm,yshift=-0.32cm]{\footnotesize $\{u,v\}$}}] at (0,0.7) {};
 \draw[thick]  (b)--(c)--(d);\end{tikzpicture}}};
    \node at  (0,0.22)  {\resizebox{2.3cm}{!}{\begin{tikzpicture}
\node (c) [circle,fill=WildStrawberry,draw=black,minimum size=0.1cm,inner sep=0.2mm,label={[xshift=-0.45cm,yshift=-0.32cm]{\footnotesize $\{x,y\}$}}] at (0,0) {};
 \node (d) [circle,fill=WildStrawberry,draw=black,minimum size=0.1cm,inner sep=0.2mm,label={[xshift=-0.45cm,yshift=-0.32cm]{\footnotesize $\{u,v\}$}}] at (0,0.35) {};
 \draw[thick]  (c)--(d);\end{tikzpicture}}};
 \node at (0,0.1) {{\resizebox{4cm}{!}{\begin{tikzpicture}
\node (b) [circle,fill=cyan,draw=black,minimum size=0.1cm,inner sep=0.2mm,label={[yshift=-0.45cm]{\footnotesize $x$}}] at (4,-0.2) {};
\node (c) [circle,fill=cyan,draw=black,minimum size=0.1cm,inner sep=0.2mm,label={[xshift=-0.45cm,yshift=-0.31cm]{\footnotesize $\{u,v\}$}}] at (3.8,0.15) {};
 \node (d) [circle,fill=cyan,draw=black,minimum size=0.1cm,inner sep=0.2mm,label={[xshift=0.2cm,yshift=-0.3cm]{\footnotesize $y$}}] at (4.2,0.15) {};
 \draw[thick]  (c)--(b)--(d);\end{tikzpicture}}}};
\node at  (0.14,0.065) { {\resizebox{2.9cm}{!}{\begin{tikzpicture}
\node (b) [circle,fill=ForestGreen,draw=black,minimum size=0.1cm,inner sep=0.2mm,label={[yshift=-0.45cm]{\footnotesize $x$}}] at (4,-0.2) {};
\node (c) [circle,fill=ForestGreen,draw=black,minimum size=0.1cm,inner sep=0.2mm,label={[xshift=0.2cm,yshift=-0.3cm]{\footnotesize $y$}}] at (4.2,0.15) {};
 \node (d) [circle,fill=ForestGreen,draw=black,minimum size=0.1cm,inner sep=0.2mm,label={[xshift=-0.2cm,yshift=-0.25cm]{\footnotesize $u$}}] at (3.8,0.15) {};
  \node (a) [circle,fill=ForestGreen,draw=black,draw=black,minimum size=0.1cm,inner sep=0.2mm,label={[xshift=-0.2cm,yshift=-0.2cm]{\footnotesize $v$}}] at (3.8,0.5) {};
 \draw[thick]  (c)--(b)--(d)--(a);\end{tikzpicture}}}};
 \node at (0.12,0.23) {\resizebox{2.3cm}{!}{\begin{tikzpicture}
\node (b) [circle,fill=cyan,draw=black,minimum size=0.1cm,inner sep=0.2mm,label={[xshift=-0.22cm,yshift=-0.25cm]{\footnotesize $v$}}] at (0,0.7) {};
\node (c) [circle,fill=cyan,draw=black,minimum size=0.1cm,inner sep=0.2mm,label={[xshift=-0.22cm,yshift=-0.26cm]{\footnotesize $u$}}] at (0,0.35) {};
 \node (d) [circle,fill=cyan,draw=black,minimum size=0.1cm,inner sep=0.2mm,label={[xshift=-0.45cm,yshift=-0.32cm]{\footnotesize $\{x,y\}$}}] at (0,0) {};
 \draw[thick]  (b)--(c)--(d);\end{tikzpicture}}};
\end{tikzpicture}}}
\end{center}
\smallskip
Observe that the rank of the composition is the sum of the corresponding ranks. 

\medskip

Let us provide the details of the construction of the composition
\begin{center}
\begin{tikzpicture}
\node at   (-0.95,0.45) {{\resizebox{0.525cm}{!}{\begin{tikzpicture}
\node (L) [circle,fill=ForestGreen,draw=black,minimum size=0.1cm,inner sep=0.2mm,label={[xshift=-0.22cm,yshift=-0.25cm]{\footnotesize $x$}}] at (0,0) {};
\node (B1) [circle,fill=ForestGreen,draw=black,minimum size=0.1cm,inner sep=0.2mm,label={[xshift=-0.22cm,yshift=-0.27cm]{\footnotesize $y$}}] at (0,0.38) {};
 \draw[thick]  (L)--(B1);\end{tikzpicture}}}}; 
\end{tikzpicture} \enspace\raisebox{1.25em}{$\circ_2$} \begin{tikzpicture}
\node at   (-0.95,0.45) {{\resizebox{0.525cm}{!}{\begin{tikzpicture}
\node (L) [circle,fill=ForestGreen,draw=black,minimum size=0.1cm,inner sep=0.2mm,label={[xshift=-0.22cm,yshift=-0.25cm]{\footnotesize $v$}}] at (0,0) {};
\node (B1) [circle,fill=ForestGreen,draw=black,minimum size=0.1cm,inner sep=0.2mm,label={[xshift=-0.22cm,yshift=-0.25cm]{\footnotesize $u$}}] at (0,0.38) {};
 \draw[thick]  (L)--(B1);\end{tikzpicture}}}}; 
\end{tikzpicture} \enspace\raisebox{1.25em}{$=$} \begin{tikzpicture}
\node at   (-0.14,0.065)  { {\resizebox{1.35cm}{!}{\begin{tikzpicture}
\node (b) [circle,fill=ForestGreen,draw=black,minimum size=0.1cm,inner sep=0.2mm,label={[yshift=-0.45cm]{\footnotesize $x$}}] at (4,-0.2) {};
\node (c) [circle,fill=ForestGreen,draw=black,minimum size=0.1cm,inner sep=0.2mm,label={[xshift=0.2cm,yshift=-0.3cm]{\footnotesize $y$}}] at (4.2,0.15) {};
 \node (d) [circle,fill=ForestGreen,draw=black,minimum size=0.1cm,inner sep=0.2mm,label={[xshift=-0.2cm,yshift=-0.25cm]{\footnotesize $v$}}] at (3.8,0.15) {};
  \node (a) [circle,fill=ForestGreen,draw=black,draw=black,minimum size=0.1cm,inner sep=0.2mm,label={[xshift=-0.2cm,yshift=-0.2cm]{\footnotesize $u$}}] at (3.8,0.5) {};
 \draw[thick]  (c)--(b)--(d)--(a);\end{tikzpicture}}}};
\end{tikzpicture}
\end{center}
By definition,  we consider the left-recursive subtrees of ${\cal T}_1$ obtained by removing the edge $x$ and we search for the one containing the vertex that used to be indexed by $1$ in ${\cal T}_1$. Since this subtree is reduced to a corolla, the resulting construct will have \raisebox{-1.1em}{\begin{tikzpicture}
\node at   (-0.95,0.45) {{\resizebox{0.525cm}{!}{\begin{tikzpicture}
\node (L) [circle,fill=ForestGreen,draw=black,minimum size=0.1cm,inner sep=0.2mm,label={[xshift=-0.22cm,yshift=-0.25cm]{\footnotesize $v$}}] at (0,0) {};
\node (B1) [circle,fill=ForestGreen,draw=black,minimum size=0.1cm,inner sep=0.2mm,label={[xshift=-0.22cm,yshift=-0.25cm]{\footnotesize $u$}}] at (0,0.38) {};
 \draw[thick]  (L)--(B1);\end{tikzpicture}}}}; 
\end{tikzpicture} } grafted to the root vertex of  \raisebox{-1.1em}{\begin{tikzpicture}
\node at   (-0.95,0.45) {{\resizebox{0.525cm}{!}{\begin{tikzpicture}
\node (L) [circle,fill=ForestGreen,draw=black,minimum size=0.1cm,inner sep=0.2mm,label={[xshift=-0.22cm,yshift=-0.25cm]{\footnotesize $x$}}] at (0,0) {};
\node (B1) [circle,fill=ForestGreen,draw=black,minimum size=0.1cm,inner sep=0.2mm,label={[xshift=-0.22cm,yshift=-0.27cm]{\footnotesize $y$}}] at (0,0.38) {};
 \draw[thick]  (L)--(B1);\end{tikzpicture}}}}; 
\end{tikzpicture}}.
\demo
\end{example}

\smallskip

 The proof that the operad ${\EuScript O}_{\infty}$ is free as an operad of vector spaces will rely on the operation of collapsing an edge in a rooted tree.
 We  recall the relevant  definitions and results below.  
 \smallskip

\begin{definition}\label{collapse}
Let  ${\cal T}\in {\tt Tree}(n)$ and let $e$ be an (internal) edge of ${\cal T}$. We define ${\cal T}\backslash e\in {\tt Tree}(n)$ to be the rooted tree obtained by collapsing the edge $e$ {\em downwards}, i.e. in such a way that the vertex that remains after $e$ is collapsed is the target vertex of $e$, i.e. the root vertex $\rho({\cal T}(\{e\}))$ of the subtree of ${\cal T}$ determined by $e$; after the collapse, the inputs of $\rho({\cal T}(\{e\}))$ will be all the inputs  of  ${\cal T}(\{e\})$  and they will be ordered as in ${\cal T}(\{e\})$. The remaining of the structure of ${\cal T}$ remains the same in ${\cal T}\backslash e$. 

\smallskip

As for the edge collapses of operadic trees $({\cal T},\sigma)$, we take the convention to consider  both ${\cal T}\backslash e$ and   ${\cal T}(\{e\})$ as  left-recursive operadic trees.  
\end{definition}
\smallskip

 The following lemma  is a straightforward consequence of Definition \ref{collapse}.
 \smallskip
\begin{lemma}\label{elementarydecomposition}
For an operadic tree  $({\cal T},\sigma)$ and  $e\in {\it e}({\cal T})$, there exists a unique permutation $\sigma_e\in\Sigma_k$, such that the equality $({\cal T},\sigma)=({\cal T}\backslash e\circ_{\rho({\cal T}(\{e\}))} {\cal T}(\{e\}))^{\sigma_e}$ holds in the operad ${\EuScript O}$.
\end{lemma}
 \smallskip
\begin{remark}[Coherence of edge collapses]\label{edgecoherence}
Note that, if $e_1,e_2\in {\it e}({\cal T})$, then $({\cal T}\backslash e_1)\backslash e_2=({\cal T}\backslash e_2)\backslash e_1$. This equality ensures that, having fixed a set of edges of a tree, the order of collapsing the edges from that set has no effect on the resulting tree.

\smallskip

 Note that, if a fixed set of edges determines a subtree ${\cal S}$ of ${\cal T}$, then the root vertex $\rho({\cal S})$ of ${\cal S}$ remains a vertex in the tree ${\cal T}\backslash{\cal S}$, obtained by collapsing all the edges of ${\cal S}$.
\end{remark}
 \smallskip
 The following result is  a consequence of Lemma \ref{elementarydecomposition} and Remark \ref{edgecoherence}.
 \smallskip
\begin{lemma}\label{subtreedecomposition}
For an operadic tree $({\cal T},\sigma)$ and a subtree ${\cal S}$ of ${\cal T}$, considered as a left-recursive operadic tree, there exists a unique permutation $\sigma_{\cal S}\in\Sigma_k$, such that the equality $({\cal T},\sigma)=({\cal T}\backslash {\cal S}\circ_{\rho({\cal S})} {\cal S})^{\sigma_{\cal S}}$ holds in the operad ${\EuScript O}$.  
\end{lemma}
 \smallskip
 
We now have all the prerequisites for proving that the ${\EuScript O}_{\infty}$ operad is free. The idea is simple and it has already been indicated in  \S\ref{edgegraph1}: constructs of arbitrary hypergraphs are non-planar trees, but if a hypergraph is the edge-graph of some operadic tree, then the hypergraph itelf,  as well as its  constructs, inherit a canonical planar embedding from that tree. This observation  gives us a way to represent each triple $({\cal T},\sigma,C)$, where $({\cal T},\sigma)$ is an operadic tree  and  $C:{\bf H}_{\cal T}$, as a planar tree of a free operad. 
\smallskip
\begin{thm}\label{free}
As a coloured operad of vector spaces, ${\EuScript O}_{\infty}$ is the free ${\mathbb N}$-coloured operad generated by the equivalence classes of left-recursive operadic trees: $${\EuScript O}_{\infty}\simeq {\cal T}_{\mathbb N}\Bigg(\, \displaystyle\bigoplus_{k\geq 2}\,\,\bigoplus_{n_1,\dots,n_k\geq 1}\,\,\bigoplus_{{\cal T}\in \,{\tt Tree}(n_1,\dots,n_k;n)}\,{\Bbbk}\,\Bigg).$$
\end{thm}
\begin{proof} We define an isomorphism $\alpha$ between  ${\EuScript O}_{\infty}$  and
the operad of  ${\mathbb N}$-coloured planar rooted trees  whose
vertices are decorated by left-recursive operadic trees and whose leaves are labeled by a permutation on the number of them, by induction on the number of vertices of the construct $C$ of a given operation $({\cal T},\sigma,C)\in {\EuScript O}_{\infty}(n_1,\dots,n_k;n)$. Denote, for $1\leq i\leq k$, with $\sigma_i$ the index  given by $\sigma$ to the  $i$-th vertex in the left-recursive ordering of ${\cal T}$.  

\smallskip

\begin{itemize}
\item If $C$ is the maximal construct $e({\cal T})$, then\begin{center}\raisebox{2.1em}{$\alpha({\cal T},\sigma,e({\cal T}))=$} 
\begin{tikzpicture}
\node (r) [rectangle,draw=black,rounded corners=.1cm,inner sep=1mm] at (0,0) {\footnotesize $\enspace{\cal T}\enspace$};
\node (rl) [circle,draw=none,inner sep=0.4mm] at (-0.8,1.1) {\footnotesize ${\bf\sigma}_1$};
\node (rr) [circle,draw=none,inner sep=0.4mm] at (0.8,1.1) {\footnotesize $\sigma_k$};
\node (rl) [circle,draw=none,inner sep=0.4mm] at (-0.8,0.8) {\footnotesize $n_{\sigma_1}$};
\node (rr) [circle,draw=none,inner sep=0.4mm] at (0.8,0.8) {\footnotesize $n_{\sigma_k}$};
\node (rc) [circle,draw=none,inner sep=0.4mm] at (0,0.8) {\footnotesize $\dots$};
\node (erc) [circle,draw=none,inner sep=0.4mm] at (0,1.1) {\footnotesize $\dots$};
\node (rb) [circle,draw=none,inner sep=0.4mm] at (0.15,-0.65) {\footnotesize $n$};
\draw (0.8,0.7)--(r)--(-0.8,0.7);
\draw (r)--(0,-0.7);
\end{tikzpicture}
\end{center}
Note that each input of $\alpha({\cal T},\sigma,e({\cal T}))$ is uniquely determined by any of the following two data: its position  in the planar structure (the number), or, the name of the vertex indexed by  ${i}$ in ${\cal T}$ (the name).\\[-0.3cm]

\item Suppose that $C=X\{C_1,\dots,C_p\}$, where ${\bf H}_{\cal T}\backslash X\leadsto {\bf H}_{{\cal T}_1},\dots, {\bf H}_{{\cal T}_p}$ and    $C_i:{\bf H}_{{\cal T}_i}$. Let ${\cal T}_X$ be the left-recursive operadic tree obtained from $({\cal T},\sigma)$ by collapsing all the edges from $e({\cal T})\backslash X$ (see Remark  \ref{edgecoherence}).  Observe  that the collapse of the edges ${\it e}({\cal T})\backslash X$ that defines ${\cal T}_X$ is, in fact, the collapse of  the subtrees ${\cal T}_i$, $1\leq i\leq p$, of ${\cal T}$. By the definition of a collapse, each ${\cal T}_i$   will collapse to the vertex $\rho({\cal T}_i)$  of ${\cal T}_X$; in particular, since the ${\cal T}_i$'s are mutually disjoint subtrees of ${\cal T}$, all the vertices $\rho({\cal T}_i)$ will be mutually distinct.   We define
$$\quad\quad\quad\quad\alpha({\cal T},\sigma,C)=\bigg((\cdots((\alpha({\cal T}_X,X)\circ_{\rho({\cal T}_{1})}  \alpha({\cal T}_{1},C_{1}) )\circ_{\rho({\cal T}_{2})}  \alpha({\cal T}_{2},C_{2}))\cdots)\circ_{{\rho({\cal T}_{p})}}\alpha({\cal T}_{p},C_{p})\bigg)^{\sigma_C},$$
where we the inputs of $\alpha({\cal T}_X,X)$ involved in grafting are represented by their names, in order to avoid the reindexing, and where $\sigma_C$ is the permutation determined uniquely thanks to (iterated application of) Lemma \ref{subtreedecomposition}. Note that each tree ${\cal T}_i$ above is considered as left-recursive.
 \end{itemize}
 For example,  for the  operadic tree 
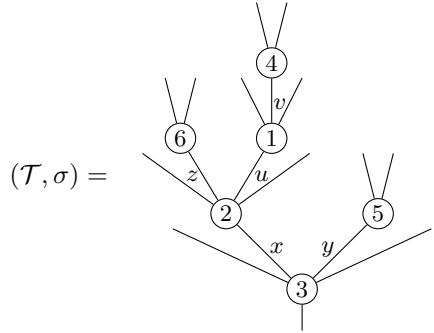
\begin{center}
\begin{tikzpicture}
   \node (E)[circle,draw=none,minimum size=4mm,inner sep=0.1mm] at (-3.3,1.5) {  $({\cal T},\sigma)=$};
    \node (E)[circle,draw=black,minimum size=4mm,inner sep=0.1mm] at (-0.1,0) {\small $3$};
    \node (F) [circle,draw=black,minimum size=4mm,inner sep=0.1mm] at (-1.1,1) {\small $2$};
    \node (A) [circle,draw=black,minimum size=4mm,inner sep=0.1mm] at (-1.7,2) {\small $6$};
 \node (q) [circle,draw=black,minimum size=4mm,inner sep=0.1mm] at (-0.5,2) {\small $1$};
 \node (u1) [circle,draw=black,minimum size=4mm,inner sep=0.1mm] at (-0.5,3) {\small $4$};
    \node (r) [circle,draw=black,minimum size=4mm,inner sep=0.1mm] at (0.9,1) {\small $5$};
 \draw[-] (1.6,0.8) -- (E)--(-1.8,0.8); 
\draw[-] (-1.9,2.8) -- (A)--(-1.5,2.8); 
 \draw[-] (0.7,1.8) -- (r)--(1.1,1.8); 
\draw[-] (-0.7,3.8) -- (u1)--(-0.3,3.8); 
 \draw[-] (E)--(-0.1,-0.55); 
 \draw[-] (-2.2,1.8) -- (F)--(0,1.8);   
 \draw[-] (-0.9,2.8) -- (q)--(-0.1,2.8);  
    \draw[-] (E)--(F) node [midway,right,xshift=-0.05cm,yshift=0.05cm] {\small $x$};
 \draw[-] (F)--(q) node [midway,right,xshift=-0.05cm,yshift=0cm] {\small $u$};
 \draw[-] (E)--(r) node [midway,left,xshift=0.05cm,yshift=0.025cm] {\small $y$};
    \draw[-] (F)--(A) node  [midway,right,xshift=-0.35cm,yshift=0cm] {\small $z$};
 \draw[-] (q)--(u1) node  [midway,right,xshift=-0.1cm,yshift=-0.05cm] {\small $v$};
   \end{tikzpicture} 
\end{center}
and constructs
\begin{center}
\raisebox{1.75em}{$C_1=$\,\,} {\resizebox{1.2cm}{!}{\begin{tikzpicture}
\node (b) [circle,fill=WildStrawberry,draw=black,minimum size=0.1cm,inner sep=0.2mm,label={[xshift=0.22cm,yshift=-0.25cm]{\footnotesize $u$}}] at (0,0.7) {};
\node (c) [circle,fill=WildStrawberry,draw=black,minimum size=0.1cm,inner sep=0.2mm,label={[xshift=0.45cm,yshift=-0.32cm]{\footnotesize $\{x,v\}$}}] at (0,0.35) {};
 \node (d) [circle,fill=WildStrawberry,draw=black,minimum size=0.1cm,inner sep=0.2mm,label={[xshift=0.45cm,yshift=-0.32cm]{\footnotesize $\{y,z\}$}}] at (0,0) {};
 \draw[thick]  (b)--(c)--(d);\end{tikzpicture}}}\quad\quad \raisebox{1.75em}{and}\quad\quad \raisebox{1.75em}{$C_2=$} {{\resizebox{1.5cm}{!}{\begin{tikzpicture}
\node (b) [circle,fill=WildStrawberry,draw=black,minimum size=0.1cm,inner sep=0.2mm,label={[yshift=-0.5cm]{\footnotesize $\{y,z,u\}$}}] at (4,-0.2) {};
\node (c) [circle,fill=WildStrawberry,draw=black,minimum size=0.1cm,inner sep=0.2mm,label={[xshift=-0.2cm,yshift=-0.25cm]{\footnotesize $x$}}] at (3.8,0.15) {};
 \node (d) [circle,fill=WildStrawberry,draw=black,minimum size=0.1cm,inner sep=0.2mm,label={[xshift=0.2cm,yshift=-0.25cm]{\footnotesize $v$}}] at (4.2,0.15) {};
 \draw[thick]  (c)--(b)--(d);\end{tikzpicture}}}}
\end{center}
of the hypergraph ${\bf H}_{\cal T}$ associated to   ${\cal T}$, we have that

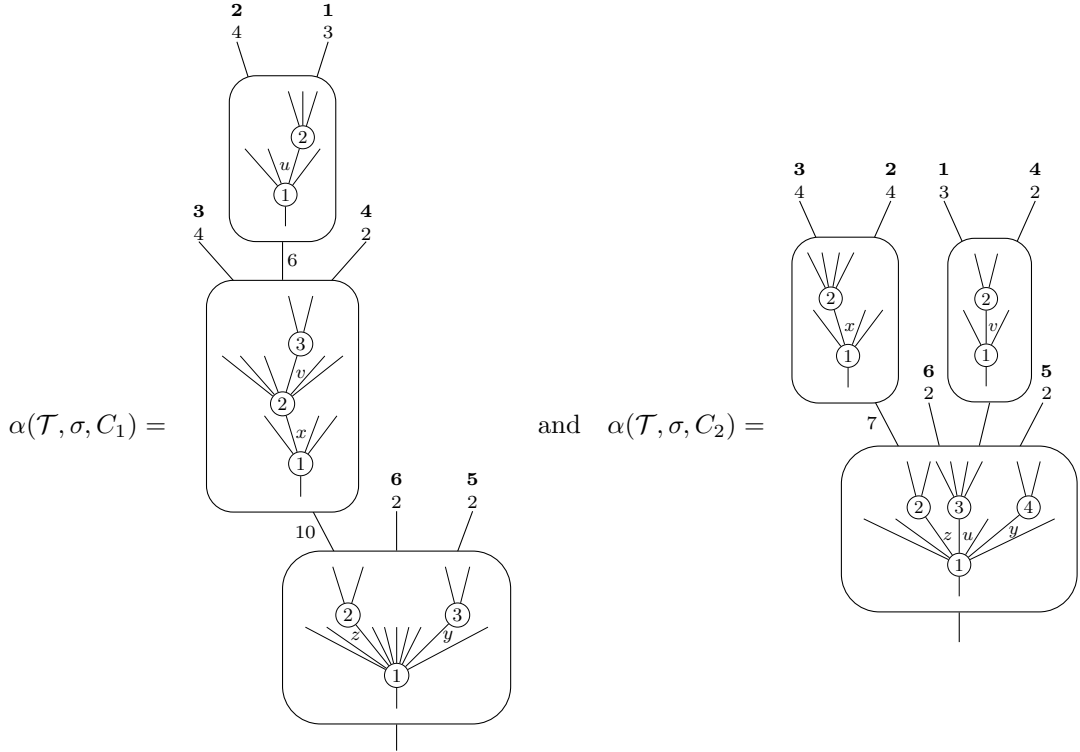
\begin{center}
\raisebox{12em}{$\alpha({\cal T},\sigma,C_1)= $}\quad \begin{tikzpicture}
\draw (0.75,0.5)--(-0.5,2.9);
\draw  (2.5,1.175)--(2.25,0.5);
\draw (-1.1,4.75)--(0,3.5)--(1.1,4.75);
\draw (-0.6,7.4)--(0,5.5)--(0.6,7.4);
\node (1) [rectangle,draw=black,fill=white,rounded corners=.5cm,inner sep=2mm] at (1.5,-0.5) {\resizebox{2.6cm}{!}{ \begin{tikzpicture}
    \node (E)[circle,draw=black,minimum size=4mm,inner sep=0.1mm] at (-0,0) {\small $1$};
    \node (F) [circle,draw=black,minimum size=4mm,inner sep=0.1mm] at (-0.8,1) {\small $2$};
\node (dd) [circle,draw=black,minimum size=4mm,inner sep=0.1mm] at (1,1) {\small $3$};
 \draw[-] (1.5,0.8) -- (E)--(-1.5,0.8);  
  \draw[-] (-1.15,0.8) -- (E)--(0,0.8);  
    \draw[-] (0.2,0.8) -- (E)--(-0.2,0.8);  
        \draw[-] (0.4,0.8) -- (E)--(-0.4,0.8);  
 \draw[-] (E)--(0,-0.55); 
 \draw[-] (-1.05,1.8) -- (F)--(-0.55,1.8);   
 \draw[-] (0.8,1.8) -- (dd)--(1.2,1.8);   
    \draw[-] (E)--(F) node [midway,left,yshift=0.15cm,xshift=0.015cm] {\small  $z$};
 \draw[-] (E)--(dd) node [midway,right,yshift=0.135cm,xshift=0.05cm] {\small  $y$};
\end{tikzpicture} }  };
\node (2) [rectangle,draw=black,fill=white,rounded corners=.35cm,inner sep=2mm] at (0,2.7) {\resizebox{1.6cm}{!}{\begin{tikzpicture}
    \node (E)[circle,draw=black,minimum size=4mm,inner sep=0.1mm] at (0.3,0) {\small $1$};
    \node (F) [circle,draw=black,minimum size=4mm,inner sep=0.1mm] at (0,1) {\small $2$};
    \node (r) [circle,draw=black,minimum size=4mm,inner sep=0.1mm] at (0.3,2) {\small $3$};
 \draw[-] (0.9,0.8) -- (E)--(-0.3,0.8);  
 \draw[-] (0.1,2.8) -- (r)--(0.5,2.8); 
 \draw[-] (0.6,0.8)--(E)--(0.3,-0.55); 
 \draw[-] (-0.7,1.8) -- (F)--(0.7,1.8);   
  \draw[-] (-1,1.8) -- (F)--(1,1.8);
    \draw[-]  (F)--(-0.3,1.8);
    \draw[-] (E)--(F) node [midway,right,xshift=-0.14cm]{\small  $x$};
 \draw[-] (F)--(r) node [midway,right,xshift=-0.14cm] {\small  $v$};
\end{tikzpicture}}};
\node (3) [rectangle,draw=black,fill=white,rounded corners=.35cm,inner sep=2mm] at (0,5.85) {\resizebox{1cm}{!}{\begin{tikzpicture}
    \node (E)[circle,draw=black,minimum size=4mm,inner sep=0.1mm] at (-0.3,0) {\small $1$};
    \node (F) [circle,draw=black,minimum size=4mm,inner sep=0.1mm] at (0,1) {\small $2$};
 \draw[-] (0.3,0.8) -- (E)--(-1,0.8); 
 \draw[-] (-0.6,0.8)--(E)--(-0.3,-0.55); 
 \draw[-] (0.25,1.8) -- (F)--(-0.25,1.8); 
  \draw[-] (F)--(0,1.8); 
    \draw[-] (E)--(F) node [midway,left,xshift=0.14cm] {\small  $u$};
\end{tikzpicture}}};
\node (n1) [rectangle,draw=none,inner sep=0mm] at (0.3,0.9) {\scriptsize $10$};
\node (n2) [rectangle,draw=none,inner sep=0mm] at (1.5,1.3) {\scriptsize $2$};
\node (n3) [rectangle,draw=none,inner sep=0mm] at (2.5,1.3) {\scriptsize $2$};
\node (n3) [rectangle,draw=none,inner sep=0mm] at (-1.1,4.85) {\scriptsize $4$};
\node (n4) [rectangle,draw=none,inner sep=0mm] at (0.13,4.5) {\scriptsize $6$};
\node (n5) [rectangle,draw=none,inner sep=0mm] at (1.1,4.85) {\scriptsize $2$};
\node (n6) [rectangle,draw=none,inner sep=0mm] at (-0.6,7.525) {\scriptsize $4$};
\node (n7) [rectangle,draw=none,inner sep=0mm] at (0.6,7.525) {\scriptsize $3$};
\node (n2') [rectangle,draw=none,inner sep=0mm] at (1.5,1.6) {\scriptsize $\bf 6$};
\node (n3') [rectangle,draw=none,inner sep=0mm] at (2.5,1.6) {\scriptsize $\bf 5$};
\node (n3') [rectangle,draw=none,inner sep=0mm] at (-1.1,5.15) {\scriptsize $\bf 3$};
\node (n5') [rectangle,draw=none,inner sep=0mm] at (1.1,5.15) {\scriptsize $\bf 4$};
\node (n6') [rectangle,draw=none,inner sep=0mm] at (-0.6,7.825) {\scriptsize $\bf 2$};
\node (n7') [rectangle,draw=none,inner sep=0mm] at (0.6,7.825) {\scriptsize $\bf 1$};
\draw (1.5,-2)--(1);
\draw (2)--(3);
\draw (1)--(1.5,1.175);
\end{tikzpicture}\quad\raisebox{12em}{and}\quad \raisebox{12em}{$\alpha({\cal T},\sigma,C_2)= $}\quad \raisebox{4.1em}{\begin{tikzpicture}
\draw (0.75,0.5)--(-0.5,2.9);
\draw  (2.6,1.175)--(2.25,0.5);
\draw (-0.6,3.85)--(0,2.5)--(0.6,3.85);
\draw (1.3,3.85)--(1.9,2.5)--(2.5,3.85);
\node (1) [rectangle,draw=black,fill=white,rounded corners=.5cm,inner sep=2mm] at (1.5,-0.5) {\resizebox{2.7cm}{!}{ \begin{tikzpicture}
    \node (E)[circle,draw=black,minimum size=4mm,inner sep=0.1mm] at (-0,0) {\small $1$};
    \node (F) [circle,draw=black,minimum size=4mm,inner sep=0.1mm] at (-0.7,1) {\small $2$};
 \node (q) [circle,draw=black,minimum size=4mm,inner sep=0.1mm] at (0,1) {\small $3$};
    \node (r) [circle,draw=black,minimum size=4mm,inner sep=0.1mm] at (1.2,1) {\small $4$};
 \draw[-] (1.65,0.8) -- (E)--(-1.65,0.8); 
  \draw[-] (-1.1,0.8) -- (E)--(0.5,0.8); 
 \draw[-] (-0.4,1.8) -- (q)--(0.4,1.8);   
  \draw[-] (-0.15,1.8) -- (q)--(0.15,1.8); 
 \draw[-] (1,1.8) -- (r)--(1.4,1.8); 
 \draw[-] (E)--(0,-0.55); 
 \draw[-] (-0.9,1.8) -- (F)--(-0.5,1.8);   
    \draw[-] (E)--(F) node [midway,right,xshift=-0.14cm] {\small  $z$};
 \draw[-] (E)--(q) node [midway,right,xshift=-0.15cm] {\small  $u$};
 \draw[-] (E)--(r) node [midway,right,xshift=0.05cm,yshift=0.075cm] {\small  $y$};
   \end{tikzpicture} }  };
\node (2) [rectangle,draw=black,fill=white,rounded corners=.35cm,inner sep=2mm] at (0,2.27) {\resizebox{1cm}{!}{\begin{tikzpicture}
    \node (E)[circle,draw=black,minimum size=4mm,inner sep=0.1mm] at (0,0) {\small $1$};
    \node (F) [circle,draw=black,minimum size=4mm,inner sep=0.1mm] at (-0.3,1) {\small $2$};
 \draw[-] (0.3,0.8)--(E)--(0.6,0.8);
 \draw[-] (-0.6,0.8)--(E)--(0,-0.55); 
 \draw[-] (-0.45,1.8) -- (F)--(-0.15,1.8); 
  \draw[-] (-0.7,1.8) -- (F)--(0.1,1.8); 
    \draw[-] (E)--(F) node [midway,right,xshift=-0.12cm] {\small  $x$};
\end{tikzpicture}}};
\node (3) [rectangle,draw=black,fill=white,rounded corners=.35cm,inner sep=2mm] at (1.9,2.27) {\resizebox{0.7cm}{!}{\begin{tikzpicture}
    \node (E)[circle,draw=black,minimum size=4mm,inner sep=0.1mm] at (0,0) {\small $1$};
    \node (F) [circle,draw=black,minimum size=4mm,inner sep=0.1mm] at (0,1) {\small $2$};
 \draw[-] (0.4,0.8) -- (E)--(-0.4,0.8); 
 \draw[-] (E)--(0,-0.55); 
 \draw[-] (0.2,1.8) -- (F)--(-0.2,1.8); 
    \draw[-] (E)--(F) node [midway,right,xshift=-0.165cm] {\small  $v$};
\end{tikzpicture} }};
\node (n1) [rectangle,draw=none,inner sep=0mm] at (0.35,0.9) {\scriptsize $7$};
\node (n2) [rectangle,draw=none,inner sep=0mm] at (1.1,1.3) {\scriptsize $2$};
\node (n3) [rectangle,draw=none,inner sep=0mm] at (2.65,1.3) {\scriptsize $2$};
\node (n4) [rectangle,draw=none,inner sep=0mm] at (-0.6,3.975) {\scriptsize $4$};
\node (n5) [rectangle,draw=none,inner sep=0mm] at (0.6,3.975) {\scriptsize $4$};
 \node (n6) [rectangle,draw=none,inner sep=0mm] at (1.3,3.975) {\scriptsize $3$};
 \node (n7) [rectangle,draw=none,inner sep=0mm] at (2.5,3.975) {\scriptsize $2$};
\node (n2') [rectangle,draw=none,inner sep=0mm] at (1.1,1.6) {\scriptsize ${\bf 6}$};
\node (n2') [rectangle,draw=none,inner sep=0mm] at (2.65,1.6) {\scriptsize ${\bf 5}$};
\node (n4') [rectangle,draw=none,inner sep=0mm] at (-0.6,4.275) {\scriptsize ${\bf 3}$};
\node (n5') [rectangle,draw=none,inner sep=0mm] at (0.6,4.275) {\scriptsize ${\bf 2}$};
 \node (n6') [rectangle,draw=none,inner sep=0mm] at (1.3,4.275) {\scriptsize ${\bf 1}$};
 \node (n7') [rectangle,draw=none,inner sep=0mm] at (2.5,4.275) {\scriptsize ${\bf 4}$};
\draw (1.5,-2)--(1);
\draw (1.9,1.175)--(1)--(1.1,1.175);
\end{tikzpicture}}
\end{center}
 where the edge and leaf colours given by natural numbers are represented using regular font, and the indexing of the leaves  is represented using bold font. Observe that, modulo leaves,   $\alpha({\cal T},C)$ has the same shape as $C$.

The inverse of $\alpha$ is defined by composing the left-recursive operadic trees that decorate the nodes of an element $(T,\sigma)$ of the free operad, in the way dictated by the edges of that element, followed by reindexing the vertices of the resulting tree as specified by $\sigma$, and by extracting the corresponding construct in the following way: first, remove all the leaves of $T$, and then, for each vertex of $T$, replace the  operadic tree that decorates that vertex by the maximal constructs of its associated hypergraph. Lemma \ref{subtreedecomposition} ensures that this correspondence is indeed an isomorphism. \end{proof}
\smallskip

 Having in mind the free operad description of ${\EuScript O}_{\infty}$, we adopt the following notational convention about   constructs.
 \smallskip

\begin{convention} If we wish to  incorporate the specification of the planar embedding of a construct  $C:{\bf H}_{\cal T}$  into the notation for $C$, we shall write   $C=X(C_1,\dots,C_p)$  instead of $C=X\{C_1,\dots,C_p\}$, if $C_1,\dots,C_p$ appear in that order in the tree $\alpha({\cal T},\sigma,C)$.
\end{convention}

\smallskip

\subsubsection{The  operad  ${\EuScript O}_{\infty}$ as a differential graded operad}\label{odg} In order to equip the ${\EuScript O}_{\infty}$ operad with a grading  and   a differential, we shall use   the free operad  structure of ${\EuScript O}_{\infty}$  and count the edges and leaves that lie in a particular position relative to some other edge or a leaf, in the way formalized by the following definition.
\begin{definition}\label{orient}
Let ${\cal T}$ be a planar rooted tree, and let $e\in e({\cal T})$ and $l\in {\it i}({\cal T})$ be an internal edge and an input leaf of ${\cal T}$, respectively.  \\[-0.3cm]
\begin{itemize}
\item The internal edges   {\em  below} $e$  (resp. $l$) in ${\cal T}$  are the internal edges of  ${\cal T}$  that lie on the unique path from the vertex $\rho({\cal T}(\{e\}))$   (resp. $\rho({\cal T}(\{l\}))$ ) to  $\rho({\cal T})$. \\[-0.3cm] 
\item The edges and leaves {\em on the left} (resp. {\em on the right})  from $e$   are the edges and leaves of ${\cal T}$   which are strictly on the left (resp.  right) from the unique path from the first (resp. last) leaf of the subtree of  ${\cal T}$ rooted at $e$, to ${r}({\cal T})$. The edges and leaves  on the left  (resp.   on the right) from $l$   are the edges and leaves of ${\cal T}$  which are strictly on the left (resp.  right) from the unique path from $l$ to ${r}({\cal T})$.
\end{itemize}
\smallskip
 For $e\in e({\cal T})$ and  $l\in {\it i}({\cal T})$, denote with $E_{\leq e}({\cal T})$ the sum of the number of all edges and leaves on the left from $e$ and the number of all edges below $e$ in ${\cal T}$, and with $E_{l>}({\cal T})$ the number of all edges and leaves on the right from $l$ in ${\cal T}$.
\end{definition}

\medskip

We grade the vector space  ${\EuScript O}_{\infty}(n_1,n_2,\dots,n_k;n)$ by setting   \smallskip
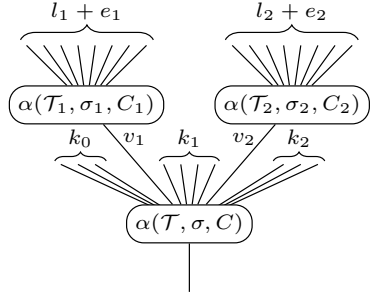
  $$|({\cal T},\sigma,C)|=  {\it e}({\cal T}) -v(C)=k-1-v(C).$$ 
  \smallskip
 Note that  the grading agrees with the rank of the construct $C:{\bf H}_{\cal T}$ in ${\cal A}({\bf H}_{\cal T})$; in particular, $0\leq |({\cal T},C)|\leq k-2$.  Observe also that  the partial composition operation of ${\EuScript O}_{\infty}$ respects the grading, in the sense that  
 \smallskip
  $$|({\cal T}_1,\sigma_1,C_1)\circ_i ({\cal T}_2,\sigma_2,C_2)|=|({\cal T}_1,\sigma_1,C_1)| +|({\cal T}_2,\sigma_2,C_2)|.$$
  \smallskip
   If $({\cal T},\sigma)$ is clear from the context,  we shall often write $|C|$ for what is actually $|({\cal T},\sigma,C)|$.

\medskip
  
In the  graded version of the ${\EuScript O}_{\infty}$ operad,  signs show up in the definition of the partial composition: we adapt the definition of $\circ_i$ by setting
\smallskip
$$({\cal T}_1,\sigma_1,C_1)\circ_i ({\cal T}_2,\sigma_2,C_2)=(-1)^{\varepsilon}({\cal T}_1\bullet_i {\cal T}_2,\sigma_1\bullet_i\sigma_2,C_1\bullet_i C_2),$$
\smallskip
where $\varepsilon$ is the number of edges and leaves of $\alpha({\cal T}_1,\sigma_1,C_1)$ on the right of the leaf indexed by $i$, multiplied by the number of all edges and leaves of $\alpha({\cal T}_2,\sigma_2,C_2)$, minus the root:
\smallskip
 $$\varepsilon= E_{i>}(\alpha({\cal T}_1,\sigma_1,C_1))\cdot (E(\alpha({\cal T}_2,\sigma_2,C_2))-1).$$
\smallskip
 \begin{example}
In the graded setting, the composition 
\begin{center}
\begin{tikzpicture}
\node at   (-0.95,0.45) {{\resizebox{0.525cm}{!}{\begin{tikzpicture}
\node (L) [circle,fill=ForestGreen,draw=black,minimum size=0.1cm,inner sep=0.2mm,label={[xshift=-0.22cm,yshift=-0.25cm]{\footnotesize $x$}}] at (0,0) {};
\node (B1) [circle,fill=ForestGreen,draw=black,minimum size=0.1cm,inner sep=0.2mm,label={[xshift=-0.22cm,yshift=-0.27cm]{\footnotesize $y$}}] at (0,0.38) {};
 \draw[thick]  (L)--(B1);\end{tikzpicture}}}}; 
\end{tikzpicture} \enspace\raisebox{1.25em}{$\circ_2$} \begin{tikzpicture}
\node at   (-0.95,0.45) {{\resizebox{0.525cm}{!}{\begin{tikzpicture}
\node (L) [circle,fill=ForestGreen,draw=black,minimum size=0.1cm,inner sep=0.2mm,label={[xshift=-0.22cm,yshift=-0.25cm]{\footnotesize $v$}}] at (0,0) {};
\node (B1) [circle,fill=ForestGreen,draw=black,minimum size=0.1cm,inner sep=0.2mm,label={[xshift=-0.22cm,yshift=-0.25cm]{\footnotesize $u$}}] at (0,0.38) {};
 \draw[thick]  (L)--(B1);\end{tikzpicture}}}}; 
\end{tikzpicture} \enspace\raisebox{1.25em}{$=$} \begin{tikzpicture}
\node at   (-0.14,0.065)  { {\resizebox{1.35cm}{!}{\begin{tikzpicture}
\node (b) [circle,fill=ForestGreen,draw=black,minimum size=0.1cm,inner sep=0.2mm,label={[yshift=-0.45cm]{\footnotesize $x$}}] at (4,-0.2) {};
\node (c) [circle,fill=ForestGreen,draw=black,minimum size=0.1cm,inner sep=0.2mm,label={[xshift=0.2cm,yshift=-0.3cm]{\footnotesize $y$}}] at (4.2,0.15) {};
 \node (d) [circle,fill=ForestGreen,draw=black,minimum size=0.1cm,inner sep=0.2mm,label={[xshift=-0.2cm,yshift=-0.25cm]{\footnotesize $v$}}] at (3.8,0.15) {};
  \node (a) [circle,fill=ForestGreen,draw=black,draw=black,minimum size=0.1cm,inner sep=0.2mm,label={[xshift=-0.2cm,yshift=-0.2cm]{\footnotesize $u$}}] at (3.8,0.5) {};
 \draw[thick]  (c)--(b)--(d)--(a);\end{tikzpicture}}}};
\end{tikzpicture}
\end{center}
from Example  \ref{compositionconstructs} gets multiplied by $+$. Indeed, 
\smallskip
\begin{center}
\raisebox{8em}{$\alpha({\cal T}_1,\sigma_1,\!\!\!\raisebox{-1.1em}{\begin{tikzpicture}
\node at   (-0.95,0.45) {{\resizebox{0.525cm}{!}{\begin{tikzpicture}
\node (L) [circle,fill=ForestGreen,draw=black,minimum size=0.1cm,inner sep=0.2mm,label={[xshift=-0.22cm,yshift=-0.25cm]{\footnotesize $x$}}] at (0,0) {};
\node (B1) [circle,fill=ForestGreen,draw=black,minimum size=0.1cm,inner sep=0.2mm,label={[xshift=-0.22cm,yshift=-0.27cm]{\footnotesize $y$}}] at (0,0.38) {};
 \draw[thick]  (L)--(B1);\end{tikzpicture}}}}; 
\end{tikzpicture}})= $}\quad \begin{tikzpicture}
\draw (0,-0.5)--(1,2.1);
\draw (0.4,3.55)--(1,2)--(1.6,3.55);
\node (1) [rectangle,draw=black,fill=white,rounded corners=.5cm,inner sep=2mm] at (0,-0.5) {\resizebox{2.4cm}{!}{\begin{tikzpicture}
    \node (E)[circle,draw=black,minimum size=4mm,inner sep=0.1mm] at (0.75,0) {\small $1$};
    \node (F) [circle,draw=black,minimum size=4mm,inner sep=0.1mm] at (-0.3,1) {\small $2$};
        \node (x)[circle,draw=none,minimum size=4mm,inner sep=0.1mm] at (0.275,0.65) {\small $x$};
 \draw[-] (-0.1,1.8) -- (F)--(-0.5,1.8); 
 \draw[-] (0.75,0.8) -- (E)--(-0.75,0.8); 
  \draw[-] (1.5,0.8) -- (E)--(2.25,0.8); 
    \draw[-] (0.45,0.8) -- (E)--(1.05,0.8); 
 \draw[-] (E)--(0.75,-0.55); 
 \draw[-] (0.15,1.8) -- (F)--(-0.75,1.8); 
    \draw[-] (E)--(F) node  {};
\end{tikzpicture}}  };
\node (2) [rectangle,draw=black,fill=white,rounded corners=.35cm,inner sep=2mm] at (1,2.1) {\resizebox{0.7cm}{!}{\begin{tikzpicture}
    \node (F) [circle,draw=black,minimum size=4mm,inner sep=0.1mm] at (-0.3,1) {\small $1$};
    \node (A) [circle,draw=black,minimum size=4mm,inner sep=0.1mm] at (-0.3,2) {\small $2$};
    \node (y) [circle,draw=none,minimum size=4mm,inner sep=0.1mm] at (-0.16,1.55) {\small $y$};
 \draw[-] (-0.1,2.8) -- (A)--(-0.5,2.8); 
 \draw[-] (0.15,1.8) -- (F)--(-0.75,1.8); 
    \draw[-] (-0.3,0.45)--(F) node  {};
    \draw[-] (F)--(A) node {};
\end{tikzpicture}}};
\node (n1) [rectangle,draw=none,inner sep=0mm] at (-0.75,1.2) {\scriptsize $7$};
\node (n2) [rectangle,draw=none,inner sep=0mm] at (0.4,3.7) {\scriptsize $3$};
\node (n3) [rectangle,draw=none,inner sep=0mm] at (1.6,3.7) {\scriptsize $2$};
\node (n1') [rectangle,draw=none,inner sep=0mm] at (-0.75,1.45) {\scriptsize ${\bf 2}$};
\node (n2') [rectangle,draw=none,inner sep=0mm] at (0.4,3.95) {\scriptsize ${\bf 1}$};
\node (n3') [rectangle,draw=none,inner sep=0mm] at (1.6,3.95) {\scriptsize ${\bf 3}$};
\node (n3) [rectangle,draw=none,inner sep=0mm] at (0.65,0.8) {\scriptsize $4$};
\draw (0,-2)--(1);
\draw (1)--(-0.75,1.075);
\end{tikzpicture}\quad\quad \raisebox{8em}{and} \quad\quad \raisebox{8em}{$\alpha({\cal T}_2,\sigma_2,\!\!\!\raisebox{-1.1em}{\begin{tikzpicture}
\node at   (-0.95,0.45) {{\resizebox{0.525cm}{!}{\begin{tikzpicture}
\node (L) [circle,fill=ForestGreen,draw=black,minimum size=0.1cm,inner sep=0.2mm,label={[xshift=-0.22cm,yshift=-0.25cm]{\footnotesize $v$}}] at (0,0) {};
\node (B1) [circle,fill=ForestGreen,draw=black,minimum size=0.1cm,inner sep=0.2mm,label={[xshift=-0.22cm,yshift=-0.27cm]{\footnotesize $u$}}] at (0,0.38) {};
 \draw[thick]  (L)--(B1);\end{tikzpicture}}}}; 
\end{tikzpicture}})= $}\quad \begin{tikzpicture}
\draw (-0.4,3.55)--(-1,2)--(-1.6,3.55);
\node (1) [rectangle,draw=black,fill=white,rounded corners=.4cm,inner sep=2mm] at (0,-0.5) {\resizebox{1.75cm}{!}{\begin{tikzpicture}
    \node (E)[circle,draw=black,minimum size=4mm,inner sep=0.1mm] at (0,0) {\small $1$};
    \node (A) [circle,draw=black,minimum size=4mm,inner sep=0.1mm] at (1,1) {\small $2$};
    \node (y) [circle,draw=none,minimum size=4mm,inner sep=0.1mm] at (0.75,0.55) {\small $v$};
 \draw[-] (-0.2,0.8) -- (E)--(0.2,0.8);   
 \draw[-] (0.8,1.8) -- (A)--(1.2,1.8);   
  \draw[-] (-0.5,0.8)--(E)--(0,-0.55); 
 \draw[-] (0.5,0.8)--(E)--(-1,0.8); 
    \draw[-] (E)--(A) node  {};
   \end{tikzpicture}}};
\node (2) [rectangle,draw=black,fill=white,rounded corners=.35cm,inner sep=2mm] at (-1,2.1) {\resizebox{1.55cm}{!}{\begin{tikzpicture}
    \node (E)[circle,draw=black,minimum size=4mm,inner sep=0.1mm] at (0,0) {\small $1$};
    \node (F) [circle,draw=black,minimum size=4mm,inner sep=0.1mm] at (0,1) {\small $2$};
            \node (x)[circle,draw=none,minimum size=4mm,inner sep=0.1mm] at (0.13,0.55) {\small $u$};
 \draw[-] (-0.2,1.8) -- (F)--(0.2,1.8);    
  \draw[-] (-0.4,0.8)--(E)--(0,-0.55); 
 \draw[-] (0.4,0.8)--(E)--(-1,0.8); 
    \draw[-] (E)--(F);
     \draw[-] (E)--(1,0.8);
   \end{tikzpicture}}};
\node (n1) [rectangle,draw=none,inner sep=0mm] at (0.6,1.2) {\scriptsize $2$};
\node (n2) [rectangle,draw=none,inner sep=0mm] at (-0.4,3.7) {\scriptsize $2$};
\node (n3) [rectangle,draw=none,inner sep=0mm] at (-1.6,3.7) {\scriptsize $5$};
\node (n1') [rectangle,draw=none,inner sep=0mm] at (0.6,1.45) {\scriptsize ${\bf 2}$};
\node (n2') [rectangle,draw=none,inner sep=0mm] at (-0.4,3.95) {\scriptsize ${\bf 1}$};
\node (n3') [rectangle,draw=none,inner sep=0mm] at (-1.6,3.95) {\scriptsize ${\bf 3}$};
\node (n3) [rectangle,draw=none,inner sep=0mm] at (-0.65,0.8) {\scriptsize $6$};
\draw (0,-2)--(1)--(2);
\draw (1)--(0.6,1.075);
\end{tikzpicture}
\end{center}
and, therefore, $\varepsilon=3\cdot 4=12$. On the other hand, we have
\begin{center}
\begin{tikzpicture}
\node at   (-0.95,0.45) {{\resizebox{0.525cm}{!}{\begin{tikzpicture}
\node (L) [circle,fill=ForestGreen,draw=black,minimum size=0.1cm,inner sep=0.2mm,label={[xshift=-0.22cm,yshift=-0.25cm]{\footnotesize $x$}}] at (0,0) {};
\node (B1) [circle,fill=ForestGreen,draw=black,minimum size=0.1cm,inner sep=0.2mm,label={[xshift=-0.22cm,yshift=-0.27cm]{\footnotesize $y$}}] at (0,0.38) {};
 \draw[thick]  (L)--(B1);\end{tikzpicture}}}}; 
\end{tikzpicture} \enspace\raisebox{1.25em}{$\circ_2$} \raisebox{0.39em}{\begin{tikzpicture}
\node at   (-0.95,0.45) {{\resizebox{1.15cm}{!}{\begin{tikzpicture}
\node (L) [circle,fill=cyan,draw=black,minimum size=0.1cm,inner sep=0.2mm,label={[xshift=0cm,yshift=-0.1cm]{\footnotesize $\{u,v\}$}}] at (0,0) {};\end{tikzpicture}}}}; 
\end{tikzpicture}} \enspace\raisebox{1.25em}{$=$\enspace\enspace $-$\!\!\!\!} \begin{tikzpicture}
\node at   (-0.14,0.065)  { {\resizebox{2.05cm}{!}{{\begin{tikzpicture}
\node (b) [circle,fill=cyan,draw=black,minimum size=0.1cm,inner sep=0.2mm,label={[yshift=-0.45cm]{\footnotesize $x$}}] at (4,-0.2) {};
\node (c) [circle,fill=cyan,draw=black,minimum size=0.1cm,inner sep=0.2mm,label={[xshift=-0.45cm,yshift=-0.31cm]{\footnotesize $\{u,v\}$}}] at (3.8,0.15) {};
 \node (d) [circle,fill=cyan,draw=black,minimum size=0.1cm,inner sep=0.2mm,label={[xshift=0.2cm,yshift=-0.3cm]{\footnotesize $y$}}] at (4.2,0.15) {};
 \draw[thick]  (c)--(b)--(d);\end{tikzpicture}}}}};
\end{tikzpicture}
\end{center}
 since $\alpha({\cal T}_2,\sigma_2,\{u,v\})$ has one edge less than $\alpha({\cal T}_2,\sigma_2,\{u\})$, which gives $\varepsilon=3\cdot 3=9$.\demo
\end{example}
\smallskip
In the following lemma, we prove that Theorem \ref{free} extends to the graded context, in which signs  show up in the computation of composition in a free operad (see \cite[Section 5.8.7]{LV}).
\smallskip
\begin{lemma}
With the composition product adapted to the graded context, ${\EuScript O}_{\infty}$ is the free ${\mathbb N}$-coloured graded operad with respect to the set of generators given in Theorem \ref{free}.
\end{lemma}
\begin{proof} By Theorem \ref{free}, it remains to be shown that, for an operation $({\cal T},\sigma, C)\in {\EuScript O}_{\infty}(n_1,\dots,n_k;n)$ and two distinct vertices $v_1,v_2\in v({\cal T})$, such that the index of $v_1$ in the left-recursive decoration of ${\cal T}$ is less than the index of $v_2$, we have that
   $$(({\cal T},\sigma,C)\circ_{v_1} ({\cal T}_1,\sigma_1,C_1))\circ_{v_2} ({\cal T}_2,\sigma_2,C_2)=(-1)^{|C_1|\cdot|C_2|}(({\cal T},\sigma,C)\circ_{v_2} ({\cal T}_2,\sigma_2,C_2))\circ_{v_1} ({\cal T}_1,\sigma_1,C_1).$$
In the free operad description of ${\EuScript O}_{\infty}$, the two compositions of the above equality are represented by the planar tree 
\begin{center}
\resizebox{4.75cm}{!}{\begin{tikzpicture}
\draw (-1.5,0.7)--(0,0)--(1.5,0.7);
\draw (-1.3,0.7)--(0,0)--(1.3,0.7);
\draw (-1.1,0.7)--(0,0)--(1.1,0.7);
\node (1) [rectangle,draw=black,fill=white,rounded corners=.2cm,inner sep=1mm] at (0,0) {\scriptsize $\alpha({\cal T},\sigma,C)$};
\node (2) [rectangle,draw=black,rounded corners=.2cm,inner sep=1mm] at (-1.2,1.4) {\scriptsize $\alpha({\cal T}_1,\sigma_1,C_1)$};
\node (3) [rectangle,draw=black,rounded corners=.2cm,inner sep=1mm] at (1.2,1.4) {\scriptsize $\alpha({\cal T}_2,\sigma_2,C_2)$};
\draw (0,-0.8)--(1);
\draw (1)--(2) node[right,yshift=-0.4cm,xshift=0.3cm] {\scriptsize $v_1$};
\draw (1)--(3) node[left,yshift=-0.4cm,xshift=-0.3cm] {\scriptsize $v_2$};
\draw (-0.5,2.1)--(2)--(-1.9,2.1);
\draw (-0.7,2.1)--(2)--(-1.7,2.1);
\draw (-0.9,2.1)--(2)--(-1.5,2.1);
\draw (-1.1,2.1)--(2)--(-1.3,2.1);
\draw (0.5,2.1)--(3)--(1.9,2.1);
\draw (0.7,2.1)--(3)--(1.7,2.1);
\draw (0.9,2.1)--(3)--(1.5,2.1);
\draw (1.1,2.1)--(3)--(1.3,2.1);
\draw (0.3,0.7)--(1)--(-0.3,0.7);
\draw (0.1,0.7)--(1)--(-0.1,0.7);
\draw [decorate,decoration={brace,amplitude=5pt,mirror,raise=-1pt},xshift=-2pt,yshift=0pt]
(-0.35,2.2) -- (-1.9,2.2)node [black,midway,above,yshift=2pt] {\scriptsize $l_1+e_1$};
\draw [decorate,decoration={brace,amplitude=5pt,mirror,raise=-1pt},xshift=2pt,yshift=0pt]
(1.9,2.2) -- (0.35,2.2)node [black,midway,above,yshift=2pt] {\scriptsize $l_2+e_2$};
\draw [decorate,decoration={brace,amplitude=3pt ,raise=-1pt},xshift=-2pt,yshift=0pt]
(-1.5,0.775) -- (-0.9,0.775)node [black,midway,above,yshift=0pt] {\scriptsize $k_0$};
\draw [decorate,decoration={brace,amplitude=3pt ,raise=-1pt},xshift=-2pt,yshift=0pt]
(1.05,0.775) -- (1.65,0.775)node [black,midway,above,yshift=0pt] {\scriptsize $k_2$};
\draw [decorate,decoration={brace,amplitude=3pt,mirror,raise=-1pt},xshift=0pt,yshift=0pt]
(0.35,0.775) -- (-0.35,0.775)node [black,midway,above,yshift=0pt] {\scriptsize $k_1$};
\end{tikzpicture}}
\end{center}
where, for $i=1,2$, $\alpha({\cal T}_i,\sigma_i,C_i)$ has $l_i$ inputs and $e_i$ edges, and where $k_0$ (resp. $k_1$, $k_2$) is the number of leaves and edges of $\alpha({\cal T},\sigma,C)$ on the left from $v_1$ (resp. between $v_1$ and $v_2$, on the right from $v_2$). The sign of the composition on the left-hand side is then determined by $\varepsilon_1=(k_1+k_2+1)(l_1+e_1)+k_2(l_2+e_2)$, while, for the right-hand side, we have $\varepsilon_2=k_2(l_2+e_2)+(k_1+k_2+l_2+e_2+1)(l_1+e_1)$. Additionally, for $i=1,2$, we have that $|C_i|=l_i-e_i-2$. A straightforward calculation shows that $\varepsilon_1=_{\mbox{mod}2} \varepsilon_2+|C_1|\cdot |C_2|$, which proves the claim.
\end{proof}
\smallskip
In order to equip the graded operad ${\EuScript O}_{\infty}$ with a differential, we now formalize the action of splitting the vertices of constructs of  edge-graphs of operadic trees. Let $({\cal T},\sigma,C)\in {\EuScript O}_{\infty}(n_1,\dots,n_k;n)$ and let $V\in v(C)$ be such that $|V|\geq 2$. Let ${\cal T}_{V}$ be the left-recursive operadic tree obtained from ${\cal T}$ by contracting all the edges of ${\cal T}$ except the ones contained in $V$.  Let $X$ and $Y$ be non-empty disjoint sets  such that $X\cup Y=V$ and such that $X\{Y\}:{\bf H}_{{\cal T}_V}$. We define  the construct  $C[X\{Y\}/V]$  of ${\bf H}_{\cal T}$ by induction on the number of vertices of $C$, as follows.\\[-0.3cm]
\begin{itemize}
\item If $C=e({\cal T})$, we set $$C[X\{Y\}/V]:=X\{Y\}.$$ 
\item  Suppose that $C=Z\{C_1,\dots,C_p\}$, where   ${\bf H}_{\cal T}\backslash Z\leadsto {\bf H}_{{\cal T}_1},\dots, {\bf H}_{{\cal T}_p}$ and    $C_i:{\bf H}_{{\cal T}_i}$.\\[-0.3cm]
\begin{itemize}
\item If there exists an index $1\leq i\leq p$, such that $V\in v(C_i)$, we define $$C[X\{Y\}/V]:=Z\{C_1,\dots,C_{i-1},C_i[X\{Y\}/V] ,C_{i+1},\dots,C_p\}.$$ 
\item Suppose that $V=Z$ and let $\{i_1,\dots,i_q\}\cup \{j_1,\dots,j_r\}$ be the partition of the set $\{1,\dots,p\}$ such that the trees ${\cal T}_{i_s}$, for $1\leq s\leq q$,   contain an edge sharing a vertex with some edge of $Y$, while the trees ${\cal T}_{i_t}$, for $1\leq t\leq q$, have no edges sharing a vertex with the edges of $Y$.  We define
$$C[X\{Y\}/V]:=X\{Y\{C_{i1},\dots C_{iq}\},C_{j1},\dots,C_{jr}\}.$$ If, exceptionally, $\{i_1,\dots,i_q\}=\emptyset$ (resp. $\{j_1,\dots,j_r\}=\emptyset$), we have that $C[X\{Y\}/V]:=X\{Y,C_1,\dots,C_p\}$ (resp.   $C[X\{Y\}/V]:=X\{Y\{C_1,\dots,C_p\}\}$).\\[-0.3cm]
\end{itemize}
\end{itemize}
 
 \smallskip
\noindent Therefore, intuitively, $C[X\{Y\}/V]$ is obtained from $C$ by splitting the vertex $V$ into the edge $X\{Y\}$. 
 \smallskip
\begin{lemma}\label{splitwelldefined} The non-planar rooted tree  $C[X\{Y\}/V]$ is indeed a construct of ${\bf H}_{\cal T}$.
\end{lemma}
\begin{proof} By induction on the number of vertices of $C$.   
 If $C$ has a single vertex  $V=e({\cal T})$, then  ${\bf H}_{{\cal T}_V}={\bf H}_{\cal T}$, and, therefore, for any decomposition $X\cup Y=e({\cal T})$, such that $X\{Y\}:{\bf H}_{{\cal T}_{V}}$, we trivially also  have that   $X\{Y\}:{\bf H}_{\cal T}$. 
 
 \smallskip

 Suppose that $C=Z(C_1,\dots,C_p)$, where  ${\bf H}_{\cal T}\backslash Z\leadsto {\bf H}_{{\cal T}_1},\dots, {\bf H}_{{\cal T}_p}$ and    $C_i:{\bf H}_{{\cal T}_i}$.\\[-0.3cm]
\begin{itemize}
\item If there exists $1\leq i\leq p$, such that $V\in v(C_i)$, we conclude by the induction hypothesis for $C_i$.\\[-0.3cm]
\item If $V=Z$ and if $\{i_1,\dots,i_q\}\cup \{j_1,\dots,j_r\}$ is the partition as above, then, since $X\{Y\}:{\cal T}_V$, it must be the case that the set of edges $Y\cup \bigcup_{i\in \{i_1,\dots,i_q\}}e({\cal T}_i)$ determines a single subtree ${\cal S}$ of ${\cal T}$. Therefore, $${\bf H}_{\cal T}\backslash X\leadsto \{{\bf H}_{{\cal S}}\}\cup \{{\bf H}_{{\cal T}_j}\,|\, j\in   \{j_1,\dots,j_r\}\}.$$
Then, if $\{i_1,\dots,i_q\}\neq \emptyset$, the conclusion follows since $Y\{C_{i1},\dots C_{iq}\}:{\bf H}_{{\cal S}}$. 
If $\{i_1,\dots,i_q\}=\emptyset$,  the conclusion follows since  $Y$ is trivially the maximal construct of ${\bf H}_{{\cal S}}$.
\end{itemize}\vspace{-0.45cm}\end{proof}

\smallskip

 The differential $d_{{\EuScript O}_{\infty}}$ of ${\EuScript O}_{\infty}$ is defined in terms of splitting the vertices of constructs, as follows: for $({\cal T},\sigma,C)\in {\EuScript O}_{\infty}(n_1,\dots,n_k;n)$, we set $$d_{{\EuScript O}_{\infty}}({\cal T},\sigma,C):=\displaystyle\sum_{\substack{V\in v(C) \\ |V|\geq 2}}\sum_{\substack{(X,Y) \\ X\cup Y=V \\  X\{Y\}:{\bf H}_{{\cal T}_V}}} (-1)^{\delta} ({\cal T},\sigma,C[X\{Y\}/V]),$$
where $\delta$ is the number of edges and leaves in $\alpha({\cal T},\sigma,C[X\{Y\}/V])$ which are on the left from or below the edge determined by $X\{Y\}$: $$\delta=E_{\leq X\{Y\}}(\alpha({\cal T},\sigma,C[X\{Y\}/V])).$$

Observe that, for $({\cal T},C)\in {\EuScript O}_{\infty}(n_1,\dots,n_k;n)$, we have $|({\cal T},C)|=k-1-v(C)$ and
$|d_{{\EuScript O}_{\infty}}({\cal T},C)|=k-1-(v(C)+1)$, which shows that $d_{{\EuScript O}_{\infty}}$ is indeed a map od degree $-1$. In particular, for generators $({\cal T},e({\cal T}))$  we have:
$$
d_{{\EuScript O}_{\infty}}({\cal T},e({\cal T}))=\displaystyle\sum_{\substack{(X,Y) \\ X\cup Y=e({\cal T})\\ X\{Y\}:{\bf H}_{\cal T}}} (-1)^{\delta} ({\cal T},X\{Y\})=\displaystyle\sum_{\substack{(X,Y) \\ X\cup Y=e({\cal T})\\ X\{Y\}:{\bf H}_{\cal T}}} (-1)^{\delta} ({\cal T}_X,X)\circ_{\rho({\cal T}_Y)} ({\cal T}_Y,Y),$$
 which shows that $d_{{\EuScript O}_{\infty}}$ is decomposable.

\begin{convention}
The differential   $d_{{\EuScript O}_{\infty}}$ of ${\EuScript O}_{\infty}$ acts on an operation $({\cal T},\sigma,C)$ by splitting the vertices of $C$, whereas the underlying operadic tree $({\cal T},\sigma)$ remains unchanged. When expressing this action, assuming that  $({\cal T},\sigma)$ is clear from the context, we shall often  write simply $d_{{\EuScript O}_{\infty}}(C)$.
\end{convention}
 
\begin{example}
We calculate the differential of $({\cal T},\sigma,\raisebox{-1.15em}{\resizebox{1.2cm}{!}{\begin{tikzpicture}
\node (c) [circle,fill=WildStrawberry,draw=black,minimum size=0.1cm,inner sep=0.2mm,label={[xshift=-0.45cm,yshift=-0.32cm]{\footnotesize $\{x,y\}$}}] at (0,0) {};
 \node (d) [circle,fill=WildStrawberry,draw=black,minimum size=0.1cm,inner sep=0.2mm,label={[xshift=-0.45cm,yshift=-0.32cm]{\footnotesize $\{u,v\}$}}] at (0,0.35) {};
 \draw[thick]  (c)--(d);\end{tikzpicture}}}\,)$, where $({\cal T},\sigma)$ is our favourite operadic tree (see Example \ref{edgegraph} and Example \ref{compositionconstructs}). By splitting the vertices $\{x,y\}$ and $\{u,v\}$, we get the following four planar trees in the free operad representation of ${\EuScript O}_{\infty}$:
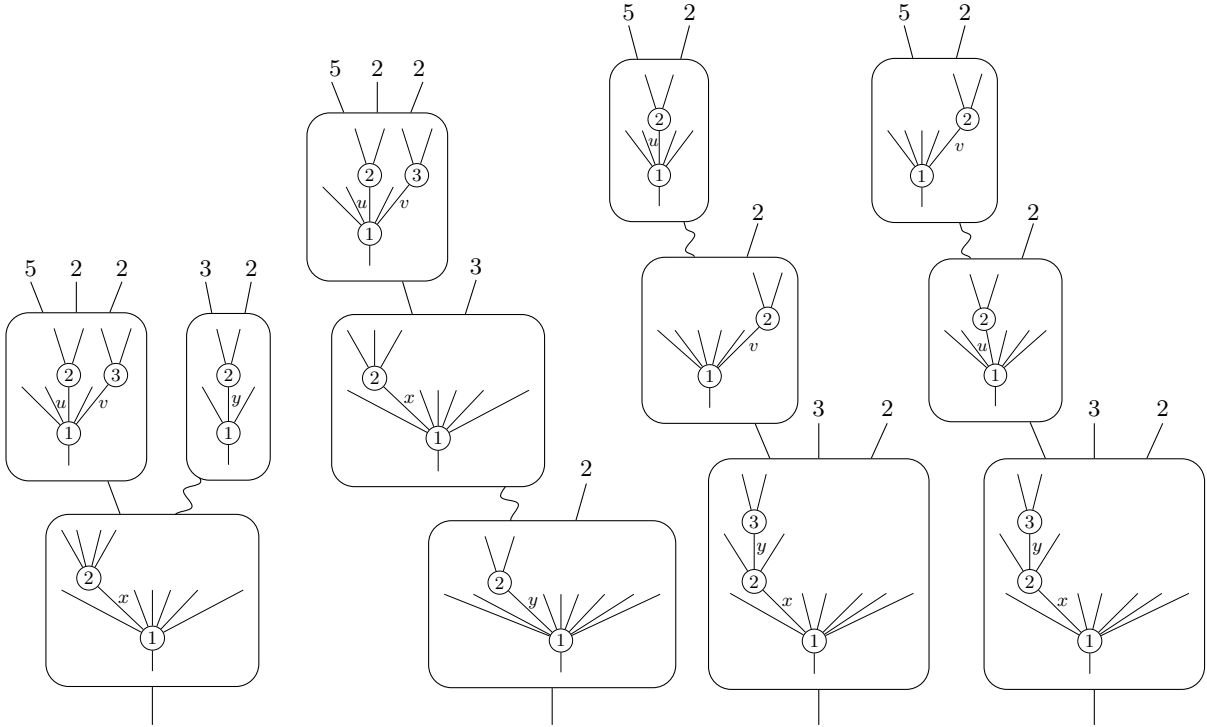
\begin{center}
\begin{tikzpicture}
\draw[snake it](0,-1.5)--(1,1.4);
\node (1) [rectangle,draw=black,fill=white,rounded corners=.3cm,inner sep=2mm] at (0,-1.5) {\resizebox{2.4cm}{!}{\begin{tikzpicture}
    \node (E)[circle,draw=black,minimum size=4mm,inner sep=0.1mm] at (0.75,0) {\small $1$};
    \node (F) [circle,draw=black,minimum size=4mm,inner sep=0.1mm] at (-0.3,1) {\small $2$};
        \node (x)[circle,draw=none,minimum size=4mm,inner sep=0.1mm] at (0.275,0.65) {\small $x$};
 \draw[-] (-0.1,1.8) -- (F)--(-0.5,1.8); 
 \draw[-] (0.75,0.8) -- (E)--(-0.75,0.8); 
  \draw[-] (1.5,0.8) -- (E)--(2.25,0.8); 
    \draw[-] (0.45,0.8) -- (E)--(1.05,0.8); 
 \draw[-] (E)--(0.75,-0.55); 
 \draw[-] (0.15,1.8) -- (F)--(-0.75,1.8); 
    \draw[-] (E)--(F) node  {};
\end{tikzpicture}}  };
\node (2) [rectangle,draw=black,fill=white,rounded corners=.3cm,inner sep=2mm] at (1,1.2) {\resizebox{0.7cm}{!}{\begin{tikzpicture}
    \node (F) [circle,draw=black,minimum size=4mm,inner sep=0.1mm] at (-0.3,1) {\small $1$};
    \node (A) [circle,draw=black,minimum size=4mm,inner sep=0.1mm] at (-0.3,2) {\small $2$};
    \node (y) [circle,draw=none,minimum size=4mm,inner sep=0.1mm] at (-0.16,1.55) {\small $y$};
 \draw[-] (-0.1,2.8) -- (A)--(-0.5,2.8); 
 \draw[-] (0.15,1.8) -- (F)--(-0.75,1.8); 
    \draw[-] (-0.3,0.45)--(F) node  {};
    \draw[-] (F)--(A) node {};
\end{tikzpicture}}};
\node (3) [rectangle,draw=black,rounded corners=.3cm,inner sep=2mm] at (-1,1.2) {\resizebox{1.45cm}{!}{\begin{tikzpicture}
    \node (E)[circle,draw=black,minimum size=4mm,inner sep=0.1mm] at (-0,0) {\small $1$};
    \node (F) [circle,draw=black,minimum size=4mm,inner sep=0.1mm] at (0,1) {\small $2$};
\node (dd) [circle,draw=black,minimum size=4mm,inner sep=0.1mm] at (0.8,1) {\small $3$};
 \draw[-]   (E)--(-0.4,0.8);  
  \draw[-] (-0.8,0.8) -- (E);  
    \draw[-] (E)--(0.4,0.8); 
 \draw[-] (E)--(0,-0.55); 
 \draw[-] (-0.25,1.8) -- (F)--(0.25,1.8);   
 \draw[-] (0.55,1.8) -- (dd)--(1.06,1.8);   
    \draw[-] (E)--(F) node [midway,left,yshift=0cm,xshift=0.17cm] {\small $u$};
 \draw[-] (E)--(dd) node [midway,right,yshift=0cm,xshift=-0.1cm] {\small $v$};
\end{tikzpicture}}};
\node (n2) [rectangle,draw=none,inner sep=0mm] at (0.7,2.9) {\small  $3$};
\node (n3) [rectangle,draw=none,inner sep=0mm] at (1.3,2.9) {\small  $2$};
\node (m1) [rectangle,draw=none,inner sep=0mm] at (-1.6,2.9) {\small $5$};
\node (m2) [rectangle,draw=none,inner sep=0mm] at (-1,2.9) {\small $2$};
\node (m3) [rectangle,draw=none,inner sep=0mm] at (-0.4,2.9) {\small $2$};
\draw (0,-3.15)--(1);
\draw (1)--(3);
\draw (3)--(-1,2.725);
\draw (-0.4,2.725)--(3)--(-1.6,2.725);
\draw (0.7,2.725)--(2)--(1.3,2.725);
\end{tikzpicture}
\quad \begin{tikzpicture}
\draw[snake it](0.2,-1.5)--(1.3,-4.2) ;
\node (1) [rectangle,draw=black,fill=white,rounded corners=.3cm,inner sep=2mm] at (-0.2,-1.5) {\resizebox{2.4cm}{!}{\begin{tikzpicture}
    \node (E)[circle,draw=black,minimum size=4mm,inner sep=0.1mm] at (0.75,0) {\small $1$};
    \node (F) [circle,draw=black,minimum size=4mm,inner sep=0.1mm] at (-0.3,1) {\small $2$};
        \node (x)[circle,draw=none,minimum size=4mm,inner sep=0.1mm] at (0.275,0.65) {\small $x$};
 \draw[-] (F)--(-0.3,1.8); 
 \draw[-] (0.75,0.8) -- (E)--(-0.75,0.8); 
  \draw[-] (1.5,0.8) -- (E)--(2.25,0.8); 
    \draw[-] (0.45,0.8) -- (E)--(1.05,0.8); 
 \draw[-] (E)--(0.75,-0.55); 
 \draw[-] (0.15,1.8) -- (F)--(-0.75,1.8); 
    \draw[-] (E)--(F) node  {};
\end{tikzpicture}}  };
\node (2) [rectangle,draw=black,fill=white,rounded corners=.3cm,inner sep=2mm] at (1.3,-4.2) {\resizebox{2.85cm}{!}{\begin{tikzpicture}
    \node (E)[circle,draw=black,minimum size=4mm,inner sep=0.1mm] at (0.75,0) {\small $1$};
    \node (F) [circle,draw=black,minimum size=4mm,inner sep=0.1mm] at (-0.3,1) {\small $2$};
        \node (x)[circle,draw=none,minimum size=4mm,inner sep=0.1mm] at (0.275,0.65) {\small $y$};
 \draw[-] (0.75,0.8) -- (E)--(-0.75,0.8); 
  \draw[-] (1.5,0.8) -- (E)--(2.45,0.8); 
    \draw[-] (0.45,0.8) -- (E)--(1.05,0.8); 
       \draw[-] (-1.25,0.8) -- (E)--(2,0.8); 
 \draw[-] (E)--(0.75,-0.55); 
 \draw[-] (-0.05,1.8) -- (F)--(-0.55,1.8); 
    \draw[-] (E)--(F) node  {};
\end{tikzpicture}}  };
\node (3) [rectangle,draw=black,rounded corners=.3cm,inner sep=2mm] at (-1,1.2) {\resizebox{1.45cm}{!}{\begin{tikzpicture}
    \node (E)[circle,draw=black,minimum size=4mm,inner sep=0.1mm] at (-0,0) {\small $1$};
    \node (F) [circle,draw=black,minimum size=4mm,inner sep=0.1mm] at (0,1) {\small $2$};
\node (dd) [circle,draw=black,minimum size=4mm,inner sep=0.1mm] at (0.8,1) {\small $3$};
 \draw[-]   (E)--(-0.4,0.8);  
  \draw[-] (-0.8,0.8) -- (E);  
    \draw[-] (E)--(0.4,0.8); 
 \draw[-] (E)--(0,-0.55); 
 \draw[-] (-0.25,1.8) -- (F)--(0.25,1.8);   
 \draw[-] (0.55,1.8) -- (dd)--(1.06,1.8);   
    \draw[-] (E)--(F) node [midway,left,yshift=0cm,xshift=0.17cm] {\small $u$};
 \draw[-] (E)--(dd) node [midway,right,yshift=0cm,xshift=-0.1cm] {\small $v$};
\end{tikzpicture}}};
\node (n3) [rectangle,draw=none,inner sep=0mm] at (1.75,-2.4) {\small  $2$};;
\node (m1) [rectangle,draw=none,inner sep=0mm] at (-1.55,2.9) {\small  $5$};
\node (m2) [rectangle,draw=none,inner sep=0mm] at (-1,2.9) {\small  $2$};
\node (m3) [rectangle,draw=none,inner sep=0mm] at (-0.45,2.9) {\small  $2$};
\node (n3) [rectangle,draw=none,inner sep=0mm] at (0.3,0.275) {\small $3$};
\draw (0.3,0.1)--(1);
\draw (2)--(1.3,-5.8);
\draw (2)--(1.75,-2.6);
\draw (1)--(3);
\draw (3)--(-1,2.725);
\draw (-0.4,2.725)--(3)--(-1.6,2.725);
\end{tikzpicture}\hspace{-1cm} \begin{tikzpicture}
\node (1) [,rectangle,draw=black,rounded corners=.3cm,inner sep=2mm] at (0.2,-1.45) {\resizebox{1.65cm}{!}{\begin{tikzpicture}
    \node (E)[circle,draw=black,minimum size=4mm,inner sep=0.1mm] at (-0,0) {\small $1$};
    \node (F) [circle,draw=black,minimum size=4mm,inner sep=0.1mm] at (1,1) {\small $2$};
 \draw[-]   (E)--(-0.6,0.8);  
  \draw[-] (-0.9,0.8) -- (E)--(-0.2,0.8);  
    \draw[-] (E)--(0.6,0.8); 
 \draw[-] (0.2,0.8)-- (E)--(0,-0.55); 
 \draw[-] (1.25,1.8) -- (F)--(0.75,1.8);    
    \draw[-] (E)--(F) node [midway,left,yshift=0cm,xshift=0.55cm] {\small $v$};
\end{tikzpicture}}};
\node (2) [rectangle,draw=black,rounded corners=.3cm,inner sep=2mm] at (1.5,-4.55) {\resizebox{2.5cm}{!}{\begin{tikzpicture}
    \node (E)[circle,draw=black,minimum size=4mm,inner sep=0.1mm] at (-0,0) {\small $1$};
    \node (F) [circle,draw=black,minimum size=4mm,inner sep=0.1mm] at (-1,1) {\small $2$};
    \node (A) [circle,draw=black,minimum size=4mm,inner sep=0.1mm] at (-1,2) {\small $3$};
            \node (x)[circle,draw=none,minimum size=4mm,inner sep=0.1mm] at (-0.45,0.65) {\small $x$};
    \node (y) [circle,draw=none,minimum size=4mm,inner sep=0.1mm] at (-0.875,1.55) {\small $y$};
 \draw[-] (0.8,0.8) -- (E)--(-1.4,0.8); 
\draw[-] (-1.2,2.8) -- (A)--(-0.8,2.8); 
 \draw[-] (-0.2,0.8) -- (E)--(0.2,0.8);   
 \draw[-] (1.65,0.8) -- (E)--(1.2,0.8); 
 \draw[-] (E)--(0,-0.55); 
 \draw[-] (-1.5,1.8) -- (F)--(-0.5,1.8);   
    \draw[-] (E)--(F) node {};
    \draw[-] (F)--(A) node {};
   \end{tikzpicture}}};
\node (3) [rectangle,draw=black,rounded corners=.3cm,inner sep=2mm] at (-0.6,1.2) {\resizebox{0.9cm}{!}{\begin{tikzpicture}
    \node (E)[circle,draw=black,minimum size=4mm,inner sep=0.1mm] at (-0,0) {\small $1$};
    \node (F) [circle,draw=black,minimum size=4mm,inner sep=0.1mm] at (0,1) {\small $2$};
  \draw[-] (-0.6,0.8) -- (E)--(0.6,0.8);  
    \draw[-] (E)--(-0.3,0.8); 
 \draw[-] (0.3,0.8)-- (E)--(0,-0.55); 
 \draw[-] (-0.25,1.8) -- (F)--(0.25,1.8);    
    \draw[-] (E)--(F) node [midway,left,yshift=0.1cm,xshift=0.2cm] {\small $u$};
\end{tikzpicture}}};
\node (n3) [rectangle,draw=none,inner sep=0mm] at (1.5,-2.35) {\small $3$};
\node (m1) [rectangle,draw=none,inner sep=0mm] at (-1,2.9) {\small $5$};
\node (m2) [rectangle,draw=none,inner sep=0mm] at (-0.2,2.9) {\small $2$};
\node (m3) [rectangle,draw=none,inner sep=0mm] at (0.7,0.25) {\small $2$};
\node (n3) [rectangle,draw=none,inner sep=0mm] at (2.4,-2.35) {\small $2$};
\draw (1)--(2)--(1.5,-6.55);
\draw (2.4,-2.55)--(2)--(1.5,-2.55);
\draw[snake it]  (1)--(3);
\draw (1)--(0.7,0.1);
\draw (-1,2.725)--(3)--(-0.2,2.725);
\end{tikzpicture} \hspace{-1cm}  \begin{tikzpicture}
\node (1) [,rectangle,draw=black,rounded corners=.3cm,inner sep=2mm] at (0.2,-1.45) {\resizebox{1.35cm}{!}{\begin{tikzpicture}
    \node (E)[circle,draw=black,minimum size=4mm,inner sep=0.1mm] at (-0,0) {\small $1$};
    \node (F) [circle,draw=black,minimum size=4mm,inner sep=0.1mm] at (-0.2,1) {\small $2$};
 \draw[-]   (E)--(-0.6,0.8);  
  \draw[-] (-0.9,0.8) -- (E)--(0.9,0.8);  
    \draw[-] (E)--(0.6,0.8); 
 \draw[-] (0.2,0.8)-- (E)--(0,-0.55); 
 \draw[-] (-0.45,1.8) -- (F)--(0.05,1.8);    
    \draw[-] (E)--(F) node [midway,left,yshift=0cm,xshift=0.17cm] {\small $u$};
\end{tikzpicture}}};
\node (2) [rectangle,draw=black,rounded corners=.3cm,inner sep=2mm] at (1.5,-4.55) {\resizebox{2.5cm}{!}{\begin{tikzpicture}
    \node (E)[circle,draw=black,minimum size=4mm,inner sep=0.1mm] at (-0,0) {\small $1$};
    \node (F) [circle,draw=black,minimum size=4mm,inner sep=0.1mm] at (-1,1) {\small $2$};
    \node (A) [circle,draw=black,minimum size=4mm,inner sep=0.1mm] at (-1,2) {\small $3$};
            \node (x)[circle,draw=none,minimum size=4mm,inner sep=0.1mm] at (-0.45,0.65) {\small $x$};
    \node (y) [circle,draw=none,minimum size=4mm,inner sep=0.1mm] at (-0.875,1.55) {\small $y$};
 \draw[-] (0.8,0.8) -- (E)--(-1.4,0.8); 
\draw[-] (-1.2,2.8) -- (A)--(-0.8,2.8); 
 \draw[-] (-0.2,0.8) -- (E)--(0.2,0.8);   
 \draw[-] (1.65,0.8) -- (E)--(1.2,0.8); 
 \draw[-] (E)--(0,-0.55); 
 \draw[-] (-1.5,1.8) -- (F)--(-0.5,1.8);   
    \draw[-] (E)--(F) node {};
    \draw[-] (F)--(A) node {};
   \end{tikzpicture}}};
\node (3) [rectangle,draw=black,rounded corners=.3cm,inner sep=2mm] at (-0.6,1.2) {\resizebox{1.25cm}{!}{\begin{tikzpicture}
    \node (E)[circle,draw=black,minimum size=4mm,inner sep=0.1mm] at (-0,0) {\small $1$};
    \node (F) [circle,draw=black,minimum size=4mm,inner sep=0.1mm] at (0.8,1) {\small $2$};
 \draw[-]   (E)--(-0.6,0.8);  
  \draw[-] (-0.3,0.8) -- (E)--(0.3,0.8);  
 \draw[-] (0,0.8)-- (E)--(0,-0.55); 
 \draw[-] (1.05,1.8) -- (F)--(0.55,1.8);    
    \draw[-] (E)--(F) node [midway,left,yshift=0cm,xshift=0.55cm] {\small $v$};
\end{tikzpicture}}};
\node (n3) [rectangle,draw=none,inner sep=0mm] at (1.5,-2.35) {\small $3$};
\node (m1) [rectangle,draw=none,inner sep=0mm] at (-1,2.9) {\small $5$};
\node (m2) [rectangle,draw=none,inner sep=0mm] at (-0.2,2.9) {\small $2$};
\node (m3) [rectangle,draw=none,inner sep=0mm] at (0.7,0.25) {\small $2$};
\node (n3) [rectangle,draw=none,inner sep=0mm] at (2.4,-2.35) {\small $2$};
\draw (1)--(2)--(1.5,-6.55);
\draw (2.4,-2.55)--(2)--(1.5,-2.55);
\draw[snake it]  (1)--(3);
\draw (1)--(0.7,0.1);
\draw (-1,2.725)--(3)--(-0.2,2.725);
\end{tikzpicture} 
\end{center}
in which the curly lines represent the edges that  arise from the splitting.
By counting the edges and the leaves on the left from and below those edges, we get $4$, $0$, $1$ and $1$, respectively.  Therefore,
$$d_{{\EuScript O}_{\infty}}(\raisebox{-1.15em}{\resizebox{1.2cm}{!}{\begin{tikzpicture}
\node (c) [circle,fill=WildStrawberry,draw=black,minimum size=0.1cm,inner sep=0.2mm,label={[xshift=-0.45cm,yshift=-0.32cm]{\footnotesize $\{x,y\}$}}] at (0,0) {};
 \node (d) [circle,fill=WildStrawberry,draw=black,minimum size=0.1cm,inner sep=0.2mm,label={[xshift=-0.45cm,yshift=-0.32cm]{\footnotesize $\{u,v\}$}}] at (0,0.35) {};
 \draw[thick]  (c)--(d);\end{tikzpicture}}}\,)=\raisebox{-1.4em}{\resizebox{2.05cm}{!}{\begin{tikzpicture}
\node (b) [circle,fill=cyan,draw=black,minimum size=0.1cm,inner sep=0.2mm,label={[yshift=-0.45cm]{\footnotesize $x$}}] at (4,-0.2) {};
\node (c) [circle,fill=cyan,draw=black,minimum size=0.1cm,inner sep=0.2mm,label={[xshift=-0.45cm,yshift=-0.31cm]{\footnotesize $\{u,v\}$}}] at (3.8,0.15) {};
 \node (d) [circle,fill=cyan,draw=black,minimum size=0.1cm,inner sep=0.2mm,label={[xshift=0.2cm,yshift=-0.3cm]{\footnotesize $y$}}] at (4.2,0.15) {};
 \draw[thick]  (c)--(b)--(d);\end{tikzpicture}}}+   \raisebox{-1.5em}{\resizebox{1.2cm}{!}{\begin{tikzpicture}
\node (b) [circle,fill=cyan,draw=black,minimum size=0.1cm,inner sep=0.2mm,label={[xshift=-0.22cm,yshift=-0.25cm]{\footnotesize $y$}}] at (0,0) {};
\node (c) [circle,fill=cyan,draw=black,minimum size=0.1cm,inner sep=0.2mm,label={[xshift=-0.22cm,yshift=-0.25cm]{\footnotesize $x$}}] at (0,0.35) {};
 \node (d) [circle,fill=cyan,draw=black,minimum size=0.1cm,inner sep=0.2mm,label={[xshift=-0.45cm,yshift=-0.32cm]{\footnotesize $\{u,v\}$}}] at (0,0.7) {};
 \draw[thick]  (b)--(c)--(d);\end{tikzpicture}}}\,-\raisebox{-1.7em}{\resizebox{1.2cm}{!}{\begin{tikzpicture}
\node (b) [circle,fill=cyan,draw=black,minimum size=0.1cm,inner sep=0.2mm,label={[xshift=-0.22cm,yshift=-0.25cm]{\footnotesize $u$}}] at (0,0.7) {};
\node (c) [circle,fill=cyan,draw=black,minimum size=0.1cm,inner sep=0.2mm,label={[xshift=-0.22cm,yshift=-0.26cm]{\footnotesize $v$}}] at (0,0.35) {};
 \node (d) [circle,fill=cyan,draw=black,minimum size=0.1cm,inner sep=0.2mm,label={[xshift=-0.45cm,yshift=-0.32cm]{\footnotesize $\{x,y\}$}}] at (0,0) {};
 \draw[thick]  (b)--(c)--(d);\end{tikzpicture}}} \,-\raisebox{-1.7em}{\resizebox{1.2cm}{!}{\begin{tikzpicture}
\node (b) [circle,fill=cyan,draw=black,minimum size=0.1cm,inner sep=0.2mm,label={[xshift=-0.22cm,yshift=-0.25cm]{\footnotesize $v$}}] at (0,0.7) {};
\node (c) [circle,fill=cyan,draw=black,minimum size=0.1cm,inner sep=0.2mm,label={[xshift=-0.22cm,yshift=-0.26cm]{\footnotesize $u$}}] at (0,0.35) {};
 \node (d) [circle,fill=cyan,draw=black,minimum size=0.1cm,inner sep=0.2mm,label={[xshift=-0.45cm,yshift=-0.32cm]{\footnotesize $\{x,y\}$}}] at (0,0) {};
 \draw[thick]  (b)--(c)--(d);\end{tikzpicture}}}\,\, .$$ 
Observe that the geometric interpretation of this differential is the boundary of the square of the 3-dimensional hemiassociahedron encoded by \raisebox{-1.15em}{\resizebox{1.2cm}{!}{\begin{tikzpicture}
\node (c) [circle,fill=WildStrawberry,draw=black,minimum size=0.1cm,inner sep=0.2mm,label={[xshift=-0.45cm,yshift=-0.32cm]{\footnotesize $\{x,y\}$}}] at (0,0) {};
 \node (d) [circle,fill=WildStrawberry,draw=black,minimum size=0.1cm,inner sep=0.2mm,label={[xshift=-0.45cm,yshift=-0.32cm]{\footnotesize $\{u,v\}$}}] at (0,0.35) {};
 \draw[thick]  (c)--(d);\end{tikzpicture}}}    (see Example \ref{compositionconstructs}).
\demo
\end{example}
In the following technical lemma, we prove that $({\EuScript O}_{\infty},d_{{\EuScript O}_{\infty}})$ is indeed a dg operad. For the sake of readability, we shall   write $({\cal T},C)$ for what is actually $({\cal T},\sigma,C)$. 
\begin{lemma}\label{dok} The map $d$ has the following properties:
\begin{itemize}
\item[1.]  $(d_{{\EuScript O}_{\infty}})^2=0$, and
\item[2.] the composition structure of ${\EuScript O}_{\infty}$ is compatible with $d_{{\EuScript O}_{\infty}}$, i.e. $$\enspace d_{{\EuScript O}_{\infty}}(({\cal T}_1,C_1)\circ_i ({\cal T}_2,C_2))=d_{{\EuScript O}_{\infty}}({\cal T}_1,C_1)\circ_i ({\cal T}_2,C_2)+(-1)^{|C_1|}({\cal T}_1,C_1)\circ_i d_{{\EuScript O}_{\infty}}({\cal T}_2,C_2).$$
\end{itemize}
\end{lemma}
  \begin{proof}\enspace 1. The proof that $d_{{\EuScript O}_{\infty}}$ squares to zero goes by  case analysis relative to the configuration of vertices of $C$ that got split in constructing the  two occurences  of the same summand $({\cal T},C')$ in $d_{{\EuScript O}_{\infty}}^2({\cal T},C)$, by showing that  the corresponding summands have the opposite sign. This is an easy analysis of the relative edge positions: in all the cases, there will exist exactly one edge that is counted in calculating the sign of one of the two instances of $({\cal T},C')$, but not in calculating the sign of the other one.

\smallskip For example, suppose that $C'$ is obtained by splitting two different vertices of $C$, i.e. that $$C'=C[X_1\{Y_1\}/V_1][X_2\{Y_2\}/V_2]=C[X_2\{Y_2\}/V_2][X_1\{Y_1\}/V_1]$$ for some $V_1,V_2\in v(C)$. Suppose, moreover, that $V_1$ is above $V_2$ in $C$, and let $V_2\{U\}$ be the first edge on the path from $V_2$ to $V_1$. Suppose finally that, after splitting $V_2$, the edge $V_2\{U\}$ splits into $X_2\{Y_2\{U\}\}$, i.e. that the vertex $Y_2$ stays on the path from $X_2$ to $X_1$.  Under these assumptions on the shape of $C'$, let $p_1$ (resp. $p_2$) be the number of internal edges between $V_1$ and $V_2$ (resp. below $V_2$) in $\alpha({\cal T},C)$, and let $l_{i}$, for $i=1,2$, be the number of edges and leaves on the left from the edge $X_i\{Y_i\}$  in $\alpha({\cal T},C[X_i\{Y_i\}/V_i])$. Finally, let $l$ be the number of edges and leaves on the left from the edge $V_2\{U\}$  in $\alpha({\cal T},C[X_1\{Y_1\}/V_1])$. The signs of the operations $C[X_1\{Y_1\}/V_1][X_2\{Y_2\}/V_2]$ and $C[X_2\{Y_2\}/V_2][X_1\{Y_1\}/V_1]$ are then induced from the sums $$\underbrace{l_1+p_1+p_2+l}_\text{$X_1\{Y_1\}/V_1$} +\underbrace{l_2+p_2}_{X_2\{Y_2\}/V_2} \quad\mbox{ and }\quad \underbrace{l_2+p_2}_{X_2\{Y_2\}/V_2}+\underbrace{l_1+p_1+p_2+l+1}_{X_1\{Y_1\}/V_1},$$ respectively, and the conclusion follows since they differ by $1$.

\medskip 

2. Denote ${\cal T}={\cal T}_1\bullet_i {\cal T}_2$ and $C=C_1\bullet_i C_2$. For the left-hand side of the equality, we  have 
$$\begin{array}{rcl}
d_{{\EuScript O}_{\infty}}(({\cal T}_1,C_1)\circ_i ({\cal T}_2,C_2))&=& \displaystyle\sum_{\substack{V\in v(C) \\ |V|\geq 2}}\sum_{\substack{(X,Y) \\ X\cup Y=V \\ X\{Y\}:{\bf H}_{{\cal T}_V}}} (-1)^{\delta+\varepsilon} ({\cal T},C[X\{Y\}/V])
\end{array}$$
where    $\delta=E_{\leq X\{Y\}}(\alpha({\cal T},C[X(Y)/V]))$ and $\varepsilon= E_{i>}(\alpha({\cal T}_1,C_1))\cdot (E(\alpha({\cal T}_2,C_2))-1)$.   We prove the stated equality by case analysis with respect to the origin of the vertex $V$ relative to $v(C_1)$ and $v(C_2)$, and, if $V\in v(C_1)$, the position of the edge $X\{Y\}$ relative to the leaf $i$ of $\alpha({\cal T}_1,C_1)$.

\medskip

 Suppose that $V\in v(C_1)$.
 \begin{itemize}
 \item   If $X\{Y\}\in E_{\leq i}(\alpha({\cal T}_1,C_1[X(Y)/V]))$, then \\[-0.3cm]
 \begin{itemize}
 \item $E_{\leq X\{Y\}}(\alpha({\cal T}_1,C_1[X\{Y\}/V]))=\delta$, and\\[-0.3cm]
 \item $E_{i>}(\alpha({\cal T}_1,C_1[X\{Y\}/V]))\cdot (E(\alpha({\cal T}_2,C_2))-1)=\varepsilon$.\\[-0.3cm]
 \end{itemize}
 \item If $X\{Y\}\in E_{i>}(\alpha({\cal T}_1,C_1[X\{Y\}/V]))$, then \\[-0.3cm]
 \begin{itemize}
 \item $E_{\leq X\{Y\}}(\alpha({\cal T}_1,C_1[X\{Y\}/V]))=\delta-(E(\alpha({\cal T}_2,C_2))-1)$,  and\\[-0.3cm]
 \item $E_{i>}(\alpha({\cal T}_1,C_1[X\{Y\}/V]))\cdot (E(\alpha({\cal T}_2,C_2))-1)=\varepsilon+E(\alpha({\cal T}_2,C_2))-1$.\\[-0.3cm]
 \end{itemize}
 \end{itemize}
 In both cases, we have that 
$$E_{\leq X\{Y\}}(\alpha({\cal T}_1,C_1[X\{Y\}/V]))+E_{i>}(\alpha( {\cal T}_1,C_1[X\{Y\}/V]))\cdot (E(\alpha({\cal T}_2,C_2))-1)=\delta+\varepsilon,$$
which means that $(-1)^{\delta+\varepsilon} ({\cal T},C[X\{Y\}/V])$ appears as a summand in
$d_{{\EuScript O}_{\infty}}({\cal T}_1,C_1)\circ_i ({\cal T}_2,C_2)$.
 
 \medskip

  If $V\in V(C_2)$, then \\[-0.3cm]
  \begin{itemize}
  \item  $E_{\leq X\{Y\}}(\alpha({\cal T}_2,C_2[X\{Y\}/V]))=\delta-E_{\leq i}(\alpha({\cal T}_1,C_1))-1$, and \\[-0.3cm]
  \item $E_{i>}(\alpha({\cal T}_1,C_1))\cdot (E(\alpha({\cal T}_2,C_2[X\{Y\}/V]))-1)=\varepsilon-E_{i>}(\alpha({\cal T}_1,C_1)).$ \\[-0.3cm]
  \end{itemize}

\noindent   Observe that 
\begin{equation}\label{set1}E_{{ i>}}(\alpha({\cal T}_1,C_1))\cup E_{\leq i}(\alpha({\cal T}_1,C_1))\cup \{i\}={\it E}(\alpha({\cal T}_1,C_1))\backslash\{\rho({\cal T}_1)\}.\end{equation} Therefore,   if $|({\cal T}_1,C_1)|$ is even (resp. odd), then the cardinality of \eqref{set1} is even (resp. odd), and, hence,  
$$\delta-E_{\leq i}(\alpha({\cal T}_1,C_1))-1+\varepsilon-E_{{ i>}}(\alpha({\cal T}_1,C_1))=_{\mbox{mod}\,2}\delta+\varepsilon$$
 $$\mbox{(resp. } \delta-E_{\leq i}(\alpha({\cal T}_1,C_1))-1+\varepsilon-E_{{ i>}}(\alpha({\cal T}_1,C_1))=_{\mbox{mod}\,2}\delta+\varepsilon+1 \mbox{ ),}$$ meaning that $(-1)^{\delta+\varepsilon} ({\cal T},C[X\{Y\}/V])$ appears as a summand in $(-1)^{|({\cal T}_1,C_1)|}({\cal T}_1,C_1)\circ_i d_{{\EuScript O}_{\infty}}({\cal T}_2,C_2)$.
 
 \medskip
 
The opposite direction is treated by an analogous analysis. \end{proof}

\smallskip

\subsubsection{The homology of $({\EuScript O}_{\infty},d_{{\EuScript O}_{\infty}})$}\label{Ohomology} We first prove that, for $k\geq 4$, $${\EuScript O}^{k-2}_{\infty}(n_1,\dots,n_k;n)\xrightarrow{\,d^{k-2}_{{\EuScript O}_{\infty}}(n_1,\dots,n_k;n)\,}\enspace\cdots\enspace\xrightarrow{\,d^{1}_{{\EuScript O}_{\infty}}(n_1,\dots,n_k;n)\,}{\EuScript O}^{0}_{\infty}(n_1,\dots,n_k;n)$$
is an exact sequence.
\smallskip
\begin{thm} For   $0<m\leq k-2$, we have that  ${\mbox{\em Ker}}\, d_{{\EuScript O}_{\infty}}^m(n_1,\dots,n_k;n)= {\mbox{\em Im}}\, d_{{\EuScript O}_{\infty}}^{m+1}(n_1,\dots,n_k;n)$. 
\end{thm}
\begin{proof}
Since $d_{{\EuScript O}_{\infty}}$ squares to zero, we have that ${\mbox{ Im}}\, d_{{\EuScript O}_{\infty}}^{m+1}(n_1,\dots,n_k;n)\subseteq {\mbox{ Ker}}\, d_{{\EuScript O}_{\infty}}^m(n_1,\dots,n_k;n)$. 

\smallskip

We prove the other direction by defining an algorithm ${\cal G}$ that takes as an input an element \linebreak $L\in${\mbox{ Ker}}\, $d_{{\EuScript O}_{\infty}}^m(n_1,\dots,n_k;n)$,  and returns an element ${\cal G}(L)$ of degree $m+1$, such that  $$d_{{\EuScript O}_{\infty}}^{m+1}(n_1,\dots,n_k;n)({\cal G}(L))=L.$$  Observe that
 we may assume, without loss of generality, that the linear combination $L$ is made of triples  whose first two components are all given by the same operadic tree ${\cal T}=({\cal T},\sigma)$, i.e.  that $$L=k_1({\cal T},C_1)+\cdots +k_p ({\cal T},C_p),$$ where, moreover, $k_i\in\{+1,-1\}$. (The proof for the general case where $k_i$'s are arbitrary integers is based on an easy adaptation of the algorithm ${\cal G}$, which takes into account the number of occurences of counstructs in $L$.)
 The element ${\cal G}(L)$ is reconstructed by the following procedure.

\medskip
 
\paragraph{\textsf{Step 1: Diamonds.}}  For each pair of indices $i,j\in \{1,\dots,p\}$, such that $i\neq j$, we are first going to construct an element $k_{ij}({\cal T},C_{ij})$ of degree $m+1$, in the following way.
\begin{itemize}
\item  If there exists a vertex $Z$ of $C_i$ and a vertex $W$ of $C_j$, such that $c_{ij}=C_i[X\{Y\}/Z]=C_j[U\{V\}/W]$ (i.e. if $C_i$ and $C_j$ have a common face $c_{ij}$ of dimension $m-1$), and such that $C_i[X\{Y\}/Z]$ and $C_j[U\{V\}/W]$ appear with the opposite sign  in $d_{{\EuScript O}_{\infty}}^{m}(n_1,\dots,n_k;n)(L)$,  then   we set $C_{ij}$ to be  the construct defined by collapsing the edge $U\{V\}$ of $C_i$ (or, equivalently, by collapsing the edge $X\{Y\}$ of $C_j$). Note that,  if, say, $Y=U$, i.e. $Z\cap W=Y$, then the constructs $C_i$ and  $C_j$ result from splitting a single vertex of $C_{ij}$, namely the vertex decorated by $Z\cup W$, while, if the sets $X$, $Y$, $U$ and $V$ are mutually disjoint, then   $C_i$ and  $C_j$ result from splitting the vertices of $C_{ij}$ decorated by $Z$ and $W$, respectively.  Observe that the coefficient $k_{ij}\in\{+1,-1\}$ is uniquely determined by the shape of $({\cal T},C_{ij})$,  in combination with values $k_i$ and $k_j$. Indeed, suppose that the two occurences of $c_{ij}$ in $d_{{\EuScript O}_{\infty}}^{m}(n_1,\dots,n_k;n)(L)$ arising as faces of $C_i$ and $C_j$, have   signs $k_i(-1)^{\delta''_i}$ and $k_j(-1)^{\delta''_j}$, respectively, where $k_i(-1)^{\delta''_i}=-k_j(-1)^{\delta''_j}$, and let $\delta'_i,   \delta'_j\geq 0$ be  the  integers tied to the  derivations leading from $C_{ij}$ to $C_i$ and $C_j$, respectively. The coefficient $k_{ij}$ is then   determined by $$k_{ij}:=k_i(-1)^{\delta'_i}=k_j(-1)^{\delta'_j},$$ which is justified by showing that $$\delta''_j-\delta''_i+1=_{\mbox{mod}\,2} \delta'_j-\delta'_i,$$ for all  $\delta'_i, \delta''_i,   \delta'_j, \delta''_j \geq 0$ that could arise with respect to the shape of $({\cal T},C_{ij})$. For example,  if $C_i$ and  $C_j$ result from splitting a single vertex of $C_{ij}$ in the way indicated above, and if $\delta'_j-\delta'_i=_{\mbox{mod}\,2} 0$, then, assuming additionally that the vertices $X$, $Y=U$ and $V$ are arranged one above the other   in  $c_{ij}$, the edge $X\{Y\}$ will contribute to $\delta''_j$, but not to $\delta''_i$, which will imply that $\delta''_j-\delta''_i+1=_{\mbox{mod}\,2} 0$.
\item Otherwise, we set $k_{ij}({\cal T},C_{ij})$ to be zero.
\end{itemize}
This procedure gives us a ``diamond'' 
\begin{equation} 
\begin{aligned} 
\begin{tikzpicture}
\node (t)[rectangle,draw=none] at (0,1) {\small $k_{ij}({\cal T},C_{ij})$};
\node (l)[rectangle,draw=none] at (-0.75,0) {\small $k_{i}({\cal T},C_{i})$};
\node (r) [rectangle,draw=none] at (0.75,0) {\small $k_{j}({\cal T},C_{j})$};
\node (b)[rectangle,draw=none] at (0,-1) {\small  $\pm c_{ij}$};
\draw (t)--(l)--(b)--(r)--(t);
\end{tikzpicture}
\end{aligned}\label{diamond}
\end{equation}
for each pair of distinct factors in $L$. Observe that the ``diamond property'' of abstract polytopes ensures that $k_{ij}({\cal T},C_{ij})$ is the unique element of degree $m+1$ that can be obtained in this way. In other words, if $k_{i}({\cal T},C_{i})$ and $k_{j}({\cal T},C_{j})$ lie in the boundary of some other element of degree $m+1$, then this element must correspond to two different faces of the polytope ${\cal A}({\bf H}_{\cal T})$. 

\smallskip

Let $G(L)$ be the  sum of the elements obtained by the above procedure: $$G(L)=\sum^p_{i=1} \sum^p_{\substack {j=1\\ j\neq i}}k_{ij} ({\cal T},C_{ij}).$$  Observe that, in general, it is not the case that   $d_{{\EuScript O}_{\infty}}^{m+1}(n_1,\dots,n_k;n)(G(L))=L$.  Indeed, a counterexample can be read from the cycle of the hemiassociahedron coloured in blue in the picture that follows.
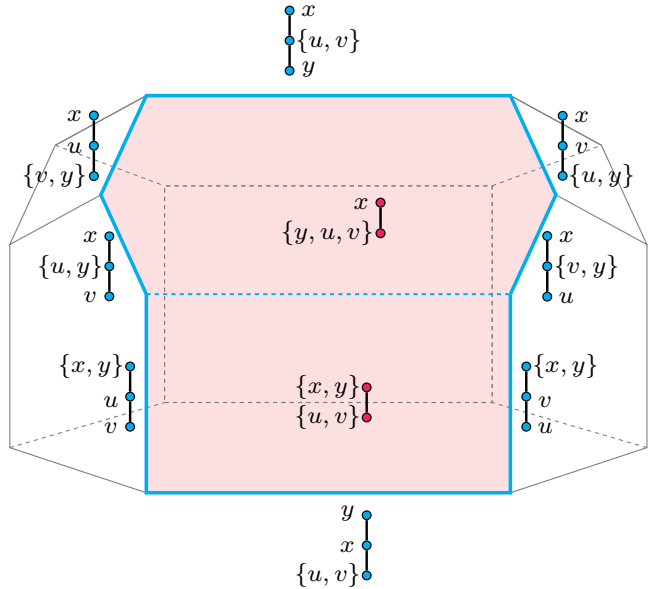
\begin{center}
 \resizebox{8.45cm}{!}{\begin{tikzpicture}[thick,scale=23]
\coordinate (A1) at (-0.2,0);
\coordinate (A2) at (0.2,0); 
\coordinate (D1) at (-0.35,0.05);
\coordinate (D2) at (0.35,0.05);
\coordinate (D3) at (-0.18,0.1);
\coordinate (D4)  at (0.18,0.1);
\coordinate (E3) at (-0.18,0.34);
\coordinate (E4) at (0.18,0.34);
\coordinate (A3) at (-0.2,0.22);
\coordinate (A4) at (0.2,0.22);
 \coordinate (C1) at (-0.25,0.33);
\coordinate (C2) at (0.25,0.33);
 \coordinate (F1) at (-0.35,0.275);
\coordinate (F2) at (0.35,0.275);
 \coordinate (G1) at (-0.3,0.385);
\coordinate (G2) at (0.3,0.385);
 \coordinate (B3) at (-0.2,0.44);
\coordinate (B4) at (0.2,0.44);
\fill[fill=WildStrawberry!20,opacity=0.75] (A1)--(A2)--(A4)--(A3)--(A1);
\fill[fill=WildStrawberry!20,opacity=0.75] (A4) -- (C2) -- (B4) -- (B3) -- (C1)-- (A3)--(A4);
\draw[gray,dashed] (D3) -- (D4) -- (E4) -- (E3) -- cycle;
\draw[gray] (G1) -- (B3) -- (B4) -- (G2);
\draw[gray,dashed]   (G1) -- (E3); 
\draw[gray,dashed] (E4) -- (G2);
 \draw[gray] (F1) -- (G1) -- (B3) -- (C1)  --  cycle;
\draw[gray] (F2) -- (G2) -- (B4) -- (C2)  --  cycle;
\draw[gray,dashed] (D3) -- (D4) -- (E4) -- (E3) -- cycle;
\draw[gray,dashed] (D1)   -- (D3);
\draw[gray,dashed] (D2)   -- (D4);
\draw[gray] (F1)--(D1)--(A1)--(A2)--(D2)--(F2);
\draw[line width=0.9mm,cyan]  (A4) -- (C2) -- (B4) -- (B3) -- (C1)-- (A3);
\draw[line width=0.9mm,cyan] (A3)--(A1) -- (A2) -- (A4) ;
\draw[line width=0.5mm,dashed,cyan] (A3)--(A4);
  \node at  (-0.26,0.11)  {\resizebox{2.2cm}{!}{\begin{tikzpicture}
\node (b) [circle,fill=cyan,draw=black,minimum size=0.1cm,inner sep=0.2mm,label={[xshift=-0.22cm,yshift=-0.25cm]{\footnotesize $v$}}] at (0,0) {};
\node (c) [circle,fill=cyan,draw=black,minimum size=0.1cm,inner sep=0.2mm,label={[xshift=-0.22cm,yshift=-0.3cm]{\footnotesize $u$}}] at (0,0.35) {};
 \node (d) [circle,fill=cyan,draw=black,minimum size=0.1cm,inner sep=0.2mm,label={[xshift=-0.45cm,yshift=-0.32cm]{\footnotesize $\{x,y\}$}}] at (0,0.7) {};
 \draw[thick]  (b)--(c)--(d);\end{tikzpicture}}};
\node at  (0.3,0.38)  {\resizebox{2.2cm}{!}{\begin{tikzpicture}
\node (b) [circle,fill=cyan,draw=black,minimum size=0.1cm,inner sep=0.2mm,label={[xshift=0.22cm,yshift=-0.25cm]{\footnotesize $x$}}] at (0,0.7) {};
\node (c) [circle,fill=cyan,draw=black,minimum size=0.1cm,inner sep=0.2mm,label={[xshift=0.22cm,yshift=-0.26cm]{\footnotesize $v$}}] at (0,0.35) {};
 \node (d) [circle,fill=cyan,draw=black,minimum size=0.1cm,inner sep=0.2mm,label={[xshift=0.45cm,yshift=-0.32cm]{\footnotesize $\{u,y\}$}}] at (0,0) {};
 \draw[thick]  (b)--(c)--(d);\end{tikzpicture}}};
\node at  (-0.3,0.38)  {\resizebox{2.2cm}{!}{\begin{tikzpicture}
\node (b) [circle,fill=cyan,draw=black,minimum size=0.1cm,inner sep=0.2mm,label={[xshift=-0.22cm,yshift=-0.25cm]{\footnotesize $x$}}] at (0,0.7) {};
\node (c) [circle,fill=cyan,draw=black,minimum size=0.1cm,inner sep=0.2mm,label={[xshift=-0.22cm,yshift=-0.26cm]{\footnotesize $u$}}] at (0,0.35) {};
 \node (d) [circle,fill=cyan,draw=black,minimum size=0.1cm,inner sep=0.2mm,label={[xshift=-0.45cm,yshift=-0.32cm]{\footnotesize $\{v,y\}$}}] at (0,0) {};
 \draw[thick]  (b)--(c)--(d);\end{tikzpicture}}};
   \node at  (0,-0.06)  {\resizebox{2.2cm}{!}{\begin{tikzpicture}
\node (b) [circle,fill=cyan,draw=black,minimum size=0.1cm,inner sep=0.2mm,label={[xshift=-0.22cm,yshift=-0.25cm]{\footnotesize $y$}}] at (0,0.7) {};
\node (c) [circle,fill=cyan,draw=black,minimum size=0.1cm,inner sep=0.2mm,label={[xshift=-0.22cm,yshift=-0.26cm]{\footnotesize $x$}}] at (0,0.35) {};
 \node (d) [circle,fill=cyan,draw=black,minimum size=0.1cm,inner sep=0.2mm,label={[xshift=-0.45cm,yshift=-0.32cm]{\footnotesize $\{u,v\}$}}] at (0,0) {};
 \draw[thick]  (b)--(c)--(d);\end{tikzpicture}}};
    \node at  (0,0.1)  {\resizebox{2.2cm}{!}{\begin{tikzpicture}
\node (c) [circle,fill=WildStrawberry,draw=black,minimum size=0.1cm,inner sep=0.2mm,label={[xshift=-0.45cm,yshift=-0.32cm]{\footnotesize $\{x,y\}$}}] at (0,0.35) {};
 \node (d) [circle,fill=WildStrawberry,draw=black,minimum size=0.1cm,inner sep=0.2mm,label={[xshift=-0.45cm,yshift=-0.32cm]{\footnotesize $\{u,v\}$}}] at (0,0) {};
 \draw[thick]  (c)--(d);\end{tikzpicture}}};
 \node at  (0,0.3)  {\resizebox{2.9cm}{!}{\begin{tikzpicture}
\node (c) [circle,fill=WildStrawberry,draw=black,minimum size=0.1cm,inner sep=0.2mm,label={[xshift=-0.22cm,yshift=-0.25cm]{\footnotesize $x$}}] at (0,0.35) {};
 \node (d) [circle,fill=WildStrawberry,draw=black,minimum size=0.1cm,inner sep=0.2mm,label={[xshift=-0.6cm,yshift=-0.32cm]{\footnotesize $\{y,u,v\}$}}] at (0,0) {};
 \draw[thick]  (c)--(d);\end{tikzpicture}}};
 \node at  (0,0.5)  {\resizebox{2.2cm}{!}{\begin{tikzpicture}
\node (b) [circle,fill=cyan,draw=black,minimum size=0.1cm,inner sep=0.2mm,label={[xshift=0.22cm,yshift=-0.25cm]{\footnotesize $y$}}] at (0,0) {};
\node (c) [circle,fill=cyan,draw=black,minimum size=0.1cm,inner sep=0.2mm,label={[xshift=0.22cm,yshift=-0.25cm]{\footnotesize $x$}}] at (0,0.7) {};
 \node (d) [circle,fill=cyan,draw=black,minimum size=0.1cm,inner sep=0.2mm,label={[xshift=0.45cm,yshift=-0.32cm]{\footnotesize $\{u,v\}$}}] at (0,0.35) {};
 \draw[thick] (b)--(d)--(c);\end{tikzpicture}}};
 \node at  (-0.283,0.25)  {\resizebox{2.2cm}{!}{\begin{tikzpicture}
\node (b) [circle,fill=cyan,draw=black,minimum size=0.1cm,inner sep=0.2mm,label={[xshift=-0.22cm,yshift=-0.25cm]{\footnotesize $v$}}] at (0,0) {};
\node (c) [circle,fill=cyan,draw=black,minimum size=0.1cm,inner sep=0.2mm,label={[xshift=-0.22cm,yshift=-0.25cm]{\footnotesize $x$}}] at (0,0.7) {};
 \node (d) [circle,fill=cyan,draw=black,minimum size=0.1cm,inner sep=0.2mm,label={[xshift=-0.45cm,yshift=-0.33cm]{\footnotesize $\{u,y\}$}}] at (0,0.35) {};
 \draw[thick] (b)--(d)--(c);\end{tikzpicture}}};
\node at  (0.283,0.25)  {\resizebox{2.2cm}{!}{\begin{tikzpicture}
\node (b) [circle,fill=cyan,draw=black,minimum size=0.1cm,inner sep=0.2mm,label={[xshift=0.22cm,yshift=-0.25cm]{\footnotesize $u$}}] at (0,0) {};
\node (c) [circle,fill=cyan,draw=black,minimum size=0.1cm,inner sep=0.2mm,label={[xshift=0.22cm,yshift=-0.25cm]{\footnotesize $x$}}] at (0,0.7) {};
 \node (d) [circle,fill=cyan,draw=black,minimum size=0.1cm,inner sep=0.2mm,label={[xshift=0.45cm,yshift=-0.33cm]{\footnotesize $\{v,y\}$}}] at (0,0.35) {};
 \draw[thick] (b)--(d)--(c);\end{tikzpicture}}};
 \node at  (0.26,0.11)  {\resizebox{2.2cm}{!}{\begin{tikzpicture}
\node (b) [circle,fill=cyan,draw=black,minimum size=0.1cm,inner sep=0.2mm,label={[xshift=0.22cm,yshift=-0.25cm]{\footnotesize $u$}}] at (0,0) {};
\node (c) [circle,fill=cyan,draw=black,minimum size=0.1cm,inner sep=0.2mm,label={[xshift=0.22cm,yshift=-0.3cm]{\footnotesize $v$}}] at (0,0.35) {};
 \node (d) [circle,fill=cyan,draw=black,minimum size=0.1cm,inner sep=0.2mm,label={[xshift=0.45cm,yshift=-0.32cm]{\footnotesize $\{x,y\}$}}] at (0,0.7) {};
 \draw[thick]  (b)--(c)--(d);\end{tikzpicture}}};
\end{tikzpicture}}
\end{center}
In this case, the diamond  for the edges 
\begin{center}
\raisebox{-1.4em}{\resizebox{1.2cm}{!}{\begin{tikzpicture}
\node (b) [circle,fill=cyan,draw=black,minimum size=0.1cm,inner sep=0.2mm,label={[xshift=-0.22cm,yshift=-0.25cm]{\footnotesize $v$}}] at (0,0) {};
\node (c) [circle,fill=cyan,draw=black,minimum size=0.1cm,inner sep=0.2mm,label={[xshift=-0.22cm,yshift=-0.25cm]{\footnotesize $x$}}] at (0,0.7) {};
 \node (d) [circle,fill=cyan,draw=black,minimum size=0.1cm,inner sep=0.2mm,label={[xshift=-0.45cm,yshift=-0.32cm]{\footnotesize $\{u,y\}$}}] at (0,0.35) {};
 \draw[thick] (b)--(d)--(c);\end{tikzpicture}}} \hspace{0.2cm} and \hspace{-0.4cm}  \raisebox{-1.4em}{\resizebox{1.2cm}{!}{\begin{tikzpicture}
\node (b) [circle,fill=cyan,draw=black,minimum size=0.1cm,inner sep=0.2mm,label={[xshift=-0.22cm,yshift=-0.25cm]{\footnotesize $v$}}] at (0,0) {};
\node (c) [circle,fill=cyan,draw=black,minimum size=0.1cm,inner sep=0.2mm,label={[xshift=-0.22cm,yshift=-0.3cm]{\footnotesize $u$}}] at (0,0.35) {};
 \node (d) [circle,fill=cyan,draw=black,minimum size=0.1cm,inner sep=0.2mm,label={[xshift=-0.45cm,yshift=-0.32cm]{\footnotesize $\{x,y\}$}}] at (0,0.7) {};
 \draw[thick]  (b)--(c)--(d);\end{tikzpicture}}} \quad\quad (resp. \enspace\raisebox{-1.4em}{\resizebox{1.2cm}{!}{\begin{tikzpicture}
\node (b) [circle,fill=cyan,draw=black,minimum size=0.1cm,inner sep=0.2mm,label={[xshift=0.22cm,yshift=-0.25cm]{\footnotesize $u$}}] at (0,0) {};
\node (c) [circle,fill=cyan,draw=black,minimum size=0.1cm,inner sep=0.2mm,label={[xshift=0.22cm,yshift=-0.25cm]{\footnotesize $x$}}] at (0,0.7) {};
 \node (d) [circle,fill=cyan,draw=black,minimum size=0.1cm,inner sep=0.2mm,label={[xshift=0.45cm,yshift=-0.32cm]{\footnotesize $\{v,y\}$}}] at (0,0.35) {};
 \draw[thick] (b)--(d)--(c);\end{tikzpicture}}} and   \raisebox{-1.4em}{\resizebox{1.2cm}{!}{\begin{tikzpicture}
\node (b) [circle,fill=cyan,draw=black,minimum size=0.1cm,inner sep=0.2mm,label={[xshift=0.22cm,yshift=-0.25cm]{\footnotesize $u$}}] at (0,0) {};
\node (c) [circle,fill=cyan,draw=black,minimum size=0.1cm,inner sep=0.2mm,label={[xshift=0.22cm,yshift=-0.3cm]{\footnotesize $v$}}] at (0,0.35) {};
 \node (d) [circle,fill=cyan,draw=black,minimum size=0.1cm,inner sep=0.2mm,label={[xshift=0.45cm,yshift=-0.32cm]{\footnotesize $\{x,y\}$}}] at (0,0.7) {};
 \draw[thick]  (b)--(c)--(d);\end{tikzpicture}}}\,) 
\end{center}
will have the pentagon
\begin{center}
\raisebox{-1em}{\resizebox{1.6cm}{!}{\begin{tikzpicture}
\node (c) [circle,fill=WildStrawberry,draw=black,minimum size=0.1cm,inner sep=0.2mm,label={[xshift=-0.6cm,yshift=-0.32cm]{\footnotesize $\{x,y,u\}$}}] at (0,0.35) {};
 \node (d) [circle,fill=WildStrawberry,draw=black,minimum size=0.1cm,inner sep=0.2mm,label={[xshift=-0.22cm,yshift=-0.25cm]{\footnotesize $v$}}] at (0,0) {};
 \draw[thick]  (c)--(d);\end{tikzpicture}}} \quad\quad (resp. \enspace  \raisebox{-1em}{\resizebox{1.6cm}{!}{\begin{tikzpicture}
\node (c) [circle,fill=WildStrawberry,draw=black,minimum size=0.1cm,inner sep=0.2mm,label={[xshift=0.6cm,yshift=-0.32cm]{\footnotesize $\{x,y,v\}$}}] at (0,0.35) {};
 \node (d) [circle,fill=WildStrawberry,draw=black,minimum size=0.1cm,inner sep=0.2mm,label={[xshift=0.22cm,yshift=-0.25cm]{\footnotesize $u$}}] at (0,0) {};
 \draw[thick]  (c)--(d);\end{tikzpicture}}})
\end{center}
 as the top element, which violates the equality $d_{{\EuScript O}_{\infty}}^{2}(n_1,\dots,n_k;n)(G(L))=L$.  

\medskip

\paragraph{\textsf{Step 2: Correction.}} In the next step of the procedure, we remove  ``bad'' factors from $G(L)$, as follows. We consider the Hasse diagram constructed by going  ``downwards'' starting from the constructs of $G(L)$, i.e., by applying two succesive derivations to $G(L)$. Observe that such a diagram may contain  diamonds which are not among the diamonds constructed in the first step of the algorithm. We shall remove from $G(L)$  each factor $k_{ij}({\cal T},C_{ij})$ for which there exist two other diamonds (in the Hasse diagram) which, together with the diamond of $k_{ij}({\cal T},C_{ij})$, make the following configuration:
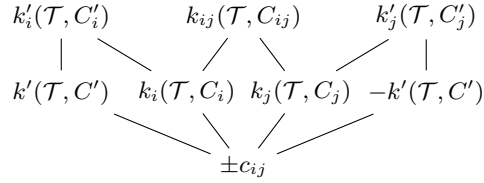
\begin{center}
\begin{tikzpicture}
\node (tl)[rectangle,draw=none] at (-2.4,1) {\small $k'_{i}({\cal T},C'_{i})$};
\node (tr)[rectangle,draw=none] at (2.4,1) {\small $k'_{j}({\cal T},C'_{j})$};
\node (t)[rectangle,draw=none] at (0,1) {\small $k_{ij}({\cal T},C_{ij})$};
\node (l)[rectangle,draw=none] at (-0.75,0) {\small $k_{i}({\cal T},C_{i})$};
\node (r) [rectangle,draw=none] at (0.75,0) {\small $k_{j}({\cal T},C_{j})$};
\node (ll)[rectangle,draw=none] at (-2.4,0) {\small $k'({\cal T},C')$};
\node (rr) [rectangle,draw=none] at (2.4,0) {\small $-k'({\cal T},C')$};
\node (b)[rectangle,draw=none] at (0,-1) {\small  $\pm c_{ij}$};
\draw (t)--(l)--(b)--(r)--(t);
\draw (ll)--(b)--(rr);
\draw (ll)--(tl)--(l);
\draw (rr)--(tr)--(r);
\end{tikzpicture}
\end{center}
(note the two occurences of $k'({\cal T},C')$ with the opposite sign at dimension $m$).  
  For the cycle $L$ from the   example above, the correction part of the algorithm will remove precisely the pentagons 
\begin{center}
\raisebox{-1em}{\resizebox{1.6cm}{!}{\begin{tikzpicture}
\node (c) [circle,fill=WildStrawberry,draw=black,minimum size=0.1cm,inner sep=0.2mm,label={[xshift=-0.6cm,yshift=-0.32cm]{\footnotesize $\{x,y,u\}$}}] at (0,0.35) {};
 \node (d) [circle,fill=WildStrawberry,draw=black,minimum size=0.1cm,inner sep=0.2mm,label={[xshift=-0.22cm,yshift=-0.25cm]{\footnotesize $v$}}] at (0,0) {};
 \draw[thick]  (c)--(d);\end{tikzpicture}}} \quad\quad  and \quad\quad  \raisebox{-1em}{\resizebox{1.6cm}{!}{\begin{tikzpicture}
\node (c) [circle,fill=WildStrawberry,draw=black,minimum size=0.1cm,inner sep=0.2mm,label={[xshift=0.6cm,yshift=-0.32cm]{\footnotesize $\{x,y,v\}$}}] at (0,0.35) {};
 \node (d) [circle,fill=WildStrawberry,draw=black,minimum size=0.1cm,inner sep=0.2mm,label={[xshift=0.22cm,yshift=-0.25cm]{\footnotesize $u$}}] at (0,0) {};
 \draw[thick]  (c)--(d);\end{tikzpicture}}}.
\end{center} 
  We define ${\cal G}(L)$ to be the sum that remains after the removal of all ``bad'' factors from $G(L)$. Observe that, if ${\cal G}(L)=G(L)$, then  $L$ is the boundary of mutually non-adjacent faces of ${\cal A}({\bf H}_{\cal T})$.
\medskip

We now prove that the algorithm ${\cal G}$ is sound, i.e. that  $d_{{\EuScript O}_{\infty}}^{m+1}(n_1,\dots,n_k;n)({\cal G}(L))=L$. We start by showing that, in the  correction step of the algorithm, we did not remove {\em too much} from $G(L)$, i.e. that all the summands of $L$ occur in $d_{{\EuScript O}_{\infty}}^{m+1}(n_1,\dots,n_k;n)({\cal G}(L))$ with the same coefficients. Notice that this holds trivially if ${\cal G}(L)=G(L)$. Suppose, therefore, that there exist  indices $i,j\in\{1,\dots,p\}$, $i\neq j$, such that the  top element $k_{ij}({\cal T},C_{ij})$ of the diamond \eqref{diamond} has been removed from $G(L)$. The proof that we can always find  other diamonds that contain $C_i$ and $C_{j}$ in dimension $m$, and whose top elements belongs to ${\cal G}(L)$
 relies on the fact that $L$ is a cycle, in the following way. We consider all the elements in the boundary of $C_i$, and, for each such element, we take all the associated diamonds that contain  $C_i$ at dimension $m$ (the existence of at least one such a diamond is ensured by the fact that $L$ is a cycle). If none of these diamonds is accepted, then all the zig-zag's of $(m-1)$-dimensional and $m$-dimensional constructs that ``go along $L$'' will end with an $(m-1)$-dimensional construct whose cancellation requires a construct outside of $L$. The rigorous proof can be extracted from the above example of the cycle $L$ of the hemiassociahedron, by taking, for example,
\begin{center}
  $C_i=$\,\,\raisebox{-1.5em}{\resizebox{1.2cm}{!}{\begin{tikzpicture}
\node (b) [circle,fill=cyan,draw=black,minimum size=0.1cm,inner sep=0.2mm,label={[xshift=0.22cm,yshift=-0.25cm]{\footnotesize $u$}}] at (0,0) {};
\node (c) [circle,fill=cyan,draw=black,minimum size=0.1cm,inner sep=0.2mm,label={[xshift=0.22cm,yshift=-0.25cm]{\footnotesize $x$}}] at (0,0.7) {};
 \node (d) [circle,fill=cyan,draw=black,minimum size=0.1cm,inner sep=0.2mm,label={[xshift=0.45cm,yshift=-0.32cm]{\footnotesize $\{v,y\}$}}] at (0,0.35) {};
 \draw[thick] (b)--(d)--(c);\end{tikzpicture}}} \quad and \quad $C_j=$\,\,\raisebox{-1.5em}{\resizebox{1.2cm}{!}{\begin{tikzpicture}
\node (b) [circle,fill=cyan,draw=black,minimum size=0.1cm,inner sep=0.2mm,label={[xshift=0.22cm,yshift=-0.25cm]{\footnotesize $u$}}] at (0,0) {};
\node (c) [circle,fill=cyan,draw=black,minimum size=0.1cm,inner sep=0.2mm,label={[xshift=0.22cm,yshift=-0.3cm]{\footnotesize $v$}}] at (0,0.35) {};
 \node (d) [circle,fill=cyan,draw=black,minimum size=0.1cm,inner sep=0.2mm,label={[xshift=0.45cm,yshift=-0.32cm]{\footnotesize $\{x,y\}$}}] at (0,0.7) {};
 \draw[thick]  (b)--(c)--(d);\end{tikzpicture}}}
\end{center}

In this case, the pentagon \raisebox{-1em}{\resizebox{1.6cm}{!}{\begin{tikzpicture}
\node (c) [circle,fill=WildStrawberry,draw=black,minimum size=0.1cm,inner sep=0.2mm,label={[xshift=0.6cm,yshift=-0.32cm]{\footnotesize $\{x,y,v\}$}}] at (0,0.35) {};
 \node (d) [circle,fill=WildStrawberry,draw=black,minimum size=0.1cm,inner sep=0.2mm,label={[xshift=0.22cm,yshift=-0.25cm]{\footnotesize $u$}}] at (0,0) {};
 \draw[thick]  (c)--(d);\end{tikzpicture}}} that has been removed from $G(L)$, has been obtained by  collapsing the top edge of both $C_i$ and $C_j$. By collapsing the bottom edge of both  $C_i$ and $C_j$, we obtain two new diamonds:

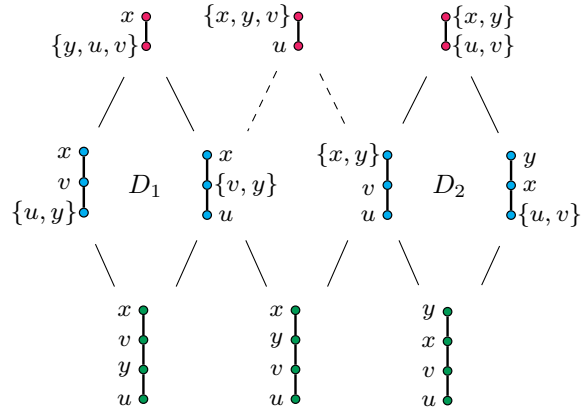
\begin{center}
\begin{tikzpicture}
\node(D1) at (0,-0.5) {$D_1$};\node(D2) at (4,-0.5) {$D_2$};
\node (t)[rectangle,draw=none] at (0,1.5) {\resizebox{2.75cm}{!}{\begin{tikzpicture}
\node(k) at (1.1,0) {};
\node (c) [circle,fill=WildStrawberry,draw=black,minimum size=0.1cm,inner sep=0.2mm,label={[xshift=-0.22cm,yshift=-0.25cm]{\footnotesize $x$}}] at (0,0.35) {};
 \node (d) [circle,fill=WildStrawberry,draw=black,minimum size=0.1cm,inner sep=0.2mm,label={[xshift=-0.6cm,yshift=-0.32cm]{\footnotesize $\{y,u,v\}$}}] at (0,0) {};
 \draw[thick]  (c)--(d);\end{tikzpicture}}};
\node (l)[rectangle,draw=none] at (-1,-0.5) {\resizebox{!}{1.3cm}{\begin{tikzpicture}
\node(k) at (0.5,0) {};
\node (b) [circle,fill=cyan,draw=black,minimum size=0.1cm,inner sep=0.2mm,label={[xshift=-0.22cm,yshift=-0.25cm]{\footnotesize $x$}}] at (0,0.7) {};
\node (c) [circle,fill=cyan,draw=black,minimum size=0.1cm,inner sep=0.2mm,label={[xshift=-0.22cm,yshift=-0.26cm]{\footnotesize $v$}}] at (0,0.35) {};
 \node (d) [circle,fill=cyan,draw=black,minimum size=0.1cm,inner sep=0.2mm,label={[xshift=-0.45cm,yshift=-0.32cm]{\footnotesize $\{u,y\}$}}] at (0,0) {};
 \draw[thick]  (b)--(c)--(d);\end{tikzpicture}}};
\node (rrr)[rectangle,draw=none] at (5,-0.55) {\resizebox{!}{1.25cm}{\begin{tikzpicture}
\node(k) at (-0.5,-0) {};
\node (b) [circle,fill=cyan,draw=black,minimum size=0.1cm,inner sep=0.2mm,label={[xshift=0.22cm,yshift=-0.32cm]{\footnotesize $y$}}] at (0,0.7) {};
\node (c) [circle,fill=cyan,draw=black,minimum size=0.1cm,inner sep=0.2mm,label={[xshift=0.22cm,yshift=-0.25cm]{\footnotesize $x$}}] at (0,0.35) {};
 \node (d) [circle,fill=cyan,draw=black,minimum size=0.1cm,inner sep=0.2mm,label={[xshift=0.45cm,yshift=-0.32cm]{\footnotesize $\{u,v\}$}}] at (0,0) {};
 \draw[thick]  (b)--(c)--(d);\end{tikzpicture}}};
\node (r) [rectangle,draw=none] at (1,-0.5) {\resizebox{!}{1.2cm}{\begin{tikzpicture}
\node(k) at (-0.5,0) {};
\node (b) [circle,fill=cyan,draw=black,minimum size=0.1cm,inner sep=0.2mm,label={[xshift=0.22cm,yshift=-0.25cm]{\footnotesize $u$}}] at (0,0) {};
\node (c) [circle,fill=cyan,draw=black,minimum size=0.1cm,inner sep=0.2mm,label={[xshift=0.22cm,yshift=-0.25cm]{\footnotesize $x$}}] at (0,0.7) {};
 \node (d) [circle,fill=cyan,draw=black,minimum size=0.1cm,inner sep=0.2mm,label={[xshift=0.45cm,yshift=-0.32cm]{\footnotesize $\{v,y\}$}}] at (0,0.35) {};
 \draw[thick] (b)--(d)--(c);\end{tikzpicture}}};
\node (b)[rectangle,draw=none] at (0,-2.75) {{\resizebox{1cm}{!}{\begin{tikzpicture}
\node (b) [circle,fill=ForestGreen,draw=black,minimum size=0.1cm,inner sep=0.2mm,label={[xshift=-0.22cm,yshift=-0.25cm]{\footnotesize $u$}}] at (0,0) {};
\node (c) [circle,fill=ForestGreen,draw=black,minimum size=0.1cm,inner sep=0.2mm,label={[xshift=-0.22cm,yshift=-0.27cm]{\footnotesize $y$}}] at (0,0.35) {};
 \node (d) [circle,fill=ForestGreen,draw=black,minimum size=0.1cm,inner sep=0.2mm,label={[xshift=-0.22cm,yshift=-0.25cm]{\footnotesize $v$}}] at (0,0.7) {};
  \node (a) [circle,fill=ForestGreen,draw=black,minimum size=0.1cm,inner sep=0.2mm,label={[xshift=-0.22cm,yshift=-0.25cm]{\footnotesize $x$}}] at (0,1.05) {};
\node(k) at (0.35,0) {};
 \draw[thick]  (b)--(c)--(d)--(a);\end{tikzpicture}}}};
\node (k1)[rectangle,draw=none] at (2,-2.75) {{\resizebox{1cm}{!}{\begin{tikzpicture}
\node (b1) [circle,fill=ForestGreen,draw=black,minimum size=0.1cm,inner sep=0.2mm,label={[xshift=-0.22cm,yshift=-0.25cm]{\footnotesize $u$}}] at (0,0) {};
\node (c1) [circle,fill=ForestGreen,draw=black,minimum size=0.1cm,inner sep=0.2mm,label={[xshift=-0.22cm,yshift=-0.25cm]{\footnotesize $v$}}] at (0,0.35) {};
 \node (d1) [circle,fill=ForestGreen,draw=black,minimum size=0.1cm,inner sep=0.2mm,label={[xshift=-0.22cm,yshift=-0.27cm]{\footnotesize $y$}}] at (0,0.7) {};
  \node (a1) [circle,fill=ForestGreen,draw=black,minimum size=0.1cm,inner sep=0.2mm,label={[xshift=-0.22cm,yshift=-0.25cm]{\footnotesize $x$}}] at (0,1.05) {};
\node(k) at (0.35,0) {};
 \draw[thick]  (b1)--(c1)--(d1)--(a1);\end{tikzpicture}}}};
\node (k12)[rectangle,draw=none] at (4,-2.75) {{\resizebox{1cm}{!}{\begin{tikzpicture}
\node (b1) [circle,fill=ForestGreen,draw=black,minimum size=0.1cm,inner sep=0.2mm,label={[xshift=-0.22cm,yshift=-0.25cm]{\footnotesize $u$}}] at (0,0) {};
\node (c1) [circle,fill=ForestGreen,draw=black,minimum size=0.1cm,inner sep=0.2mm,label={[xshift=-0.22cm,yshift=-0.25cm]{\footnotesize $v$}}] at (0,0.35) {};
 \node (d1) [circle,fill=ForestGreen,draw=black,minimum size=0.1cm,inner sep=0.2mm,label={[xshift=-0.22cm,yshift=-0.25cm]{\footnotesize $x$}}] at (0,0.7) {};
  \node (a1) [circle,fill=ForestGreen,draw=black,minimum size=0.1cm,inner sep=0.2mm,label={[xshift=-0.22cm,yshift=-0.27cm]{\footnotesize $y$}}] at (0,1.05) {};
\node(k) at (0.35,0) {};
 \draw[thick]  (b1)--(c1)--(d1)--(a1);\end{tikzpicture}}}};
\node (tq)[rectangle,draw=none] at (2,1.59) {\resizebox{2.75cm}{!}{\begin{tikzpicture}
\node(kq) at (1.1,0) {};
\node (cq) [circle,fill=WildStrawberry,draw=black,minimum size=0.1cm,inner sep=0.2mm,label={[xshift=-0.22cm,yshift=-0.25cm]{\footnotesize $u$}}] at (0,0) {};
 \node (dq) [circle,fill=WildStrawberry,draw=black,minimum size=0.1cm,inner sep=0.2mm,label={[xshift=-0.6cm,yshift=-0.32cm]{\footnotesize $\{x,y,v\}$}}] at (0,0.35) {};
 \draw[thick]  (cq)--(dq);\end{tikzpicture}}};
  \node (jo) at  (4,1.55)  {\resizebox{2cm}{!}{\begin{tikzpicture}
\node (aa) at (-0.75,0) {};
\node (c) [circle,fill=WildStrawberry,draw=black,minimum size=0.1cm,inner sep=0.2mm,label={[xshift=0.45cm,yshift=-0.32cm]{\footnotesize $\{x,y\}$}}] at (0,0.35) {};
 \node (d) [circle,fill=WildStrawberry,draw=black,minimum size=0.1cm,inner sep=0.2mm,label={[xshift=0.45cm,yshift=-0.32cm]{\footnotesize $\{u,v\}$}}] at (0,0) {};
 \draw[thick]  (c)--(d);\end{tikzpicture}}};
\node (mq1) [rectangle,draw=none] at (3,-0.45) {\resizebox{!}{1.3cm}{\begin{tikzpicture}
\node (kq1) at (0.5,0) {};
\node (bq1) [circle,fill=cyan,draw=black,minimum size=0.1cm,inner sep=0.2mm,label={[xshift=-0.22cm,yshift=-0.25cm]{\footnotesize $u$}}] at (0,0) {};
\node (cq1) [circle,fill=cyan,draw=black,minimum size=0.1cm,inner sep=0.2mm,label={[xshift=-0.22cm,yshift=-0.3cm]{\footnotesize $v$}}] at (0,0.35) {};
 \node (dq1) [circle,fill=cyan,draw=black,minimum size=0.1cm,inner sep=0.2mm,label={[xshift=-0.45cm,yshift=-0.32cm]{\footnotesize $\{x,y\}$}}] at (0,0.7) {};
 \draw[thick]  (bq1)--(cq1)--(dq1);\end{tikzpicture}}};
\draw (t)--(l)--(b)--(r)--(t);
\draw (r)--(k1)--(mq1)--(jo);
\draw (mq1)--(k12)--(rrr)--(jo);
\draw[dashed] (mq1)--(tq)--(r);
\end{tikzpicture}
\end{center}
Observe that, in the construction of $D_1$ and $D_2$, the fact that $L$ is a cycle has been used in order to ensure  that the constructs  
\begin{center}
\resizebox{!}{1.3cm}{\begin{tikzpicture}
\node (b) [circle,fill=cyan,draw=black,minimum size=0.1cm,inner sep=0.2mm,label={[xshift=-0.22cm,yshift=-0.25cm]{\footnotesize $x$}}] at (0,0.7) {};
\node (c) [circle,fill=cyan,draw=black,minimum size=0.1cm,inner sep=0.2mm,label={[xshift=-0.22cm,yshift=-0.26cm]{\footnotesize $v$}}] at (0,0.35) {};
 \node (d) [circle,fill=cyan,draw=black,minimum size=0.1cm,inner sep=0.2mm,label={[xshift=-0.45cm,yshift=-0.32cm]{\footnotesize $\{u,y\}$}}] at (0,0) {};
 \draw[thick]  (b)--(c)--(d);\end{tikzpicture}} \quad\quad \raisebox{1.55em}{and} \quad\quad \resizebox{!}{1.25cm}{\begin{tikzpicture}
\node (b) [circle,fill=cyan,draw=black,minimum size=0.1cm,inner sep=0.2mm,label={[xshift=0.22cm,yshift=-0.32cm]{\footnotesize $y$}}] at (0,0.7) {};
\node (c) [circle,fill=cyan,draw=black,minimum size=0.1cm,inner sep=0.2mm,label={[xshift=0.22cm,yshift=-0.25cm]{\footnotesize $x$}}] at (0,0.35) {};
 \node (d) [circle,fill=cyan,draw=black,minimum size=0.1cm,inner sep=0.2mm,label={[xshift=0.45cm,yshift=-0.32cm]{\footnotesize $\{u,v\}$}}] at (0,0) {};
 \draw[thick]  (b)--(c)--(d);\end{tikzpicture}}
\end{center}
are indeed present in $L$, since, otherwise, the constructs 
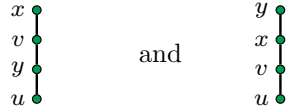
\begin{center}
\resizebox{1cm}{!}{\begin{tikzpicture}
\node (b) [circle,fill=ForestGreen,draw=black,minimum size=0.1cm,inner sep=0.2mm,label={[xshift=-0.22cm,yshift=-0.25cm]{\footnotesize $u$}}] at (0,0) {};
\node (c) [circle,fill=ForestGreen,draw=black,minimum size=0.1cm,inner sep=0.2mm,label={[xshift=-0.22cm,yshift=-0.27cm]{\footnotesize $y$}}] at (0,0.35) {};
 \node (d) [circle,fill=ForestGreen,draw=black,minimum size=0.1cm,inner sep=0.2mm,label={[xshift=-0.22cm,yshift=-0.25cm]{\footnotesize $v$}}] at (0,0.7) {};
  \node (a) [circle,fill=ForestGreen,draw=black,minimum size=0.1cm,inner sep=0.2mm,label={[xshift=-0.22cm,yshift=-0.25cm]{\footnotesize $x$}}] at (0,1.05) {};
\node(k) at (0.35,0) {};
 \draw[thick]  (b)--(c)--(d)--(a);\end{tikzpicture}} \quad\quad \raisebox{2em}{and} \quad\quad \resizebox{1cm}{!}{\begin{tikzpicture}
\node (b1) [circle,fill=ForestGreen,draw=black,minimum size=0.1cm,inner sep=0.2mm,label={[xshift=-0.22cm,yshift=-0.25cm]{\footnotesize $u$}}] at (0,0) {};
\node (c1) [circle,fill=ForestGreen,draw=black,minimum size=0.1cm,inner sep=0.2mm,label={[xshift=-0.22cm,yshift=-0.25cm]{\footnotesize $v$}}] at (0,0.35) {};
 \node (d1) [circle,fill=ForestGreen,draw=black,minimum size=0.1cm,inner sep=0.2mm,label={[xshift=-0.22cm,yshift=-0.25cm]{\footnotesize $x$}}] at (0,0.7) {};
  \node (a1) [circle,fill=ForestGreen,draw=black,minimum size=0.1cm,inner sep=0.2mm,label={[xshift=-0.22cm,yshift=-0.27cm]{\footnotesize $y$}}] at (0,1.05) {};
\node(k) at (0.35,0) {};
 \draw[thick]  (b1)--(c1)--(d1)--(a1);\end{tikzpicture}}
\end{center}
would not cancel out. Finally, the fact that the top elements of $D_1$ and $D_2$ are accepted by the algorithm can be demonstrated as follows. Suppose, for example, that $D_1$ is rejected by the algorithm. Starting from $C_i$, 
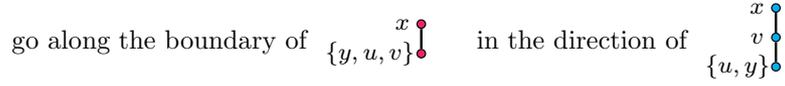
\begin{center}
\raisebox{1em}{go along the boundary of} \resizebox{2.75cm}{!}{\begin{tikzpicture}
\node(k) at (1.1,0) {};
\node (c) [circle,fill=WildStrawberry,draw=black,minimum size=0.1cm,inner sep=0.2mm,label={[xshift=-0.22cm,yshift=-0.25cm]{\footnotesize $x$}}] at (0,0.35) {};
 \node (d) [circle,fill=WildStrawberry,draw=black,minimum size=0.1cm,inner sep=0.2mm,label={[xshift=-0.6cm,yshift=-0.32cm]{\footnotesize $\{y,u,v\}$}}] at (0,0) {};
 \draw[thick]  (c)--(d);\end{tikzpicture}}\hspace{-0.75cm} \raisebox{1em}{in the direction of} \raisebox{-0.5em}{\resizebox{!}{1.253cm}{\begin{tikzpicture}
\node (b) [circle,fill=cyan,draw=black,minimum size=0.1cm,inner sep=0.2mm,label={[xshift=-0.22cm,yshift=-0.25cm]{\footnotesize $x$}}] at (0,0.7) {};
\node (c) [circle,fill=cyan,draw=black,minimum size=0.1cm,inner sep=0.2mm,label={[xshift=-0.22cm,yshift=-0.26cm]{\footnotesize $v$}}] at (0,0.35) {};
 \node (d) [circle,fill=cyan,draw=black,minimum size=0.1cm,inner sep=0.2mm,label={[xshift=-0.45cm,yshift=-0.32cm]{\footnotesize $\{u,y\}$}}] at (0,0) {};
 \draw[thick]  (b)--(c)--(d);\end{tikzpicture}}}
\end{center}
 making a zig-zag of vertices and edges, as long as the corresponding  edges  occur  in $L$. Suppose that \begin{center}\resizebox{!}{1.253cm}
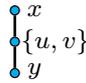
{\begin{tikzpicture}
\node (b) [circle,fill=cyan,draw=black,minimum size=0.1cm,inner sep=0.2mm,label={[xshift=0.22cm,yshift=-0.25cm]{\footnotesize $y$}}] at (0,0) {};
\node (c) [circle,fill=cyan,draw=black,minimum size=0.1cm,inner sep=0.2mm,label={[xshift=0.22cm,yshift=-0.25cm]{\footnotesize $x$}}] at (0,0.7) {};
 \node (d) [circle,fill=cyan,draw=black,minimum size=0.1cm,inner sep=0.2mm,label={[xshift=0.45cm,yshift=-0.32cm]{\footnotesize $\{u,v\}$}}] at (0,0.35) {};
 \draw[thick] (b)--(d)--(c);\end{tikzpicture}}\end{center}
is the last edge of  \raisebox{-1em}{\resizebox{2.75cm}{!}{\begin{tikzpicture}
\node(k) at (1.1,0) {};
\node (c) [circle,fill=WildStrawberry,draw=black,minimum size=0.1cm,inner sep=0.2mm,label={[xshift=-0.22cm,yshift=-0.25cm]{\footnotesize $x$}}] at (0,0.35) {};
 \node (d) [circle,fill=WildStrawberry,draw=black,minimum size=0.1cm,inner sep=0.2mm,label={[xshift=-0.6cm,yshift=-0.32cm]{\footnotesize $\{y,u,v\}$}}] at (0,0) {};
 \draw[thick]  (c)--(d);\end{tikzpicture}}} \hspace{-0.9cm} that occurs in $L$. Since $L$ is a cycle, in order for the vertex \begin{center}
\resizebox{1cm}{!}{\begin{tikzpicture}
\node (b) [circle,fill=ForestGreen,draw=black,minimum size=0.1cm,inner sep=0.2mm,label={[xshift=-0.22cm,yshift=-0.27cm]{\footnotesize $y$}}] at (0,0) {};
\node (c) [circle,fill=ForestGreen,draw=black,minimum size=0.1cm,inner sep=0.2mm,label={[xshift=-0.22cm,yshift=-0.25cm]{\footnotesize $v$}}] at (0,0.35) {};
 \node (d) [circle,fill=ForestGreen,draw=black,minimum size=0.1cm,inner sep=0.2mm,label={[xshift=-0.22cm,yshift=-0.25cm]{\footnotesize $u$}}] at (0,0.7) {};
  \node (a) [circle,fill=ForestGreen,draw=black,minimum size=0.1cm,inner sep=0.2mm,label={[xshift=-0.22cm,yshift=-0.25cm]{\footnotesize $x$}}] at (0,1.05) {};
\node(k) at (0.35,0) {};
 \draw[thick]  (b)--(c)--(d)--(a);\end{tikzpicture}}\end{center}  to cancel out, the edge \raisebox{-1.5em}{\resizebox{1.2cm}{!}{\begin{tikzpicture}
\node (b) [circle,fill=cyan,draw=black,minimum size=0.1cm,inner sep=0.2mm,label={[xshift=-0.22cm,yshift=-0.27cm]{\footnotesize $y$}}] at (0,0) {};
\node (c) [circle,fill=cyan,draw=black,minimum size=0.1cm,inner sep=0.2mm,label={[xshift=-0.22cm,yshift=-0.25cm]{\footnotesize $v$}}] at (0,0.35) {};
 \node (d) [circle,fill=cyan,draw=black,minimum size=0.1cm,inner sep=0.2mm,label={[xshift=-0.45cm,yshift=-0.32cm]{\footnotesize $\{u,x\}$}}] at (0,0.7) {};
 \draw[thick]  (b)--(c)--(d);\end{tikzpicture}}}\, must occur in $L$. We then have two possible cases, both leading to a contradiction.
\begin{itemize}
\item If the top hexagon \raisebox{-0.7em}{\resizebox{2.95cm}{!}{\begin{tikzpicture}
\node(k) at (1.1,0) {};
\node (c) [circle,fill=WildStrawberry,draw=black,minimum size=0.1cm,inner sep=0.2mm,label={[xshift=-0.6cm,yshift=-0.32cm]{\footnotesize $\{x,u,v\}$}}] at (0,0.35) {};
 \node (d) [circle,fill=WildStrawberry,draw=black,minimum size=0.1cm,inner sep=0.2mm,label={[xshift=-0.22cm,yshift=-0.27cm]{\footnotesize $y$}}] at (0,0) {};
 \draw[thick]  (c)--(d);\end{tikzpicture}}} \hspace{-1cm} is accepted by the algorithm, then the square \raisebox{-1em}{\resizebox{1.2cm}{!}{\begin{tikzpicture}
\node (c) [circle,fill=WildStrawberry,draw=black,minimum size=0.1cm,inner sep=0.2mm,label={[xshift=-0.45cm,yshift=-0.32cm]{\footnotesize $\{x,v\}$}}] at (0,0.35) {};
 \node (d) [circle,fill=WildStrawberry,draw=black,minimum size=0.1cm,inner sep=0.2mm,label={[xshift=-0.45cm,yshift=-0.32cm]{\footnotesize $\{u,y\}$}}] at (0,0) {};
 \draw[thick]  (c)--(d);\end{tikzpicture}}} on the right must also be accepted, since, otherwise, the vertex 
\begin{center}
\resizebox{1cm}{!}{\begin{tikzpicture}
\node (b) [circle,fill=ForestGreen,draw=black,minimum size=0.1cm,inner sep=0.2mm,label={[xshift=-0.22cm,yshift=-0.27cm]{\footnotesize $y$}}] at (0,0) {};
\node (c) [circle,fill=ForestGreen,draw=black,minimum size=0.1cm,inner sep=0.2mm,label={[xshift=-0.22cm,yshift=-0.25cm]{\footnotesize $u$}}] at (0,0.35) {};
 \node (d) [circle,fill=ForestGreen,draw=black,minimum size=0.1cm,inner sep=0.2mm,label={[xshift=-0.22cm,yshift=-0.25cm]{\footnotesize $v$}}] at (0,0.7) {};
  \node (a) [circle,fill=ForestGreen,draw=black,minimum size=0.1cm,inner sep=0.2mm,label={[xshift=-0.22cm,yshift=-0.25cm]{\footnotesize $x$}}] at (0,1.05) {};
\node(k) at (0.35,0) {};
 \draw[thick]  (b)--(c)--(d)--(a);\end{tikzpicture}}\end{center} would not cancel out. Consequently, in order for the vertex \begin{center}
\resizebox{1cm}{!}{\begin{tikzpicture}
\node (b) [circle,fill=ForestGreen,draw=black,minimum size=0.1cm,inner sep=0.2mm,label={[xshift=-0.22cm,yshift=-0.25cm]{\footnotesize $u$}}] at (0,0) {};
\node (c) [circle,fill=ForestGreen,draw=black,minimum size=0.1cm,inner sep=0.2mm,label={[xshift=-0.22cm,yshift=-0.27cm]{\footnotesize $y$}}] at (0,0.35) {};
 \node (d) [circle,fill=ForestGreen,draw=black,minimum size=0.1cm,inner sep=0.2mm,label={[xshift=-0.22cm,yshift=-0.25cm]{\footnotesize $v$}}] at (0,0.7) {};
  \node (a) [circle,fill=ForestGreen,draw=black,minimum size=0.1cm,inner sep=0.2mm,label={[xshift=-0.22cm,yshift=-0.25cm]{\footnotesize $x$}}] at (0,1.05) {};
\node(k) at (0.35,0) {};
 \draw[thick]  (b)--(c)--(d)--(a);\end{tikzpicture}}
\end{center} to cancel out, the pentagon \raisebox{-1em}{\resizebox{1.6cm}{!}{\begin{tikzpicture}
\node (c) [circle,fill=WildStrawberry,draw=black,minimum size=0.1cm,inner sep=0.2mm,label={[xshift=0.6cm,yshift=-0.32cm]{\footnotesize $\{x,y,v\}$}}] at (0,0.35) {};
 \node (d) [circle,fill=WildStrawberry,draw=black,minimum size=0.1cm,inner sep=0.2mm,label={[xshift=0.22cm,yshift=-0.25cm]{\footnotesize $u$}}] at (0,0) {};
 \draw[thick]  (c)--(d);\end{tikzpicture}}} must be accepted by the algorithm.  
\item By the analogous agruments, one concludes that, if the top hexagon is not accepted by the algorithm, then   the square \raisebox{-1em}{\resizebox{1.2cm}{!}{\begin{tikzpicture}
\node (c) [circle,fill=WildStrawberry,draw=black,minimum size=0.1cm,inner sep=0.2mm,label={[xshift=-0.45cm,yshift=-0.32cm]{\footnotesize $\{x,u\}$}}] at (0,0.35) {};
 \node (d) [circle,fill=WildStrawberry,draw=black,minimum size=0.1cm,inner sep=0.2mm,label={[xshift=-0.45cm,yshift=-0.32cm]{\footnotesize $\{v,y\}$}}] at (0,0) {};
 \draw[thick]  (c)--(d);\end{tikzpicture}}} and the pentagon \raisebox{-1em}{\resizebox{1.6cm}{!}{\begin{tikzpicture}
\node (c) [circle,fill=WildStrawberry,draw=black,minimum size=0.1cm,inner sep=0.2mm,label={[xshift=-0.6cm,yshift=-0.32cm]{\footnotesize $\{x,y,u\}$}}] at (0,0.35) {};
 \node (d) [circle,fill=WildStrawberry,draw=black,minimum size=0.1cm,inner sep=0.2mm,label={[xshift=-0.22cm,yshift=-0.25cm]{\footnotesize $v$}}] at (0,0) {};
 \draw[thick]  (c)--(d);\end{tikzpicture}}} on the left must be accepted.
\end{itemize}

Finally, the fact that all the summands of  $d_{{\EuScript O}_{\infty}}^{m+1}(n_1,\dots,n_k;n)({\cal G}(L))$ occur in $L$, i.e. that, if the top element of a diamond \eqref{diamond} is accepted in ${\cal G}(L)$, and if $C_{ij}$ contains a face $C$ that does not appear in $L$, then there exists another diamond, whose top element is also accepted in ${\cal G}(L)$ and   has $C$ as a face,  and in  which $C$ occurs with the opposite sign  than the one of its copy in  \eqref{diamond}, follows simply by the construction of ${\cal G}(L)$. Indeed, if such a diamond would not exist, then $C_{ij}$ would be removed in the correction step of the algorithm.\end{proof}
\smallskip
 For $m=0$, we have that  ${\mbox{Ker}}\, d^{0}_{{\EuScript O}_{\infty}}(n_1,\dots,n_k;n)={\EuScript O}^{0}_{\infty}(n_1,\dots,n_k;n)$,
since constructions   have no vertices that could be split. As for the image of   $d_{{\EuScript O}_{\infty}}^{1}(n_1,\dots,n_k;n)$, we have
\smallskip
$${\mbox{Im}}\, d_{{\EuScript O}_{\infty}}^{1}(n_1,\dots,n_k;n)={\textsf{Span}}_{{\Bbbk}}\Bigg(\displaystyle \bigoplus_{({\cal T},\sigma)\in {\tt {Tree}}(n_1,\dots,n_k;n)}\, \bigoplus_{C:{\bf H}_{\cal T}, |C|=1} \pm C[x\{y\}/\{x,y\}] \pm C[y\{x\}/\{x,y\}] \Bigg),$$ \smallskip
where $\{x,y\}$ is the unique two-element vertex of $C$ and the signs are determined 
by the position of the vertices $x$ and $y$ in ${\bf H}_{\cal T}$ (considered as the edge-graph with levels), using the criterion from \ref{edgegraph1}: if the shortest path between $x$ and $y$ is made of vertical edges only, then one of the vertices $x$ and $y$   is above the other and the   construction which respects this position gets multiplied by $-$, and the other one by $+$; otherwise, both  constructions   get the sign $+$. By collapsing ${\mbox{Im}}\, d_{1}$ to zero, for each ${\cal T}\in {\EuScript O}(n_1,\dots,n_k;n)$,  all the constructions of ${\bf H}_{\cal T}$  get glued into a single equivalence class; in particular, different trees from ${\EuScript O}(n_1,\dots,n_k;n)$ give rise to different classes. Therefore, 
\smallskip $${\mbox{Ker}}\, d^m_{{\EuScript O}_{\infty}}(n_1,\dots,n_k;n)/{\mbox{Im}}\, d^{m+1}_{{\EuScript O}_{\infty}}(n_1,\dots,n_k;n)=\begin{cases}
\{0\}, \enspace m\neq 0\\
{\EuScript O}(n_1,\dots,n_k;n), \enspace m= 0.
\end{cases}$$
\smallskip
which entails that $H({\EuScript O}_{\infty},d_{{\EuScript O}_{\infty}})\cong H({\EuScript O},0)$. The witnessing  quasi-isomorphism $\alpha_{\EuScript O}:{\EuScript O}_{\infty}\rightarrow {\EuScript O}$    is  simply  the first projection on degree zero, and the zero map elsewhere. 

\medskip

Having established the freeness of ${\EuScript O}_{\infty}$, the decomposability of $d_{{\EuScript O}_{\infty}}$ and the quasi-isomorphism with ${\EuScript O}$, we can now finally conclude.
\smallskip
\begin{thm}\label{Omin}
The operad ${\EuScript O}_{\infty}$ is the minimal model for the operad ${\EuScript O}$.
\end{thm}
\smallskip
\begin{remark}
By employing the elements of Koszul duality for coloured operads, Theorem \ref{Omin} can be established in a more economical way. Recall that the ${\EuScript O}_{\infty}$ operad is originally defined by ${\EuScript O}_{\infty}:=\Omega{\EuScript O}^{\ac}$, i.e., by  $${\EuScript O}_{\infty}:=({\cal T}_{\mathbb N}(s^{-1} \overline{{\EuScript O}^{\ac}}),d),$$ where   $s^{-1} \overline{{\EuScript O}^{\ac}}$ is the desuspension of the coaugmentation coideal of ${\EuScript O}^{\ac}:={\cal T}^c_{\mathbb N}(sE)/ (s^2 R)$, where $E$ and $R$ are as in Definition \ref{1} and $s$ denotes the degree $+1$ suspension shift, and  $d$ is the unique derivation extending the cooperad structure on  ${\EuScript O}^{\ac}$. Since ${\EuScript O}$ is self dual (cf. \cite[Theorem 4.3]{VdL}),   the operations of ${\EuScript O}_{\infty}$ are given by ``composite trees of operadic trees'', and the differential of ${\EuScript O}_{\infty}$ is given by the ``degrafting of composite trees of operadic trees'' (cf. Theorem \ref{free}). As explained in \cite[Section 5.1]{DCV} for homotopy cooperads,   the spaces of such operations  can be realized as lattices of {\em nested left-recursive operadic trees} under the partial order given by {\em refinement of  nestings}, i.e. by a subfamily of the family of graph associahedra of   Carr-Devadoss introduced in \cite{CD}.  Therefore, Theorem \ref{Omin} follows by showing that, for a fixed left-recursive operadic tree ${\cal T}$, the lattice ${\cal A}({\bf H}_{\cal T})$ is isomorphic with the   lattice of nestings of ${\cal T}$. This, in turn, is a direct consequence of \cite[Proposition 2]{CIO}.  
\end{remark}
\smallskip
\subsubsection{Stasheff's  associahedra as a suboperad of ${\EuScript O}_{\infty}$} The earliest example of an explicit description of the minimal model of a dg operad, predating even  the notions of  operad and minimal model themselves, is given by the dg $A_{\infty}$-operad, the minimal model of the non-symmetric operad ${\it As}$ for associative algebras, often described in terms of Stasheff's associahedra \cite{Sta1}.

\medskip 

 Recall that the non-symmetric associative operad ${\it As}$, encoding the category of non-unital associative algebras, is defined by ${\it As}(n)={\textsf{Span}}_{{\Bbbk}}(\{t_n\})$, for $n\geq 2$, where $t_n$ the isomorphism class of a planar corolla with $n$ inputs.   The one-dimensional space ${\it As}(n)$ is concentrated in degree is zero and the differential is trivial.

\smallskip
In the  standard  dg framework, the $A_{\infty}$-operad is the  quasi-free  dg operad $A_{\infty}={\cal T}(\bigoplus_{n\geq 2}{\textsf{Span}}_{{\Bbbk}}(\{t_n\}))$, where $|t_n|=n-2$, with the differential given by $$d(t_n)=\displaystyle\sum_{\substack{n=p+q+r\\ k=p+r+1\\ k,q\geq 2}} (-1)^{p} \,\, t_k \circ_{p+1} t_q.$$

The dg $A_{\infty}$-operad is the minimal model for ${\it As}$. Indeed, the map $\alpha_{\it As}:A_{\infty}\rightarrow {\it As}$, defined as the identity on $t_2$ and as the zero map elsewhere, induces a homology  isomorphism  $H_{\bullet}(A_{\infty},d)\cong {\it As}$, whereas $d$ is clearly defined in terms of decomposable elements of $A_{\infty}$.

\smallskip

Let ${\cal A}_{\infty}$ be the suboperad of the ${\EuScript O}_{\infty}$ operad determined by linear  operadic trees, i.e. operadic trees ${\cal T}$ with univalent vertices such that  a vertex $i$ is always adjacent to the vertex $i-1$ (and not to a vertex $i-j$, for some $j>1$). Observe that the univalency requirement ensures that linear operadic trees are closed under the operation of substitution of trees. 

\begin{thm}
The operad ${\cal A}_{\infty}$ is isomorphic to the $A_{\infty}$-operad, and, therefore, it is the minimal model for the operad ${\it As}$.
\end{thm}
 \begin{proof}
The restriction to linear operadic trees that defines ${\cal A}_{\infty}$ collapses the set of colours ${\mathbb N}$ of the operad ${\EuScript O}$ to the singleton set $\{1\}$, making therefore ${\cal A}_{\infty}$ a monochrome operad. The conclusion follows from the correspondence between the construct description of associahedra and the  standard description underlying the definition of the $A_{\infty}$-operad, established in \S \ref{assoc}.
\end{proof}

\smallskip

\begin{remark}  Let $K_n$, for $n\geq 2$, denote the $(n-2)$-dimensional   associahedron, i.e. a CW complex
whose cells   of dimension $k$   are in bijection with  rooted planar trees having
$n$ leaves and $n-k-1$ vertices.
 The sequence ${\cal K}=\{K_n\}_{n\geq 2}$ is naturally endowed with the structure of a non-symmetric topological operad: the composition $$\circ_i: K_r \times K_s\rightarrow K_{r+s-1}$$ is defined as follows: for faces $k_1\in K_{r}$ and $k_2\in K_s$, $\circ_i(k_1,k_2)$ is the face of $K_{r+s-1}$ obtained by grafting the tree encoding the face $k_2$ to the leaf $i$ of the tree encoding the face $k_1$.  The topological operad ${\cal K}$ is turned into the dg  $A_{\infty}$-operad  by taking cellular chain complexes on ${\cal K}$; we refer to \cite[Proposition 9.2.4.]{LV} for the details of this  transition. 
 
 \smallskip
 
For what concerns geometric realizations of the associahedra, Stasheff initially considered the CW complexes  $K_n$  as curvilinear polytopes. They were later given coordinates as convex polytopes in Euclidean spaces in different ways:  as convex hulls of points, as  intersections of half-spaces, and by truncations of standard simplices.  However, it was only in the recent paper \cite{mttv} of Masuda,  Thomas, Tonks and Vallette that  a non-symmetric operad structure on a family of convex polytopal realizations of the associahedra has been introduced. This raises the question of an appropriate ``geometric realization'' of the ${\EuScript O}_{\infty}$  operad,  which we leave to a future work.
\end{remark}

 \smallskip
\subsection{The combinatorial Boardman-Vogt-Berger-Moerdijk resolution of ${\EuScript O}$}\label{Wconstruction}
\noindent In \cite{BM2}, Berger and Moerdijk constructed a cofibrant resolution for coloured operads in arbitrary monoidal model categories, by generalizing  the Boardman-Vogt $W$-construction for topological operads \cite{BV}.  In this section, by introducing a cubical subdivision of the faces of operadic polytopes, we define the operad ${\EuScript O}^{\circ}_{\infty}$, which is precisely the $W$-construction applied on ${\EuScript O}$.

 \medskip

In order to provide the combinatorial description of the  $W$-construction of ${\EuScript O}$, we  are going to generalize the notion of a construct of the edge-graph ${\bf H}_{\cal T}$ of an operadic tree ${\cal T}$, to the notion of a circled construct of ${\bf H}_{\cal T}$. A circled construct  should be thought of as a  two-level construct,  i.e. a construct whose vertices are constructs themselves; the idea is that circles determine those ``higher'' vertices, in  the same way as in the definition of the  monad of trees. The circles that we  add to a construct of ${\bf H}_{\cal T}$  arise from decompositions of ${\cal T}$, and themselves determine a   decomposition  of that construct -- this construction is dual to the one defining the isomorphism $\alpha$ in the  proof of Theorem \ref{free}.

\smallskip

\begin{definition}\label{dec}  Let $({\cal T},\sigma)\in {\EuScript O}(n_1,\dots,n_k;n)$. The set of {\em circled constructs} of the hypergraph ${\bf H}_{\cal T}$ is generated by the following two rules.
\begin{itemize}
\item For each (ordinary) construct $C:{\bf H}_{\cal T}$,  the construct $C$ together with a single circle that entirely surrounds it, is a circled construct of ${\bf H}_{\cal T}$.
\item If $({\cal T},\sigma)=({\cal T}_1,\sigma_1)\bullet_i ({\cal T}_2,\sigma_2)$ and if $C_1$ and $C_2$ are circled constructs of ${\bf H}_{{\cal T}_1}$ and ${\bf H}_{{\cal T}_2}$, respectively, then the construct $C_1\bullet_i C_2:{\bf H}_{\cal T}$, determined uniquely by the composition $\circ_i$ of ${\EuScript O}_{\infty}$,  together with all the  circles  of $C_1$ and $C_2$ (and no other circle), is a circled construct of ${\bf H}_{\cal T}$.
\end{itemize}
\end{definition}
 In what follows, we shall write $C^{\circ}:{\bf H}_{\cal T}$ to denote that $C^{\circ}$ is a circled construct of ${\bf H}_{\cal T}$; if $C^{\circ}:{\bf H}_{\cal T}$, we shall denote with    $C$    the ordinary construct of ${\bf H}_{\cal T}$ obtained by forgetting the circles of $C^{\circ }$. We shall write $C_p^{\circ }$ to indicate that the circled construct $C^{\circ}$ has $p$ circles. Denote with ${A}^{\circ}({\bf H}_{\cal T})$ the set of all circled constructs of ${\bf H}_{\cal T}$.

\begin{remark}
Note that the set of edges of a circled construct $C^{\circ}:{\bf H}_{\cal T}$ can be decomposed  into two disjoint subsets: the subset of {\em circled edges}, i.e. of the edges that lie within a circle, and the subset of {\em connecting edges}, i.e. of the edges  that connect two adjacent circles.  The disjointness is ensured by the fact that  Definition \ref{dec} disallowes nested circles.
\end{remark}
  
\begin{remark}
Circled constructs are generalizations of {\em circled trees}, in the terminology of \cite[Appendix C.2.3.]{LV}. 
\end{remark}

Let, for $k\geq 2$, $n_1,\dots,n_k\geq 1$, and $n=\bigl(\sum^k_{i=1} n_i\bigr)-k+1$, ${\EuScript O}^{\circ}_{\infty}(n_1,\dots,n_k;n)$ be the vector space spanned by triples $({\cal T},\sigma,C^{\circ})$, where $({\cal T},\sigma)\in {\tt{Tree}}(n_1,\dots,n_k;n)$ and $C^{\circ}:{\bf H}_{\cal T}$, subject to the equivalence relation generated by:
\begin{quote}
$({\cal T}_{1},\sigma_1,C^{\circ}_1)\sim ({\cal T}_{2},\sigma_2,C^{\circ}_2)$ if there exists an isomorphism $\varphi:{\cal T}_1\rightarrow {\cal T}_2$, such that  $\varphi\circ\sigma_1=\sigma_2$, $C_{1}\sim_{\varphi} C_{2}$, and such that, modulo the renaming $\varphi$, the circles of $C^{\circ}_1$ are exactly the circles of $C^{\circ}_2$.
\end{quote}
  
%
\noindent Therefore, 

 $${\EuScript O}^{\circ}_{\infty}(n_1,\dots,n_k;n)={\textsf{Span}}_{{\Bbbk}} \Bigg(\bigoplus_{({\cal T},\sigma)\in {\tt{Tree}}(n_1,\dots,n_k;n)}  {{A}^{\circ}({\bf H}_{\cal T})} \Bigg).$$

\noindent The space ${\EuScript O}^{\circ}_{\infty}(n_1,\dots,n_k;n)$ is graded by $|({\cal T},C_p^{\circ})|=|v(C)|-p$. 

\medskip

 The ${\mathbb N}$-coloured graded collection $$\{{\EuScript O}^{\circ}_{\infty}(n_1,\dots,n_k;n)\,|\, n_1,\dots,n_k\geq 1\}$$ admits the following operad structure: the composition operation $$\circ_i : {\EuScript O}^{\circ}_{\infty}\displaystyle(n_1,\dots,n_k;n) \otimes {\EuScript O}^{\circ}_{\infty}\displaystyle(m_1,\dots,m_l;n_i)\rightarrow {\EuScript O}^{\circ}_{\infty}\displaystyle(n_1,\dots,n_{i-1},m_1,\dots,m_l,n_{i+1},\dots ,n_k;n)$$
 is defined by
 $$({\cal T}_1,\sigma_1,C_1^{\circ})\circ_i({\cal T}_2,\sigma_2,C_2^{\circ})=(-1)^{\varepsilon}({\cal T}_1\bullet_i {\cal T}_2,{\sigma}_1\bullet_i {\sigma}_2,C^{\circ}_1\bullet_i C^{\circ}_2),$$
where  $C^{\circ}_1\bullet_i C^{\circ}_2$  is defined by the second rule of Definition \ref{dec}, and the sign $(-1)^{\varepsilon}$ is determined as follows. Observe that, for an operation $({\cal T},\sigma,C^{\circ})$, the circles of $C^{\circ}$ carry over to $\alpha({\cal T},\sigma,C)$, where $\alpha$ is the isomorphism from the proof of Theorem $\ref{free}$. This decomposes the set of edges of $\alpha({\cal T},\sigma,C)$ into the set of circled edges and the set of connecting edges. Then, $\varepsilon$ is the number of connecting edges and leaves of $\alpha({\cal T}_1,\sigma_1,C_1)$ on the right from the leaf indexed by $i$, multiplied by the number of all connecting edges and leaves of $\alpha({\cal T}_2,\sigma_2,C_2)$, minus the root. Therefore,    the sign   is calculated in the analogous way as for the partial composition of ${\EuScript O}_{\infty}$, save that the  vertices of operations of ${\EuScript O}_{\infty}$ are identified with the circles of operations of ${\EuScript O}^{\circ}_{\infty}$.

\medskip

Finally, the derivative $d^{\circ}$ of ${\EuScript O}^{\circ}_{\infty}$ will be the difference $d^{\circ}_1-d^{\circ}_0$ of two derivatives. The derivative $d^{\circ}_1$   acts on  $({\cal T},\sigma,C^{\circ})$ by turning   circled edges of $C^{\circ}$  into connecting edges, by splitting the circles in two. The associated signs are determined as follows. Let us fix a summand in $d^{\circ}_1({\cal T},\sigma,C^{\circ})$. Suppose that  $C'$ is the construct surrounded by the circle of $C^{\circ}$ that we split in two,  let $C'_1$ and $C'_2$ be the constructs surrounded by the two resulting circles in the summand, and let $X$ and $Y$ be the unions of the sets decorating the vertices of $C'_1$ and $C'_2$  (i.e. $X$ and $Y$ are the sets obtained by collapsing all the edges of $C'_1$ and $C'_2$), respectively. Suppose, without loss of generality,  that the circle surrounding $C'_1$ is below the circle surrounding $C'_2$. The sign of the resulting summand is given by $(-1)^{\delta+e(C'_1)+e(C'_2)}$, where $(-1)^{\delta}$ is the  sign of the summand of $d_{{\EuScript O}_{\infty}}({\cal T},\sigma,C[(X\cup Y)/X\{Y\}])$  whose new edge is determined  by $X\{Y\}$. Here, $C[(X\cup Y)/X\{Y\}]$ denotes the construct obtained from $C$ by collapsing the edge $X\{Y\}$.  Observe that  the component  $d^{\circ}_1$ does not affect the overall configuration of edges of $C$ as a plain construct.   The component  $d^{\circ}_0$ collapses circled edges; each resulting summand $({\cal T},\sigma,C[(U\cup V)/U\{V\}]^{\circ})$ will be multiplied by the  sign that  $({\cal T},\sigma,C)$   gets as a summand of    $d_{{\EuScript O}_{\infty}}({\cal T},\sigma,C[(U\cup V)/U\{V\}])$. The component $d^{\circ}_0$ does not affect circles. One might think of $d^{\circ}_1$ as a coercion of the derivative $d_{{\EuScript O}_{\infty}}$ to ${\EuScript O}^{\circ}_{\infty}$, by the identification of the vertices of the operations of ${\EuScript O}$ with the circles of the operations of ${\EuScript O}^{\circ}_{\infty}$, since both $d^{\circ}_1$ and $d_{{\EuScript O}_{\infty}}$ split the corresponding entity. On the other hand, $d^{\circ}_0$ can be seen as the inverse of $d_{{\EuScript O}_{\infty}}$: $d^{\circ}_0$ collapses the edges, whereas $d_{{\EuScript O}_{\infty}}$ creates new edges by splitting the vertices. Therefore, intuitively, $d^{\circ}_1$ acts {\em globally}, and $d^{\circ}_0$ {\em locally}. The sum $e(C'_1)+e(C'_2)$ in the sign $(-1)^{\delta+e(C'_1)+e(C'_2)}$ pertaining to $d_1^{\circ}$ has to be added in order  to ensure that   $d^{\circ}$   squares to zero. More precisely, it is needed for the pairs of summands   in $(d^{\circ})^2({\cal T},\sigma,C^{\circ})$ that should be cancelled out and  which   arise  by applying $d^{\circ}_0$ and  $d^{\circ}_1$ in the opposite order. As a consequence of Lemma \ref{dok},   $d^{\circ}$    agrees with the composition of ${\EuScript O}^{\circ}_{\infty}$.

  \medskip In the following two examples, we describe two cubical decompositions of operadic polytopes, given by the appropriate posets of circled constructs, where, like it was the case for ordinary constructs, the  partial order is induced by the action of the differential  $d^{\circ}$.

\begin{example} For the linear tree 
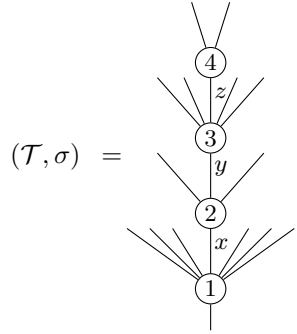
\begin{center}
\begin{tikzpicture}
  \node (t1)[circle,draw=none] at (-1.9,1.75) { $({\cal T},\sigma)\enspace=$};
    \node (E)[circle,draw=black,minimum size=4mm,inner sep=0.1mm] at (0,0) {\small $1$};
    \node (F) [circle,draw=black,minimum size=4mm,inner sep=0.1mm] at (0,1) {\small $2$};
    \node (A) [circle,draw=black,minimum size=4mm,inner sep=0.1mm] at (0,2) {\small $3$};
\node (D) [circle,draw=black,minimum size=4mm,inner sep=0.1mm] at (0,3) {\small $4$};
 \draw[-] (0.7,2.8) -- (A)--(-0.7,2.8); 
  \draw[-] (0.35,2.8) -- (A)--(-0.35,2.8); 
\draw[-] (0.25,3.8) -- (D)--(-0.25,3.8); 
 \draw[-] (1.1,0.8) -- (E)--(-1.1,0.8); 
  \draw[-] (0.8,0.8) -- (E)--(-0.8,0.8); 
   \draw[-] (0.5,0.8) -- (E)--(-0.5,0.8); 
 \draw[-] (E)--(0,-0.55); 
 \draw[-] (0.7,1.8) -- (F)--(-0.7,1.8); 
    \draw[-] (E)--(F) node [midway,right,xshift=-0.075cm,yshift=0.1cm] {\small $x$};
    \draw[-] (F)--(A) node  [midway,right,xshift=-0.075cm,yshift=0.1cm] {\small $y$};
\draw[-] (A)--(D) node  [midway,right,xshift=-0.075cm,yshift=0.1cm] {\small  $z$};
\end{tikzpicture} 
\end{center}
by taking all the circled  constructs of ${\bf H}_{\cal T}$, we recover a familiar  picture of the cubical decomposition of the pentagon from \cite[Appendix C.2.3.]{LV}, in which we fully labeled only the faces of the shaded square:
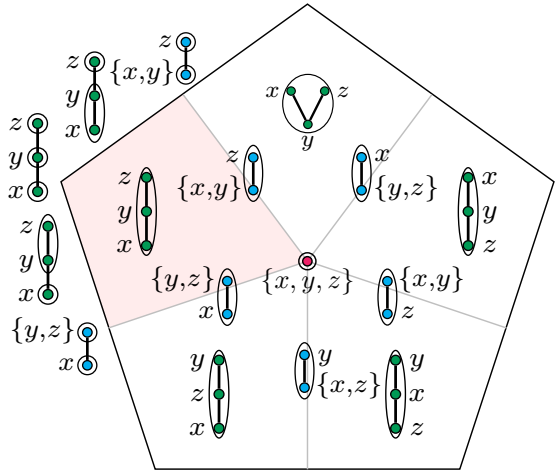
\begin{center}
\resizebox{7.5cm}{!}{\begin{tikzpicture}
\draw[draw=none,fill=pink!30] (-2.37,0.77)--(-1.2,1.63)--(0,0)--(-1.92,-0.65)--cycle;
\node[draw=black,minimum size=5cm,regular polygon,regular polygon sides=5] (a) at (0,0) {};
\draw[lightgray] (-1.92,-0.65)--(0,0)--(0,-2.025);
\draw[lightgray] (1.92,-0.65)--(0,0);
\draw[lightgray] (-1.2,1.62)--(0,0)--(1.2,1.62);
 \node (E) at (0,1.4)  {\resizebox{1cm}{!}{\begin{tikzpicture}
\draw (4,0) ellipse (0.3cm and 0.35cm);
\node (L) [circle,fill=ForestGreen,draw=black,minimum size=0.1cm,inner sep=0.2mm,label={[yshift=-0.5cm]{\footnotesize $y$}}] at (4,-0.25) {};
\node (B1) [circle,fill=ForestGreen,draw=black,minimum size=0.1cm,inner sep=0.2mm,label={[xshift=-0.22cm,yshift=-0.25cm]{\footnotesize $x$}}] at (3.8,0.15) {};
 \node (B4) [circle,fill=ForestGreen,draw=black,minimum size=0.1cm,inner sep=0.2mm,label={[xshift=0.22cm,yshift=-0.25cm]{\footnotesize $z$}}] at (4.2,0.15) {};
 \draw[thick]  (L)--(B1);
\draw [thick] (L)--(B4);\end{tikzpicture}}};
 \node (E1) at (-1.7,0.475) {\resizebox{0.5cm}{!}{\begin{tikzpicture}
\draw (0,0.35) ellipse (0.1cm and 0.45cm);
\node (L) [circle,fill=ForestGreen,draw=black,minimum size=0.1cm,inner sep=0.2mm,label={[xshift=-0.22cm,yshift=-0.25cm]{\footnotesize $x$}}] at (0,0) {};
\node (B1) [circle,fill=ForestGreen,draw=black,minimum size=0.1cm,inner sep=0.2mm,label={[xshift=-0.22cm,yshift=-0.3cm]{\footnotesize $y$}}] at (0,0.35) {};
 \node (B4) [circle,fill=ForestGreen,draw=black,minimum size=0.1cm,inner sep=0.2mm,label={[xshift=-0.22cm,yshift=-0.25cm]{\footnotesize $z$}}] at (0,0.7) {};
 \draw[thick]  (L)--(B1)--(B4);
\end{tikzpicture}}};
  \node (P41) at (1.7,0.475) {\resizebox{0.5cm}{!}{\begin{tikzpicture}
\draw (0,0.35) ellipse (0.1cm and 0.45cm);
\node (L) [circle,fill=ForestGreen,draw=black,minimum size=0.1cm,inner sep=0.2mm,label={[xshift=0.22cm,yshift=-0.25cm]{\footnotesize $z$}}] at (0,0) {};
\node (B1) [circle,fill=ForestGreen,draw=black,minimum size=0.1cm,inner sep=0.2mm,label={[xshift=0.22cm,yshift=-0.3cm]{\footnotesize $y$}}] at (0,0.35) {};
 \node (B4) [circle,fill=ForestGreen,draw=black,minimum size=0.1cm,inner sep=0.2mm,label={[xshift=0.22cm,yshift=-0.25cm]{\footnotesize $x$}}] at (0,0.7) {};
 \draw[thick]  (L)--(B1)--(B4);
\end{tikzpicture}}};
\node (Asubt1) at (1,-1.3) {\resizebox{0.5cm}{!}{\begin{tikzpicture}
\draw (0,0.35) ellipse (0.1cm and 0.45cm);
\node (L) [circle,fill=ForestGreen,draw=black,minimum size=0.1cm,inner sep=0.2mm,label={[xshift=0.22cm,yshift=-0.25cm]{\footnotesize $z$}}] at (0,0) {};
\node (B1) [circle,fill=ForestGreen,draw=black,minimum size=0.1cm,inner sep=0.2mm,label={[xshift=0.22cm,yshift=-0.25cm]{\footnotesize $x$}}] at (0,0.35) {};
 \node (B4) [circle,fill=ForestGreen,draw=black,minimum size=0.1cm,inner sep=0.2mm,label={[xshift=0.22cm,yshift=-0.3cm]{\footnotesize $y$}}] at (0,0.7) {};
 \draw[thick]  (L)--(B1)--(B4);
\end{tikzpicture}}};
  \node (A1) at (-1,-1.3) {\resizebox{0.5cm}{!}{\begin{tikzpicture}
\draw (0,0.35) ellipse (0.1cm and 0.45cm);
\node (L) [circle,fill=ForestGreen,draw=black,minimum size=0.1cm,inner sep=0.2mm,label={[xshift=-0.22cm,yshift=-0.25cm]{\footnotesize $x$}}] at (0,0) {};
\node (B1) [circle,fill=ForestGreen,draw=black,minimum size=0.1cm,inner sep=0.2mm,label={[xshift=-0.22cm,yshift=-0.25cm]{\footnotesize $z$}}] at (0,0.35) {};
 \node (B4) [circle,fill=ForestGreen,draw=black,minimum size=0.1cm,inner sep=0.2mm,label={[xshift=-0.22cm,yshift=-0.3cm]{\footnotesize $y$}}] at (0,0.7) {};
 \draw[thick]  (L)--(B1)--(B4);
 \end{tikzpicture}}};
 \node (A1d) at (-0.9,0.83) {\resizebox{0.95cm}{!}{\begin{tikzpicture}
\draw (0,0.18) ellipse (0.1cm and 0.3cm);
\node (L) [circle,fill=cyan,draw=black,minimum size=0.1cm,inner sep=0.2mm,label={[xshift=-0.47cm,yshift=-0.3cm]{\footnotesize $\{x,\!y\}$}}] at (0,0) {};
\node (B1) [circle,fill=cyan,draw=black,minimum size=0.1cm,inner sep=0.2mm,label={[xshift=-0.22cm,yshift=-0.225cm]{\footnotesize $z$}}] at (0,0.35) {};
 \draw[thick]  (L)--(B1);\end{tikzpicture}}};
  \node (A1dd) at  (0.9,0.83)  {\resizebox{0.95cm}{!}{\begin{tikzpicture}
\draw (0,0.18) ellipse (0.1cm and 0.3cm);
\node (L) [circle,fill=cyan,draw=black,minimum size=0.1cm,inner sep=0.2mm,label={[xshift=0.47cm,yshift=-0.3cm]{\footnotesize $\{y,\!z\}$}}] at (0,0) {};
\node (B1) [circle,fill=cyan,draw=black,minimum size=0.1cm,inner sep=0.2mm,label={[xshift=0.22cm,yshift=-0.225cm]{\footnotesize $x$}}] at (0,0.35) {};
 \draw[thick]  (L)--(B1);\end{tikzpicture}}};
  \node (fr) at (1.15,-0.3)  {\resizebox{0.95cm}{!}{\begin{tikzpicture}
\draw (0,0.18) ellipse (0.1cm and 0.3cm);
\node (L) [circle,fill=cyan,draw=black,minimum size=0.1cm,inner sep=0.2mm,label={[xshift=0.22cm,yshift=-0.225cm]{\footnotesize $z$}}] at (0,0) {};
\node (B1) [circle,fill=cyan,draw=black,minimum size=0.1cm,inner sep=0.2mm,label={[xshift=0.47cm,yshift=-0.3cm]{\footnotesize $\{x,\!y\}$}}] at (0,0.35) {};
 \draw[thick]  (L)--(B1);\end{tikzpicture}}};
\node (A1dsddd) at (-1.15,-0.3)   {\resizebox{0.95cm}{!}{\begin{tikzpicture}
\draw (0,0.18) ellipse (0.1cm and 0.3cm);
\node (L) [circle,fill=cyan,draw=black,minimum size=0.1cm,inner sep=0.2mm,label={[xshift=-0.22cm,yshift=-0.225cm]{\footnotesize $x$}}] at (0,0) {};
\node (B1) [circle,fill=cyan,draw=black,minimum size=0.1cm,inner sep=0.2mm,label={[xshift=-0.47cm,yshift=-0.3cm]{\footnotesize $\{y,\!z\}$}}] at (0,0.35) {};
 \draw[thick]  (L)--(B1);\end{tikzpicture}}};

\node (A1dddd) at (0.35,-1.1)  {\resizebox{0.95cm}{!}{\begin{tikzpicture}
\draw (0,0.18) ellipse (0.1cm and 0.3cm);
\node (L) [circle,fill=cyan,draw=black,minimum size=0.1cm,inner sep=0.2mm,label={[xshift=0.47cm,yshift=-0.3cm]{\footnotesize $\{x,\!z\}$}}] at (0,0) {};
\node (B1) [circle,fill=cyan,draw=black,minimum size=0.1cm,inner sep=0.2mm,label={[xshift=0.22cm,yshift=-0.3cm]{\footnotesize $y$}}] at (0,0.35) {};
 \draw[thick]  (L)--(B1);\end{tikzpicture}}};

 \node (Cs1) at (0,0) {\begin{tikzpicture}
\node (L3) [circle,fill=none,draw=black,minimum size=0.165cm,inner sep=0.2mm] at (4,-0.25) {};
\node (L) [circle,fill=WildStrawberry,draw=black,minimum size=0.09cm,inner sep=0mm] at (4,-0.25) {};
\end{tikzpicture}
};
\node (aa) at (0,-0.2) {\resizebox{0.9cm}{!}{$\{x,y,z\}$}};

\node(j) at (-2.65,0)  {\resizebox{0.5cm}{!}{\begin{tikzpicture}
\draw (0,0.52) ellipse (0.1cm and 0.3cm);
\node (L3) [circle,fill=none,draw=black,minimum size=0.2cm,inner sep=0.2mm] at (0,0) {};
\node (L) [circle,fill=ForestGreen,draw=black,minimum size=0.1cm,inner sep=0.2mm,label={[xshift=-0.22cm,yshift=-0.25cm]{\footnotesize $x$}}] at (0,0) {};
\node (B1) [circle,fill=ForestGreen,draw=black,minimum size=0.1cm,inner sep=0.2mm,label={[xshift=-0.22cm,yshift=-0.3cm]{\footnotesize $y$}}] at (0,0.35) {};
 \node (B4) [circle,fill=ForestGreen,draw=black,minimum size=0.1cm,inner sep=0.2mm,label={[xshift=-0.22cm,yshift=-0.25cm]{\footnotesize $z$}}] at (0,0.7) {};
 \draw[thick]  (L)--(B1);
\draw [thick] (B1)--(B4);\end{tikzpicture}}};

\node(j1) at (-2.2,1.6) {\resizebox{0.5cm}{!}{\begin{tikzpicture}
\draw (0,0.173) ellipse (0.1cm and 0.3cm);
\node (L3) [circle,fill=none,draw=black,minimum size=0.2cm,inner sep=0.2mm] at (0,0.7) {};
\node (L) [circle,fill=ForestGreen,draw=black,minimum size=0.1cm,inner sep=0.2mm,label={[xshift=-0.22cm,yshift=-0.25cm]{\footnotesize $x$}}] at (0,0) {};
\node (B1) [circle,fill=ForestGreen,draw=black,minimum size=0.1cm,inner sep=0.2mm,label={[xshift=-0.22cm,yshift=-0.3cm]{\footnotesize $y$}}] at (0,0.35) {};
 \node (B4) [circle,fill=ForestGreen,draw=black,minimum size=0.1cm,inner sep=0.2mm,label={[xshift=-0.22cm,yshift=-0.25cm]{\footnotesize $z$}}] at (0,0.7) {};
 \draw[thick]  (L)--(B1);
\draw [thick] (B1)--(B4);\end{tikzpicture}}};

\node(j2) at (-2.75,1) {\resizebox{0.5cm}{!}{\begin{tikzpicture}
\node (L1) [circle,fill=none,draw=black,minimum size=0.2cm,inner sep=0.2mm] at (0,0) {};
\node (L2) [circle,fill=none,draw=black,minimum size=0.2cm,inner sep=0.2mm] at (0,0.35) {};
\node (L3) [circle,fill=none,draw=black,minimum size=0.2cm,inner sep=0.2mm] at (0,0.7) {};
\node (L) [circle,fill=ForestGreen,draw=black,minimum size=0.1cm,inner sep=0.2mm,label={[xshift=-0.22cm,yshift=-0.25cm]{\footnotesize $x$}}] at (0,0) {};
\node (B1) [circle,fill=ForestGreen,draw=black,minimum size=0.1cm,inner sep=0.2mm,label={[xshift=-0.22cm,yshift=-0.3cm]{\footnotesize $y$}}] at (0,0.35) {};
 \node (B4) [circle,fill=ForestGreen,draw=black,minimum size=0.1cm,inner sep=0.2mm,label={[xshift=-0.22cm,yshift=-0.25cm]{\footnotesize $z$}}] at (0,0.7) {};
 \draw[thick]  (L)--(B1);
\draw [thick] (B1)--(B4);\end{tikzpicture}}};

\node(j3) at (-1.55,1.95) {\resizebox{0.95cm}{!}{\begin{tikzpicture}
\node (L1) [circle,fill=none,draw=black,minimum size=0.2cm,inner sep=0.2mm] at (0,0) {};
\node (L2) [circle,fill=none,draw=black,minimum size=0.2cm,inner sep=0.2mm] at (0,0.35) {};
\node (L) [circle,fill=cyan,draw=black,minimum size=0.1cm,inner sep=0.2mm,label={[xshift=-0.22cm,yshift=-0.23cm]{\footnotesize $z$}}] at (0,0.35) {};
\node (B1) [circle,fill=cyan,draw=black,minimum size=0.1cm,inner sep=0.2mm,label={[xshift=-0.47cm,yshift=-0.3cm]{\footnotesize $\{x,\!y\}$}}] at (0,0) {};
 \draw[thick]  (L)--(B1);\end{tikzpicture}}};

\node(j4) at (-2.5,-0.8) {\resizebox{0.95cm}{!}{\begin{tikzpicture}
\node (L1) [circle,fill=none,draw=black,minimum size=0.2cm,inner sep=0.2mm] at (0,0) {};
\node (L2) [circle,fill=none,draw=black,minimum size=0.2cm,inner sep=0.2mm] at (0,0.35) {};
\node (L) [circle,fill=cyan,draw=black,minimum size=0.1cm,inner sep=0.2mm,label={[xshift=-0.22cm,yshift=-0.23cm]{\footnotesize $x$}}] at (0,0) {};
\node (B1) [circle,fill=cyan,draw=black,minimum size=0.1cm,inner sep=0.2mm,label={[xshift=-0.47cm,yshift=-0.3cm]{\footnotesize $\{y,\!z\}$}}] at (0,0.35) {};
 \draw[thick]  (L)--(B1);\end{tikzpicture}}};
\end{tikzpicture}}
\end{center}

\smallskip
 
By calculating the derivative of the circled construction $C^{\circ}=$\raisebox{-1.3em}{{\resizebox{0.6cm}{!}{\begin{tikzpicture}
\draw (0,0.35) ellipse (0.1cm and 0.45cm);
\node (L) [circle,fill=ForestGreen,draw=black,minimum size=0.1cm,inner sep=0.2mm,label={[xshift=-0.22cm,yshift=-0.25cm]{\footnotesize $x$}}] at (0,0) {};
\node (B1) [circle,fill=ForestGreen,draw=black,minimum size=0.1cm,inner sep=0.2mm,label={[xshift=-0.22cm,yshift=-0.3cm]{\footnotesize $y$}}] at (0,0.35) {};
 \node (B4) [circle,fill=ForestGreen,draw=black,minimum size=0.1cm,inner sep=0.2mm,label={[xshift=-0.22cm,yshift=-0.25cm]{\footnotesize $z$}}] at (0,0.7) {};
 \draw[thick]  (L)--(B1)--(B4);
\end{tikzpicture}}}}\,, we get the boundary of the  square:

\begin{center}
\raisebox{1.25em}{$(d^{\circ}_1-d^{\circ}_0)\Bigg($}\!\!{\raisebox{-0.25em}{\resizebox{0.65cm}{!}{\begin{tikzpicture}
\draw (0,0.35) ellipse (0.1cm and 0.45cm);
\node (L) [circle,fill=ForestGreen,draw=black,minimum size=0.1cm,inner sep=0.2mm,label={[xshift=-0.22cm,yshift=-0.25cm]{\footnotesize $x$}}] at (0,0) {};
\node (B1) [circle,fill=ForestGreen,draw=black,minimum size=0.1cm,inner sep=0.2mm,label={[xshift=-0.22cm,yshift=-0.3cm]{\footnotesize $y$}}] at (0,0.35) {};
 \node (B4) [circle,fill=ForestGreen,draw=black,minimum size=0.1cm,inner sep=0.2mm,label={[xshift=-0.22cm,yshift=-0.25cm]{\footnotesize $z$}}] at (0,0.7) {};
 \draw[thick]  (L)--(B1)--(B4);\end{tikzpicture}}}}\,\raisebox{1.2em}{$\Bigg)=$}\,\,\,\raisebox{1.1em}{$- $}\! \raisebox{0cm}{\resizebox{0.65cm}{!}{{\begin{tikzpicture}
\draw (0,0.173) ellipse (0.1cm and 0.3cm);
\node (L3) [circle,fill=none,draw=black,minimum size=0.2cm,inner sep=0.2mm] at (0,0.7) {};
\node (L) [circle,fill=ForestGreen,draw=black,minimum size=0.1cm,inner sep=0.2mm,label={[xshift=-0.22cm,yshift=-0.25cm]{\footnotesize $x$}}] at (0,0) {};
\node (B1) [circle,fill=ForestGreen,draw=black,minimum size=0.1cm,inner sep=0.2mm,label={[xshift=-0.22cm,yshift=-0.3cm]{\footnotesize $y$}}] at (0,0.35) {};
 \node (B4) [circle,fill=ForestGreen,draw=black,minimum size=0.1cm,inner sep=0.2mm,label={[xshift=-0.22cm,yshift=-0.25cm]{\footnotesize $z$}}] at (0,0.7) {};
 \draw[thick]  (L)--(B1);
\draw [thick] (B1)--(B4);\end{tikzpicture}}}}\,\, \raisebox{1.1em}{$+$} \! \raisebox{0cm}{\resizebox{0.65cm}{!}{{\begin{tikzpicture}
\draw (0,0.52) ellipse (0.1cm and 0.3cm);
\node (L3) [circle,fill=none,draw=black,minimum size=0.2cm,inner sep=0.2mm] at (0,0) {};
\node (L) [circle,fill=ForestGreen,draw=black,minimum size=0.1cm,inner sep=0.2mm,label={[xshift=-0.22cm,yshift=-0.25cm]{\footnotesize $x$}}] at (0,0) {};
\node (B1) [circle,fill=ForestGreen,draw=black,minimum size=0.1cm,inner sep=0.2mm,label={[xshift=-0.22cm,yshift=-0.3cm]{\footnotesize $y$}}] at (0,0.35) {};
 \node (B4) [circle,fill=ForestGreen,draw=black,minimum size=0.1cm,inner sep=0.2mm,label={[xshift=-0.22cm,yshift=-0.25cm]{\footnotesize $z$}}] at (0,0.7) {};
 \draw[thick]  (L)--(B1);
\draw [thick] (B1)--(B4);\end{tikzpicture}}}}\,\,\,\raisebox{1.1em}{$+$}\raisebox{0.25em}{\resizebox{1.3cm}{!}{\begin{tikzpicture}
\draw (0,0.18) ellipse (0.1cm and 0.3cm);
\node (L) [circle,fill=cyan,draw=black,minimum size=0.1cm,inner sep=0.2mm,label={[xshift=-0.47cm,yshift=-0.3cm]{\footnotesize $\{x,\!y\}$}}] at (0,0) {};
\node (B1) [circle,fill=cyan,draw=black,minimum size=0.1cm,inner sep=0.2mm,label={[xshift=-0.22cm,yshift=-0.225cm]{\footnotesize $z$}}] at (0,0.35) {};
 \draw[thick]  (L)--(B1);\end{tikzpicture}}}\,\,\,\raisebox{1.1em}{$+$}\raisebox{0.4em}{{\resizebox{1.3cm}{!}{\begin{tikzpicture}
\draw (0,0.18) ellipse (0.1cm and 0.3cm);
\node (L) [circle,fill=cyan,draw=black,minimum size=0.1cm,inner sep=0.2mm,label={[xshift=-0.22cm,yshift=-0.225cm]{\footnotesize $x$}}] at (0,0) {};
\node (B1) [circle,fill=cyan,draw=black,minimum size=0.1cm,inner sep=0.2mm,label={[xshift=-0.47cm,yshift=-0.3cm]{\footnotesize $\{y,\!z\}$}}] at (0,0.35) {};
 \draw[thick]  (L)--(B1);\end{tikzpicture}}}}\, .
\end{center} 
Here is how we calculated some of the signs   above. The first  summand on the right-hand side is obtained by applying $d_1^{\circ}$ on $C^{\circ}$, by turning the circled edge $y\{z\}$ into a connecting edge.  The sign rule says that this summand will be multiplied by   $(-1)^{\delta+1+0}$, where $(-1)^{\delta}$ is the sign that the ordinary construct \raisebox{-0.95em}{\resizebox{1.2cm}{!}{\begin{tikzpicture}
\node (L) [circle,fill=cyan,draw=black,minimum size=0.1cm,inner sep=0.2mm,label={[xshift=-0.47cm,yshift=-0.3cm]{\footnotesize $\{x,\!y\}$}}] at (0,0) {};
\node (B1) [circle,fill=cyan,draw=black,minimum size=0.1cm,inner sep=0.2mm,label={[xshift=-0.22cm,yshift=-0.225cm]{\footnotesize $z$}}] at (0,0.35) {};
 \draw[thick]  (L)--(B1);\end{tikzpicture}}} \, gets as a summand of $d_{{\EuScript O}_{\infty}}({\cal T},\sigma,\!\raisebox{-0.55em}{\resizebox{1.5cm}{!}{\begin{tikzpicture}
\node (B1) [circle,fill=cyan,draw=black,minimum size=0.1cm,inner sep=0.2mm,label={[xshift=-0.575cm,yshift=-0.31cm]{\footnotesize $\{x,\!y,\!z\}$}}] at (0,0.35) {};
\end{tikzpicture}}}\,)$. Since  $(-1)^{\delta}=1$,  the final sign is $-$. The third summand on the right-hand side is obtained by applying $d^{\circ}_0$ on $C^{\circ}$, by collapsing the circled edge $x\{y\}$. By the sign rule, it will be multiplied by the sign that $C$ gets as a summand of $d_{{\EuScript O}_{\infty}}({\cal T},\sigma,\!\raisebox{-0.95em}{\resizebox{1.2cm}{!}{\begin{tikzpicture}
\node (L) [circle,fill=cyan,draw=black,minimum size=0.1cm,inner sep=0.2mm,label={[xshift=-0.47cm,yshift=-0.3cm]{\footnotesize $\{x,\!y\}$}}] at (0,0) {};
\node (B1) [circle,fill=cyan,draw=black,minimum size=0.1cm,inner sep=0.2mm,label={[xshift=-0.22cm,yshift=-0.225cm]{\footnotesize $z$}}] at (0,0.35) {};
 \draw[thick]  (L)--(B1);\end{tikzpicture}}}\,)$, which is $-$, which, together with the $-$ in front of $d_0^{\circ}$, results in  $+$. The reader may readily check that applying the differential $d^{\circ}$ on the sum on the right-hand side of the above equality results in $0$.\demo
\end{example}

Since  the basis of  ${\EuScript O}^{\circ}_{\infty}(n_1,\dots,n_k;n)$ is given by the faces of cubical subdivisions of arbitrary   operadic polytopes, and not just the associahedra, we give below an  example of the cubical subdivision of one of the hexagonal faces of the hemiassociahedron (see \ref{sec.hemiassociahedron}).

\begin{example}
The cubical subdivision of the shaded hexagonal face 
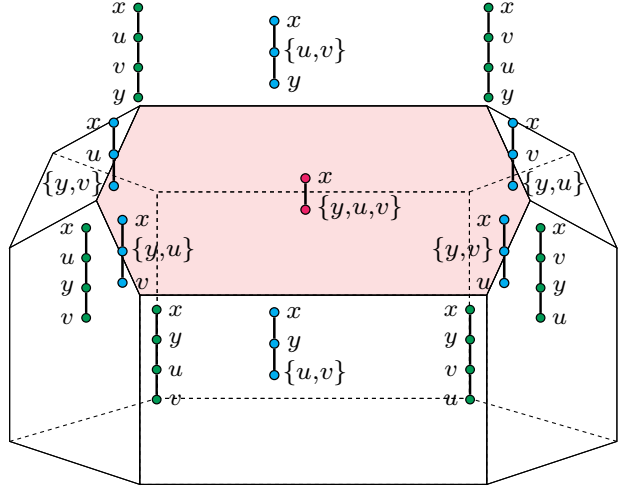
\begin{center}
 \resizebox{8cm}{!}{\begin{tikzpicture}[thick,scale=23]
\coordinate (A1) at (-0.2,0);
\coordinate (A2) at (0.2,0); 
\coordinate (D1) at (-0.35,0.05);
\coordinate (D2) at (0.35,0.05);
\coordinate (D3) at (-0.18,0.1);
\coordinate (D4)  at (0.18,0.1);
\coordinate (E3) at (-0.18,0.34);
\coordinate (E4) at (0.18,0.34);
\coordinate (A3) at (-0.2,0.22);
\coordinate (A4) at (0.2,0.22);
 \coordinate (C1) at (-0.25,0.33);
\coordinate (C2) at (0.25,0.33);
 \coordinate (F1) at (-0.35,0.275);
\coordinate (F2) at (0.35,0.275);
 \coordinate (G1) at (-0.3,0.385);
\coordinate (G2) at (0.3,0.385);
 \coordinate (B3) at (-0.2,0.44);
\coordinate (B4) at (0.2,0.44);
 \fill[fill=WildStrawberry!20,opacity=0.75] (B3)--(B4)--(C2)--(A4)--(A3)--(C1)--(B3);
\draw (A1) -- (A2) -- (A4) -- (A3) -- cycle;
\draw (A3) -- (A4) -- (C2) -- (B4) -- (B3) -- (C1)-- cycle;
\draw[dashed] (D3) -- (D4) -- (E4) -- (E3) -- cycle;
\draw[dashed] (A1) -- (A2) -- (A4) -- (A3) -- (A1);
\draw (A3) -- (A4) -- (C2) -- (B4) -- (B3) -- (C1)-- (A3);
\draw (G1) -- (B3) -- (B4) -- (G2);
\draw (G2) -- (F2);
\draw (G1) -- (F1);
\draw (A3)--(A1) -- (A2)--(A4);
\draw[dashed]   (G1) -- (E3); 
\draw[dashed] (E4) -- (G2);
 \draw  (A1) -- (D1) -- (F1) -- (C1) -- (A3) --  cycle;
 \draw  (A2) -- (D2) -- (F2) -- (C2) -- (A4) --  cycle;
 \draw  (F1) -- (G1) -- (B3) -- (C1)  --  cycle;
\draw (F2) -- (G2) -- (B4) -- (C2)  --  cycle;
\draw[dashed] (D3) -- (D4) -- (E4) -- (E3) -- cycle;
\draw[dashed] (D1)   -- (D3);
\draw[dashed] (D2)   -- (D4);
\draw (F1)--(D1)--(A1)--(A2)--(D2)--(F2);
\node at    (0.22,0.5)  {\resizebox{1.1cm}{!}{\begin{tikzpicture}
\node (b) [circle,fill=ForestGreen,draw=black,minimum size=0.1cm,inner sep=0.2mm,label={[xshift=0.22cm,yshift=-0.27cm]{\footnotesize $y$}}] at (0,0) {};
\node (c) [circle,fill=ForestGreen,draw=black,minimum size=0.1cm,inner sep=0.2mm,label={[xshift=0.22cm,yshift=-0.25cm]{\footnotesize $u$}}] at (0,0.35) {};
 \node (d) [circle,fill=ForestGreen,draw=black,minimum size=0.1cm,inner sep=0.2mm,label={[xshift=0.22cm,yshift=-0.25cm]{\footnotesize $v$}}] at (0,0.7) {};
  \node (a) [circle,fill=ForestGreen,draw=black,minimum size=0.1cm,inner sep=0.2mm,label={[xshift=0.22cm,yshift=-0.25cm]{\footnotesize $x$}}] at (0,1.05) {};
 \draw[thick]  (b)--(c)--(d)--(a);\end{tikzpicture}}};
\node at   (-0.22,0.5)  {\resizebox{1.1cm}{!}{\begin{tikzpicture}
\node (b) [circle,fill=ForestGreen,draw=black,minimum size=0.1cm,inner sep=0.2mm,label={[xshift=-0.22cm,yshift=-0.27cm]{\footnotesize $y$}}] at (0,0) {};
\node (c) [circle,fill=ForestGreen,draw=black,minimum size=0.1cm,inner sep=0.2mm,label={[xshift=-0.22cm,yshift=-0.25cm]{\footnotesize $v$}}] at (0,0.35) {};
 \node (d) [circle,fill=ForestGreen,draw=black,minimum size=0.1cm,inner sep=0.2mm,label={[xshift=-0.22cm,yshift=-0.25cm]{\footnotesize $u$}}] at (0,0.7) {};
  \node (a) [circle,fill=ForestGreen,draw=black,minimum size=0.1cm,inner sep=0.2mm,label={[xshift=-0.22cm,yshift=-0.25cm]{\footnotesize $x$}}] at (0,1.05) {};
 \draw[thick]  (b)--(c)--(d)--(a);\end{tikzpicture}}};
\node at  (0.28,0.245) {\resizebox{1.1cm}{!}{\begin{tikzpicture}
\node (b) [circle,fill=ForestGreen,draw=black,minimum size=0.1cm,inner sep=0.2mm,label={[xshift=0.22cm,yshift=-0.25cm]{\footnotesize $u$}}] at (0,0) {};
\node (c) [circle,fill=ForestGreen,draw=black,minimum size=0.1cm,inner sep=0.2mm,label={[xshift=0.22cm,yshift=-0.27cm]{\footnotesize $y$}}] at (0,0.35) {};
 \node (d) [circle,fill=ForestGreen,draw=black,minimum size=0.1cm,inner sep=0.2mm,label={[xshift=0.22cm,yshift=-0.25cm]{\footnotesize $v$}}] at (0,0.7) {};
  \node (a) [circle,fill=ForestGreen,draw=black,minimum size=0.1cm,inner sep=0.2mm,label={[xshift=0.22cm,yshift=-0.25cm]{\footnotesize $x$}}] at (0,1.05) {};
 \draw[thick]  (b)--(c)--(d)--(a);\end{tikzpicture}}};
\node at   (-0.28,0.245) {\resizebox{1.1cm}{!}{\begin{tikzpicture}
\node (b) [circle,fill=ForestGreen,draw=black,minimum size=0.1cm,inner sep=0.2mm,label={[xshift=-0.22cm,yshift=-0.25cm]{\footnotesize $v$}}] at (0,0) {};
\node (c) [circle,fill=ForestGreen,draw=black,minimum size=0.1cm,inner sep=0.2mm,label={[xshift=-0.22cm,yshift=-0.27cm]{\footnotesize $y$}}] at (0,0.35) {};
 \node (d) [circle,fill=ForestGreen,draw=black,minimum size=0.1cm,inner sep=0.2mm,label={[xshift=-0.22cm,yshift=-0.25cm]{\footnotesize $u$}}] at (0,0.7) {};
  \node (a) [circle,fill=ForestGreen,draw=black,minimum size=0.1cm,inner sep=0.2mm,label={[xshift=-0.22cm,yshift=-0.25cm]{\footnotesize $x$}}] at (0,1.05) {};
 \draw[thick]  (b)--(c)--(d)--(a);\end{tikzpicture}}};
\node at  (0.1625,0.15)  {\resizebox{1.1cm}{!}{\begin{tikzpicture}
\node (b) [circle,fill=ForestGreen,draw=black,minimum size=0.1cm,inner sep=0.2mm,label={[xshift=-0.22cm,yshift=-0.25cm]{\footnotesize $u$}}] at (0,0) {};
\node (c) [circle,fill=ForestGreen,draw=black,minimum size=0.1cm,inner sep=0.2mm,label={[xshift=-0.22cm,yshift=-0.25cm]{\footnotesize $v$}}] at (0,0.35) {};
 \node (d) [circle,fill=ForestGreen,draw=black,minimum size=0.1cm,inner sep=0.2mm,label={[xshift=-0.22cm,yshift=-0.27cm]{\footnotesize $y$}}] at (0,0.7) {};
  \node (a) [circle,fill=ForestGreen,draw=black,minimum size=0.1cm,inner sep=0.2mm,label={[xshift=-0.22cm,yshift=-0.25cm]{\footnotesize $x$}}] at (0,1.05) {};
 \draw[thick]  (b)--(c)--(d)--(a);\end{tikzpicture}}};
\node at  (-0.1625,0.15)  {\resizebox{1.1cm}{!}{\begin{tikzpicture}
\node (b) [circle,fill=ForestGreen,draw=black,minimum size=0.1cm,inner sep=0.2mm,label={[xshift=0.22cm,yshift=-0.25cm]{\footnotesize $v$}}] at (0,0) {};
\node (c) [circle,fill=ForestGreen,draw=black,minimum size=0.1cm,inner sep=0.2mm,label={[xshift=0.22cm,yshift=-0.25cm]{\footnotesize $u$}}] at (0,0.35) {};
 \node (d) [circle,fill=ForestGreen,draw=black,minimum size=0.1cm,inner sep=0.2mm,label={[xshift=0.22cm,yshift=-0.27cm]{\footnotesize $y$}}] at (0,0.7) {};
  \node (a) [circle,fill=ForestGreen,draw=black,minimum size=0.1cm,inner sep=0.2mm,label={[xshift=0.22cm,yshift=-0.25cm]{\footnotesize $x$}}] at (0,1.05) {};
 \draw[thick]  (b)--(c)--(d)--(a);\end{tikzpicture}}};
 \node(j3) at (0.05,0.335) {\resizebox{3cm}{!}{\begin{tikzpicture}
\node (L) [circle,fill=WildStrawberry,draw=black,minimum size=0.1cm,inner sep=0.2mm,label={[xshift=0.22cm,yshift=-0.23cm]{\footnotesize $x$}}] at (0,0.35) {};
\node (B1) [circle,fill=WildStrawberry,draw=black,minimum size=0.1cm,inner sep=0.2mm,label={[xshift=0.6cm,yshift=-0.3cm]{\footnotesize $\{y,\!u,\!v\}$}}] at (0,0) {};
 \draw[thick]  (L)--(B1);\end{tikzpicture}}};
  \node(j1) at (-0.275,0.38) {\resizebox{2.35cm}{!}{\begin{tikzpicture}
\node (L) [circle,fill=cyan,draw=black,minimum size=0.1cm,inner sep=0.2mm,label={[xshift=-0.45cm,yshift=-0.3cm]{\footnotesize $\{y,\!v\}$}}] at (0,0) {};
\node (B1) [circle,fill=cyan,draw=black,minimum size=0.1cm,inner sep=0.2mm,label={[xshift=-0.22cm,yshift=-0.25cm]{\footnotesize $u$}}] at (0,0.35) {};
 \node (B4) [circle,fill=cyan,draw=black,minimum size=0.1cm,inner sep=0.2mm,label={[xshift=-0.22cm,yshift=-0.25cm]{\footnotesize $x$}}] at (0,0.7) {};
 \draw[thick]  (L)--(B1);
\draw [thick] (B1)--(B4);\end{tikzpicture}}};
  \node(j1) at (0.275,0.38) {\resizebox{2.35cm}{!}{\begin{tikzpicture}
\node (L) [circle,fill=cyan,draw=black,minimum size=0.1cm,inner sep=0.2mm,label={[xshift=0.45cm,yshift=-0.3cm]{\footnotesize $\{y,\!u\}$}}] at (0,0) {};
\node (B1) [circle,fill=cyan,draw=black,minimum size=0.1cm,inner sep=0.2mm,label={[xshift=0.22cm,yshift=-0.25cm]{\footnotesize $v$}}] at (0,0.35) {};
 \node (B4) [circle,fill=cyan,draw=black,minimum size=0.1cm,inner sep=0.2mm,label={[xshift=0.22cm,yshift=-0.25cm]{\footnotesize $x$}}] at (0,0.7) {};
 \draw[thick]  (L)--(B1);
\draw [thick] (B1)--(B4);\end{tikzpicture}}};
 \node(j1) at (-0.175,0.27) {\resizebox{2.35cm}{!}{\begin{tikzpicture}
\node (L) [circle,fill=cyan,draw=black,minimum size=0.1cm,inner sep=0.2mm,label={[xshift=0.45cm,yshift=-0.3cm]{\footnotesize $\{y,\!u\}$}}] at (0,0.35) {};
\node (B1) [circle,fill=cyan,draw=black,minimum size=0.1cm,inner sep=0.2mm,label={[xshift=0.22cm,yshift=-0.25cm]{\footnotesize $v$}}] at (0,0) {};
 \node (B4) [circle,fill=cyan,draw=black,minimum size=0.1cm,inner sep=0.2mm,label={[xshift=0.22cm,yshift=-0.25cm]{\footnotesize $x$}}] at (0,0.7) {};
 \draw[thick] (B4)--(L)--(B1);
 \end{tikzpicture}}};
  \node(j1) at (0.175,0.27) {\resizebox{2.35cm}{!}{\begin{tikzpicture}
\node (L) [circle,fill=cyan,draw=black,minimum size=0.1cm,inner sep=0.2mm,label={[xshift=-0.45cm,yshift=-0.3cm]{\footnotesize $\{y,\!v\}$}}] at (0,0.35) {};
\node (B1) [circle,fill=cyan,draw=black,minimum size=0.1cm,inner sep=0.2mm,label={[xshift=-0.22cm,yshift=-0.25cm]{\footnotesize $u$}}] at (0,0) {};
 \node (B4) [circle,fill=cyan,draw=black,minimum size=0.1cm,inner sep=0.2mm,label={[xshift=-0.22cm,yshift=-0.25cm]{\footnotesize $x$}}] at (0,0.7) {};
 \draw[thick] (B4)--(L)--(B1);
 \end{tikzpicture}}};
  \node(j1) at (0,0.5) {\resizebox{2.35cm}{!}{\begin{tikzpicture}
\node (L) [circle,fill=cyan,draw=black,minimum size=0.1cm,inner sep=0.2mm,label={[xshift=0.45cm,yshift=-0.3cm]{\footnotesize $\{u,\!v\}$}}] at (0,0.35) {};
\node (B1) [circle,fill=cyan,draw=black,minimum size=0.1cm,inner sep=0.2mm,label={[xshift=0.22cm,yshift=-0.275cm]{\footnotesize $y$}}] at (0,0) {};
 \node (B4) [circle,fill=cyan,draw=black,minimum size=0.1cm,inner sep=0.2mm,label={[xshift=0.22cm,yshift=-0.25cm]{\footnotesize $x$}}] at (0,0.7) {};
 \draw[thick] (B4)--(L)--(B1);
 \end{tikzpicture}}};
   \node(j1) at (0,0.16) {\resizebox{2.35cm}{!}{\begin{tikzpicture}
\node (L) [circle,fill=cyan,draw=black,minimum size=0.1cm,inner sep=0.2mm,label={[xshift=0.45cm,yshift=-0.3cm]{\footnotesize $\{u,\!v\}$}}] at (0,0) {};
\node (B1) [circle,fill=cyan,draw=black,minimum size=0.1cm,inner sep=0.2mm,label={[xshift=0.22cm,yshift=-0.275cm]{\footnotesize $y$}}] at (0,0.35) {};
 \node (B4) [circle,fill=cyan,draw=black,minimum size=0.1cm,inner sep=0.2mm,label={[xshift=0.22cm,yshift=-0.25cm]{\footnotesize $x$}}] at (0,0.7) {};
 \draw[thick]  (L)--(B1);
\draw [thick] (B1)--(B4);\end{tikzpicture}}};
\end{tikzpicture}}
\end{center}
of the hemiassociahedron is given by
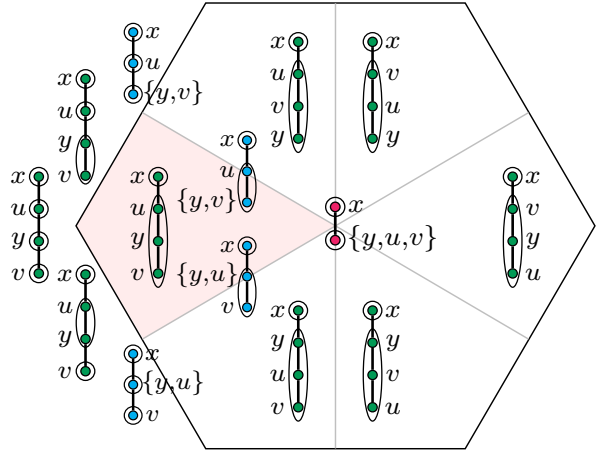
\begin{center}
\resizebox{8cm}{!}{\begin{tikzpicture}
\draw[draw=none,fill=pink!30] (-2.5,0)--(-1.85,1.1)--(0,0)--(-1.85,-1.1)--cycle;
\node[draw=black,minimum size=5cm,regular polygon,regular polygon sides=6] (a) at (0,0) {};
\draw[lightgray] (-1.86,1.1)--(0,0)--(1.86,1.1);
\draw[lightgray] (-1.86,-1.1)--(0,0)--(1.86,-1.1);
\draw[lightgray] (0,2.1625)--(0,0)--(0,-2.1625);
\node(j3) at (0.5,0) {\resizebox{1.2cm}{!}{\begin{tikzpicture}
\node (L1) [circle,fill=none,draw=black,minimum size=0.2cm,inner sep=0.2mm] at (0,0) {};
\node (L2) [circle,fill=none,draw=black,minimum size=0.2cm,inner sep=0.2mm] at (0,0.35) {};
\node (L) [circle,fill=WildStrawberry,draw=black,minimum size=0.1cm,inner sep=0.2mm,label={[xshift=0.22cm,yshift=-0.23cm]{\footnotesize $x$}}] at (0,0.35) {};
\node (B1) [circle,fill=WildStrawberry,draw=black,minimum size=0.1cm,inner sep=0.2mm,label={[xshift=0.6cm,yshift=-0.3cm]{\footnotesize $\{y,\!u,\!v\}$}}] at (0,0) {};
 \draw[thick]  (L)--(B1);\end{tikzpicture}}};
\node at    (0.5,1.3)  {\resizebox{0.475cm}{!}{\begin{tikzpicture}
\draw (0,0.35) ellipse (0.1cm and 0.5cm);
\node (L1) [circle,fill=none,draw=black,minimum size=0.2cm,inner sep=0.2mm] at (0,1.05) {};
\node (b) [circle,fill=ForestGreen,draw=black,minimum size=0.1cm,inner sep=0.2mm,label={[xshift=0.22cm,yshift=-0.27cm]{\footnotesize $y$}}] at (0,0) {};
\node (c) [circle,fill=ForestGreen,draw=black,minimum size=0.1cm,inner sep=0.2mm,label={[xshift=0.22cm,yshift=-0.25cm]{\footnotesize $u$}}] at (0,0.35) {};
 \node (d) [circle,fill=ForestGreen,draw=black,minimum size=0.1cm,inner sep=0.2mm,label={[xshift=0.22cm,yshift=-0.25cm]{\footnotesize $v$}}] at (0,0.7) {};
  \node (a) [circle,fill=ForestGreen,draw=black,minimum size=0.1cm,inner sep=0.2mm,label={[xshift=0.22cm,yshift=-0.25cm]{\footnotesize $x$}}] at (0,1.05) {};
 \draw[thick]  (b)--(c)--(d)--(a);\end{tikzpicture}}};
 \node at   (-0.5,1.3)  {\resizebox{0.475cm}{!}{\begin{tikzpicture}
\draw (0,0.35) ellipse (0.1cm and 0.5cm);
 \node (L1) [circle,fill=none,draw=black,minimum size=0.2cm,inner sep=0.2mm] at (0,1.05) {};
\node (b) [circle,fill=ForestGreen,draw=black,minimum size=0.1cm,inner sep=0.2mm,label={[xshift=-0.22cm,yshift=-0.27cm]{\footnotesize $y$}}] at (0,0) {};
\node (c) [circle,fill=ForestGreen,draw=black,minimum size=0.1cm,inner sep=0.2mm,label={[xshift=-0.22cm,yshift=-0.25cm]{\footnotesize $v$}}] at (0,0.35) {};
 \node (d) [circle,fill=ForestGreen,draw=black,minimum size=0.1cm,inner sep=0.2mm,label={[xshift=-0.22cm,yshift=-0.25cm]{\footnotesize $u$}}] at (0,0.7) {};
  \node (a) [circle,fill=ForestGreen,draw=black,minimum size=0.1cm,inner sep=0.2mm,label={[xshift=-0.22cm,yshift=-0.25cm]{\footnotesize $x$}}] at (0,1.05) {};
 \draw[thick]  (b)--(c)--(d)--(a);\end{tikzpicture}}};
 \node at   (-1.85,0) {\resizebox{0.475cm}{!}{\begin{tikzpicture}
\draw (0,0.35) ellipse (0.1cm and 0.5cm);
  \node (L1) [circle,fill=none,draw=black,minimum size=0.2cm,inner sep=0.2mm] at (0,1.05) {};
\node (b) [circle,fill=ForestGreen,draw=black,minimum size=0.1cm,inner sep=0.2mm,label={[xshift=-0.22cm,yshift=-0.25cm]{\footnotesize $v$}}] at (0,0) {};
\node (c) [circle,fill=ForestGreen,draw=black,minimum size=0.1cm,inner sep=0.2mm,label={[xshift=-0.22cm,yshift=-0.27cm]{\footnotesize $y$}}] at (0,0.35) {};
 \node (d) [circle,fill=ForestGreen,draw=black,minimum size=0.1cm,inner sep=0.2mm,label={[xshift=-0.22cm,yshift=-0.25cm]{\footnotesize $u$}}] at (0,0.7) {};
  \node (a) [circle,fill=ForestGreen,draw=black,minimum size=0.1cm,inner sep=0.2mm,label={[xshift=-0.22cm,yshift=-0.25cm]{\footnotesize $x$}}] at (0,1.05) {};
 \draw[thick]  (b)--(c)--(d)--(a);\end{tikzpicture}}};
  \node at   (-2.55,-0.95) {\resizebox{0.475cm}{!}{\begin{tikzpicture}
\draw (0,0.525) ellipse (0.1cm and 0.265cm);
  \node (L2) [circle,fill=none,draw=black,minimum size=0.2cm,inner sep=0.2mm] at (0,0) {};
  \node (L1) [circle,fill=none,draw=black,minimum size=0.2cm,inner sep=0.2mm] at (0,1.05) {};
\node (b) [circle,fill=ForestGreen,draw=black,minimum size=0.1cm,inner sep=0.2mm,label={[xshift=-0.22cm,yshift=-0.25cm]{\footnotesize $v$}}] at (0,0) {};
\node (c) [circle,fill=ForestGreen,draw=black,minimum size=0.1cm,inner sep=0.2mm,label={[xshift=-0.22cm,yshift=-0.27cm]{\footnotesize $y$}}] at (0,0.35) {};
 \node (d) [circle,fill=ForestGreen,draw=black,minimum size=0.1cm,inner sep=0.2mm,label={[xshift=-0.22cm,yshift=-0.25cm]{\footnotesize $u$}}] at (0,0.7) {};
  \node (a) [circle,fill=ForestGreen,draw=black,minimum size=0.1cm,inner sep=0.2mm,label={[xshift=-0.22cm,yshift=-0.25cm]{\footnotesize $x$}}] at (0,1.05) {};
 \draw[thick]  (b)--(c)--(d)--(a);\end{tikzpicture}}};
  \node at   (-2.55,0.95) {\resizebox{0.475cm}{!}{\begin{tikzpicture}
\draw (0,0.175) ellipse (0.1cm and 0.265cm);
  \node (L2) [circle,fill=none,draw=black,minimum size=0.2cm,inner sep=0.2mm] at (0,0.7) {};
  \node (L1) [circle,fill=none,draw=black,minimum size=0.2cm,inner sep=0.2mm] at (0,1.05) {};
\node (b) [circle,fill=ForestGreen,draw=black,minimum size=0.1cm,inner sep=0.2mm,label={[xshift=-0.22cm,yshift=-0.25cm]{\footnotesize $v$}}] at (0,0) {};
\node (c) [circle,fill=ForestGreen,draw=black,minimum size=0.1cm,inner sep=0.2mm,label={[xshift=-0.22cm,yshift=-0.27cm]{\footnotesize $y$}}] at (0,0.35) {};
 \node (d) [circle,fill=ForestGreen,draw=black,minimum size=0.1cm,inner sep=0.2mm,label={[xshift=-0.22cm,yshift=-0.25cm]{\footnotesize $u$}}] at (0,0.7) {};
  \node (a) [circle,fill=ForestGreen,draw=black,minimum size=0.1cm,inner sep=0.2mm,label={[xshift=-0.22cm,yshift=-0.25cm]{\footnotesize $x$}}] at (0,1.05) {};
 \draw[thick]  (b)--(c)--(d)--(a);\end{tikzpicture}}};
 \node at   (-3,00) {\resizebox{0.475cm}{!}{\begin{tikzpicture}
   \node (L0) [circle,fill=none,draw=black,minimum size=0.2cm,inner sep=0.2mm] at (0,0) {};
  \node (L3) [circle,fill=none,draw=black,minimum size=0.2cm,inner sep=0.2mm] at (0,0.35) {};
  \node (L2) [circle,fill=none,draw=black,minimum size=0.2cm,inner sep=0.2mm] at (0,0.7) {};
  \node (L1) [circle,fill=none,draw=black,minimum size=0.2cm,inner sep=0.2mm] at (0,1.05) {};
\node (b) [circle,fill=ForestGreen,draw=black,minimum size=0.1cm,inner sep=0.2mm,label={[xshift=-0.22cm,yshift=-0.25cm]{\footnotesize $v$}}] at (0,0) {};
\node (c) [circle,fill=ForestGreen,draw=black,minimum size=0.1cm,inner sep=0.2mm,label={[xshift=-0.22cm,yshift=-0.27cm]{\footnotesize $y$}}] at (0,0.35) {};
 \node (d) [circle,fill=ForestGreen,draw=black,minimum size=0.1cm,inner sep=0.2mm,label={[xshift=-0.22cm,yshift=-0.25cm]{\footnotesize $u$}}] at (0,0.7) {};
  \node (a) [circle,fill=ForestGreen,draw=black,minimum size=0.1cm,inner sep=0.2mm,label={[xshift=-0.22cm,yshift=-0.25cm]{\footnotesize $x$}}] at (0,1.05) {};
 \draw[thick]  (b)--(c)--(d)--(a);\end{tikzpicture}}};
 \node at   (1.85,0) {\resizebox{0.475cm}{!}{\begin{tikzpicture}
\draw (0,0.35) ellipse (0.1cm and 0.5cm);
   \node (L1) [circle,fill=none,draw=black,minimum size=0.2cm,inner sep=0.2mm] at (0,1.05) {};
\node (b) [circle,fill=ForestGreen,draw=black,minimum size=0.1cm,inner sep=0.2mm,label={[xshift=0.22cm,yshift=-0.25cm]{\footnotesize $u$}}] at (0,0) {};
\node (c) [circle,fill=ForestGreen,draw=black,minimum size=0.1cm,inner sep=0.2mm,label={[xshift=0.22cm,yshift=-0.27cm]{\footnotesize $y$}}] at (0,0.35) {};
 \node (d) [circle,fill=ForestGreen,draw=black,minimum size=0.1cm,inner sep=0.2mm,label={[xshift=0.22cm,yshift=-0.25cm]{\footnotesize $v$}}] at (0,0.7) {};
  \node (a) [circle,fill=ForestGreen,draw=black,minimum size=0.1cm,inner sep=0.2mm,label={[xshift=0.22cm,yshift=-0.25cm]{\footnotesize $x$}}] at (0,1.05) {};
 \draw[thick]  (b)--(c)--(d)--(a);\end{tikzpicture}}};
 \node at  (-0.5,-1.3)  {\resizebox{0.475cm}{!}{\begin{tikzpicture}
 \node (L1) [circle,fill=none,draw=black,minimum size=0.2cm,inner sep=0.2mm] at (0,1.05) {};
\draw (0,0.35) ellipse (0.1cm and 0.5cm);
\node (b) [circle,fill=ForestGreen,draw=black,minimum size=0.1cm,inner sep=0.2mm,label={[xshift=-0.22cm,yshift=-0.25cm]{\footnotesize $v$}}] at (0,0) {};
\node (c) [circle,fill=ForestGreen,draw=black,minimum size=0.1cm,inner sep=0.2mm,label={[xshift=-0.22cm,yshift=-0.25cm]{\footnotesize $u$}}] at (0,0.35) {};
 \node (d) [circle,fill=ForestGreen,draw=black,minimum size=0.1cm,inner sep=0.2mm,label={[xshift=-0.22cm,yshift=-0.27cm]{\footnotesize $y$}}] at (0,0.7) {};
  \node (a) [circle,fill=ForestGreen,draw=black,minimum size=0.1cm,inner sep=0.2mm,label={[xshift=-0.22cm,yshift=-0.25cm]{\footnotesize $x$}}] at (0,1.05) {};
 \draw[thick]  (b)--(c)--(d)--(a);\end{tikzpicture}}};
  \node at  (0.5,-1.3)  {\resizebox{0.475cm}{!}{\begin{tikzpicture}
  \node (L1) [circle,fill=none,draw=black,minimum size=0.2cm,inner sep=0.2mm] at (0,1.05) {};
\draw (0,0.35) ellipse (0.1cm and 0.5cm);
\node (b) [circle,fill=ForestGreen,draw=black,minimum size=0.1cm,inner sep=0.2mm,label={[xshift=0.22cm,yshift=-0.25cm]{\footnotesize $u$}}] at (0,0) {};
\node (c) [circle,fill=ForestGreen,draw=black,minimum size=0.1cm,inner sep=0.2mm,label={[xshift=0.22cm,yshift=-0.25cm]{\footnotesize $v$}}] at (0,0.35) {};
 \node (d) [circle,fill=ForestGreen,draw=black,minimum size=0.1cm,inner sep=0.2mm,label={[xshift=0.22cm,yshift=-0.27cm]{\footnotesize $y$}}] at (0,0.7) {};
  \node (a) [circle,fill=ForestGreen,draw=black,minimum size=0.1cm,inner sep=0.2mm,label={[xshift=0.22cm,yshift=-0.25cm]{\footnotesize $x$}}] at (0,1.05) {};
 \draw[thick]  (b)--(c)--(d)--(a);\end{tikzpicture}}};
 \node(j1) at (-1.2,0.5) {\resizebox{0.88cm}{!}{\begin{tikzpicture}
\draw (0,0.1675) ellipse (0.1cm and 0.275cm);
\node (L1) [circle,fill=none,draw=black,minimum size=0.2cm,inner sep=0.2mm] at (0,0.7) {};
\node (L) [circle,fill=cyan,draw=black,minimum size=0.1cm,inner sep=0.2mm,label={[xshift=-0.45cm,yshift=-0.3cm]{\footnotesize $\{y,\!v\}$}}] at (0,0) {};
\node (B1) [circle,fill=cyan,draw=black,minimum size=0.1cm,inner sep=0.2mm,label={[xshift=-0.22cm,yshift=-0.25cm]{\footnotesize $u$}}] at (0,0.35) {};
 \node (B4) [circle,fill=cyan,draw=black,minimum size=0.1cm,inner sep=0.2mm,label={[xshift=-0.22cm,yshift=-0.25cm]{\footnotesize $x$}}] at (0,0.7) {};
 \draw[thick]  (L)--(B1);
\draw [thick] (B1)--(B4);\end{tikzpicture}}};
 \node(j1) at (-1.2,-0.5) {\resizebox{0.88cm}{!}{\begin{tikzpicture}
 \node (L1) [circle,fill=none,draw=black,minimum size=0.2cm,inner sep=0.2mm] at (0,0.7) {};
\draw (0,0.1675) ellipse (0.1cm and 0.275cm);
\node (L) [circle,fill=cyan,draw=black,minimum size=0.1cm,inner sep=0.2mm,label={[xshift=-0.45cm,yshift=-0.3cm]{\footnotesize $\{y,\!u\}$}}] at (0,0.35) {};
\node (B1) [circle,fill=cyan,draw=black,minimum size=0.1cm,inner sep=0.2mm,label={[xshift=-0.22cm,yshift=-0.25cm]{\footnotesize $v$}}] at (0,0) {};
 \node (B4) [circle,fill=cyan,draw=black,minimum size=0.1cm,inner sep=0.2mm,label={[xshift=-0.22cm,yshift=-0.25cm]{\footnotesize $x$}}] at (0,0.7) {};
 \draw[thick] (B4)--(L)--(B1);
 \end{tikzpicture}}};
 \node(j1) at (-1.6,-1.55) {\resizebox{0.88cm}{!}{\begin{tikzpicture}
 \node (L1) [circle,fill=none,draw=black,minimum size=0.2cm,inner sep=0.2mm] at (0,0.7) {};
   \node (L0) [circle,fill=none,draw=black,minimum size=0.2cm,inner sep=0.2mm] at (0,0) {};
  \node (L3) [circle,fill=none,draw=black,minimum size=0.2cm,inner sep=0.2mm] at (0,0.35) {};
\node (L) [circle,fill=cyan,draw=black,minimum size=0.1cm,inner sep=0.2mm,label={[xshift=0.45cm,yshift=-0.3cm]{\footnotesize $\{y,\!u\}$}}] at (0,0.35) {};
\node (B1) [circle,fill=cyan,draw=black,minimum size=0.1cm,inner sep=0.2mm,label={[xshift=0.22cm,yshift=-0.25cm]{\footnotesize $v$}}] at (0,0) {};
 \node (B4) [circle,fill=cyan,draw=black,minimum size=0.1cm,inner sep=0.2mm,label={[xshift=0.22cm,yshift=-0.25cm]{\footnotesize $x$}}] at (0,0.7) {};
 \draw[thick] (B4)--(L)--(B1);
 \end{tikzpicture}}};
 \node(j1) at (-1.6,1.55) {\resizebox{0.88cm}{!}{\begin{tikzpicture}
 \node (L1) [circle,fill=none,draw=black,minimum size=0.2cm,inner sep=0.2mm] at (0,0.7) {};
   \node (L0) [circle,fill=none,draw=black,minimum size=0.2cm,inner sep=0.2mm] at (0,0) {};
  \node (L3) [circle,fill=none,draw=black,minimum size=0.2cm,inner sep=0.2mm] at (0,0.35) {};
\node (L) [circle,fill=cyan,draw=black,minimum size=0.1cm,inner sep=0.2mm,label={[xshift=0.45cm,yshift=-0.3cm]{\footnotesize $\{y,\!v\}$}}] at (0,0) {};
\node (B1) [circle,fill=cyan,draw=black,minimum size=0.1cm,inner sep=0.2mm,label={[xshift=0.22cm,yshift=-0.25cm]{\footnotesize $u$}}] at (0,0.35) {};
 \node (B4) [circle,fill=cyan,draw=black,minimum size=0.1cm,inner sep=0.2mm,label={[xshift=0.22cm,yshift=-0.25cm]{\footnotesize $x$}}] at (0,0.7) {};
 \draw[thick]  (L)--(B1);
\draw [thick] (B1)--(B4);\end{tikzpicture}}};
\end{tikzpicture}}
\end{center}
where, once again, we fully labeled only the shaded square. \demo
\end{example}

\begin{thm}
The coloured operad ${\EuScript O}^{\circ}_{\infty}$ is the $W$-construction ${W}(H,{\EuScript O})$ of the dg coloured operad ${\EuScript O}$ and the interval $H=N_{\ast}(\Delta^1)$ of normalized chains on the 1-simplex $\Delta^1$.
\end{thm}
\begin{proof}

Recall from \cite[Section 8.5.2.]{BM1} that the operad ${W}(H,{\EuScript O})$   is constructed like the free ${\mathbb N}$-coloured operad over ${\EuScript O}$, but  with the additional assignment of a ``length in $H$'' to each edge of the corresponding trees. More precisely, since $H=N_{\ast}(\Delta^1)$ is the complex concentrated in degrees $0$ and $1$  defined by: $$H_0={\textsf{Span}}_{{\Bbbk}} (\{\lambda_0\})\oplus {\textsf{Span}}_{{\Bbbk}} (\{\lambda_1\}),\quad H_1={\textsf{Span}}_{{\Bbbk}}(\{\lambda\}), \quad \mbox{ and } d_1(\lambda)=\lambda_1-\lambda_0,$$ and since the lengths of the edges of a planar rooted tree $T$ are formally defined by the  tensor product $
H(T):=\bigotimes_{e\in e(T)} H$,
the space ${W}(H,{\EuScript O})(n_1,\dots,n_k;n)$ is spanned by the equivalence classes of quadruples $({\cal T},\sigma,C,h)$, where $({\cal T},\sigma,C)\in {\EuScript O}_{\infty}(n_1,\dots,n_k;n)$ and  $h\in H(\alpha({\cal T},C))$, where $\alpha$ is the isomorphism from Theorem \ref{free}, under the equivalence relation generated by
$$({\cal T},\sigma,C[X\{Y\}/V],s_{X\{Y\}}(h))\sim ({\cal T},\sigma,C,h),$$ 
for  $X\{Y\}:{\cal T}_V$ and $s_e:H(\alpha({\cal T},C))\rightarrow H(\alpha({\cal T},C[X\{Y\}/V]))$ defined by setting the length of $X\{Y\}$ to be $\lambda_0$. 
The partial composition operation of ${W}(H,{\EuScript O})$ is defined by grafting trees and assigning the  length $\lambda_1$ to the new 
edge. Finally, the differential $\partial_W$ of ${W}(H,{\EuScript O})$  is the sum $\partial_W=\partial_{\EuScript O}+\partial_{H}$ of the {\em internal differential} $\partial_{\EuScript O}$, which, being induced from the differential of the operad ${\EuScript O}$, is trivially $0$,  and the {\em external differential} $\partial_H$, which is itself a sum of the operators $\partial_H^1$ and $\partial_H^2$, induced by the  degree zero elements of $H$, which act by assigning the corresponding  length  to the edges.

\medskip

The isomorphism between $W(H,{\EuScript O})$ and ${\EuScript O}^{\circ}_{\infty}$ is reflected by the fact that the elements of the basis of ${W}(H,{\EuScript O})(n_1,\dots,n_k;n)$ determined by an operadic tree ${\cal T}$ are in one-to-one correspondence with the circled constructs of ${\bf H}_{{\cal T}}$. 
 Indeed,   a circled construct $C^{\circ}:{\bf H}_{{\cal T}}$ determines $h\in H(\alpha({\cal T},\sigma,C))$ as follows:
if $X$ is a vertex of the underlying  construct  $C:{\bf H}_{\cal T}$,   then, in $\alpha({\cal T},\sigma,C)$,   the edges   determined by the set $X$ are ``already collapsed'' and hence their length is taken to be $\lambda_0$; the connecting edges of $C^{\circ}$ will get the length $\lambda_1$ and the circled edges will get the length $\lambda$. With this identification, it becomes obvious that the   external differential $\partial_H$ is precisely  $d^{\circ}_1-d^{\circ}_0$.  \end{proof}

\smallskip
  \subsection{Combinatorial homotopy theory for cyclic operads}\label{cyc} In \cite[Section 1.6.3]{luc}, Luk\' acs defined the ${\mathbb N}$-coloured operad ${\EuScript C}$ governing non-symmetric cyclic operads. His definition is based on {\em cyclic operadic trees}, i.e. operadic trees additionally equipped with  bijections labeling clockwise all their leaves  in a cyclic way with respect to the planar embedding, and the appropriate modification of the substitution operation. In this section, we show that, by switching from operadic to cyclic operadic trees, while preserving the faces of operadic polytopes as components of operations, one obtains the minimal model ${\EuScript C}_{\infty}$ for the operad ${\EuScript C}$.  

\medskip

\subsubsection{Cyclic operads as algebras over the coloured operad ${\EuScript C}$} Consider the equivalence classes of {\em cyclic operadic trees}, i.e. of triples $({\cal T},\sigma,{\tau})$, where  $({\cal T},\sigma)$ is an operadic tree and  $\tau:\{0,1,\dots,n\}\rightarrow l({\cal T})$ is a bijection labeling clockwise all the leaves of ${\cal T}$ in a cyclic way with respect to the planar embedding of ${\cal T}$, under the equivalence relation generated by:
\begin{quote}
$({\cal T}_1,\sigma_1,\tau_1)\sim ({\cal T}_2,\sigma_2,\tau_2)$ if there exists an isomorphism   $\varphi: {\cal U}({\cal T}_1)\rightarrow {\cal U}({\cal T}_2)$ of planar unrooted trees ${\cal U}({\cal T}_1)$ and ${\cal U}({\cal T}_2)$, obtained from ${\cal T}_1$ and ${\cal T}_2$ by forgetting the respective roots,  such that  $\sigma_1=\varphi\circ\sigma_2$ and $\tau_1=\varphi\circ\tau_2$.
\end{quote} 
 Let ${\tt{CTree}}(n_1,\dots,n_k;n)$ denote the set of equivalence classes of cyclic operadic trees $({\cal T},\sigma,{\tau})$, such that  $({\cal T},\sigma)\in {\tt{Tree}}(n_1,\dots,n_k;n)$.

\smallskip

 \begin{example}\label{classesunrooted} 
The following three cyclic operadic trees, in which the marked leaf is the $0$-th in the cyclic order,  are equivalent: 
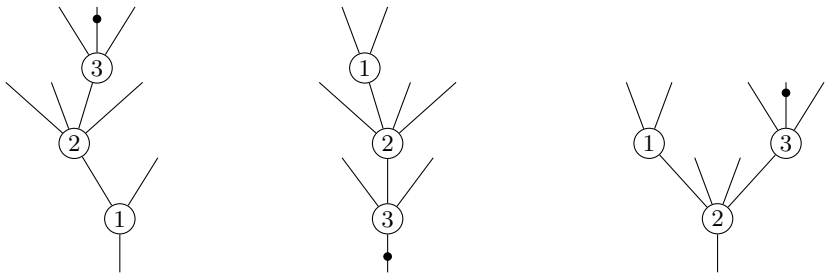
\begin{center}
\begin{tikzpicture}
\node (A) [circle,draw=black,minimum size=4mm,inner sep=0.1mm] at (0.6,0) {\small $1$}; 
\node (B) [circle,draw=black,minimum size=4mm,inner sep=0.1mm] at (0,1) {\small $2$}; 
\node (C) [circle,draw=black,minimum size=4mm,inner sep=0.1mm] at (0.3,2) {\small $3$}; 
\draw (0.6,-0.71)--(A)--(B)--(C);
\draw (A)--(1.1,0.81);
\draw (B)--(-0.3,1.81);
\draw (0.9,1.81)--(B)--(-0.9,1.81);
\draw (-0.2,2.81)--(C)--(0.8,2.81);
\draw (C)--(0.3,2.81);
\node (r) [circle,draw=black,fill=black,minimum size=1mm,inner sep=0.0cm] at (0.3,2.65) {};
\end{tikzpicture}\quad \quad \quad\quad \quad\quad
\begin{tikzpicture}
\node (C) [circle,draw=black,minimum size=4mm,inner sep=0.1mm] at (0.6,0) {\small $3$}; 
\node (B) [circle,draw=black,minimum size=4mm,inner sep=0.1mm] at (0.6,1) {\small $2$}; 
\node (A) [circle,draw=black,minimum size=4mm,inner sep=0.1mm] at (0.3,2) {\small $1$}; 
\draw (0,2.81)--(A)--(0.6,2.81);
\draw (B)--(A);
\draw (-0.3,1.81)--(B)--(0.9,1.81);
\draw (B)--(1.5,1.81);
\draw (0,0.81)--(C)--(1.2,0.81);
\draw (0.6,-0.71)--(C)--(B);
\node (r) [circle,draw=black,fill=black,minimum size=1mm,inner sep=0.0cm] at (0.6,-0.5) {};
\end{tikzpicture} \quad \quad  \quad\quad \quad\quad \begin{tikzpicture}
\node (C) [circle,draw=black,minimum size=4mm,inner sep=0.1mm] at (1.5,2) {\small $3$}; 
\node (B) [circle,draw=black,minimum size=4mm,inner sep=0.1mm] at (0.6,1) {\small $2$}; 
\node (A) [circle,draw=black,minimum size=4mm,inner sep=0.1mm] at (-0.3,2) {\small $1$}; 
\draw (C)--(B)--(A);
\draw (0.3,1.81)--(B)--(0.9,1.81);
\draw (B)--(0.6,0.29);
\draw (-0.6,2.81)--(A)--(0,2.81);
\draw (1,2.81)--(C)--(2,2.81);
\draw (C)--(1.5,2.81);
\node (r) [circle,draw=black,fill=black,minimum size=1mm,inner sep=0.0cm] at (1.5,2.67) {};
\end{tikzpicture} \vspace{-0.45cm}
\end{center}\demo
\end{example}

\noindent Observe that each  equivalence class of a cyclic operadic tree $({\cal T},\sigma,\tau)$   can be represented by a triple $(\underline{\cal T},\underline{\sigma},\underline{\tau})$ ,  for which $\underline{\tau}(0)$ is the root of $\underline{\cal T}$.  In Example \ref{classesunrooted},  this canonical representative is the tree in the middle.   

\medskip

Let, for $k\geq 1$, $n_1,\dots,n_k\geq 1$ and $n=\big(\sum^k_{i=1} n_i\big)-k+1$, ${\EuScript C}(n_1,\dots,n_k;n)$ be the vector space spanned by  the equivalence classes of cyclic operadic trees $({\cal T},\sigma,{\tau})\in{\tt{CTree}}(n_1,\dots,n_k;n)$.
Note that this definition allows  cyclic operadic trees with only one vertex (and no internal edges).

\smallskip
 
  The operad structure on the collection ${\EuScript C}$ is defined in terms of substitution of trees that takes into account the cyclic orderings of leaves:  when identifying the leaves of a vertex of the first tree with the leaves of the second tree, the root of that vertex is considered as $0$-th in the planar order. More precisely, for $({\cal T}_1,\sigma_1,{\tau}_1)\in {\EuScript C}(n_1,\dots,n_k;n)$ and $({\cal T}_2,\sigma_2,{\tau_2})\in {\EuScript C}(m_1,\dots,m_l;n_i)$,  we have 
$$({\cal T}_1,\sigma_1,\tau_1)\circ_i ({\cal T}_2,\sigma_2,\tau_2)=({\cal T}_1\bullet_i \underline{{\cal T}_2}, \sigma_1\bullet_i \underline{\sigma_2},\tau_1),$$ where the $\bullet_i$ operation is the one defined for operadic trees in \S\ref{operadictrees}, and where $(\underline{{\cal T}_2},\underline{\sigma_2},\underline{\tau_2})$ is the canonical representative of the class determined by $({\cal T}_2,\sigma_2,{\tau_2})$. The action of the symmetric group is defined by $({\cal T},\sigma,\tau)^{\kappa}=({\cal T},\sigma\circ\kappa,\tau)$.

\smallskip

\begin{lemma} 
Algebras over ${\EuScript C}$ are (non-unital, non-symmetric, reduced) dg cyclic operads. 
\end{lemma}
\begin{proof}
This has been proven in \cite[Proposition 1.6.4]{luc}. An alternative proof can be obtained by translating Luk\' acs' definition of ${\EuScript C}$ into an equivalent definition that describes ${\EuScript C}$ in terms of generators (i.e.   composite trees)  and relations.  Such a definition is obtained  by extending the set of generators $E$ and the set of relations $R$ from Definition \ref{1} by adding the unary generators encoding cyclic permutations and   the relations  governing them, as follows.

\smallskip

Indeed, relying on Lemma \ref{representation}, it can be  shown that    ${\EuScript C}\simeq{\cal T}_{\mathbb N}(\hat{E})/(\hat{R})$, where the set of generators $\hat{E}$ is given by  binary operations 
\begin{center}  \raisebox{1.8em}{$\hat{E}(n,k;n+k-1)=\Bigg\{$}
\raisebox{0.5em}{\resizebox{!}{1.3cm}{\begin{tikzpicture}  
\node(n) [rectangle,draw=none,minimum size=0mm,inner sep=0cm] at (-0.3,0.6) {\tiny $1$};
\node(k) [rectangle,draw=none,minimum size=0mm,inner sep=0cm] at (0.3,0.6) {\tiny $2$};
\node(n) [rectangle,draw=none,minimum size=0mm,inner sep=0cm] at (-0.4,0.4) {\tiny $n$};
\node(k) [rectangle,draw=none,minimum size=0mm,inner sep=0cm] at (0.4,0.4) {\tiny $k$};
\node(k) [rectangle,draw=none,minimum size=0mm,inner sep=0cm] at (0.5,-0.45) {\tiny $n\!+\!k\!-\!1$};
\node(j) [rectangle,draw=none,minimum size=0mm,inner sep=0cm] at (0,-0.5) {};
 \node(i) [circle,draw=black,thick,minimum size=0.2mm,inner sep=0.05cm] at (0,0) {$i$}; 
\draw[-,thick] (i)--(-0.3,0.45);\draw[-,thick] (i)--(0.3,0.45);
\draw[-,thick] (i)--(j);\end{tikzpicture}}} \raisebox{1.8em}{, \enspace $1\leq i\leq n$ \Bigg\} $\cup$\, $\Bigg\{$}
\raisebox{0.5em}{\resizebox{!}{1.3cm}{\begin{tikzpicture}  
\node(n) [rectangle,draw=none,minimum size=0mm,inner sep=0cm] at (-0.3,0.6) {\tiny $2$};
\node(k) [rectangle,draw=none,minimum size=0mm,inner sep=0cm] at (0.3,0.6) {\tiny $1$};
\node(n) [rectangle,draw=none,minimum size=0mm,inner sep=0cm] at (-0.4,0.4) {\tiny $k$};
\node(k) [rectangle,draw=none,minimum size=0mm,inner sep=0cm] at (0.4,0.4) {\tiny $n$};
\node(k) [rectangle,draw=none,minimum size=0mm,inner sep=0cm] at (0.5,-0.45) {\tiny $n\!+\!k\!-\!1$};
\node(j) [rectangle,draw=none,minimum size=0mm,inner sep=0cm] at (0,-0.5) {};
 \node(i) [circle,draw=black,thick,minimum size=0.2mm,inner sep=0.0325cm] at (0,0) {\small $j$}; 
\draw[-,thick] (i)--(-0.3,0.45);\draw[-,thick] (i)--(0.3,0.45);
\draw[-,thick] (i)--(j);\end{tikzpicture}}} \raisebox{1.8em}{, \enspace $1\leq j\leq k$ \Bigg\},}  
\end{center}
\smallskip
for $n,k\geq 1$, equipped with the action of the transposition $(21)$ that sends \raisebox{-1em}{\resizebox{!}{1.3cm}{\begin{tikzpicture}  
\node(n) [rectangle,draw=none,minimum size=0mm,inner sep=0cm] at (-0.3,0.6) {\tiny $1$};
\node(k) [rectangle,draw=none,minimum size=0mm,inner sep=0cm] at (0.3,0.6) {\tiny $2$};
\node(n) [rectangle,draw=none,minimum size=0mm,inner sep=0cm] at (-0.4,0.4) {\tiny $n$};
\node(k) [rectangle,draw=none,minimum size=0mm,inner sep=0cm] at (0.4,0.4) {\tiny $k$};
\node(k) [rectangle,draw=none,minimum size=0mm,inner sep=0cm] at (0.5,-0.45) {\tiny $n\!+\!k\!-\!1$};
\node(j) [rectangle,draw=none,minimum size=0mm,inner sep=0cm] at (0,-0.5) {};
 \node(i) [circle,draw=black,thick,minimum size=0.2mm,inner sep=0.05cm] at (0,0) {$i$}; 
\draw[-,thick] (i)--(-0.3,0.45);\draw[-,thick] (i)--(0.3,0.45);
\draw[-,thick] (i)--(j);\end{tikzpicture}}}\, to  \raisebox{-1em}{\resizebox{!}{1.3cm}{\begin{tikzpicture}  
\node(n) [rectangle,draw=none,minimum size=0mm,inner sep=0cm] at (-0.3,0.6) {\tiny $2$};
\node(k) [rectangle,draw=none,minimum size=0mm,inner sep=0cm] at (0.3,0.6) {\tiny $1$};
\node(n) [rectangle,draw=none,minimum size=0mm,inner sep=0cm] at (-0.4,0.4) {\tiny $n$};
\node(k) [rectangle,draw=none,minimum size=0mm,inner sep=0cm] at (0.41,0.4) {\tiny $k$};
\node(k) [rectangle,draw=none,minimum size=0mm,inner sep=0cm] at (0.5,-0.45) {\tiny $n\!+\!k\!-\!1$};
\node(j) [rectangle,draw=none,minimum size=0mm,inner sep=0cm] at (0,-0.5) {};
 \node(i) [circle,draw=black,thick,minimum size=0.2mm,inner sep=0.05cm] at (0,0) {$i$}; 
\draw[-,thick] (i)--(-0.3,0.45);\draw[-,thick] (i)--(0.3,0.45);
\draw[-,thick] (i)--(j);\end{tikzpicture}}}, and unary operations
\begin{center}  \raisebox{1.75em}{$\hat{E}(n;n)=\Bigg\{$}\enspace 
\resizebox{!}{1.3cm}{\begin{tikzpicture}  
\node(n) [rectangle,draw=none,minimum size=0mm,inner sep=0cm] at (0.17,-0.03) {\scriptsize $\tau$};
\node(k) [rectangle,draw=none,minimum size=0mm,inner sep=0cm] at (0.15,0.45) {\tiny $n$};
\node(k) [rectangle,draw=none,minimum size=0mm,inner sep=0cm] at (0.15,-0.45) {\tiny $n$};
 \node(i) [circle,draw=black,fill=black,thick,minimum size=1mm,inner sep=0.01cm] at (0,0) {}; 
\draw[-,thick] (0,0.5)--(i)--(0,-0.5);\end{tikzpicture}}\raisebox{1.8em}{\, , \enspace $\tau\in {\mathbb Z}^{\it op}_{n+1}\backslash\{{\it id}\}$ $\Bigg\}$} \end{center}
for $n\geq 1$,  realized as cyclic permutations (i.e. bijections preserving the cyclic order) of the cyclically ordered set $(0,1,\dots,n)$, and the set of relations $\hat{R}$ is given by the relation \hyperlink{A1a}{\texttt{(A1)}} from Definition \ref{1}, together with the following   equalities, concerning the action of cyclic permutations:
\begin{center}
\hypertarget{C1}{\texttt{(C1)}}   \enspace\raisebox{-4.1ex}{\resizebox{!}{1.4cm}{\begin{tikzpicture}  
\node(n) [rectangle,draw=none,minimum size=0mm,inner sep=0cm] at (0.22,-0.03) {\scriptsize $\tau_1$};
\node(m) [rectangle,draw=none,minimum size=0mm,inner sep=0cm] at (0.22,-0.37) {\scriptsize $\tau_2$};
\node(k) [rectangle,draw=none,minimum size=0mm,inner sep=0cm] at (0.15,0.3) {\tiny $n$};
 \node(i) [circle,draw=black,fill=black,thick,minimum size=1mm,inner sep=0.01cm] at (0,0) {}; 
 \node(j) [circle,draw=black,fill=black,thick,minimum size=1mm,inner sep=0.01cm] at (0,-0.35) {}; 
\draw[-,thick] (0,0.35)--(i)--(j)--(0,-0.7);\end{tikzpicture}}} \enspace $=$\enspace  \raisebox{-4.1ex}{\resizebox{!}{1.4cm}{\begin{tikzpicture}  
\node(n) [rectangle,draw=none,minimum size=0mm,inner sep=0cm] at (0.35,-0.205) {\scriptsize $\tau_2\tau_1$};
\node(k) [rectangle,draw=none,minimum size=0mm,inner sep=0cm] at (0.15,0.3) {\tiny $n$};
 \node(i) [circle,draw=black,fill=black,thick,minimum size=1mm,inner sep=0.01cm] at (0,-0.175) {}; 
\draw[-,thick] (0,0.35)--(i)--(0,-0.7);\end{tikzpicture}}} \quad\quad\quad \hypertarget{C2}{\texttt{(C2)}}   \enspace \raisebox{-4.1ex}{\resizebox{!}{1.4cm}{\begin{tikzpicture}  
\node(n) [rectangle,draw=none,minimum size=0mm,inner sep=0cm] at (0.175,-0.03) {\scriptsize $\tau$};
\node(m) [rectangle,draw=none,minimum size=0mm,inner sep=0cm] at (0.34,-0.33) {\scriptsize $\tau^{-1}$};
\node(k) [rectangle,draw=none,minimum size=0mm,inner sep=0cm] at (0.15,0.3) {\tiny $n$};
 \node(i) [circle,draw=black,fill=black,thick,minimum size=1mm,inner sep=0.01cm] at (0,0) {}; 
 \node(j) [circle,draw=black,fill=black,thick,minimum size=1mm,inner sep=0.01cm] at (0,-0.35) {}; 
\draw[-,thick] (0,0.35)--(i)--(j)--(0,-0.7);\end{tikzpicture}}} \enspace =\enspace  \raisebox{-4.1ex}{\resizebox{!}{1.4cm}{\begin{tikzpicture}  
\node(k) [rectangle,draw=none,minimum size=0mm,inner sep=0cm] at (0.15,0.3) {\tiny $n$};
\draw[-,thick] (0,0.35)--(0,-0.7);\end{tikzpicture}}}      \\[1cm]
  \hypertarget{C3}{\texttt{(C3)}} \enspace\raisebox{-4.5 ex}{\resizebox{!}{2.1cm}{\begin{tikzpicture}  
\node(1) [rectangle,draw=none,minimum size=0mm,inner sep=0cm] at (0,-0.22) {\tiny $1$};
\node(2) [rectangle,draw=none,minimum size=0mm,inner sep=0cm] at (1,-0.22) {\tiny $2$};
\node(k) [rectangle,draw=none,minimum size=0mm,inner sep=0cm] at (-0.12,-0.42) {\tiny $n$};
\node(k1) [rectangle,draw=none,minimum size=0mm,inner sep=0cm] at (1.12,-0.42) {\tiny $k$};
\node(m) [rectangle,draw=none,minimum size=0mm,inner sep=0cm] at (0.675,-1.3) {\scriptsize $\tau$};
 \node(i1) [circle,draw=black,thick,minimum size=0.2mm,inner sep=0.03cm] at (0.5,-0.8) {\scriptsize $j$}; 
 \node(1j) [circle,draw=black,fill=black,thick,minimum size=1mm,inner sep=0.01cm] at (0.5,-1.3) {}; 
\draw[-,thick] (0,-0.35)--(i1)--(0.5,-1.75);
\draw[-,thick] (i1)--(1,-0.35);\end{tikzpicture}}} \enspace $=$\enspace  \raisebox{-4.5 ex}{\resizebox{!}{2.1cm}{\begin{tikzpicture}  
\node(1) [rectangle,draw=none,minimum size=0mm,inner sep=0cm] at (0,0.25) {\tiny $2$};
\node(2) [rectangle,draw=none,minimum size=0mm,inner sep=0cm] at (1,0.25) {\tiny $1$};
\node(k) [rectangle,draw=none,minimum size=0mm,inner sep=0cm] at (-0.15,0.05) {\tiny $k$};
\node(k1) [rectangle,draw=none,minimum size=0mm,inner sep=0cm] at (1.15,0.05) {\tiny $n$};
\node(n) [rectangle,draw=none,minimum size=0mm,inner sep=0cm] at (0.22,-0.35) {\scriptsize $\tau_1$};
\node(m) [rectangle,draw=none,minimum size=0mm,inner sep=0cm] at (0.78,-0.35) {\scriptsize $\tau_2$};
 \node(j) [circle,draw=black,fill=black,thick,minimum size=1mm,inner sep=0.01cm] at (0,-0.35) {}; 
 \node(1j) [circle,draw=black,fill=black,thick,minimum size=1mm,inner sep=0.01cm] at (1,-0.35) {}; 
 \node(i1) [circle,draw=black,thick,minimum size=0.2mm,inner sep=0.05cm] at (0.5,-0.8) {\scriptsize $i$}; 
\draw[-,thick] (0,0.1)--(j)--(i1)--(0.5,-1.35);
\draw[-,thick] (i1)--(1j)--(1,0.1);\end{tikzpicture}}}  \quad\quad\quad  \hypertarget{C4}{\texttt{(C4)}} \enspace\raisebox{-4.5 ex}{\resizebox{!}{2.1cm}{\begin{tikzpicture}  
\node(1) [rectangle,draw=none,minimum size=0mm,inner sep=0cm] at (0,-0.22) {\tiny $1$};
\node(2) [rectangle,draw=none,minimum size=0mm,inner sep=0cm] at (1,-0.22) {\tiny $2$};
\node(k) [rectangle,draw=none,minimum size=0mm,inner sep=0cm] at (-0.12,-0.42) {\tiny $n$};
\node(k1) [rectangle,draw=none,minimum size=0mm,inner sep=0cm] at (1.12,-0.42) {\tiny $k$};
\node(m) [rectangle,draw=none,minimum size=0mm,inner sep=0cm] at (0.675,-1.3) {\scriptsize $\tau$};
 \node(i1) [circle,draw=black,thick,minimum size=0.2mm,inner sep=0.03cm] at (0.5,-0.8) {\scriptsize $j$}; 
 \node(1j) [circle,draw=black,fill=black,thick,minimum size=1mm,inner sep=0.01cm] at (0.5,-1.3) {}; 
\draw[-,thick] (0,-0.35)--(i1)--(0.5,-1.75);
\draw[-,thick] (i1)--(1,-0.35);\end{tikzpicture}}} \enspace $=$\enspace   \raisebox{-4.5 ex}{\resizebox{!}{2.1cm}{\begin{tikzpicture}  
\node(1) [rectangle,draw=none,minimum size=0mm,inner sep=0cm] at (0,0.25) {\tiny $1$};
\node(2) [rectangle,draw=none,minimum size=0mm,inner sep=0cm] at (1,0.25) {\tiny $2$};
\node(n) [rectangle,draw=none,minimum size=0mm,inner sep=0cm] at (0.2,-0.29) {\scriptsize $\tau_1$};
\node(k) [rectangle,draw=none,minimum size=0mm,inner sep=0cm] at (-0.15,0.05) {\tiny $n$};
\node(k1) [rectangle,draw=none,minimum size=0mm,inner sep=0cm] at (1.15,0.05) {\tiny $k$};
 \node(j) [circle,draw=black,fill=black,thick,minimum size=1mm,inner sep=0.01cm] at (0,-0.35) {}; 
 \node(i1) [circle,draw=black,thick,minimum size=0.2mm,inner sep=0.05cm] at (0.5,-0.8) {\scriptsize $i$}; 
\draw[-,thick] (0,0.1)--(j)--(i1)--(0.5,-1.35);
\draw[-,thick] (i1)--(1,0.1);\end{tikzpicture}}}  
\end{center}
where,   in  \hyperlink{C3}{\texttt{(C3)}},  it is assumed that $j\leq\tau(0)\leq k+j-1$, i.e. that $\tau(0)=k-i+j$ for some  $1\leq i\leq k$,  and $\tau_1$ is determined by $\tau_1(i)=0$ and $\tau_2$ by $\tau_2(0)=j$, and, in \hyperlink{C4}{\texttt{(C4)}}, it is assumed that $1\leq\tau(0)\leq j-1$ (resp. $k+j\leq\tau(0)\leq n+k-1$), and $\tau_1$ is determined by $\tau_1(0)=\tau(0)$ (resp. $\tau_1(0)=\tau(0)-k+1$). Note that the relation \hyperlink{A2}{\texttt{(A2)}} from Definition \ref{1} does  not appear in the above description of ${\EuScript C}$, since it can be proven by a sequence of equalities witnessed by the relations   \hyperlink{A1}{\texttt{(A1)}}, \hyperlink{C1}{\texttt{(C1)}}, \hyperlink{C2}{\texttt{(C2)}}, \hyperlink{C3}{\texttt{(C3)}} and \hyperlink{C4}{\texttt{(C4)}}.

\smallskip

A ${\EuScript C}$-algebra is then a dg ${\mathbb N}$-module $(A,d)=\{(A(n),d_{A(n)})\}_{n\geq 1}$, equipped with the obvious unary and binary operations. Observe that the unary generators  determine a right ${\mathbb Z}_{n+1}$  action  on each $A(n)$, making $\{A(n)\}_{n\geq 1}$ a cyclic ${\mathbb N}$-module, the relations \hyperlink{A1}{\texttt{(A1)}} and \hyperlink{A2}{\texttt{(A2)}} ensure the associativity of  the operadic composition, 
while the relations  \hyperlink{C3}{\texttt{(C3)}} and \hyperlink{C4}{\texttt{(C4)}} establish the compatibility of the action of cyclic permutations with the operadic composition. That this is indeed a cyclic operad follows according to \cite[Proposition 42.]{opsprops}.
 \end{proof}
\smallskip
 \subsubsection{The combinatorial ${\EuScript C}_{\infty}$ operad}  The ${\EuScript C}_{\infty}$ operad is the dg operad whose structure is induced by   the one of the ${\EuScript O}_{\infty}$ operad, by replacing operadic trees with cyclic operadic trees,  while preserving the faces of operadic polytopes as components of operations. Here, the  operadic polytope associated to a cyclic operadic tree $({\cal T},\sigma,\tau)$ remains simply the hypergraph polytope  for the edge-graph ${\bf H}_{\cal T}$ of the tree ${\cal T}$. In particular, if ${\cal T}$ has only one vertex, its associated edge-graph is the empty hypergraph ${\bf H}_{\emptyset}$ (i.e. the unique hypergraph whose set of vertices is empty), whose set of constructs is the singleton containing the empty construct. Observe that this modification is well defined, since equivalent cyclic operadic trees have the same associated edge-graph. 

\medskip

More precisely, for $k\geq 1$, $n_1,\dots,n_k\geq 1$ and $n=(\sum^k_{i=1} n_i)-k+1$, we define ${\EuScript C}_{\infty}(n_1,\dots,n_k;n)$ to be the vector space spanned by the equivalence classes of quadruples $({\cal T},\sigma,\tau,C)$, where $({\cal T},\sigma,\tau)\in {\tt{CTree}}(n_1,\dots,n_k;n)$ and $C:{\bf H}_{{\cal T}}$. For $({\cal T}_1,\sigma_1,\tau_1,C_1)\in {\EuScript C}_{\infty}(n_1,\dots,n_k;n)$ and $({\cal T}_2,\sigma_2,\tau_2,C_2)\in \linebreak{\EuScript C}_{\infty}(m_1,\dots,m_l;n_i)$, we define the partial composition operation $\circ_i$  by  $$({\cal T}_1,\sigma_1,\tau_1,C_1)\circ_i ({\cal T}_2,\sigma_2,\tau_2,C_2):=({\cal T}_1\bullet_i \underline{{\cal T}_2},\sigma_1\bullet_i \underline{\sigma_2},C_1\bullet_i C_2,\tau_1),$$ where   $(\underline{{\cal T}_2},\underline{\sigma_2},\underline{\tau_2})$ is the canonical representative of the class determined by $({\cal T}_2,\sigma_2,{\tau_2})$,  and $C_1\bullet_i C_2$ is defined exactly as in \S\ref{Oinf}.
This determines an operad on vector spaces which is easily proven to be free over the equivalence classes represented of left-recursive cyclic operadic trees (i.e. cyclic operadic trees whose underlying operadic trees are left-recursive): $${\EuScript C}_{\infty}\simeq {\cal T}_{\mathbb N}\Bigg(\, \displaystyle\bigoplus_{k\geq 1}\,\,\bigoplus_{n_1,\dots,n_k\geq 1}\,\bigoplus_{({\cal T},{\tau})\in {\tt{CTree}}(n_1,\dots,n_k;n)}\,{\Bbbk}\Bigg).$$

\noindent Indeed, the corresponding isomorphism $\alpha$ is defined by introducing the decomposition of an operation $({\cal T},\sigma,\tau,C)\in{\EuScript C}_{\infty}(n_1,\dots,n_k;n)$  by the construction analogous to the one from the proof of Theorem \ref{free}, which, in addition, takes into account the data  given by $\tau$, as follows. 
 Denote, for $1\leq i\leq k$, with $\sigma_i$ the index  given by $\sigma$ to the  $i$-th vertex in the left-recursive ordering of ${\cal T}$.  

\smallskip

\begin{itemize}
\item If $C$ is the maximal construct $e({\cal T})$, then\begin{center}\raisebox{2.1em}{$\alpha({\cal T},\sigma,\tau,e({\cal T}))=$} 
\begin{tikzpicture}
\node (r) [rectangle,draw=black,rounded corners=.1cm,inner sep=1mm] at (0,0) {\footnotesize $\enspace({\cal T},\tau)\enspace$};
\node (rl) [circle,draw=none,inner sep=0.4mm] at (-0.8,1.1) {\footnotesize ${\bf\sigma}_1$};
\node (rr) [circle,draw=none,inner sep=0.4mm] at (0.8,1.1) {\footnotesize $\sigma_k$};
\node (rl) [circle,draw=none,inner sep=0.4mm] at (-0.8,0.8) {\footnotesize $n_{\sigma_1}$};
\node (rr) [circle,draw=none,inner sep=0.4mm] at (0.8,0.8) {\footnotesize $n_{\sigma_k}$};
\node (rc) [circle,draw=none,inner sep=0.4mm] at (0,0.8) {\footnotesize $\dots$};
\node (erc) [circle,draw=none,inner sep=0.4mm] at (0,1.1) {\footnotesize $\dots$};
\node (rb) [circle,draw=none,inner sep=0.4mm] at (0.15,-0.65) {\footnotesize $n$};
\draw (0.8,0.7)--(r)--(-0.8,0.7);
\draw (r)--(0,-0.7);
\end{tikzpicture}
\end{center}

\item Suppose that $C=X\{C_1,\dots,C_p\}$, where ${\bf H}_{\cal T}\backslash X\leadsto {\bf H}_{{\cal T}_1},\dots, {\bf H}_{{\cal T}_p}$ and    $C_i:{\bf H}_{{\cal T}_i}$. Let $({\cal T}_X,\tau)$ be the left-recursive cyclic operadic tree obtained from $({\cal T},\sigma,\tau)$ by collapsing all the edges from $e({\cal T})\backslash X$. Note that this construction preserves  $\tau$.  Once again, the collapse of the edges ${\it e}({\cal T})\backslash X$ that defines ${\cal T}_X$ is, in fact, the collapse of the subtrees ${\cal T}_i$, $1\leq i\leq p$, of ${\cal T}$.  Denote with $\tau_i$ the indexing of the leaves of ${\cal T}_i$ that sends $0$ to the root of ${\cal T}_i$. We define
$$\quad\quad\quad\alpha({\cal T},\sigma,\tau,C)=\bigg((\cdots(\alpha({\cal T}_X,\tau,X)\circ_{\rho({\cal T}_{1})}  \alpha({\cal T}_{1},\tau_1,C_{1}) )\cdots)\circ_{{\rho({\cal T}_{p})}}\alpha({\cal T}_{p},\tau_p,C_{p})\bigg)^{\sigma_C},$$
where  $\sigma_C$ is the permutation determined uniquely thanks to Lemma \ref{subtreedecomposition}.
\end{itemize}

\begin{example} The free operad representation of the operation $({\cal T},\sigma,\tau,C)$, where  
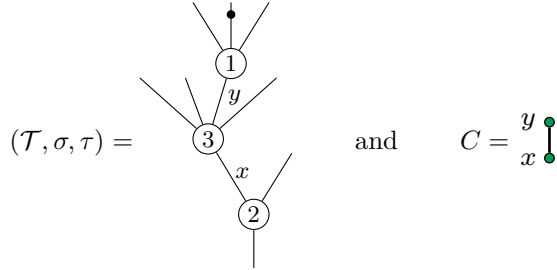
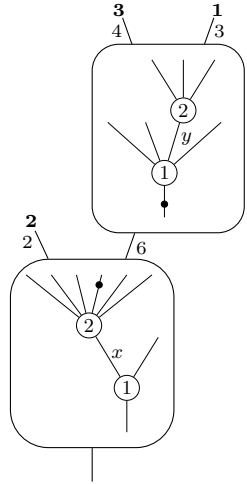
\begin{center}
\begin{tikzpicture}
\node (T) [circle,draw=none,minimum size=4mm,inner sep=0.1mm] at (-1.8,0.95) { $({\cal T},\sigma,\tau)=$}; 
\node (A) [circle,draw=black,minimum size=4mm,inner sep=0.1mm] at (0.6,0) {\small $2$}; 
\node (B) [circle,draw=black,minimum size=4mm,inner sep=0.1mm] at (0,1) {\small $3$}; 
\node (C) [circle,draw=black,minimum size=4mm,inner sep=0.1mm] at (0.3,2) {\small $1$}; 
            \node (x)[circle,draw=none,minimum size=4mm,inner sep=0.1mm] at (0.45,0.55) {\small $x$};
    \node (y) [circle,draw=none,minimum size=4mm,inner sep=0.1mm] at (0.35,1.55) {\small $y$};
\draw (0.6,-0.71)--(A)--(B)--(C);
\draw (A)--(1.1,0.81);
\draw (B)--(-0.3,1.81);
\draw (0.9,1.81)--(B)--(-0.9,1.81);
\draw (-0.2,2.81)--(C)--(0.8,2.81);
\draw (C)--(0.3,2.81);
\node (r) [circle,draw=black,fill=black,minimum size=1mm,inner sep=0.0cm] at (0.3,2.65) {};
\end{tikzpicture} \quad\quad \raisebox{4.5em}{and}  \quad\quad  
\raisebox{4.5em}{$C=$}\raisebox{3.5em}{\resizebox{0.6cm}{!}{\begin{tikzpicture}
\node (L) [circle,fill=ForestGreen,draw=black,minimum size=0.1cm,inner sep=0.2mm,label={[xshift=-0.22cm,yshift=-0.25cm]{\footnotesize $x$}}] at (0,0) {};
\node (B1) [circle,fill=ForestGreen,draw=black,minimum size=0.1cm,inner sep=0.2mm,label={[xshift=-0.22cm,yshift=-0.27cm]{\footnotesize $y$}}] at (0,0.38) {};
 \draw[thick]  (L)--(B1);\end{tikzpicture}}}
\end{center}
  is given by   
\begin{center}
\begin{tikzpicture}
\draw (0.4,3.7)--(1,2)--(1.6,3.7);
\node (1) [rectangle,draw=black,fill=white,rounded corners=.5cm,inner sep=2mm] at (0,-0.75) {\resizebox{1.75cm}{!}{\begin{tikzpicture}
\node (A) [circle,draw=black,minimum size=4mm,inner sep=0.1mm] at (0.6,0) {\small $1$}; 
\node (B) [circle,draw=black,minimum size=4mm,inner sep=0.1mm] at (0,1) {\small $2$}; 
            \node (x)[circle,draw=none,minimum size=4mm,inner sep=0.1mm] at (0.45,0.55) {\small $x$};
\draw (0.6,-0.71)--(A)--(B);
\draw (A)--(1.1,0.81);
\draw (-0.2,1.81)--(B)--(0.2,1.81);
\draw (0.6,1.81)--(B)--(-0.6,1.81);
\draw (1,1.81)--(B)--(-1,1.81);
\node (r) [circle,draw=black,fill=black,minimum size=1mm,inner sep=0.0cm] at (0.16,1.65) {};
\end{tikzpicture}}  };
\node (2) [rectangle,draw=black,fill=white,rounded corners=.35cm,inner sep=2mm] at (1,2.1) {\resizebox{1.6cm}{!}{\begin{tikzpicture}
\node (B) [circle,draw=black,minimum size=4mm,inner sep=0.1mm] at (0,1) {\small $1$}; 
\node (C) [circle,draw=black,minimum size=4mm,inner sep=0.1mm] at (0.3,2) {\small $2$}; 
    \node (y) [circle,draw=none,minimum size=4mm,inner sep=0.1mm] at (0.35,1.55) {\small $y$};
\draw (0,0.29)--(B)--(C);
\draw (B)--(-0.3,1.81);
\draw (0.9,1.81)--(B)--(-0.9,1.81);
\draw (-0.2,2.81)--(C)--(0.8,2.81);
\draw (C)--(0.3,2.81);
\node (r) [circle,draw=black,fill=black,minimum size=1mm,inner sep=0.0cm] at (0,0.5)  {};
\end{tikzpicture} }};
\node (s1) [rectangle,draw=none,inner sep=0mm] at (-0.8,1) {\scriptsize ${\bf 2}$};
\node (s3) [rectangle,draw=none,inner sep=0mm] at (1.65,3.8) {\scriptsize ${\bf 1}$};
\node (s2) [rectangle,draw=none,inner sep=0mm] at (0.35,3.8) {\scriptsize ${\bf 3}$};
\node (n1) [rectangle,draw=none,inner sep=0mm] at (-0.85,0.725) {\scriptsize $2$};
\node (n2) [rectangle,draw=none,inner sep=0mm] at (0.325,3.525) {\scriptsize $4$};
\node (n3) [rectangle,draw=none,inner sep=0mm] at (1.675,3.525) {\scriptsize $3$};
\node (n3) [rectangle,draw=none,inner sep=0mm] at (0.65,0.65) {\scriptsize $6$};
\draw (0,-2.45)--(1)--(2);
\draw (1)--(-0.75,0.875);
\end{tikzpicture}
\end{center}

 \demo
\end{example}
The space ${\EuScript C}_{\infty}(n_1,\dots,n_k;n)$ is then graded by setting $|({\cal T},\sigma,\tau,C)|:=e({\cal T})-v(C)$.
Analogously as we did for the operad ${\EuScript O}_{\infty}$, the free operad structure of ${\EuScript C}_{\infty}$ is used for the introduction of signs for the corresponding dg extension.   In particular, the differential $d_{{\EuScript C}_{\infty}}$ of ${\EuScript C}_{\infty}$   acts like the differential $d_{{\EuScript O}_{\infty}}$   of ${\EuScript O}_{\infty}$, i.e. it splits the vertices of  constructs and does  not change the underlying cyclic operadic tree; exceptionally, $d_{{\EuScript C}_{\infty}}$ maps the empty construct to $0$.  

\medskip

We can, therefore, conclude.

\begin{thm}
The operad ${\EuScript C}_{\infty}$ is the minimal model for the operad ${\EuScript C}$.
\end{thm}

\bibliographystyle{amsplain}

\end{document}